\documentclass[hidelinks]{article}
\usepackage[margin=1in]{geometry}
\usepackage{setspace}
\usepackage{amsmath}
\usepackage{amsfonts}
\usepackage{amsthm}
\usepackage{amssymb}
\usepackage{latexsym}
\usepackage{enumitem} 
\usepackage{mathrsfs}
\usepackage{mathtools}
\usepackage{xcolor}
\usepackage{inputenc}
\usepackage{bbm}
\usepackage{csquotes}
\usepackage[pdftex,pdfpagelabels,bookmarks,hyperindex,hyperfigures]{hyperref}

\hypersetup{
    colorlinks=true,
    linkcolor=blue,
    filecolor=blue,      
    urlcolor=cyan,
    citecolor=red
}

\def\R{\mathbb{R}}

\def\Tr{\mathrm{Tr}}

\DeclareMathOperator{\Var}{Var}

\DeclareMathOperator{\Rademacher}{Rademacher}

\DeclareMathOperator{\supp}{supp}
\DeclareMathOperator{\spn}{span}

\DeclareMathOperator{\median}{median}
\newcommand{\indep}{\perp \!\!\! \perp}
\DeclareMathOperator*{\argmax}{argmax}
\DeclareMathOperator*{\argmin}{argmin}
\DeclareMathOperator{\SRE}{SRE}
\DeclareMathOperator{\dKL}{d_{KL}}
\DeclareMathOperator{\dTV}{d_{TV}}

\newtheorem{theorem}{Theorem}
\newtheorem{proposition}{Proposition}
\newtheorem{lemma}{Lemma}
\newtheorem{corollary}{Corollary}
\newtheorem{definition}{Definition}
\newtheorem{assumption}{Assumption}

\theoremstyle{definition}
\newtheorem{remark}{Remark}

\title{Sparsity meets correlation in Gaussian sequence model}
\author{Subhodh Kotekal$^1$ and Chao Gao$^2$ \\ \textit{University of Chicago}}
\date{}

\begin{document}
\maketitle
\footnotetext[1]{Email: \texttt{skotekal@uchicago.edu}. The research of SK is supported in part by NSF
Grants ECCS-2216912 and DMS-1547396.}
	\footnotetext[2]{Email: \texttt{chaogao@uchicago.edu}. The research of CG is supported in part by NSF
Grants ECCS-2216912 and DMS-2310769, NSF Career Award DMS-1847590, and an Alfred Sloan fellowship.}

\abstract{We study estimation of an \(s\)-sparse signal in the \(p\)-dimensional Gaussian sequence model with equicorrelated observations and derive the minimax rate. A new phenomenon emerges from correlation, namely the rate scales with respect to \(p-2s\) and exhibits a phase transition at \(p-2s \asymp \sqrt{p}\). Correlation is shown to be a blessing provided it is sufficiently strong, and the critical correlation level exhibits a delicate dependence on the sparsity level. Due to correlation, the minimax rate is driven by two subproblems: estimation of a linear functional (the average of the signal) and estimation of the signal's \((p-1)\)-dimensional projection onto the orthogonal subspace. The high-dimensional projection is estimated via sparse regression and the linear functional is cast as a robust location estimation problem. Existing robust estimators turn out to be suboptimal, and we show a kernel mode estimator with a widening bandwidth exploits the Gaussian character of the data to achieve the optimal estimation rate.}

%%%%%%%%%%%%%%%%%%%%%%%%%%%%%%%%%%%%%%%%%
\section{Introduction}\label{section:introduction}
%%%%%%%%%%%%%%%%%%%%%%%%%%%%%%%%%%%%%%%%%

A remarkably successful suite of tools \cite{wainwright_high-dimensional_2019,vershynin_high-dimensional_2018,boucheron_concentration_2013,negahban_unified_2012} for estimating a sparse high-dimensional parameter has been most extensively developed and deployed in models for independent data. Though some recent works \cite{basu_regularized_2015,loh_high-dimensional_2012,shu_estimation_2019,melnyk_estimating_2016} employ sophisticated analyses to furnish upper bounds of similar flavors under dependent data, there is a great dearth of minimax lower bounds; complete and sharp estimation rates for even seemingly simple settings are absent in the literature. We consider a simple ``signal plus noise'' setup in the form of sparse Gaussian sequence model with correlated observations. To fix notation, consider the model 
\begin{equation}\label{model:additive}
    X_i = \theta_i + \sqrt{\gamma} W + \sqrt{1-\gamma} Z_i,
\end{equation}
where the signal \(\theta = (\theta_1,...,\theta_p) \in \R^p\) is \(s\)-sparse, \(\gamma \in [0, 1]\) denotes the correlation level\footnote{The restriction \(\gamma \in [0, 1]\) is needed to express (\ref{model:additive}), but it is not required to express (\ref{model:vector}) in which the covariance matrix is positive semi-definite if and only if \(-\frac{1}{p-1} \leq \gamma \leq 1\). However, it can be shown the minimax estimation rate when \(-\frac{1}{p-1} \leq \gamma < 0\) is the same as when \(\gamma = 0\), and so nothing is lost in restricting to \(\gamma \in [0, 1]\).}, and \(W, Z_1,...,Z_p \overset{iid}{\sim} N(0, 1)\). The shared factor \(W\) in (\ref{model:additive}) implies the observations are all equicorrelated with correlation \(\gamma\). For a first investigation, the correlation \(\gamma\) and the sparsity \(s\) is taken to be known; adaptation is addressed later. In vector form with \(X = (X_1,...,X_p)\), the model can be written as
\begin{equation}\label{model:vector}
    X \sim N\left(\theta, (1-\gamma)I_p + \gamma \mathbf{1}_p\mathbf{1}_p^\intercal\right), 
\end{equation}
where \(\mathbf{1}_p \in \R^p\) is the vector with all entries equal to one. This article is concerned with rate-optimal estimation of the sparse signal \(\theta\). 

The model (\ref{model:additive}) has a long history and is motivated by many statistical considerations. For example, (\ref{model:additive}) is a prototypical mixed effects model in which \(\theta \in \R^p\) are the fixed effects and \(W\) is the random effect. Mixed models are ubiquitous in applications, especially in the social and biomedical sciences, and so an understanding of fundamental limits is quite important, particularly in the modern environment of sparse effects. Aside from standard use-cases of mixed models, the model (\ref{model:additive}) provides a stylized one-factor model for studying correlation, which is especially important in genetics applications \cite{efron_correlation_2007,leek_capturing_2007,leek_general_2008,qiu_correlation_2005,qiu_effects_2005,sun_solving_2018}. 

Additionally, the model (\ref{model:additive}) is quite fundamental in large-scale inference. Consider the statistician faced with \(p\)-many \(z\)-scores \(X_i\), each associated with testing an individual hypothesis. In a series of highly original articles, Efron \cite{efron_microarrays_2008,efron_large-scale_2004,efron_correlation_2007} convincingly criticizes the conventional practice of assuming that, under the null hypothesis, the \(z\)-score \(X_i\) is distributed exactly according to the \(N(0, 1)\) distribution. Armed with a number of illustrative datasets, Efron illustrates the dangers of basing downstream inferences on assuming a \emph{theoretical null} when it is actually misspecified. Given most hypotheses are null in typical large-scale inference settings, Efron points out it is now actually possible to estimate the null distribution, and suggests basing inferences on an \emph{empirical null}. For theoretical study, Efron proposes the Gaussian two-groups model, 
\begin{equation}\label{model:two-groups}
    X_i \overset{ind}{\sim} 
    \begin{cases}
        N(\mu, \sigma^2) &\textit{if } i \in \mathcal{H}_0, \\
        N(\mu + \theta_i, \sigma^2) &\textit{if } i \in \mathcal{H}_0^c,
    \end{cases}
\end{equation}
where now the null parameters \(\mu, \sigma^2\) are unknown and are to be estimated. Here, \(\mathcal{H}_0\) denotes the set of nulls and \(\mathcal{H}_0^c\) denotes the set of nonnulls. The connection to (\ref{model:additive}) is apparent in that (\ref{model:additive}) imposes the prior \(\mu \sim N(0, \gamma)\). The problem of estimating \(\theta\) in (\ref{model:additive}) has the interpretation of estimating the nonnull effects in the two-groups model. Thus, fundamental understanding of model (\ref{model:additive}) is foundational to large-scale inference.

\subsection{Related work}\label{section:related_work}

Some existing work has addressed estimation in the Gaussian sequence model with various forms of dependence. Johnstone and Silverman \cite{johnstone_wavelet_1997} consider estimation in the model \(Y \sim N(\theta, V)\) with the parameter space \(\theta \in \R^p\) and with known \(V\). Those authors provide an oracle inequality for the soft-thresholding estimator \(\hat{\theta}\) with thresholds \(\lambda_i = \sqrt{2V_{ii}\log p}\). This estimator is shown to achieve, within a \(\log p\) factor, the oracle risk amongst hard-thresholding estimators \(\hat{\theta}_i = \delta_i X_i\) with \(\delta_i \in \{0, 1\}\). The oracle threshold is \(\delta_i = \mathbbm{1}_{\left\{|\theta_i| \geq \sqrt{V_{ii}}\right\}}\). Under some conditions on \(V\) (which are violated once \(1-\gamma = o(1)\) in (\ref{model:vector})), they also provide an asymptotic lower bound stating every estimator must incur estimation risk which is at least a \(\log p\) factor of the oracle risk amongst hard-thresholding estimators. Johnstone and Silverman \cite{johnstone_wavelet_1997} go on to explore a few examples of short-range and long-range dependence along with different parameter spaces. 

In \cite{johnstone_wavelet_1999} (see also \cite{wang_function_1996,hall_nonparametric_1990,opsomer_nonparametric_2001,beran_statistical_1992} for related results), Johnstone considers a nonparametric regression model with regular design \(Y_i = f(i/n) + \xi_i\) where the error \(\xi_i\) is drawn from a stationary Gaussian process. Transforming to sequence space via a wavelet transform, Johnstone shows the estimators proposed in \cite{johnstone_wavelet_1997} achieve the asymptotic minimax rate of estimation over Besov classes under some specific models of short- and long-range dependence. Central to the arguments is the decorrelating effect of the wavelet transform which makes the dependence under consideration manageable in sequence space. Estimation can also be done adaptively without knowledge of the smoothness or the dependence index. Though these results are impressive in their scope, they do not say anything about strong, non-decaying correlation as is present in the model (\ref{model:additive}).

Let us turn to a focused discussion of estimation in (\ref{model:additive}). In the case \(\gamma = 0\), it is well known the minimax rate in square error is \(s \log\left(\frac{ep}{s}\right)\). The maximum likelihood estimator \(\hat{\theta}_{\text{MLE}} := \arg\min_{||\theta||_0 \leq s} ||X - \theta||^2\) is rate-optimal but requires knowledge of the sparsity level. The penalized likelihood estimator
\begin{equation*}
    \hat{\theta}_{\text{pen}} = \argmin_{\mu \in \R^p} \left\{||X - \mu||^2 + \text{pen}(||\mu||_0)\right\}
\end{equation*}
does not require knowledge of the sparsity level and is also rate-optimal under some conditions on the penalty function. Moreover, Birg\'{e} and Massart \cite{birge_gaussian_2001} establish the oracle inequality
\begin{equation*}
    ||\hat{\theta}_{\text{pen}} - \theta||^2 \lesssim \min_{1 \leq s \leq p} \left\{ \min_{||\mu||_0 \leq s} ||\mu - \theta||^2 + \text{pen}(s)\right\},
\end{equation*}
which holds with high probability and is applicable for any \(\theta \in \R^p\) (potentially not sparse). In an asymptotic setup with \(p \to \infty\) and \(s = o(p)\), Donoho and Johnstone (Theorem 8.21 in \cite{johnstone_gaussian_nodate}) obtain the sharp constant by establishing the asymptotic minimax estimation risk \(2(1 + o(1))s \log\left(\frac{ep}{s}\right)\). On the other extreme, consider \(\gamma \in [0, 1]\) and \(s = p\). It can be shown the minimax estimation rate is of order \(\Tr\left( (1-\gamma)I_p + \gamma \mathbf{1}_p\mathbf{1}_p^\intercal \right) = p\). Clearly the maximum likelihood estimator \(\hat{\theta}_{\text{MLE}} = X\) is optimal. In fact, for a general known covariance matrix \(\Sigma\), the minimax estimation rate is of the order \(\Tr(\Sigma)\) when \(s = p\). When \(s < p\), one might hope to apply a standard sparse estimator (e.g. hard-thresholding) to the whitened data 
\begin{equation*}
    \Sigma^{-1/2} X \sim N(\Sigma^{-1/2}\theta, I_p).
\end{equation*}
The problem is \(\Sigma^{-1/2}\theta\) may not be sparse and an out-of-the-box application of a sparse estimator may be inappropriate. This issue has been recognized before by Johnstone and Silverman in \cite{johnstone_wavelet_1997}, who write on pages 342-343, 
\begin{quote}
    ``A possible alternative procedure...is to use a prewhitening transformation...This method would have the advantage that wavelet thresholding is applied to a version of the data with homoscedastic uncorrelated noise. However, the wavelet decomposition of the signal in the original domain may have sparsity properties that are lost in the prewhitening transformation...''
\end{quote}
The subtle interaction between \(\Sigma\) and sparsity needs to be carefully studied in order to characterize the minimax estimation rate, which we denote \(\varepsilon^*(p, s, \gamma)\). 
\begin{definition}
    We say \(\varepsilon^*(p, s, \gamma)\) is the minimax estimation rate if for any \(\delta \in (0, 1)\), 
    \begin{enumerate}[label=(\roman*)]
        \item there exists \(C_\delta > 0\) depending only on \(\delta\) such that 
        \begin{equation*}
            \inf_{\hat{\theta}} \sup_{||\theta||_0 \leq s} P_{\theta, \gamma}\left\{||\hat{\theta} - \theta|| \geq C_\delta \varepsilon^*(p, s, \gamma)\right\} \leq \delta,
        \end{equation*}
        \item there exists \(c_\delta > 0\) depending only on \(\delta\) such that 
        \begin{equation*}
            \inf_{\hat{\theta}} \sup_{||\theta||_0 \leq s} P_{\theta, \gamma}\left\{||\hat{\theta} - \theta|| \geq c_\delta \varepsilon^*(p, s, \gamma)\right\} \geq 1-\delta. 
        \end{equation*}
    \end{enumerate}
\end{definition}
This definition of minimax rate captures only the scaling of the problem's difficulty on \(p, s,\) and \(\gamma\). Though it is an interesting problem to pin down the tight scaling with respect to \(\delta\), this article focuses on \(\varepsilon^*(p, s, \gamma)\) since even this coarser quantity has not been sharply characterized in the literature. 

\subsection{A preview of the interaction between sparsity and correlation}\label{section:preview}
To give a preview of how correlation and sparsity can interact in an interesting way, consider the case \(\gamma = 1\). In this setting, the observations are all perfectly correlated, that is \(X = \theta + W \mathbf{1}_p\). In this case, the minimax rate is
\begin{equation}\label{rate:perfect_correlation}
    \varepsilon^*(p, s, 1)^2 \asymp 
    \begin{cases}
        0 &\textit{if } s < \frac{p}{2}, \\
        p &\textit{if } s \geq \frac{p}{2}.
    \end{cases}
\end{equation}
When \(s \geq \frac{p}{2}\), the rate \(p\) is achieved by \(\hat{\theta} = X\). To see why perfect estimation is possible when \(s < \frac{p}{2}\), consider estimating the orthogonal pieces \(\theta - \bar{\theta}\mathbf{1}_p\) and \(\bar{\theta}\mathbf{1}_p\) separately, as combining the separate estimators yields an estimator for \(\theta\). Here, the notation \(\bar{\theta} := \frac{1}{p} \sum_{i=1}^{p} \theta_i\) denotes the average. With this strategy in mind, break the data \(X\) into the two orthogonal pieces
\begin{align*}
    X - \bar{X}\mathbf{1}_p &= \theta - \bar{\theta}\mathbf{1}_p, \\
    \bar{X} &= \bar{\theta} + W \sim N(\bar{\theta}, 1). 
\end{align*}
Clearly \(X - \bar{X}\mathbf{1}_p\) perfectly estimates \(\theta - \bar{\theta}\mathbf{1}_p\). However, \(\bar{X}\) does not perfectly estimate \(\bar{\theta}\) due to the presence of \(W\). So \(\bar{X}\) cannot be used if perfect estimation of \(\theta\) is to be achieved. Counterintuitively, \(X - \bar{X}\mathbf{1}_p\) can be used to perfectly estimate \(\bar{\theta}\) despite their orthogonality. Sparsity of \(\theta\) is crucial. Consider 
\begin{equation*}
    (X - \bar{X}\mathbf{1}_p)_i = 
    \begin{cases}
        \theta_i - \bar{\theta} &\textit{if } i \in \supp(\theta), \\
        -\bar{\theta} &\textit{if } i \in \supp(\theta)^c. 
    \end{cases}
\end{equation*}
Note \(|\supp(\theta)^c| \geq p - s > \frac{p}{2}\). Therefore, \(\bar{\theta} = \text{mode}\left( \left\{ \bar{X} - X_i \right\}_{i=1}^{p} \right)\), meaning \(\bar{\theta}\) can be perfectly estimated from \(X - \bar{X}\mathbf{1}_p\). Furthermore, it is clear \(s < \frac{p}{2}\) is both sufficient \emph{and} necessary for \(\bar{\theta}\) to be identifiable from \(X - \bar{X}\mathbf{1}_p\). When \(s \geq \frac{p}{2}\), one is forced to use \(\bar{X}\) to estimate \(\bar{\theta}\) and so only the rate \(p\) can ever be achieved for estimation of \(\theta\). 

The case \(\gamma = 1\) showcases how sparsity and correlation can interact to yield new phenomena in the minimax estimation rate. Additionally, the analysis here teases the strategy for the general case \(\gamma \in [0, 1]\). Namely, we separately investigate estimation of \(\theta - \bar{\theta}\mathbf{1}_p\) and \(\bar{\theta}\mathbf{1}_p\) from the pieces \(X - \bar{X}\mathbf{1}_p\) and \(\bar{X}\).

\subsection{The ``decorrelate-then-regress" strategy fails}\label{section:decorrelate_then_regress}

Though Johnstone and Silverman \cite{johnstone_wavelet_1997} recognize the issues associated to whitening the data before applying a sparse estimator, it may be argued this problem is specific to the application of an estimator designed for a sequence model to the whitened data. The interlocutor may point out the statistician should use an estimator designed for sparse regression by treating the whitening matrix as the design. It turns out this approach also has issues. 

To illustrate, assume \(\gamma < 1\) so the covariance matrix \(\Sigma = (1-\gamma)I_p + \gamma \mathbf{1}_p\mathbf{1}_p^\intercal\) is invertible. Decorrelating the observations \(\Sigma^{-1/2} X \sim N\left(\Sigma^{-1/2} \theta, I_p\right)\) consider a natural sparse regression estimator  
\begin{equation*}
    \hat{\theta} := \argmin_{\beta \in \R^p} \left\{\frac{1}{L}\left|\left| \Sigma^{-1/2} X - \Sigma^{-1/2} \beta \right|\right|^2 + \lambda ||\beta||_1\right\}. 
\end{equation*}
Here, the scaling factor \(L > 0\) is included just to allow appropriate normalization of the ``design" matrix \(\Sigma^{-1/2}\). Note \(\Sigma^{-1} = \frac{1}{1-\gamma} \left(I_p - \frac{1}{p}\mathbf{1}_p\mathbf{1}_p^\intercal\right) + \frac{1}{1-\gamma+\gamma p} \cdot \frac{1}{p}\mathbf{1}_p\mathbf{1}_p^\intercal\). It can be checked the design matrix satisfies the usual restricted eigenvalue-type conditions (e.g. see \cite{bellec_slope_2018}) when \(s < \delta p\) for a sufficiently small universal constant \(\delta \in (0, 1)\). Therefore, picking the penalty \(\lambda\) in a rate-optimal way \cite{bellec_slope_2018},  one obtains
\begin{equation}\label{eqn:decorrelate_regress_err}
    \frac{E_{\theta, \gamma}\left(\left| \left| \left(I_p - \frac{1}{p}\mathbf{1}_p\mathbf{1}_p^\intercal\right)(\hat{\theta} - \theta) \right|\right|^2\right)}{1-\gamma} + \frac{E_{\theta, \gamma}\left(\left|\left|\left(\frac{1}{p}\mathbf{1}_p\mathbf{1}_p^\intercal\right) (\hat{\theta} - \theta)\right|\right|^2\right)}{1-\gamma+\gamma p} \lesssim s \log\left(\frac{ep}{s}\right).
\end{equation}
\noindent The first term in (\ref{eqn:decorrelate_regress_err}) corresponds to estimating \(\theta - \bar{\theta}\mathbf{1}_p\), that is to say 
\begin{equation*}
    E_{\theta, \gamma}\left( \left| \left| (\hat{\theta} - \overline{\hat{\theta}}\mathbf{1}_p) - (\theta - \bar{\theta}\mathbf{1}_p) \right|\right|^2\right) \lesssim (1-\gamma) s \log\left(\frac{ep}{s}\right).
\end{equation*}
The bound improves as the correlation gets stronger; correlation is a blessing. The second term in (\ref{eqn:decorrelate_regress_err}) corresponds to estimating \(\bar{\theta}\), that is to say 
\begin{equation*}
    E_{\theta, \gamma}\left(p \left(\overline{\hat{\theta}} - \bar{\theta}\right)^2 \right) \lesssim (1-\gamma+\gamma p) s \log\left(\frac{ep}{s}\right). 
\end{equation*}
The bound weakens as the correlation gets stronger; correlation is a curse. In fact, the bound does not even go to \(0\) as \(\gamma \to 1\) when \(s < \frac{p}{2}\), failing to match the rate established in Section \ref{section:preview}. Noting \(\theta = (\theta - \bar{\theta}\mathbf{1}_p) + \bar{\theta}\mathbf{1}_p\) and summing the two error bounds above, it follows the final estimator satisfies 
\begin{equation*}
    E_{\theta, \gamma}\left(||\hat{\theta} - \theta||^2\right) \lesssim (1-\gamma+\gamma p) s \log\left(\frac{ep}{s}\right).
\end{equation*}
The bound has lost all the gains from correlation. It appears the ``decorrelate-then-regress" strategy is suboptimal due to the problematic piece \(\bar{\theta}\).  

Pinning down the sharp minimax rate for estimating \(\theta\) requires studying the functional estimation problem (estimating \(\bar{\theta}\)) separately from the high-dimensional estimation problem (estimating \(\theta - \bar{\theta}\mathbf{1}_p\)). Functional estimation not only boasts a rich history but has also witnessed modern developments \cite{collier_optimal_2018,collier_estimating_2018,collier_minimax_2017,carpentier_estimating_2021,carpentier_adaptive_2019,bickel_estimating_1988,lepski_estimation_1999,cai_testing_2011}. A defining feature of functional estimation is the appearance of distinctive minimax rates not frequently seen when estimating multidimensional parameters. Indeed, it turns out the need to estimate \(\bar{\theta}\) is the principal reason for the new rate phenomena described in this paper.

\subsection{Main contribution}\label{section:main_contribution}
Our main contribution is a characterization of the minimax estimation rate of \(\theta\) in squared error given an observation from (\ref{model:additive}). For \(1 \leq s \leq p\) and \(\gamma \in [0, 1]\), we show the minimax rate is given, up to universal constant factors, by 
\begin{align}\label{rate:correlation}
    \begin{split}
    \varepsilon^*(p, s, \gamma)^2 &\asymp (1-\gamma)s \log\left(\frac{ep}{s \wedge (p-2s)_{+}^2}\right) \wedge p\\
    &\asymp 
    \begin{cases}
        (1-\gamma) s \log\left(\frac{ep}{s}\right) &\textit{if } 1 \leq s \leq \frac{p}{2} - \sqrt{p}, \\
        \left((1-\gamma) p \log\left(\frac{ep}{(p-2s)^2}\right) \right) \wedge p &\textit{if } \frac{p}{2} - \sqrt{p} < s < \frac{p}{2}, \\
        p &\textit{if } \frac{p}{2} \leq s \leq p,
    \end{cases}
\end{split}
\end{align}
where we adopt the convention \(\log(x/0) = \infty\) for any \(x > 0\). A glance at \(\varepsilon^*(p, s, \gamma)\) immediately reveals the blessing of correlation. In the sparsity regime \(\frac{p}{2} - s \gtrsim \sqrt{p}\), we have \(\varepsilon^*(p, s, \gamma) = o(\varepsilon^*(p, s, 0))\) if and only if \(1 - \gamma = o(1)\). When \(\frac{p}{2} - s \lesssim \sqrt{p}\) and \(s < \frac{p}{2}\), we have \(\varepsilon^*(p, s, \gamma) = o\left(\varepsilon^*(p, s, 0)\right)\) if and only if \(1-\gamma = o\left(\log^{-1}\left(\frac{ep}{(p-2s)^2}\right)\right)\). The critical correlation level bestowing a blessing exhibits only a logarithmic dependence on the dimension \(p\). The scaling with \(p-2s\) is a curious feature whose appearance is a direct consequence of needing to estimate \(\bar{\theta}\).  

The minimax rate exhibits a discontinuity at \(s = \frac{p}{2}\) when \(1-\gamma = o\left(\log^{-1}(ep)\right)\). This discontinuity is most severe at \(\gamma = 1\) as seen in (\ref{rate:perfect_correlation}). Moreover, for dense signals (i.e. \(s \geq \frac{p}{2}\)), the minimax rate in the correlated and independent settings match, \(\varepsilon^*(p, s, \gamma) = \varepsilon^*(p, s, 0)\) for all \(\gamma \in [0, 1]\). The irrelevance of correlation in the dense regime is related to nonidentifiability in a robust statistics problem (see Section \ref{section:lower_bound_dense} for further discussion). 

Figure \ref{fig:alternative_scaling} presents plots of the minimax rate against the sparsity level \(s\) for \(p = 100\) with the choice of scaling \(\gamma = 1-p^{-\kappa}\). Figure \ref{fig:alternative_scaling} illustrates the blessing of correlation (discussed further in Remark \ref{remark:estimation_vs_testing}) as well as the hard jump in the rate. Furthermore, note how the transition from the rate \((1-\gamma)p\log\left(\frac{ep}{(p-2s)^2}\right)\) to \(p\) occurs for larger \(s\) as \(\gamma\) increases. Specifically, the elbows move to the right as \(\kappa\) increases in Figure \ref{fig:alternative_scaling}.

\begin{figure}[h]
    \centering 
    \includegraphics[scale=0.5]{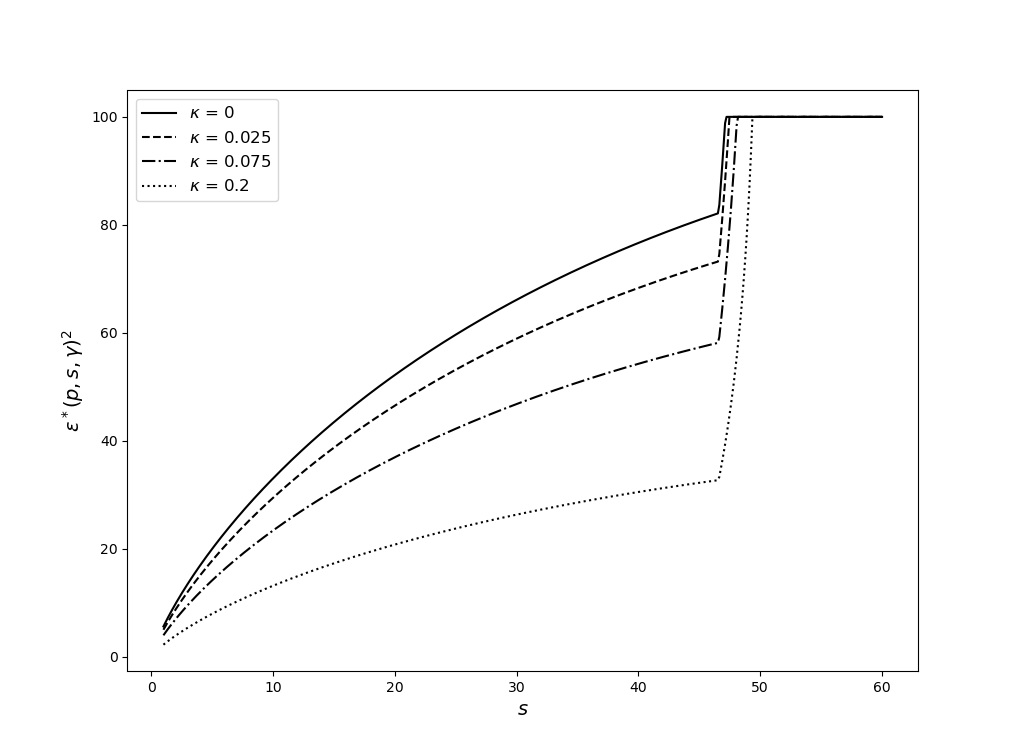}
    \caption{Plots of the rate \(\varepsilon^*(p, s, \gamma)^2\) against \(s\) with \(p = 100\) and \(\gamma = 1-p^{-\kappa}\) for various choices of \(\kappa\).}\label{fig:alternative_scaling}
\end{figure}

\begin{remark}[Testing vs estimation]\label{remark:estimation_vs_testing}
    As has been repeatedly pointed out in the minimax testing literature \cite{ingster_nonparametric_2003}, the fundamental limits of hypothesis testing and estimation are often different in high-dimensional models. The same turns out to be true for model (\ref{model:additive}). In \cite{kotekal_minimax_2023}, the authors derived the minimax separation rate for the hypothesis testing problem 
    \begin{align*}
        H_0 &: \theta = 0, \\
        H_1 &: ||\theta|| \geq \rho \text{ and } ||\theta||_0 \leq s. 
    \end{align*}
    The minimax separation rate is given by 
    \begin{equation*}
        \rho^*(p, s, \gamma)^2 \asymp 
        \begin{cases}
            (1-\gamma)s \log\left(\frac{ep}{s^2}\right) &\textit{if } 1 \leq s < \sqrt{p}, \\ 
            (1-\gamma)\sqrt{p} &\textit{if } \sqrt{p} \leq s \leq \frac{p}{2}, \\
            (1-\gamma)\sqrt{p} + \frac{(1-\gamma)p^{3/2}}{p-s} \wedge (1-\gamma + \gamma p) &\textit{if } \frac{p}{2} < s < p - \sqrt{p}, \\
            (1-\gamma)\sqrt{p} + \left((1-\gamma)p \log\left(\frac{ep}{(p-s)^2}\right)\right) \wedge (1-\gamma+\gamma p) &\textit{if } p - \sqrt{p} \leq s \leq p. 
        \end{cases}
    \end{equation*}
    The display above is exactly the same as in \cite{kotekal_minimax_2023}, but presented in a slightly different way to ease comparison to the rate (\ref{rate:correlation}). Unsurprisingly, the separation rate never exceeds the estimation rate, and in many regimes the separation rate is strictly faster. 
    
    Also, correlation exhibits different effects for testing and estimation. For example, when \(\frac{p}{2} \leq s < p\), correlation can be a blessing for testing whereas it is totally irrelevant for estimation. The necessary strength for correlation to help also differs between the two problems. For example, in the regime \(s < \frac{p}{2}\), testing requires \(1-\gamma = o(1)\) in order for correlation to be a blessing. On the other hand, estimation requires the stronger condition \(1-\gamma = o\left(\log^{-1}\left(\frac{ep}{(p-2s)^2}\right)\right)\) in the regime \(\frac{p}{2} - s \lesssim \sqrt{p}\). Moreover, \cite{kotekal_minimax_2023} showed there exist regimes in which correlation can be a curse for hypothesis testing. In contrast, correlation is always a blessing for estimation. 

    The intuition for differing effect of correlation on the two problems can be understood roughly as follows. At its core, what issue does correlation introduce? As pointed out in Section \ref{section:decorrelate_then_regress}, it makes the scalar quantity \(\bar{\theta}\) difficult to discern from the data; this affects both estimation and testing. In the estimation problem, the \(p\)-dimensional target \(\theta\) splits into a \((p-1)\)-dimensional piece \(\theta - \bar{\theta}\mathbf{1}_p\) and a \(1\)-dimensional piece \(\bar{\theta}\mathbf{1}_p\). Roughly speaking, the gains in estimating \(\theta - \bar{\theta}\mathbf{1}_p\) due to correlation can be quite large as it is high-dimensional, outweighing the increased difficulty in estimating the scalar \(\bar{\theta}\). Thus, correlation is always a blessing. In the testing problem, we can use a nonrigorous heuristic that the complexity of testing should be comparable to that of estimating the quadratic functional \(||\theta||^2\). In contrast to \(\theta\), the target \(||\theta||^2\) splits into two \(1\)-dimensional pieces, \(||\theta-\bar{\theta}\mathbf{1}_p||^2\) and \(||\bar{\theta}\mathbf{1}_p||^2\). Now, the gains from correlation in estimating \(||\theta-\bar{\theta}\mathbf{1}_p||^2\) need not be so large since it is a scalar, and thus need not outweigh the increased difficulty in estimating \(\bar{\theta}\). There is now room for correlation to be a curse.
\end{remark}

\begin{remark}[Large-scale inference]\label{remark:large-scale_inference}
    Model (\ref{model:additive}) can be interpreted as a descendant of the two-groups model (\ref{model:two-groups}) in which a Gaussian prior on the unknown \(\mu\) is imposed. If \(\mu\) is known, then (\ref{model:two-groups}) is equivalent to the sparse Gaussian sequence model after transforming the data \(\{X_i - \mu\}_{i=1}^{n}\). Thus, the oracle rate of estimation if \(\mu\) were known is 
    \begin{equation*}
        \varepsilon_{\text{oracle}}^*(p, s, \sigma)^2 \asymp \sigma^2 s \log\left(\frac{ep}{s}\right).
    \end{equation*}
    When \(\mu \in \R\) is completely unknown (no prior is imposed), then the minimax rate for estimating the signal \(\theta\) can be established in a straightforward way from our result, namely
    \begin{equation}\label{rate:LSI}
        \varepsilon^*(p, s, \sigma)^2 \asymp 
        \begin{cases}
            \sigma^2 s \log\left(\frac{ep}{s}\right) &\textit{if } 1 \leq s \leq \frac{p}{2} - \sqrt{p}, \\
            \sigma^2 p \log\left(\frac{ep}{(p-2s)^2}\right) &\textit{if } \frac{p}{2} - \sqrt{p} < s < \frac{p}{2}, \\
            \infty &\textit{if } \frac{p}{2} \leq s \leq p. 
        \end{cases}
    \end{equation}
    Remarkably, the oracle rate as if \(\mu\) were known is achieved in the sparsity regime \(\frac{p}{2} - s \gtrsim \sqrt{p}\). No price is paid in the rate for abandoning the assumption the theoretical null holds. Likewise, the price paid in the regime \(s < \frac{p}{2}\) and \(\frac{p}{2} - s \lesssim \sqrt{p}\) is mild as it is at most only logarithmic in the dimension. Section \ref{section:LSI} discusses large-scale inference further. 
\end{remark}

\subsection{Notation}
This section defines frequently used notation. For a natural number \(n\), denote \([n] := \{1,..., n\}\). For \(a, b \in \R\) the notation \(a \lesssim b\) denotes the existence of a universal constant \(c > 0\) such that \(a \leq cb\). The notation \(a \gtrsim b\) is used to denote \(b \lesssim a\). Additionally \(a \asymp b\) denotes \(a \lesssim b\) and \(a \gtrsim b\). The symbol \(:=\) is frequently used when defining a quantity or object. Furthermore, we frequently use \(a \vee b := \max(a, b)\) and \(a \wedge b := \min(a, b)\). We also use the notation \((x)_{+}\) to denote \(x \vee 0\) for \(x \in \R\). We generically use the notation \(\mathbbm{1}_A\) to denote the indicator function for an event \(A\). For a vector \(v \in \R^p\) and a subset \(S \subset [p]\), we sometimes use the notation \(v_S \in \R^{p}\) to denote the vector with coordinate \(i\) equal to \(v_i\) if \(i \in S\) and zero otherwise. In other cases, the notation \(v_S \in \R^{|S|}\) denotes the subvector of dimension \(|S|\) corresponding to the coordinates in \(S\). The context will clarify between the two different notational uses of \(v_S\). In particular, we will frequently make use of the notation \(\mathbf{1}_S := (\mathbf{1}_p)_S\) in this way. Additionally, \(||v||_0 := \sum_{i=1}^{p} \mathbbm{1}_{\{v_i \neq 0\}}\), \(||v||_1 := \sum_{i=1}^{p} |v_i|\), and \(||v||^2 := \sum_{i=1}^{p} v_i^2\). We also frequently make use of the notation \(\bar{v} = p^{-1}\sum_{i=1}^{p} v_i\), though in some cases the notation is used to denote something specified in advance. For a vector \(v \in \R^p\), the notation \(\supp(v)\) refers to the support of \(v\), namely the set \(\{i \in [p] : v_i \neq 0\}\). For two probability measures \(P\) and \(Q\) on a measurable space \((\mathcal{X}, \mathcal{A})\), the total variation distance is defined as \(\dTV(P, Q) := \sup_{A \in \mathcal{A}} |P(A) - Q(A)|\). If \(P\) is absolutely continuous with respect to \(Q\), then the \(\chi^2\)-divergence is defined as \(\chi^2(P||Q) := \int_{\mathcal{X}} \left(\frac{dP}{dQ} - 1\right)^2 \, dQ\). For sequences \(\{a_k\}_{k=1}^{\infty}\) and \(\{b_k\}_{k=1}^{\infty}\), the notation \(a_k = o(b_k)\) denotes \(\lim_{k \to \infty} \frac{a_k}{b_k} = 0\) and the notation \(a_k = \omega(b_k)\) is used to denote \(b_k = o(a_k)\). For a matrix \(A \in \R^{m \times n}\), the Frobenius norm of \(A\) is denoted as \(||A||_F = \sqrt{\sum_{i=1}^{m}\sum_{j=1}^{n} a_{ij}^2}\). For \(1 \leq j \leq p\), the vector \(e_j \in \R^p\) denotes the \(j\)th standard basis vector of \(\R^p\).

%%%%%%%%%%%%%%%%%%%%%%%%%%%%%%%%%%%%%%%%%%%%%%
\section{Estimation of a projection: sparse regression}\label{section:sparse_regression}
%%%%%%%%%%%%%%%%%%%%%%%%%%%%%%%%%%%%%%%%%%%%%%

Before describing the core of our methodology in the interesting regime \(s < \frac{p}{2}\), it is convenient to quickly address the case \(s \geq \frac{p}{2}\). Furnishing an estimator which achieves error of order \(p\) is trivial. The raw data \(X\) can be used. The following result is a simple consequence of Markov's inequality along with \(E_{\theta, \gamma}(||X - \theta||^2) = p\). 
\begin{proposition}\label{prop:dense_ubound}
    If \(1 \leq s \leq p\) and \(\gamma \in [0, 1]\), then for any \(\delta \in (0, 1)\) we have
    \begin{equation*}
        \sup_{||\theta||_0 \leq s} P_{\theta, \gamma}\left\{ ||X - \theta||^2 > \frac{p}{\delta}\right\} \leq \delta.
    \end{equation*}
\end{proposition}

As alluded in earlier discussion, for the regime \(s < \frac{p}{2}\) our methodology relies on the simple decomposition \(\theta = (\theta - \bar{\theta}\mathbf{1}_p) + \bar{\theta} \mathbf{1}_p\). Separate estimators will be constructed for each piece, and later combined to furnish an estimator for \(\theta\). In preparation for the definition of our estimators, consider the following decorrelation transformation
\begin{equation*}
    \widetilde{X} := X - \bar{X}\mathbf{1}_p + \frac{\sqrt{1-\gamma}}{\sqrt{p}} \xi \mathbf{1}_p,
\end{equation*}
where \(\xi \sim N(0, 1)\) is drawn by the statistician independently of the data. It follows 
\begin{equation}\label{def:Xtilde}
    \widetilde{X} \sim N(\theta - \bar{\theta}\mathbf{1}_p, (1-\gamma)I_p),
\end{equation}
and so \(\widetilde{X}\) has independent components. Note also \(\bar{X} \sim N\left(\bar{\theta}, \frac{1-\gamma+\gamma p}{p}\right)\) and \(\bar{X}\) is independent of \(\widetilde{X}\).

We can estimate \(\theta - \bar{\theta}\mathbf{1}_p\) via regression; recall from Section \ref{section:decorrelate_then_regress} there are potential gains from correlation when estimating this piece. In what follows, we will assume that \(p\) is larger than a sufficiently large universal constant to avoid technical distractions. We will use the Lasso estimator with a rate-optimal choice of the tuning parameter (see \cite{bellec_slope_2018}). Some preliminary scaling is necessary. Define \(Y := \sqrt{p} \widetilde{X}\) and note 
\begin{equation*}
    Y \sim N\left(M \theta, \sigma^2I_p\right),
\end{equation*}
where \(M = \sqrt{p} \left(I_p - \frac{1}{p}\mathbf{1}_p\mathbf{1}_p^\intercal\right)\) and \(\sigma^2 = (1-\gamma)p\). For a choice of the tuning parameter \(\lambda > 0\), define 
\begin{equation}\label{problem:lasso}
    \hat{\beta} = \argmin_{\beta \in \R^p} \left\{ \frac{1}{p} ||Y - M\beta||^2 + 2\lambda ||\beta||_1 \right\}.
\end{equation}
The choice of penalty \(\lambda \asymp \sqrt{(1-\gamma) \log\left(ep/s\right)}\) will be made as prescribed by Bellec et al. \cite{bellec_slope_2018}. Some rescaling is necessary to estimate \(\theta - \bar{\theta}\mathbf{1}_p\), and so \(p^{-1/2}M \hat{\beta}\) will be used.

\begin{proposition}\label{prop:lasso_error}
    Suppose \(1 \leq s \leq \frac{p}{784}\) and \(\gamma \in [0, 1)\). If \(\lambda \geq 2(4 + \sqrt{2}) \sqrt{(1-\gamma)\log\left(\frac{2ep}{s}\right)}\) and \(\delta \in (0, 1)\), then 
    \begin{equation*}
        \sup_{||\theta||_0 \leq s} P_{\theta, \gamma} \left\{ \left|\left|\frac{1}{\sqrt{p}} M \hat{\beta} - \left(\theta - \bar{\theta}\mathbf{1}_p\right)\right|\right|^2 > 16\lambda^2s \delta^{-1} \right\} \leq \delta,
    \end{equation*}
    where \(\hat{\beta}\) solves (\ref{problem:lasso}).
\end{proposition}

The condition \(s \leq \frac{p}{784}\) ensures \(M\) satisfies a restricted eigenvalue type condition by way of a sparse eigenvalue condition (see Proposition 8.1 in \cite{bellec_slope_2018}). To provide intuition on why a condition like \(s \leq c p\) for some constant \(c \in (0, 1)\) is needed, consider
\begin{equation*}
    \min_{\substack{\Delta \in \R^p \setminus \{0\} \\ ||\Delta||_0 \leq s}} \frac{\frac{1}{p} ||M\Delta||^2 }{||\Delta||^2} = \min_{\substack{\Delta \in \R^p \setminus \{0\}\\ ||\Delta||_0 \leq s}} \frac{||\Delta - \bar{\Delta}\mathbf{1}_p||^2}{||\Delta||^2} = \frac{p-s}{p}. 
\end{equation*}
Thus, the condition \(p-s \asymp p\) is necessary for ensuring \(\min_{\substack{\Delta \in \R^p\setminus\{0\}, ||\Delta||_0 \leq s}} \frac{1}{p} ||M\Delta||^2 / ||\Delta||^2 \gtrsim 1\). Here, we have used \(||\Delta - \bar{\Delta}\mathbf{1}_p||^2 \geq ||\Delta||^2 \cdot \frac{p-s}{p}\) for any \(||\Delta||_0 \leq s\) (see Corollary 1 in \cite{kotekal_minimax_2023}), which is tight whenever \(\Delta\) is constant on its support.

\begin{remark}[Strong restricted eigenvalue condition]
    Though the preceding discussion provides intuition, it does not explain why sparse regression cannot be applied when \(s\) is close to \(\frac{p}{2}\). The condition \(p-s \asymp p\) is still satisfied here. Proposition \ref{prop:lasso_error}, at its core, relies on Theorem 4.2 of \cite{bellec_slope_2018}, which requires \(M\) to satisfy the \(\SRE(s, c_0)\) (strong restricted eigenvalue) condition for some \(c_0 > 1\). Specifically, it is required \(\frac{1}{p} ||Me_j||^2 \leq 1\) for all \(1 \leq j \leq p\) and 
    \begin{equation*}
        \kappa(s, c_0)^2 := \min_{\substack{\Delta \in \mathcal{C}_{\SRE}(s, c_0) \\ \Delta \neq 0}} \frac{\frac{1}{p} ||M\Delta||^2}{||\Delta||^2} > 0,
    \end{equation*}
    where \(\mathcal{C}_{\SRE}(s, c_0) = \left\{ \Delta \in \R^p : ||\Delta||_1 \leq (1 + c_0) \sqrt{s}||\Delta||_2 \right\}\). When \(s \geq (1+c_0)^{-2}p\), it is straightforward to see \(\mathbf{1}_p \in \mathcal{C}_{\SRE}(s, c_0)\), which directly implies \(\kappa(s, c_0) = 0\) since \(M\mathbf{1}_p = 0\). Since \(c_0 > 1\), there is only hope to apply Theorem 4.2 for \(s \leq \frac{p}{4}\). In other words, the condition \(\frac{p}{2} - s \gtrsim p\) is needed to apply the sparse regression strategy.
\end{remark}

When \(s > \frac{p}{784}\), we will use the data \(\widetilde{X}\) to estimate \(\theta - \bar{\theta}\mathbf{1}_p\). This estimator exhibits risk of order \((1-\gamma) p\). The following theorem summarizes our upper bound. 
\begin{theorem}\label{thm:orthog_ubound}
    Suppose \(1 \leq s < \frac{p}{2}\) and \(\gamma \in [0, 1)\). Set 
    \begin{equation*}
        \hat{v} := 
        \begin{cases}
            \frac{1}{\sqrt{p}} M\hat{\beta} &\textit{if } s \leq \frac{p}{784}, \\
            \widetilde{X} &\textit{if } s > \frac{p}{784},
        \end{cases}
    \end{equation*}
    where \(\hat{\beta}\) solves (\ref{problem:lasso}) with \(\lambda = 2(4+\sqrt{2}) \sqrt{(1-\gamma) \log\left(\frac{2ep}{s}\right)}\). For any \(\delta \in (0, 1)\), there exists \(C_\delta > 0\) depending only on \(\delta\) such that
    \begin{equation*}
        \sup_{||\theta||_0 \leq s} P_{\theta, \gamma}\left\{\left|\left| \hat{v} - (\theta - \bar{\theta}\mathbf{1}_p) \right|\right|^2 > C_\delta (1-\gamma) s \log\left(\frac{ep}{s}\right)\right\} \leq \delta. 
    \end{equation*}
\end{theorem}
\noindent To summarize, the ``decorrelate-then-regress'' strategy discussed in Section \ref{section:decorrelate_then_regress} is a good one for estimating \(\theta - \bar{\theta}\mathbf{1}_p\). 

%%%%%%%%%%%%%%%%%%%%%%%%%%%%%%%%%%%%%%%%%%%%%%
\section{Estimation of a linear functional: kernel mode estimator} \label{section:linear_functional}
%%%%%%%%%%%%%%%%%%%%%%%%%%%%%%%%%%%%%%%%%%%%%%
The linear functional \(\bar{\theta}\) can always be estimated by the sample mean \(\bar{X}\). However, since \(\bar{X} \sim N(\bar{\theta}, p^{-1}(1-\gamma+\gamma p))\), sample mean only has good risk when the correlation level is quite low. In particular, the risk does not scale with \(1-\gamma\) as in (\ref{rate:correlation}), and so \(\bar{X}\) misses out on potential gains in the setting of strong correlation (e.g. most severely in \(\gamma = 1\) discussed in Section \ref{section:preview}).

Collier et al. \cite{collier_minimax_2017} consider linear functional estimation in the setting \(\gamma = 0\), where the coordinates are all independent. At a high level, to exploit the sparsity \(||\theta||_0\leq s\), they estimate the linear functional \(\bar{\theta}\) by thresholding; their estimator is \(\hat{L} = \frac{1}{p}\sum_{j=1}^{p} X_j \mathbbm{1}_{\{|X_j| > (2\log(1+p/s^2))^{1/2}\}}\) if \(s < \sqrt{p}\) and \(\hat{L} = \bar{X}\) if \(s \geq \sqrt{p}\). Focusing on the regime \(s < \sqrt{p}\), the intuition for thresholding is the usual one; if \(|X_j|\) is small, then \(\theta_j\) is ``likely'' to be zero (or near-zero), in which case \(\theta_j\) can be estimated by \(0\) instead of \(X_j\). Since the \(Z_j\) in (\ref{model:additive}) are i.i.d. mean zero, this line of reasoning essentially holds ``on average'' across coordinates.  

The situation becomes more complicated when \(\gamma > 0\). With their thresholding estimator \(\hat{L} = \frac{1}{p}\sum_{j=1}^{p} X_j \mathbbm{1}_{\{|X_j| > t\}}\) for some threshold choice \(t > 0\), it becomes quickly clear there is trouble at hand. If \(|X_j|\) is small, we can now only really say \(\theta_j + \sqrt{\gamma} W\) is near-zero. It is not quite right to argue that since \(W\) is mean zero, it then follows \(\theta_j\) must be near-zero; there is only a single draw \(W\) which is shared across all coordinates. Therefore, the term \(\sqrt{\gamma}W\) is an unknown, nuisance ``background'' level which causes trouble in picking the threshold. A natural idea is to estimate \(\sqrt{\gamma}W\) and use it to center the data or set a data-driven threshold. However, the task of estimating \(\sqrt{\gamma}W\) is essentially equivalent to the task of estimating \(\bar{\theta}\) itself. Indeed, note \(\bar{X} \,|\, W \sim N(\bar{\theta} + \sqrt{\gamma}W, (1-\gamma)p^{-1})\), and so estimating \(\sqrt{\gamma}W\) is the same as estimating \(\bar{\theta}\) up to only a parametric rate slowdown. Consequently, it appears a different estimation strategy for \(\bar{\theta}\) needs to be developed.

In this section, we describe two estimators for the regimes \(s \leq \frac{p}{784}\) and \(s > \frac{p}{784}\). In the former, the linear functional \(\bar{\theta}\) can be essentially estimated using the Lasso estimator discussed in Section \ref{section:sparse_regression}. However, the problem becomes more delicate when \(s > \frac{p}{784}\). We cast the problem as a mode estimation problem and investigate a natural kernel mode estimator. Section \ref{section:kme_illustration} illustrates the methodological idea in a simple, special case, and Section \ref{section:kme_bandwidth} describes the optimal choice of the bandwidth. Section \ref{section:kme_robust_connection} highlights the connection of mode estimation to robust statistics, and illustrates why the kernel mode estimator can outperform the sample median, which is a standard robust location estimator, for estimating \(\bar{\theta}\).

When \(1 \leq s \leq \frac{p}{784}\), it turns out estimation of the linear functional \(\bar{\theta}\) via sparse regression suffices.
\begin{proposition}\label{prop:small_s_linear}
    Suppose \(1 \leq s \leq \frac{p}{784}\) and \(\gamma \in [0, 1)\). If \(\lambda = 2(4+\sqrt{2})\sqrt{(1-\gamma)\log\left(\frac{2ep}{s}\right)}\) and \(\delta \in (0, 1)\), then there exists \(C_\delta > 0\) depending only on \(\delta\) such that
    \begin{equation*}
        \sup_{||\theta||_0 \leq s} P_{\theta, \gamma} \left\{ \left|\left| \left(\frac{1}{p}\sum_{i=1}^{p} \hat{\beta}_i\right) \mathbf{1}_p - \bar{\theta}\mathbf{1}_p\right|\right|^2 > C_\delta (1-\gamma) s \log\left(\frac{ep}{s}\right)\right\} \leq \delta, 
    \end{equation*}
    where \(\hat{\beta}\) solves (\ref{problem:lasso}).  
\end{proposition}
\noindent The estimator benefits from strong correlation for the same reason the sparse regression estimator of Section \ref{section:sparse_regression} enjoys strong correlation. Consequently, for \(s \leq \frac{p}{784}\), either \(\bar{X}\) or \(\frac{1}{p}\sum_{i=1}^{p}\hat{\beta}_i\) will be used depending on the correlation level. 

Estimating \(\bar{\theta}\) for \(s > \frac{p}{784}\) is a much more delicate problem. Linear functional estimation has been extensively studied in various models other than (\ref{model:additive}). For example, in the context of the white noise model, perhaps the simplest example is estimation of the regression function at a point \cite{ibragimov_nonparametric_1984}. In the problem of estimating a linear functional over a convex parameter space, Donoho and Liu \cite{donoho_geometrizing_1991} establish a connection between the minimax estimation rate and a certain modulus of continuity. Most of the contemporaneous literature on linear functional estimation assumed convex parameter spaces (which precludes sparsity). Cai and Low \cite{cai_minimax_2004} generalized to the case of a finite union of convex parameter spaces and were able to furnish a bound in sparse Gaussian sequence model with sparsity \(s < p^{1/2-\delta}\). Their bound is sharp up to a logarithmic factor in \(p\). The sharp nonasymptotic rate for estimating \(\sum_{i=1}^{p} m_i\) in the model \(Y \sim N(m, \sigma^2I_p)\) with \(||m||_0 \leq s\) was finally established in \cite{collier_minimax_2017} and is given by \(\sigma^2 s^2 \log(1 + p/s^2)\). The rate can be reexpressed in a more evocative form
\begin{equation*}
    \sigma^2 s^2\log\left(1+\frac{p}{s^2}\right) \asymp 
    \begin{cases}
        \sigma^2 s^2 \log\left(\frac{ep}{s^2}\right) &\textit{if } s < \sqrt{p}, \\
        \sigma^2 p &\textit{if } s \geq \sqrt{p}. 
    \end{cases}
\end{equation*}
The rate exhibits a phase transition at \(s \asymp \sqrt{p}\), which is a distinguishing feature compared to the rate \(\sigma^2 s \log\left(\frac{ep}{s}\right)\) for estimating the full vector \(m\).

Returning to the problem of estimating \(\bar{\theta}\), recall \(\widetilde{X} \sim N\left(\theta - \bar{\theta}\mathbf{1}_p, (1-\gamma)I_p\right)\). Examining \(Y := -\widetilde{X}\), note
\begin{equation}\label{model:Gaussian_contamination}
    Y_i \overset{ind}{\sim}
    \begin{cases}
        N(\eta_i, 1-\gamma) &\textit{if } i \in \mathcal{O}, \\
        N(\mu, 1-\gamma) &\textit{if } i \in \mathcal{I},
    \end{cases}
\end{equation}
where \(\mathcal{O} = \supp(\theta)\), \(\mathcal{I} = \supp(\theta)^c\), \(\mu = \bar{\theta}\), and \(\eta_i = \bar{\theta} - \theta_i\). Since \(||\theta||_0 \leq s\), it follows that the majority of data have mean \(\bar{\theta}\). In the case \(\gamma = 1\) addressed in Section \ref{section:preview}, the empirical mode was used to perfectly estimate \(\bar{\theta}\). In the case \(\gamma < 1\), the intuition from the perfectly correlated setting suggests viewing estimation of \(\bar{\theta}\) as a mode estimation problem.

We will use the kernel mode estimator with the box kernel. For a bandwidth \(h > 0\), define 
\begin{equation}\label{def:Ghat_h}
    \hat{G}_h(t) := \frac{1}{2ph} \sum_{i = 1}^{p} \mathbbm{1}_{\{|t - Y_i| \leq h\}}
\end{equation}
and 
\begin{equation}\label{def:mhat}
    \hat{\mu} := \argmax_{t \in \R} \hat{G}_h(t).
\end{equation}
The population counterpart is denoted as \(G_h(t) = E_{\theta, \gamma}\left(\hat{G}_h(t)\right)\). Mode estimation has a long history. In one of the most foundational papers in the history of statistics (where kernel density estimators were proposed for the problem of density estimation, see also \cite{rosenblatt_remarks_1956}), Parzen \cite{parzen_estimation_1962} discussed estimation of the mode. Given independent and identically distributed data \(X_1,...,X_n\) drawn from a distribution with probability density function \(f\), Parzen proposed the estimator 
\begin{equation*}
    \argmax_{x \in \R} \frac{1}{nh} \sum_{i=1}^{n} K\left(\frac{x-X_i}{h}\right),
\end{equation*}
where \(K\) is a kernel and \(h > 0\) is the bandwidth. Parzen proved that when \(f\) is uniformly continuous, the mode of \(f\) is unique, \(K\) satisfies some mild conditions, and \(\lim_{n\to \infty} nh^2 = \infty\), the estimator is consistent for the mode. The contents of Parzen's article do not cover the simple, yet important box kernel; Chernoff \cite{chernoff_estimation_1964} studied consistency and the asymptotic distribution of the associated mode estimator. Eddy \cite{eddy_optimum_1980}, under the same iid setup, derived optimal rates of convergence associated to each kernel under sufficient restrictions. Jiang \cite{jiang_uniform_2017} considered a finite-sample setting and derived high-probability bounds on kernel density estimates as well as associated functionals including the mode. Arias-Castro, Qiao, and Zheng \cite{arias-castro_estimation_2022} studied a closely related estimator based on histograms.  

These existing results in the literature cannot be directly applied to our problem. The core difficulty is that the data drawn from (\ref{model:additive}) are not identically distributed, violating the typical assumption in the literature. The presence of outliers drawn from different distributions in (\ref{model:additive}) demands new analysis; the problem is now one of ``robust'' mode estimation. Furthermore, it turns out we need to choose a widening, instead of shrinking, bandwidth. In particular, good mode estimation is possible through the choice of bandwidth which happens to yield a bad density estimator. This understanding is different from that found in the older articles, where the resulting mode estimator is good because the density estimator is good.

There are many articles (especially from the causal inference and semiparametric statistics communities) which have found that optimal functional estimation can be done using hyperparameters for nuisance function estimators which are tuned in a way that is suboptimal for estimating the nuisance \cite{newey_crossfit_2018,mcclean_double_2024,hall_bias_1992,gine_simple_2008,paninski_undersmoothed_2008,mcgrath_nuisance_2024}. This can be in the form of undersmoothing or oversmoothing, and often involves some kind of sample splitting scheme \cite{mcclean_double_2024}. A major issue with a simple plugin estimator using an optimal nuisance function estimator is excessive bias, and these techniques aim to reduce the bias. In a variety of settings and for a variety of functionals, many such estimators have been shown to not only be minimax rate-optimal but also semiparametrically efficient in the \(\sqrt{n}\)-consistent regime. Casting estimation of \(\bar{\theta}\) as a mode estimation problem can be viewed in this light; the underlying density is the nuisance function and the mode is the target of interest. Though there is no sample splitting since data are not i.i.d., the kernel mode estimator we investigate indeed oversmooths to estimate the mode well. 

Oversmoothing has also be employed in the robust statistics literature. In particular, the authors of \cite{liu_density_2019} consider estimating a density at a point under Huber's \(\epsilon\)-contamination model and shows that a kernel-based estimator with a larger choice of bandwidth than the usual, non-robust choice turns out to be minimax-rate optimal. A similar result was established by \cite{zhang_robust_2023} when considering estimating the full density in \(L_p\) norm. The presence of outliers drives the oversmoothing, and, viewing estimation of \(\bar{\theta}\) as a robust statistics problem, the same phenomenon seems to occur.  

\subsection{An illustration in a special case}\label{section:kme_illustration}
For illustration, ignore the constraint \(\sum_{i \in \supp(\theta)^c} \mu + \sum_{i \in \supp(\theta)} \eta_i = 0\) and consider the model (\ref{model:Gaussian_contamination}) in the simple case where all the nonzero \(\eta_i\) are all equal to some \(\eta \neq 0\) and \(\gamma = 0\). For fixed \(t \in \R\), the random variable \(\hat{G}_h(t)\) is unbiased for 
\begin{equation*}
    G_h(t) = \frac{1}{2h} \left( \frac{p-s}{p} \left(\Phi\left( t - \mu + h\right) - \Phi\left( t - \mu - h\right) \right) + \frac{s}{p} \left(\Phi\left( t - \eta + h\right) - \Phi\left( t - \eta - h\right) \right) \right).
\end{equation*}
Here, \(\Phi\) denotes the cumulative distribution function for the standard normal distribution. Define 
\begin{equation}\label{def:global_maximizer}
    m := \argmax_{t \in \R} G_h(t).
\end{equation}
Consider \(m\) is the global maximizer of the function \(G_h\) which can also be written as 
\begin{equation*}
    G_h(t) = \frac{1}{2h} \cdot \frac{p-2s}{p} \left(\Phi(t - \mu +h) - \Phi(t - \mu - h)\right) + J_h(t),
\end{equation*}
where \(J_h(t) = \frac{1}{2h} \cdot \frac{s}{p} \left(\left(\Phi(t-\mu+h) - \Phi(t-\mu- h)\right) + \left(\Phi(t - \eta +h) - \Phi(t - \eta - h)\right)\right)\). Observe \(J_h\) is symmetric about the point \(\frac{\mu+\eta}{2}\) and exhibits two global maxima, one which is close to \(\mu\) and one which is close to \(\eta\). Looking at \(G_h\), the first term is monotone decreasing in \(|t - \mu|\), and it is precisely its presence which ensures \(m\) is close to \(\mu\). Figure \ref{fig:JvG} shows an example of the functions \(J_h\) and \(G_h\). The global maximizer \(m\) is indeed close to \(\mu\).  

\begin{figure}[ht]
    \centering
    \includegraphics[scale=0.5]{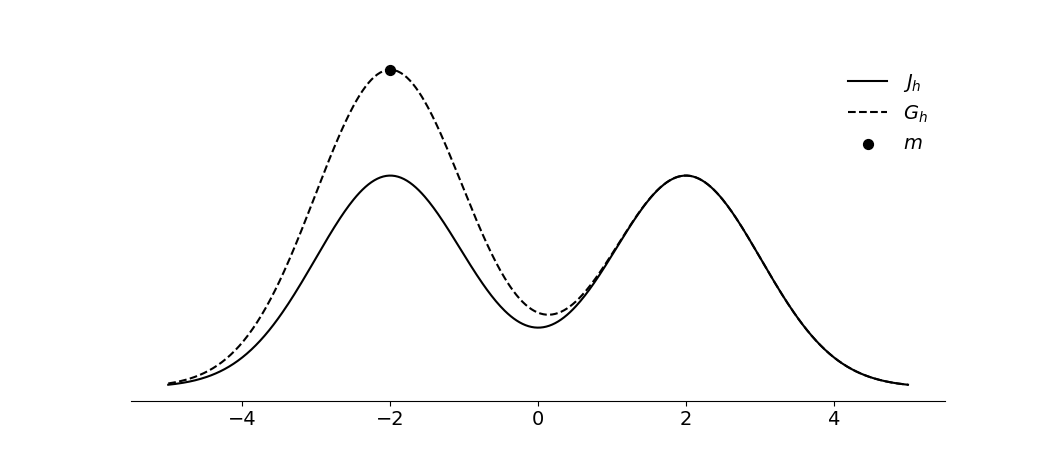}
    \caption[Figure 1.]{Plots of \(G_h\) and \(J_h\) with \(\mu = -2, \eta = 2, h = 0.25, p = 10000, s = \left\lfloor \frac{p}{2} - 10\sqrt{p}\right\rfloor = 4000, m \approx -1.999\).}\label{fig:JvG}
\end{figure}

\subsection{A widening, instead of shrinking, bandwidth}\label{section:kme_bandwidth}
For development of the methodology, let us exit the special case and return to the general setting where \(\gamma \in [0, 1)\) and \(\eta_i = \bar{\theta} - \theta_i\) need not all be the same. Curiously, it turns out one should take the bandwidth to essentially widen rather than shrink. By ``oversmoothing'', variance is traded off for bias. It turns out the correct choice is 
\begin{equation}\label{eqn:h_order}
    h \asymp \sqrt{(1-\gamma)\left(1 \vee \log\left(\frac{ep}{(p-2s)^2}\right)\right)}. 
\end{equation}
The factor \(\sqrt{1-\gamma}\) is the standard deviation in (\ref{model:Gaussian_contamination}) and thereby denotes the scale of the data. Notably, the scale-free quantity \(\frac{h}{\sqrt{1-\gamma}}\), which can be effectively understood as a bandwidth, grows in \(p\) when \(p-2s = o(\sqrt{p})\). Throughout the following discussion and in the proofs, it will be assumed \(p\) is larger than a sufficiently large universal constant. 

To illustrate why this choice of bandwidth \(h\) is the correct one, recall the high level goal is to show the global maximizer of \(\hat{G}_h\) is close to \(\mu\). It suffices to find universal positive constants \(C_1\) and \(C_2\) such that with high probability, there exists a point \(x\) with \(|x - \mu| \leq C_1 h\) so for every point \(t\) with \(|t - \mu| > C_2 h\) we have \(\hat{G}_h(x) > \hat{G}_h(t)\). On this event, it would then follow \(|\hat{\mu} - \mu| \leq C_2 h\). 

To show \(\hat{G}_h(x) > \hat{G}_h(t)\), consider 
\begin{equation}\label{eqn:kme_master_display}
    \hat{G}_h(x) - \hat{G}_h(t) \geq \left(G_h(x) - G_h(t)\right) - \left|\hat{G}_h(x) - G_h(x)\right| - \left|\hat{G}_h(t) - G_h(t)\right|. 
\end{equation}
A lower bound for the signal \(\left(G_h(x) - G_h(t)\right)\) is needed, along with control of the stochastic deviations \(\left|\hat{G}_h(x) - G_h(x)\right|\) and \(\left|\hat{G}_h(t) - G_h(t)\right|\) uniformly over \(|t - \mu| \geq C_2 h\). The following bound on the signal is available by Proposition \ref{prop:mean_diff_fix}.

\begin{proposition}\label{prop:mean_diff_fix}
    Suppose \(C_a\) is a sufficiently large universal constant. Further suppose \(\frac{h}{\sqrt{1-\gamma}}\) is larger than a sufficiently large universal constant. There exists \(|x - \mu| \leq \frac{h}{4}\) such for any \(|t - \mu| \geq C_a h\), we have 
    \begin{equation}\label{eqn:mean_diff}
        G_h(x) - G_h(t) \geq \left(\frac{1}{2} \cdot \frac{p-2|\mathcal{O}|}{2ph} \right) \vee \left( \frac{p-2|\mathcal{O}|}{2ph} - \frac{e^{-\frac{Ch^2}{1-\gamma}}}{h} + \frac{1}{2ph} \sum_{i \in \mathcal{O}} P_{\theta, \gamma}\left\{|t - Y_i| > h\right\} \right),
    \end{equation}
    where \(C > 0\) is a universal constant. 
\end{proposition}
Specifically, we will use the lower bound provided from Proposition \ref{prop:mean_diff_fix} with the choice \(C_1 = \frac{1}{4}\) and \(C_2 = C_a\). The proof of Proposition \ref{prop:mean_diff_fix} is quite involved. The main technical component is a precise characterization of the mode of a Gaussian mixture. See Theorem \ref{thm:mixture} in Appendix \ref{appendix:mixture} and a full treatment of this topic there.

With the bound (\ref{eqn:mean_diff}) on the signal in hand, the proof proceeds by considering two disjoint regions, namely 
\begin{align*}
    \mathcal{U} &:= \left\{ t \in \R : |t - \mu| \geq C_2 h \text{ and } \frac{1}{p} \sum_{i \in \mathcal{O}} P_{\theta, \gamma}\left\{|t - Y_i| > h\right\} \geq 4e^{-\frac{Ch^2}{1-\gamma}} \right\}, \\
    \mathcal{V} &:= \left\{ t \in \R : |t - \mu| > C_2 h \text{ and } t \in \mathcal{U}^c \right\}.
\end{align*}
Take a point \(t \in \R\) with \(|t - \mu| > C_2 h\). Note either \(t \in \mathcal{U}\) or \(t \in \mathcal{V}\). Suppose \(t \in \mathcal{V}\). From the first term in (\ref{eqn:mean_diff}) we have 
\begin{equation*}
    G_h(x) - G_h(t) \geq \frac{1}{2} \cdot \frac{p-2s}{2ph}. 
\end{equation*}
From (\ref{eqn:kme_master_display}), the stochastic deviations needed to be bounded. Clearly standard concentration (say the Dvoretzky-Kiefer-Wolfowitz inequality) giving \(\sup_{t \in \R}|\hat{G}_h(t) - G_h(t)| \lesssim h^{-1}p^{-1/2}\) will not suffice in the regime \(p-2s \lesssim \sqrt{p}\). In other words, the signal is too small and so faster concentration is needed. To illustrate why the choice (\ref{eqn:h_order}) is the right one, consider the variance
\begin{equation*}
    \Var\left(\hat{G}_h(t)\right) = \frac{1}{4p^2h^2} \sum_{i \in \mathcal{I}} \Var\left(\mathbbm{1}_{\{|t - Y_i| \leq h\}}\right) + \frac{1}{4p^2h^2} \sum_{i \in \mathcal{O}} \Var\left(\mathbbm{1}_{\{|t - Y_i| \leq h\}}\right).
\end{equation*}
Since \(|t - \mu| \geq C_2 h\) with \(C_2\) sufficiently large, we have 
\begin{equation*}
    \Var\left(\hat{G}_h(t)\right) \lesssim \frac{1}{ph^2} e^{-\frac{Ch^2}{1-\gamma}} + \frac{1}{p^2h^2} \sum_{i \in \mathcal{O}} P_{\theta, \gamma}\left\{|t - Y_i| > h\right\} \lesssim \frac{1}{ph^2}e^{-\frac{Ch^2}{1-\gamma}},
\end{equation*}
where we have used the definition of \(\mathcal{V}\). By selecting \(h\) as in (\ref{eqn:h_order}), the standard deviation is of no larger order than \(\frac{p-2s}{ph}\), which is the signal magnitude. Of course, controlling the stochastic deviation at just the one point \(t\) is insufficient. Rather, we need to bound the deviation uniformly over \(t \in \mathcal{V}\) and we use well-known empirical process theory tools \cite{boucheron_concentration_2013}. Nevertheless, this variance calculation captures the intuition for why this choice of \(h\) is the right one. 

Now suppose \(t \in \mathcal{U}\). By the definition of $\mathcal{U}$, extra signal is available as the second term in (\ref{eqn:mean_diff}) gives
\begin{equation}\label{eqn:kme_signal_U}
    G_h(x) - G_h(t) \geq \frac{p-2s}{2ph} + \frac{1/2}{2ph} \sum_{i \in \mathcal{O}} P_{\theta, \gamma}\left\{|t - Y_i| > h\right\}.     
\end{equation}
From (\ref{eqn:kme_master_display}), the stochastic deviations needed to be bounded. Since \(|t - \mu| \geq C_2 h\) with \(C_2\) sufficiently large, we have
\begin{align*}
    \Var\left(\hat{G}_h(t)\right) &\leq \frac{1}{4ph^2} e^{-\frac{Ch^2}{1-\gamma}} + \frac{1}{4p^2h^2} \sum_{i \in \mathcal{O}} P_{\theta, \gamma}\left\{|t - Y_i| > h\right\} \\
    &\leq \frac{1}{4ph^2} e^{-\frac{Ch^2}{1-\gamma}} + \frac{4}{4p^2h^2} + \left(\frac{1/2}{2ph} \sum_{i \in \mathcal{O}} P_{\theta, \gamma}\left\{|t - Y_i| > h\right\}\right)^2,
\end{align*}
where we have used the inequality \(ab \leq a^2 + b^2\) for all \(a, b \in \R\). By selecting \(h\) as in (\ref{eqn:h_order}), the standard deviation is of no larger order than the signal (\ref{eqn:kme_signal_U}). Again, controlling the stochastic deviation at just the one point \(t\) is insufficient. Uniform control is obtained through empirical process theory tools \cite{boucheron_concentration_2013}. 

The argument to bound the stochastic deviation \(|\hat{G}_h(x) - G_h(x)|\) has a similar flavor to the above analysis, but is slightly different and so we relegate discussion of it to the proof. Putting together the results for \(\mathcal{U}\) and \(\mathcal{V}\) thus establishes \(|\hat{\mu} - \mu| \lesssim h\) with high probability provided we choose \(h\) as in (\ref{eqn:h_order}). 

\begin{theorem}\label{thm:kme_known_var}
    Suppose \(1 \leq s < \frac{p}{2}\) and \(\gamma \in [0, 1)\). There exist universal constants \(C_1, C_2> 0\) such that the following holds. For any \(\delta \in (0, 1)\), there exists \(L_\delta\) depending only on \(\delta\) such that if \(p\) is sufficiently large depending only on \(\delta\) and 
    \begin{equation*}
        h = C_1 \sqrt{(1-\gamma)\left(1 \vee \log\left(\frac{L_\delta p}{(p-2s)^2}\right)\right)},
    \end{equation*}
    then 
    \begin{equation*}
        \sup_{||\theta||_0 \leq s} P_{\theta, \gamma} \left\{|\hat{\mu} - \bar{\theta}| > C_2 h\right\} \leq \delta,
    \end{equation*}
    where \(\hat{\mu}\) is given by (\ref{def:mhat}).
\end{theorem}

\subsection{Connection to robust statistics}\label{section:kme_robust_connection}
In (\ref{model:Gaussian_contamination}), estimation of \(\mu\) can be viewed as a robust estimation problem of a location parameter. Indeed, the notation \(\mathcal{I}\) and \(\mathcal{O}\) is suggestive, calling to mind the categorization of ``inliers'' and ``outliers''. The model (\ref{model:Gaussian_contamination}) is an instantiation of the mean-shift contamination model which has been extensively studied in the robust statistics literature, mainly in the context of regression \cite{gannaz_robust_2007,mccann_robust_2007,antoniadis_wavelet_2007,foygel_corrupted_2014,nguyen_robust_2013,dalalyan_outlier-robust_2019,collier_multidimensional_2019}. This literature largely contains results in the case \(s \leq \delta p\) where \(\delta > 0\) is a sufficiently small constant. To the best of our knowledge, the regime where \(s\) is close to \(\frac{p}{2}\), that is to say \(p-2s = o(p)\), is unaddressed. Moreover, in the regime \(s \leq \delta p\), the results specialized to the location estimation problem do not deliver rates faster than that achieved by sample median. The following proposition gives a bound on the error obtained by sample median, and in fact provides content when \(p-2s = o(p)\).

\begin{proposition}\label{prop:sample_median}
    Suppose \(1 \leq s < \frac{p}{2}\) and \(\gamma \in [0, 1)\). If \(\delta \in (0, 1)\) and \(p\) is sufficiently large depending only on \(\delta\), then there exists \(C_\delta > 0\) depending only on \(\delta\) such that
    \begin{equation*}
        \sup_{||\theta||_0 \leq s} P_{\theta, \gamma}\left\{ \left|\left| \hat{T}\mathbf{1}_p - \bar{\theta}\mathbf{1}_p \right|\right|^2 > C_\delta(1-\gamma)\left(1 + \frac{s^2}{p}\log\left(\frac{ep}{p-2s}\right)\right) \right\} \leq \delta,
    \end{equation*}
    where \(\hat{T} = \median\left(Y_1,...,Y_p\right)\). 
\end{proposition}
\noindent The proof of Proposition \ref{prop:sample_median} can be found in Appendix \ref{section:misc}. If the better of \(\hat{T}\mathbf{1}_p\) or \(\bar{X}\mathbf{1}_p\) is used to estimate \(\bar{\theta}\mathbf{1}_p\) and the sparse regression estimator from Section \ref{section:sparse_regression} is used to estimate \(\theta - \bar{\theta}\mathbf{1}_p\), then the combined estimator achieves the rate
\begin{align*}
    &(1-\gamma) s \log\left(\frac{ep}{s}\right) + (1-\gamma)\left(1 + \frac{s^2}{p}\log\left(\frac{ep}{p-2s}\right)\right) \wedge (1-\gamma+\gamma p) \\
    &\asymp 
    \begin{cases}
        (1-\gamma)s\log\left(\frac{ep}{s}\right) &\textit{if } s \leq \frac{p}{4}, \\
        (1-\gamma)p \log\left(\frac{ep}{p-2s}\right) \wedge p &\textit{if } \frac{p}{4} < s < \frac{p}{2}. 
    \end{cases}
\end{align*}
Compare to the minimax rate (\ref{rate:correlation}),
\begin{equation*}
    \varepsilon^*(p, s, \gamma)^2 \asymp 
    \begin{cases}
        (1-\gamma) s \log\left(\frac{ep}{s}\right) &\textit{if } 1 \leq s \leq \frac{p}{2} - \sqrt{p}, \\
        (1-\gamma) p \log\left(\frac{ep}{(p-2s)^2}\right) \wedge p&\textit{if } \frac{p}{2} - \sqrt{p} < s < \frac{p}{2}.
    \end{cases}
\end{equation*}
When \(\frac{p}{2}-s = o(p)\), using the sample median yields an estimator which may have risk suboptimal by a factor logarithmic in \(p\). Sample median can be improved upon when \(\frac{p}{2}-s\) is near \(\sqrt{p}\). 

It is not immediately obvious why the kernel mode estimator is a better robust location estimator in (\ref{model:Gaussian_contamination}) than the sample median. In fact, it is well known \cite{chen_general_2016} sample median is minimax rate-optimal in Huber's contamination model. It turns out the model (\ref{model:Gaussian_contamination}) is more structured than Huber's contamination model. In particular, the Gaussian character of the data in (\ref{model:Gaussian_contamination}) can be exploited to achieve faster rates of convergence. To illustrate how the kernel mode estimator exploits the Gaussian character of the data, suppose \(\gamma = 0\) and the data were actually generated from Huber's contamination model 
\begin{equation*}
    Y_i \overset{ind}{\sim} 
    \begin{cases}
        N(\mu, 1) &\textit{if } i \in \mathcal{I}, \\
        \delta_{\eta_i} &\textit{if } i \in \mathcal{O}. 
    \end{cases}
\end{equation*}
For illustration, consider the special case where \(\eta_i = \eta\) for all \(i \in \mathcal{O}\). Then, \(\hat{G}_h(t)\) is unbiased for 
\begin{equation*}
    G_{\text{Huber}, h}(t) := \frac{1}{2h} \cdot \frac{p-2s}{p} \left( \Phi(t - \mu + h) - \Phi(t - \mu - h)\right) + J_{\text{Huber}, h}(t),
\end{equation*}
where \(J_{\text{Huber}, h}(t) := \frac{1}{2h} \cdot \frac{s}{p} \left( \left(\Phi(t - \mu + h) - \Phi(t - \mu - h)\right) + \mathbbm{1}_{\{|t - \eta| \leq h\}}\right)\). Figure \ref{fig:JHuber_v_GHuber} shows any global maximizer \(m_{\text{Huber}}\) (which is not unique) of \(G_{\text{Huber}, h}\) is not at all close to \(\mu\). Figure \ref{fig:JHuber_v_GHuber} is markedly different from Figure \ref{fig:JvG}. Thus, the Gaussianity is essential for \(m\) to be close to \(\mu\), and the kernel mode estimator's success crucially depends on this.

\begin{figure}[!ht]
    \centering
    \includegraphics[scale=0.5]{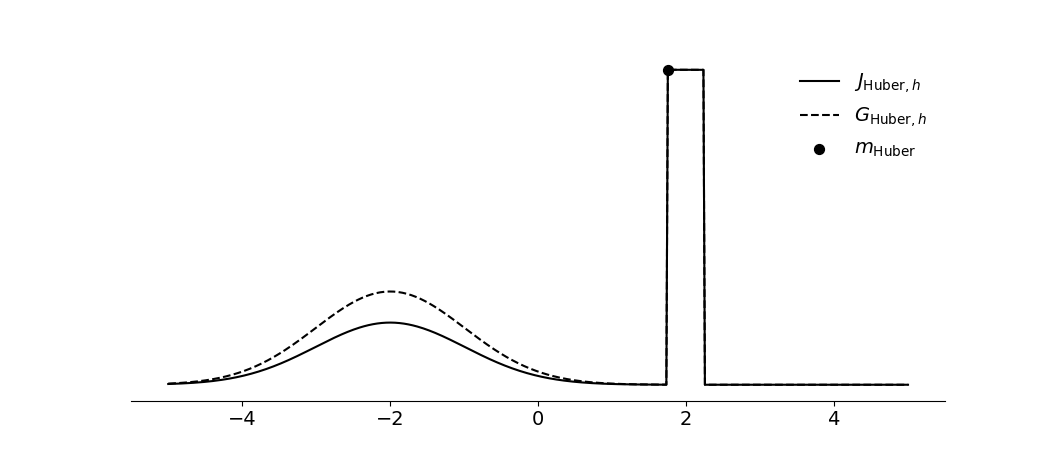}
    \caption[Figure 1.]{Plots of \(G_{\text{Huber}, h}\) and \(J_{\text{Huber}, h}\) with \(\mu = -2, \eta = 2, h = 0.25, p = 10000, s = \left\lfloor \frac{p}{2} - 10\sqrt{p}\right\rfloor = 4000, m_{\text{Huber}} \approx 1.750\).}\label{fig:JHuber_v_GHuber}
\end{figure}

%%%%%%%%%%%%%%%%%%%%%%%%%%%%%%%%%%%%%%%%%%%%%%
\section{Upper bound}\label{section:upper_bound}
%%%%%%%%%%%%%%%%%%%%%%%%%%%%%%%%%%%%%%%%%%%%%%

The two components for estimating an orthogonal projection and a linear functional from Sections \ref{section:sparse_regression} and \ref{section:linear_functional} are combined to obtain a final estimator of the signal. To estimate \(\theta - \bar{\theta}\mathbf{1}_p\), the estimator \(\hat{v}\) from Theorem \ref{thm:orthog_ubound} is used. To estimate \(\bar{\theta}\), the three estimators discussed in Section \ref{section:linear_functional} need to be combined.

As discussed in Section \ref{section:linear_functional}, if the correlation is strong enough, sparse regression is used to estimate \(\bar{\theta}\) for \(s \leq \frac{p}{784}\) and a kernel mode estimator is used for \(s > \frac{p}{784}\). On the other hand, if the correlation is not strong enough, then \(\bar{X}\) is used since \(\bar{X} \sim N(\bar{\theta}, p^{-1}(1-\gamma+\gamma p))\) as the risk can be low for weak correlation. The following result provides detail and is stated without proof.

\begin{corollary}\label{corollary:linear_functional_estimation}
    Suppose \(1 \leq s < \frac{p}{2}\) and \(\gamma \in [0, 1)\). Set 
    \begin{equation*}
        \hat{T} = 
        \begin{cases}
            \frac{1}{p} \sum_{i=1}^{p} \hat{\beta}_i &\textit{if } 1 \leq s \leq \frac{p}{784} \text{ and } (1-\gamma)s \log\left(\frac{ep}{s}\right) \leq 1-\gamma+\gamma p, \\
            \hat{\mu} &\textit{if } \frac{p}{784} < s < \frac{p}{2} \text{ and } (1-\gamma)p\left(1\vee \log\left(\frac{ep}{(p-2s)^2}\right)\right)\leq 1-\gamma+\gamma p, \\
            \bar{X} &\textit{otherwise},
        \end{cases}
    \end{equation*}
    where \(\hat{\beta}\) is given in Proposition \ref{prop:small_s_linear} and \(\hat{\mu}\) is given in Theorem \ref{thm:kme_known_var}. For any \(\delta \in (0, 1)\), there exists \(C_\delta > 0\) depending only on \(\delta\) such that 
    \begin{equation*}
        \sup_{||\theta||_0 \leq s} P_{\theta, \gamma}\left\{ ||\hat{T}\mathbf{1}_p - \bar{\theta}\mathbf{1}_p||^2 > C_\delta \left( \epsilon(p, s, \gamma)^2 \wedge (1-\gamma+\gamma p) \right)\right\} \leq \delta, 
    \end{equation*}
    where
    \begin{equation*}
        \epsilon(p, s, \gamma)^2 = 
        \begin{cases}
            (1-\gamma)s\log\left(\frac{ep}{s}\right) &\textit{if } 1 \leq s \leq \frac{p}{2} - \sqrt{p}, \\
            (1-\gamma)p\log\left(\frac{ep}{(p-2s)^2}\right) &\textit{if } \frac{p}{2} - \sqrt{p} < s < \frac{p}{2}. 
        \end{cases}
    \end{equation*}
\end{corollary}
In fact, the keen reader will point out the estimator \(p^{-1} \sum_{i=1}^{p}\hat{\beta}_i\) can be used even in the case \((1-\gamma) s \log\left(\frac{ep}{s}\right) > 1 - \gamma + \gamma p\). To see this, consider Proposition \ref{prop:small_s_linear} asserts \(\left|\left|\left(p^{-1}\sum_{i=1}^{p} \hat{\beta}_i\right) \mathbf{1}_p - \bar{\theta}\mathbf{1}_p\right|\right|^2 \lesssim (1-\gamma)s\log\left(\frac{ep}{s}\right)\) with high probability for \(s \leq \frac{p}{784}\), which actually matches the desired rate for estimating the entire vector \(\theta\) as seen in (\ref{rate:correlation}). In other words, though \(\bar{X}\) may be better for estimating \(\bar{\theta}\), ultimately it brings no benefit from the perspective of rates for estimating \(\theta\).  

The estimators for the two components \(\theta - \bar{\theta}\mathbf{1}_p\) and \(\bar{\theta}\) can be combined to achieve the following rate, which we state without proof.
\begin{theorem}\label{thm:upper_bound}
    Suppose \(1 \leq s \leq p\) and \(\gamma \in [0, 1)\). Set 
    \begin{equation*}
        \hat{\theta} = 
        \begin{cases}
            \hat{v} + \hat{T}\mathbf{1}_p &\textit{if } 1 \leq s < \frac{p}{2}, \\
            X &\textit{if } \frac{p}{2} \leq s \leq p, 
        \end{cases}
    \end{equation*}
    where \(\hat{v}\) is given in Theorem \ref{thm:orthog_ubound} and \(\hat{T}\) is given in Corollary \ref{corollary:linear_functional_estimation}. For any \(\delta \in (0, 1)\), there exists \(C_\delta > 0\) depending only on \(\delta\) such that
    \begin{equation*}
        \sup_{||\theta||_0 \leq s} P_{\theta, \gamma}\left\{ ||\hat{\theta} - \theta||^2 > C_\delta \varepsilon^*(p, s, \gamma)^2\right\} \leq \delta, 
    \end{equation*}
    where \(\varepsilon^*(p, s, \gamma)^2\) is given by (\ref{rate:correlation}). 
\end{theorem}

%%%%%%%%%%%%%%%%%%%%%%%%%%%%%%%%%%%%%%%%%%%%%%
\section{Lower bound}\label{section:lower_bounds}
%%%%%%%%%%%%%%%%%%%%%%%%%%%%%%%%%%%%%%%%%%%%%%

In this section, we present a matching minimax lower bound by considering various sparsity regimes in turn. 

\subsection{\texorpdfstring{Regime \(1 \leq s \leq \frac{p}{2} - \sqrt{p}\)}{Regime 1 <= s <= p/2 - sqrt(p)}}
As seen in (\ref{rate:correlation}), the rate \((1-\gamma)s\log\left(\frac{ep}{s}\right)\) in the regime \(s \leq \frac{p}{2} - \sqrt{p}\) will be quite familiar to the reader acquainted with high-dimensional statistics. 

\begin{proposition}\label{prop:sparse_lbound}
    Suppose \(1 \leq s \leq p\) and \(\gamma \in [0, 1]\). If \(\delta \in (0, 1)\) and \(p\) is sufficiently large depending only on \(\delta\), then there exists \(c_\delta > 0\) depending only on \(\delta\) such that 
    \begin{equation*}
        \inf_{\hat{\theta}} \sup_{||\theta||_0 \leq s} P_{\theta, \gamma}\left\{ ||\hat{\theta} - \theta||^2 \geq c_\delta (1-\gamma) s \log\left(\frac{ep}{s}\right) \right\} \geq 1 - \delta.
    \end{equation*}
\end{proposition}

This lower bound is entirely driven by the difficulty of estimating \(\theta - \bar{\theta}\mathbf{1}_p\). The set \(\{ \theta - \bar{\theta}\mathbf{1}_p : ||\theta||_0 \leq s\}\) is simply a \(p-1\) dimensional orthogonal projection of the set of \(s\)-sparse vectors. The data \(X\) can be broken into two independent pieces, \(X - \bar{X}\mathbf{1}_p\) and \(\bar{X}\mathbf{1}_p\). Since \(\bar{X} \sim N\left(\bar{\theta}, \frac{1-\gamma+\gamma p}{p}\right)\) has large variance compared to the desired scaling of \(1-\gamma\) in the rate, we have the intuition that the information of \(\bar{X}\) is negligible. Putting it aside, consider \(X - \bar{X}\mathbf{1}_p\) has a distribution which is nearly a spherical Gaussian with variance \(1-\gamma\) and mean \(\theta - \bar{\theta}\mathbf{1}_p\), suggesting it should suffice for estimation of \(\theta - \bar{\theta}\mathbf{1}_p\). This intuition materializes when applying Fano's method, as it turns out \(\bar{X}\) does not contribute much and the situation is as if only \(X - \bar{X}\mathbf{1}_p\) were available to estimate. Fano's method yields a lower bound of order \((1-\gamma) s \log\left(\frac{ep}{s}\right)\), and the correlation structure of (\ref{model:additive}) presents no serious technical challenge. 

\subsection{\texorpdfstring{Regime \(\frac{p}{2} - \sqrt{p} < s < \frac{p}{2} \)}{Regime p/2 - sqrt(p) < s < p/2}}

In the regime \(s < \frac{p}{2}\) with \(\frac{p}{2} - s \lesssim \sqrt{p}\), the need to estimate \(\bar{\theta}\mathbf{1}_p\) affects the difficulty of the problem. Though the lower bound argument uses the same technique found in the literature on functional estimation, the details are not standard. 

\begin{proposition}\label{prop:thetabar_lbound}
    Suppose \(\frac{p}{2} - \sqrt{p} < s < \frac{p}{2}\) and \(\gamma \in [0, 1]\). If \(\delta \in (0, 1)\) and \(p\) is sufficiently large depending only on \(\delta\), then there exists \(c_\delta > 0\) depending only on \(\delta\) such that 
    \begin{equation*}
        \inf_{\hat{\theta}} \sup_{||\theta||_0 \leq s} P_{\theta, \gamma}\left\{ ||\hat{\theta} - \theta||^2 \geq c_\delta \left((1-\gamma)p \log\left(1 + \frac{p}{(p-2s)^2}\right) \wedge p\right) \right\} \geq 1 - \delta. 
    \end{equation*}
\end{proposition}

As is typical in the functional estimation literature, Le Cam's method (or the ``method of two fuzzy hypotheses'' \cite{tsybakov_introduction_2009}) is used to prove the lower bound. Namely, one seeks two priors \(\pi_0\) and \(\pi_1\) which maximize \(||\bar{\theta}_0 \mathbf{1}_p - \bar{\theta}_1\mathbf{1}_p||^2\) for \(\theta_0 \sim \pi_0\) and \(\theta_1 \sim \pi_1\) while keeping the total variation distance between the mixtures \(P_{\pi_0} := \int_{\theta} P_{\theta, \gamma} \pi_0(d\theta)\) and \(P_{\pi_1} := \int_{\theta} P_{\theta, \gamma} \pi_1(d\theta)\) small. 

In parametric problems, it typically suffices to pick both \(\pi_0\) and \(\pi_1\) to be point masses and typically \(\dTV(P_{\pi_1}, P_{\pi_0})\) can be controlled explicitly. For more modern settings where the underlying parameter is high dimensional, \(\pi_1\) usually must be chosen to be a nontrivial mixture. In some problems, it suffices to pick \(\pi_0\) to be a point mass. The lower bound arguments are highly related to the minimax testing literature; perhaps the most typical technique is the Ingster-Suslina method \cite{ingster_nonparametric_2003} which is used to bound \(\chi^2(P_{\pi_1}||P_{\pi_0})\) and deduce a corresponding bound for \(\dTV(P_{\pi_1}, P_{\pi_0})\). 

Though this approach has borne fruit in some problems, other problems require choosing both \(\pi_0\) and \(\pi_1\) to be nontrivial mixtures. This case is the most technically challenging in terms of bounding the total variation; the Ingster-Suslina method \cite{ingster_nonparametric_2003}, which has seen massive success in delivering sharp testing results, is no longer applicable. In the literature thus far, the technique of moment matching is perhaps most popular \cite{wu_polynomial_2020,lepski_estimation_1999,cai_testing_2011}. The priors \(\pi_0\) and \(\pi_1\) are constructed, sometimes in an implicit manner, to share as many moments as possible. Recently, Fourier-based approaches have seen success, in which \(\pi_0\) and \(\pi_1\) are constructed to have characteristic functions agreeing on a large interval around the origin \cite{carpentier_adaptive_2019,cai_optimal_2010}. 

In the problem of estimating \(\bar{\theta}\) where \(\theta\) is sparse, our lower bound construction involves choosing \(\pi_0\) and \(\pi_1\) to be nontrivial mixtures. However, we are able to avoid intricate constructions involving matching moments or characteristic functions by defining the priors explicitly. Our explicit construction has the added benefit of illustrating once again why the \(s = \frac{p}{2}\) transition exists in the rate. For sake of illustration, let us take \(p\) to be even. For ease of notation, set 
\begin{equation*}
    \psi^2 = \psi(p, s, \gamma)^2 = (1-\gamma) p \log\left(1 + \frac{p}{(p-2s)^2}\right) \wedge (1-\gamma+\gamma p).
\end{equation*}
A draw \(\theta_0 \sim \pi_0\) is defined by drawing uniformly at random a size \(s\) subset \(S_0 \subset E_0 := \left\{1,...,\frac{p}{2}\right\}\) and setting \(\theta_0 = \frac{c \psi}{\sqrt{s}}\mathbf{1}_{S_0}\). A draw \(\theta_1 \sim \pi_1\) is defined by drawing uniformly at random a size \(s\) subset \(S_1 \subset E_1 := \left\{\frac{p}{2} + 1,..., p\right\}\) and setting \(\theta_1 = - \frac{c\psi}{\sqrt{s}}\mathbf{1}_{S_1}\). Here, \(c\) is a suitably small positive constant and \(\mathbf{1}_{S_j} \in \R^p\) denotes the vector with entry \(i\) equal to \(1\) if \(i \in S_j\) and \(0\) otherwise, for \(j = 0, 1\). Note \(||\bar{\theta}_0 \mathbf{1}_p - \bar{\theta}_1\mathbf{1}_p||^2 \asymp \psi^2 \) almost surely since \(s \asymp p\). 

The key observation which avoids having to match moments or characteristic functions is the following. For \(j = 0, 1\), it follows immediately from the additive representation (\ref{model:additive}) that for \(Y \sim P_{\pi_j}\) we have that the random quantities \(Y_{E_0} - \bar{Y}_{E_0}\mathbf{1}_{E_0}\), \(Y_{E_1} - \bar{Y}_{E_1}\mathbf{1}_{E_1}\), and \((\bar{Y}_{E_0}, \bar{Y}_{E_1})\) are all mutually independent, where \(\bar{Y}_{E_j} = \frac{1}{|E_j|} \sum_{i \in E_j} Y_i\). Writing \(P_{\pi_j}^{I}, P_{\pi_j}^{II}\), and \(P_{\pi_j}^{III}\) to denote the marginal distributions of these three random vectors, we have \(\dTV(P_{\pi_0}, P_{\pi_1}) \leq \dTV(P_{\pi_0}^{I}, P_{\pi_1}^{I}) + \dTV(P_{\pi_0}^{II}, P_{\pi_1}^{II}) + \dTV(P_{\pi_0}^{III}, P_{\pi_1}^{III})\). The third term can be handled explicitly since the distribution of \((\bar{\theta}_{E_0}, \bar{\theta}_{E_1})\) is a point mass under either \(\pi_0\) or \(\pi_1\). It is here where \(\psi^2 \leq 1-\gamma+\gamma p\) is needed. By symmetry, handling the first term is exactly like handling the second term, so attention can be focused on the second term. 

By the definition of \(E_1\) and \(\pi_0\), we have \(Y_{E_1} = 0\) almost surely when \(Y \sim P_{\pi_0}\). As a consequence of the Neyman-Pearson lemma, the quantity \(1 - \dTV(P_{\pi_1}^{II}, P_{\pi_2}^{II})\) corresponds to the minimal Type I plus Type II error of the hypothesis testing problem 
\begin{align*}
    H_0 &: U \sim N(0, (1-\gamma)I_{p/2} + \gamma \mathbf{1}_{p/2}\mathbf{1}_{p/2}^\intercal), \\
    H_1 &: U \sim \int N\left(-\frac{c\psi}{\sqrt{s}}\mathbf{1}_{S_1}, (1-\gamma)I_{p/2} + \gamma \mathbf{1}_{p/2}\mathbf{1}_{p/2}^\intercal\right)\, \pi_1(dS_1),
\end{align*}
given the data \(U - \bar{U}\mathbf{1}_{p/2}\). The difficult problem of testing a mixture null against a mixture alternative has been conveniently reduced to the simpler problem of testing a point null against a mixture alternative. Additionally, this testing problem is essentially the same as Problem II in \cite{kotekal_minimax_2023}, with the only notable change being that the dimension has halved to \(\frac{p}{2}\). Since \(\frac{p}{2} - s \gtrsim \sqrt{p}\), the result of \cite{kotekal_minimax_2023} suggests the term \(\dTV(P_{\pi_1}^{II}, P_{\pi_2}^{II})\) is suitably bounded since \(\psi^2 \lesssim (1-\gamma) \frac{p}{2} \log\left(1 + \left(\frac{p}{2}\right)/\left(\frac{p}{2} - s\right)^2\right)\). 

The first term \(\dTV(P_{\pi_1}^{I}, P_{\pi_2}^{I})\) is bounded by symmetry, and so appropriate control over \(\dTV(P_{\pi_1}, P_{\pi_0})\) has been established. In other words, \(\psi^2\) is indeed a minimax lower bound, up to universal constant factors. Combining the \((1-\gamma)p\) lower bound from Proposition \ref{prop:sparse_lbound} with \(\psi^2\) yields the lower bound stated in Proposition \ref{prop:thetabar_lbound}. 

\subsection{\texorpdfstring{Regime \(\frac{p}{2} \leq s \leq p\)}{Regime p/2 <= s <= p}}\label{section:lower_bound_dense}
In the regime \(s \geq \frac{p}{2}\), the lower bound of order \(p\) in (\ref{rate:correlation}) admits a simple proof once the key observation is made. 

\begin{proposition}\label{prop:dense_thetabar_lbound}
    Suppose \(\frac{p}{2} \leq s \leq p\) and \(\gamma \in [0, 1]\). If \(\delta \in (0, 1)\), then there exists \(c_\delta > 0\) depending only on \(\delta\) such that 
    \begin{equation*}
        \inf_{\hat{\theta}} \sup_{||\theta||_0 \leq s} P_{\theta, \gamma}\left\{ ||\hat{\theta} - \theta||^2 \geq c_\delta p \right\} \geq 1 - \delta. 
    \end{equation*}
\end{proposition}

The key observation is the same as that discussed in Section \ref{section:preview}, namely that \(\bar{\theta}\) is not identifiable from \(X - \bar{X}\mathbf{1}_p\) when \(s \geq \frac{p}{2}\). To elaborate, consider the typical approach to proving minimax lower bounds. At a high level, we seek two parameters \(\theta_0\) and \(\theta_1\) such that \(||\bar{\theta}_0\mathbf{1}_p - \bar{\theta}_1\mathbf{1}_p||^2\) is large while \(\dTV(P_{\theta_0, \gamma}, P_{\theta_1, \gamma})\) is small. The principal eigenvector of the covariance matrix is \(\frac{1}{\sqrt{p}}\mathbf{1}_p\), meaning it is most difficult to distinguish two parameters \(\theta_0\) and \(\theta_1\) which exhibit a difference vector \(\theta_0 - \theta_1\) that lies in \(\spn\{\mathbf{1}_p\}\). Le Cam's two point method is used with the choices \(\theta_0 := \frac{c\sqrt{1-\gamma+\gamma p}}{\sqrt{s}}\mathbf{1}_S\) and \(\theta_1 := -\frac{c\sqrt{1-\gamma+\gamma p}}{\sqrt{s}}\mathbf{1}_T\) where \(S = \{1,...,s\}\) and \(T = S^c\). Note that both \(\theta_0\) and \(\theta_1\) are \(s\)-sparse since \(|S| = s\) and \(|T| = p-s \leq \frac{p}{2} \leq s\). The condition \(s \geq \frac{p}{2}\) is critical. With this choice, we importantly have \(\theta_0 - \theta_1 \in \spn\{\mathbf{1}_p\}\) and it can be shown that \(\dTV(P_{\theta_0, \gamma}, P_{\theta_1, \gamma})\) is small. Since \(||\bar{\theta}_0\mathbf{1}_p - \bar{\theta}_1 \mathbf{1}_p||^2 \asymp 1-\gamma+\gamma p\), a lower bound of order \(1-\gamma+\gamma p\) is thus established. Combining with the \((1-\gamma) p\) lower bound from Proposition \ref{prop:sparse_lbound} yields the lower bound claimed in Proposition \ref{prop:dense_thetabar_lbound}.

%%%%%%%%%%%%%%%%%%%%%%%%%%%%%%%%%%%%%%%%%%%%%%
\section{Adaptation to sparsity and correlation}\label{section:adaptation}
%%%%%%%%%%%%%%%%%%%%%%%%%%%%%%%%%%%%%%%%%%%%%%

In this section, estimators adaptive to the sparsity and correlation levels are constructed. At a high level, a pilot estimator is furnished which estimates the order of \(1-\gamma\) with high probability. The pilot estimator is used in the choice of the penalty parameter \(\lambda\) in Section \ref{section:sparse_regression} and the choice of bandwidth \(h\) in Section \ref{section:linear_functional}. To adapt to the sparsity level, Lepski's method is employed. 

\subsection{Correlation estimation}
Since Lepski's method is eventually used for adaptation to the sparsity, it turns out it suffices to only consider correlation estimation in the regime \(s < \frac{p}{2}\). From (\ref{model:additive}), consider for \(i \in \supp(\theta)^c\) we have 
\begin{equation*}
    X_i = \sqrt{\gamma}W + \sqrt{1-\gamma}Z_i.
\end{equation*}
In effect, \(\sqrt{\gamma}W\) is a shared location shift for \(X_i\) such that \(i \in \supp(\theta)^c\). If \(\supp(\theta)^c\) were known, a natural estimator for \(1-\gamma\) would be the sample variance computed from that subset. Of course, \(\supp(\theta)^c\) is unknown. A key observation is that if sample variance were computed on a subset \(D\) of the data such that \(D \cap \supp(\theta) \neq \emptyset\), then it would overestimate \(1-\gamma\) due to the presence of nonzero \(\theta_i\). Conversely, a good choice \(D \subset \supp(\theta)^c\) does not overestimate. Since \(|\supp(\theta)^c| \geq \frac{p}{2}\), one idea for an estimator is 
\begin{equation*}
    \min_{|D| = \left\lceil\frac{p}{2}\right\rceil} \frac{1}{|D| - 1}\sum_{i \in D} (X_i - \bar{X}_D)^2,
\end{equation*}
where \(\bar{X}_D = \frac{1}{|D|}\sum_{i \in D} X_i\). This estimator is essentially that suggested in \cite{collier_optimal_2018}, except with a modification to handle the nuisance location shift \(\sqrt{\gamma}W\). 

Though conceptually clean, it appears an exhaustive search of exponential time is required. The setting of \cite{collier_optimal_2018} enables one to write down a simple polynomial-time algorithm, but the nuisance \(\sqrt{\gamma}W\) here appears to hinder this. However, it turns out the idea can be rescued by random sampling. Independently draw subsets \(E_1,...,E_m \subset \{1,...,p\}\) of size \(\ell\) uniformly at random and define
\begin{equation}\label{def:gamma_hat}
    1 - \hat{\gamma} = \min_{1 \leq r \leq m} \frac{1}{\ell-1} \sum_{i \in E_r} (X_i - \bar{X}_{E_r})^2.
\end{equation}
If \(m\) is chosen to scale polynomially in \(p\), then \(1-\hat{\gamma}\) can be computed in polynomial time. 

\begin{proposition}\label{prop:correlation_estimation}
    Suppose \(1 \leq s < \frac{p}{2}\). Fix \(\delta \in (0, 1)\). There exist constants \(C_1, C_2, L > 0\) depending only on \(\delta\) such that if \(m = \left\lceil p^{C_1}\right\rceil\) and \(2 < \ell < \left\lceil C_2\log p\right\rceil\), then 
    \begin{equation*}
        \inf_{\substack{||\theta||_0 \leq s \\ \gamma \in [0, 1)}} P_{\theta, \gamma}\left\{ L^{-1} \leq \frac{1 - \hat{\gamma}}{1-\gamma} \leq L\right\} \geq 1 - \delta,
    \end{equation*}
    where \(\hat{\gamma}\) is given by (\ref{def:gamma_hat}).
\end{proposition}
The reader should note the definition of \(\hat{\gamma}\) in (\ref{def:gamma_hat}) does not require knowledge of \(s\), but the bound established in Proposition \ref{prop:correlation_estimation} relies on \(s < \frac{p}{2}\). 

\subsection{Adaptive sparse regression}\label{section:lasso_adapt}
A major component to our methodology is working with the decorrelated data \(\widetilde{X}\) given by (\ref{def:Xtilde}). Forming \(\widetilde{X}\) requires knowledge of \(\gamma\) in order to add the right amount of Gaussian noise to \(X - \bar{X}\mathbf{1}_p\). To furnish estimators which adapt to \(\gamma\), we will forgo the decorrelation step and directly work with the correlated data. It turns out the same sparse regression approach of Section \ref{section:sparse_regression} can be employed directly to \(X - \bar{X}\mathbf{1}_p\).

Letting \(Y = \sqrt{p}(X - \bar{X}\mathbf{1}_p)\) and \(M = \sqrt{p}\left(I_p - \frac{1}{p}\mathbf{1}_p\mathbf{1}_p^\intercal\right)\), observe 
\begin{equation}\label{def:correlated_regression}
    Y \sim N\left(M\theta, \sigma^2\left(I_p - \frac{1}{p}\mathbf{1}_p\mathbf{1}_p^\intercal\right)\right),
\end{equation}
where \(\sigma^2 = (1-\gamma) p\). In other words, we are faced with a sparse regression problem with negatively correlated noise. Note Theorem 4.2 of \cite{bellec_slope_2018} is a deterministic result which holds on the event (4.1) in \cite{bellec_slope_2018}. It turns out one can easily show the event (4.1) continues to have high probability even with the correlated noise in (\ref{def:correlated_regression}), and so the sparse regression result goes through. For a choice of tuning parameter \(\lambda > 0\), define 
\begin{equation}\label{def:correlated_beta}
    \hat{\beta} = \argmin_{\beta \in \R^p} \left\{ \frac{1}{p} ||Y - M\beta||^2 + 2\lambda||\beta||_1\right\}.
\end{equation}
To furnish an adaptive estimator, Lepski's method is used. Define the set 
\begin{equation}\label{grid:lasso}
    \mathcal{S} = \left\{2^k : k = 0,1,\ldots,\left\lfloor \log_2\left( \frac{p}{784} \right) \right\rfloor\right\} \cup \left\{p\right\}. 
\end{equation}
For each \(s \in \mathcal{S}\), define 
\begin{equation}\label{eqn:lepski_lambda}
    \hat{\lambda}(s) = 2(4+\sqrt{2}) L_\eta^{1/2} \sqrt{(1-\hat{\gamma}) \log\left(\frac{2ep}{s}\right)},
\end{equation}
where \(\hat{\gamma}\) is the correlation estimator (\ref{def:gamma_hat}) at confidence level \(\eta \in (0, 1)\) and \(L_\eta\) is the corresponding constant from Proposition \ref{prop:correlation_estimation}. Let \(\hat{\beta}(s)\) denote the solution to (\ref{def:correlated_beta}) with the choice of penalty \(\hat{\lambda}(s)\), and define the estimators 
\begin{equation}\label{def:lepski_vs}
    \hat{v}(s) = 
    \begin{cases}
        \frac{1}{\sqrt{p}} M\hat{\beta}(s) &\textit{if } s \leq \frac{p}{784}, \\
        X - \bar{X}\mathbf{1}_p &\textit{if } s > \frac{p}{784}.
    \end{cases}
\end{equation}
The ingredients to which Lepski's method can be applied are now in hand. Using the notation \(B(x, r) = \left\{z \in \R^p : ||x - z|| \leq r\right\}\) for \(x \in \R^p\) and \(r > 0\) to denote the ball of radius \(r\) centered at \(x\), define the estimator \(\hat{v}\) to be any element of the set
\begin{align*}
    \bigcap_{\substack{s \in \mathcal{S} \setminus \{p\} \\ s \geq s'}} B\left(\hat{v}(s), \sqrt{13 s\hat{\lambda}(s)^2}\right),
\end{align*}
where \(s'\) is the smallest value in \(\mathcal{S}\) such that the set is nonempty. If no such \(s'\) exists, set \(\hat{v} = \hat{v}(p)\). Note the computation of \(\hat{v}\) requires no knowledge of the true sparsity nor the true correlation.

\begin{proposition}\label{prop:lasso_lepski}
    Suppose \(1 \leq s^* \leq p\). If \(\eta, \delta \in (0, 1)\) and \(p\) is sufficiently large depending only on \(\delta\) and \(\eta\), then there exists \(C_{\delta, \eta} > 0\) depending only on \(\delta\) and \(\eta\) such that
    \begin{equation*}
        \sup_{\substack{||\theta||_0 \leq s^* \\ \gamma \in [0, 1)}} P_{\theta, \gamma}\left\{ ||\hat{v} - (\theta - \bar{\theta}\mathbf{1}_p)||^2 > C_{\delta, \eta}(1-\gamma) s^*\log\left(\frac{ep}{s^*}\right) \right\} \leq \delta + \eta.  
    \end{equation*}
\end{proposition}
\noindent Hence, the rate achieved in Section \ref{section:sparse_regression} can still be achieved without knowledge of the sparsity level nor the correlation level. 

\subsection{Adaptive linear functional estimation}\label{section:kme_adaptation}
The kernel mode estimator of Section \ref{section:linear_functional} can be directly applied to \(-(X - \bar{X}\mathbf{1}_p)\). For \(h > 0\), define
\begin{equation}\label{def:KME_correlated}
    \hat{G}_h(x) = \frac{1}{2ph} \sum_{i=1}^{p} \mathbbm{1}_{\{|x - (\bar{X} - X_i)| \leq h\}}.
\end{equation}
For a specific choice of \(h\), our estimator is 
\begin{equation*}
    \hat{\mu} = \argmax_{t \in \R} \hat{G}_h(t).
\end{equation*}
Note \(\hat{G}_h\) is now a sum of correlated random variables, and so \(\hat{\mu}\) appears complicated to analyze. It turns out a simple line of reasoning enables the analysis of Section \ref{section:linear_functional} to still be of use. To elaborate, define 
\begin{equation*}
    \tilde{G}_h(x) = \frac{1}{2ph} \sum_{i=1}^{p} \mathbbm{1}_{\{|x - Y_i| \leq h\}},
\end{equation*}
where \(Y\) is given by (\ref{model:Gaussian_contamination}), and define 
\begin{equation}\label{def:oracle_kme}
    \tilde{\mu} = \argmax_{t \in \R} \tilde{G}_h(t).
\end{equation}
Note the results of Section \ref{section:linear_functional} apply to \(\tilde{\mu}\). In a sense, \(\tilde{\mu}\) can be thought of as an oracle estimator, namely since it uses \(Y\) which is obtained as a consequence of knowing \(\gamma\). Consider \(|\hat{\mu} - \bar{\theta}| \leq |\hat{\mu} - \tilde{\mu}| + |\tilde{\mu} - \bar{\theta}|\). The error term \(|\tilde{\mu} - \bar{\theta}|\) is handled by appeals to Section \ref{section:linear_functional}. The first term \(|\hat{\mu} - \tilde{\mu}|\) can be bounded by noting \(Y = (\bar{X} - X_i) + \frac{\sqrt{1-\gamma}}{\sqrt{p}}\xi\) where \(\xi \sim N(0, 1)\). Therefore, \(\hat{G}_h(x) = \tilde{G}_h\left(x + \sqrt{\frac{1-\gamma}{p}}\xi\right)\). Consequently, for every maximizer of \(\hat{G}_h\), there exists a maximizer of \(\tilde{G}_h\) which is located away of distance exactly \(\sqrt{\frac{1-\gamma}{p}}|\xi|\). Therefore,
\begin{equation}\label{def:kme_bypass_decorrelation}
    |\hat{\mu} - \bar{\theta}| \leq \sqrt{\frac{1-\gamma}{p}}|\xi| + |\tilde{\mu} - \bar{\theta}|. 
\end{equation}
Since \(|\xi| \lesssim 1\) with high probability, the first term turns out to be negligible. Hence, the order of the error of \(\hat{\mu}\), in which we avoided explicit decorrelation, can be bounded by order of the error of \(\tilde{\mu}\), which is obtained after decorrelation. Therefore, it suffices to consider \(\tilde{\mu}\) when investigating adaptation to \(s\) and \(\gamma\). As mentioned earlier, Lepski's method will be employed to adapt to the sparsity level. For use in the method, an estimator which knows \(s\) but adapts to \(\gamma\) is needed.
\begin{theorem}\label{thm:kme_unknown_var}
    Suppose \(1 \leq s < \frac{p}{2}\) and suppose \(\tilde{C} > 0\) is a sufficiently large universal constant. Suppose \(\delta, \eta \in (0, 1)\). If either \(s \geq \frac{p}{2} - p^{1/4}\) and \(p \geq (2/\delta)^2\) or \(s < \frac{p}{2} - p^{1/4}\) and \(p \geq \tilde{C}\log^{16}(1/\delta)\), then the choice
    \begin{equation*}
        \hat{h} = C_1 \sqrt{(1-\hat{\gamma})\left(1 \vee \left(\log\left(\frac{ep}{(p-2s)^2}\right) + \log\log\left(\frac{1}{\delta}\right)\right)\right)},
    \end{equation*}
    with \(C_1\) sufficiently large depending only on \(\eta\) yields  
    \begin{equation*}
        \sup_{\substack{||\theta||_0 \leq s \\ \gamma \in [0, 1)}} P_{\theta, \gamma}\left\{ \frac{|\tilde{\mu} - \bar{\theta}|}{\sqrt{1-\gamma}} > C \sqrt{1 \vee \left(\log\left(\frac{ep}{(p-2s)^2}\right) + \log\log\left(\frac{1}{\delta}\right)\right)} \right\} \leq \delta + \eta,
    \end{equation*}
    where \(C > 0\) is a constant depending only on \(\eta\). Here, \(\hat{\gamma}\) is the correlation estimator in Proposition \ref{prop:correlation_estimation} at confidence level \(\eta\) and \(\tilde{\mu}\) is given by (\ref{def:oracle_kme}) with the (random) bandwidth \(\hat{h}\).
\end{theorem}
\noindent The proof is largely the same as in the case of known \(\gamma\). The only difference is that the bandwidth is now random, but this is easily accommodated. 

We can now apply Lepski's method to adapt to the sparsity level. Let \(\hat{\gamma}\) denote the correlation estimator (\ref{def:gamma_hat}) at confidence level \(\eta \in (0, 1)\) and let \(L_\eta\) be the corresponding constant from Proposition \ref{prop:correlation_estimation}. Define the set 
\begin{align*}
    \mathcal{S} = &\left\{ 2^k : k = 0,1,2,...,\left\lfloor \log_2\left(\frac{p}{784}\right)\right\rfloor\right\} \\
	&\;\;\;\; \cup \left\{\left\lfloor \frac{p}{2} \right\rfloor - 2^k : k =\lceil \log_2 p^{1/4} \rceil, \lceil\log_2 p^{1/4}\rceil + 1,\ldots, \lfloor \log_2\sqrt{p}\rfloor\right\} \cup \left\{\left\lfloor \frac{p}{2} \right\rfloor - 1, p\right\}. 
\end{align*}
For \(s \leq \frac{p}{784}\), sparse regression will be used. Specifically, let \(\hat{\lambda}(s)\) be given by (\ref{eqn:lepski_lambda}) and let \(\hat{\beta}(s)\) denote the solution to (\ref{def:correlated_beta}) with the choice of penalty \(\hat{\lambda}(s)\). The kernel mode estimator will be used for other sparsity levels. For each \(s \in \mathcal{S}\), define the confidence level 
\begin{equation*}
    \delta_s = 
    \begin{cases}
        \exp\left(-K \left(\frac{p}{(p-2s)^2} \wedge p^{1/32} \right)\right) &\textit{if } s \leq \left\lfloor \frac{p}{2} \right\rfloor - 2^{\lceil \log_2p^{1/4}\rceil}, \\
        e^{-K} &\textit{otherwise},
    \end{cases}
\end{equation*}
where \(K > 0\). For each \(s \in \mathcal{S}\), define the bandwidths
\begin{equation*}
    \hat{h}(s) = C_1L_\eta^{1/2}\sqrt{(1-\hat{\gamma}) \left(1 \vee \left(\log\left(\frac{ep}{(p-2s)^2}\right) + \log\log\left(\frac{1}{\delta_s}\right)\right)\right)},
\end{equation*}
where \(C_1\) is the constant depending only on \(\eta\) from Theorem \ref{thm:kme_unknown_var}. Define the estimators for \(s \in \mathcal{S}\), 
\begin{equation*}
    \hat{T}(s) =
    \begin{cases}
        \frac{1}{p}\sum_{i=1}^{p} \hat{\beta}_i(s) &\textit{if } s \leq \frac{p}{784},\\
        \argmax_{t \in \R} \hat{G}_{\hat{h}(s)}(t) &\textit{if } \frac{p}{784} < s < \frac{p}{2} \text{ and } 1-\hat{\gamma} \leq C_\eta\log^{-1}\left(\frac{ep}{(p-2s)^2}\right),\\
        \bar{X} &\textit{otherwise},
    \end{cases}
\end{equation*}
where \(\hat{G}_h\) is given by (\ref{def:KME_correlated}) and \(C_\eta > 0\) is sufficiently large depending only on \(\eta\). Note \(\bar{X}\) is not used when \(s \leq \frac{p}{784}\) since it is not needed to achieve the minimax rate (see Section \ref{section:upper_bound} for discussion). Define the radii
\begin{equation*}
    r(s) = 
    \begin{cases}
        R_{K, \eta} \sqrt{\frac{s}{p}\hat{\lambda}(s)^2}&\textit{if } s \leq \frac{p}{784}, \\
        R_{K, \eta}\hat{h}(s) &\textit{if } \frac{p}{784} < s < \frac{p}{2} \text{ and } 1-\hat{\gamma} \leq C_\eta\log^{-1}\left(\frac{ep}{(p-2s)^2}\right), \\
        R_{K, \eta} &\textit{otherwise},
    \end{cases}
\end{equation*}
where \(R_{K, \eta} > 0\) is sufficiently large depending only on \(\eta\) and \(K\). Now define the adaptive estimator \(\hat{T}\) to be an element of the set 
\begin{equation*}
    \bigcap_{\substack{s \in \mathcal{S} \\ s \geq s'}} B\left(\hat{T}(s), r(s)\right)
\end{equation*}
where \(s'\) is the smallest value in \(\mathcal{S}\) such that the set is nonempty. If no such \(s'\) exists, set \(\hat{T} = \hat{T}(p)\). 

\begin{proposition}\label{prop:linear_lepski}
    Suppose \(1 \leq s^* \leq p\). If \(\eta \in (0, 1)\) and \(p\) is sufficiently large depending only on \(K\), then there exists \(C_1, C_\eta > 0\) depending only on \(\eta\) and \(R_{K, \eta}\) depending only on \(K\) and \(\eta\) such that 
    \begin{equation*}
        \sup_{\substack{||\theta||_0 \leq s^* \\ \gamma \in [0, 1)}} P_{\theta, \gamma}\left\{ \left|\hat{T} - \bar{\theta}\right| > C_{K, \eta} \epsilon(p, s^*, \gamma) \right\} \leq \eta + \frac{5}{e^{K/4}-1},
    \end{equation*}
    where \(C_{K, \eta} > 0\) depends only on \(K\) and \(\eta\), and 
    \begin{equation*}
        \epsilon(p, s, \gamma)^2 = 
        \begin{cases}
            (1-\gamma) \frac{s}{p}\log\left(\frac{ep}{s}\right) &\textit{if } s \leq \frac{p}{2} - \sqrt{p}, \\
            (1-\gamma)\log\left(\frac{ep}{(p-2s)^2}\right) \wedge 1 &\textit{if } \frac{p}{2} - \sqrt{p} < s < \frac{p}{2}, \\
            1 &\textit{if } \frac{p}{2} \leq s \leq p. 
        \end{cases}
    \end{equation*}
\end{proposition}
\noindent Hence, the adaptive estimator constructed by Lepski's method achieves the same rate in (\ref{section:linear_functional}). Consequently, \(\hat{T}\) can be combined with the adaptive estimator \(\hat{v}\) developed in Section \ref{section:lasso_adapt} to achieve the minimax rate for estimating \(\theta\) without knowledge of the true correlation nor sparsity level.

\subsection{An adaptive procedure}

Sections \ref{section:lasso_adapt} and \ref{section:kme_adaptation} give adaptive estimators for the components \(\theta - \bar{\theta}\mathbf{1}_p\) and \(\bar{\theta}\). The two pieces can be directly combined to furnish an adaptive estimator of \(\theta\), as stated in the following theorem without proof. 

\begin{theorem}
    Suppose \(1 \leq s^* \leq p\). If \(\delta \in (0, 1)\) and \(p\) is sufficiently large depending only on \(\delta\), then 
    \begin{equation*}
        \sup_{\substack{||\theta||_0 \leq s^* \\ \gamma \in [0, 1)}} P_{\theta, \gamma}\left\{||\hat{v} + \hat{T}\mathbf{1}_p - \theta||^2 \geq C_{\delta} \varepsilon^*(p, s^*, \gamma)^2\right\} \leq \delta
    \end{equation*}
    where \(C_{\delta} > 0\) depends only on \(\delta\), \(\varepsilon^*(p, s, \gamma)^2\) is given by (\ref{rate:correlation}), \(\hat{v}\) is given by Proposition \ref{prop:lasso_lepski}, and \(\hat{T}\) is given by Proposition \ref{prop:linear_lepski}.
\end{theorem}

\noindent Hence, simultaneous adaptation to the sparsity and the correlation levels is possible. 

%%%%%%%%%%%%%%%%%%%%%%%%%%%%%%%%%%%%%%%%%%%%%%
\section{Discussion}
%%%%%%%%%%%%%%%%%%%%%%%%%%%%%%%%%%%%%%%%%%%%%%

A couple of finer points are explored in this section.

\subsection{Large-scale inference}\label{section:LSI}
As mentioned in Remark \ref{remark:large-scale_inference}, the minimax rate (\ref{rate:LSI}) for estimation of the signal \(\theta\) in the two-groups model (\ref{model:two-groups}) can be directly obtained from our results. To fix notation, we will take \(\theta \in \R^p\) with the understanding \(\mathcal{H}_0 = \supp(\theta)^c\) in (\ref{model:two-groups}) and \(||\theta||_0 \leq s\). For the sake of discussion, we will assume \(s\) and \(\sigma^2\) are known, as adaptation to the sparsity and variance can be straightforwardly addressed via ideas from Section \ref{section:adaptation}. With these preliminaries in place, we have the representation 
\begin{equation}\label{model:additive_LSI}
    X_i = \mu + \theta_i + \sigma Z_i
\end{equation}
where \(Z_1,\ldots, Z_p \overset{iid}{\sim} N(0, 1)\). Clearly the model (\ref{model:additive}) is exactly (\ref{model:additive_LSI}) with \(\sigma^2 = 1-\gamma\) and the prior \(\mu \sim N(0, \gamma)\). The same approach of estimating \(\theta - \bar{\theta}\mathbf{1}_p\) and \(\bar{\theta}\mathbf{1}_p\) separately can be taken. In Appendix \ref{appendix:LSI}, we describe in more detail how the rate (\ref{rate:LSI}) can be established.

\subsection{Correlation causes impossibility of adaptation in expectation}
The success of adaptation seen in Section \ref{section:adaptation} may seem odd to some readers. They may intuit it is not known whether \(s < \frac{p}{2}\), in which case \(\bar{\theta}\) can be identified from \(X-\bar{X}\mathbf{1}_p\) and can be exclusively used to achieve a faster rate, or \(s \geq \frac{p}{2}\), in which case \(\bar{\theta}\) cannot be identified from \(X - \bar{X}\mathbf{1}_p\) and its exclusive use may result in unbounded risk. The keen reader's intuition turns out to be correct when considering adaptation in expectation, but not so when considering adaptation in probability. In other words, adaptation in probability is possible as seen in Section \ref{section:adaptation}, but adaptation in expectation is not possible. 

Generally speaking, the key difference between estimation in expectation and in probability is the following. Successful estimation in probability allows the possibility of an estimator to have unbounded risk on an event of small probability; in other words, good risk only needs to be achieved on a high-probability event. In contrast, successful estimation in expectation does not allow the estimator to have unbounded risk on the bad event; rather, the risk must be suitably bounded so it can be canceled out by the small probability of the bad event. In the context of adaptation, if the true sparsity satisfies \(s^* \geq \frac{p}{2}\), and an adaptive estimator only makes a mistake in using a strategy designed for \(s < \frac{p}{2}\) on a small probability event, then it may very well achieve adaptation in probability while failing to adapt in expectation. 

The following result rigorously states the impossibility of adaptation in expectation. 

\begin{theorem}\label{thm:adapt_in-exp}
    Suppose \(1 \leq s < p\) and \(\gamma \in [0, 1]\). There exist two universal constants \(C_0, C_1 > 0\) such that the following holds. For any \(r \geq 1\), if \(\hat{\theta}\) is an estimator such that 
    \begin{equation*}
        \sup_{||\theta||_0 \leq s} E_{\theta, \gamma}\left(||\hat{\theta} - \theta||^2\right) \leq C_0(1-\gamma+\gamma p)e^{-C_1r},
    \end{equation*}
    then 
    \begin{equation*}
        \sup_{||\theta||_0 \leq p} E_{\theta, \gamma}\left(||\hat{\theta} - \theta||^2\right) \geq C_0 (1-\gamma+\gamma p)r.
    \end{equation*}
\end{theorem}
Theorem \ref{thm:adapt_in-exp} is proved in Appendix \ref{section:misc}. To illustrate Theorem \ref{thm:adapt_in-exp}, consider \(s = 1\) in which the minimax (squared) rate in expectation is \((1-\gamma)\log(ep)\). If \(\hat{\theta}\) is an estimator which achieves the minimax rate for \(s = 1\), then Theorem \ref{thm:adapt_in-exp} with \(r \asymp 1 \vee \log\left(\frac{1-\gamma+\gamma p}{(1-\gamma) \log(ep)}\right)\) implies \(\hat{\theta}\) must satisfy 
\begin{equation*}
    \sup_{||\theta||_0 \leq p} E_{\theta, \gamma}\left(||\hat{\theta} - \theta||^2\right) \gtrsim (1-\gamma+\gamma p) \left(1 \vee \log\left(\frac{1-\gamma+\gamma p}{(1-\gamma)\log(ep)}\right)\right).
\end{equation*}
At \(s = p\), the minimax (squared) rate in expectation is \(p\). Clearly, \(\hat{\theta}\) suffers a worse rate if \(\gamma\) is sufficiently strong. For example, if \(\gamma \asymp 1\), then \(\sup_{||\theta||_0 \leq p} E_{\theta, \gamma}\left(||\hat{\theta} - \theta||^2\right) \gtrsim p \log\left(\frac{p}{1-\gamma}\right)\). Notably, the impossibility of adaptation in expectation is a phenomenon which appears only with the presence of correlation. In particular, if \(\gamma = 0\) then Theorem \ref{thm:adapt_in-exp} has no content since the upper bound condition \(C_0(1-\gamma+\gamma p)e^{-C_1 r} = C_0 e^{-C_1 r} \lesssim 1\), and no estimator achieves this rate since the minimax rate is \(\log(ep)\). Indeed, it is well known adaptation in expectation is possible in the independent setting (i.e. \(\gamma = 0\)) \cite{birge_gaussian_2001}, and Theorem \ref{thm:adapt_in-exp} poses no contradiction. In general, Theorem \ref{thm:adapt_in-exp} can only be meaningfully applied when the upper bound condition is of order at least the minimax rate, since no estimator can achieve risk of smaller order by definition. Similarly, Theorem \ref{thm:adapt_in-exp} only delivers a nontrivial conclusion when \((1-\gamma+\gamma p)r \gtrsim p\). 

Theorem \ref{thm:adapt_in-exp} thus highlights an interesting consequence of correlation. Furthermore, it identifies a problem in which adaptation in probability is actually different from adaptation in expectation, which is a phenomenon which has received very little attention in the statistics literature; to the best of our knowledge, only \cite{cai_adaptation_2006} systematically investigates the difference between in-expectation and in-probability adaptation. Furthermore, \cite{cai_adaptation_2006} only shows adaptation in probability for linear functional estimation is possible when adapting to a bounded number of convex classes. Our result in Section \ref{section:adaptation} establishes adaptation is, in fact, possible when adapting to a growing number of classes which are actually unions of convex classes. While our work contributes to the interesting difference between in-probability and in-expectation adaptation, Theorem \ref{thm:adapt_in-exp} only asserts adaptation is not possible. It is an intriguing open problem to sharply characterize the exact cost for adaptation in expectation.

\subsection{Other correlation structures}\label{section:other_correlation}
A natural generalization of (\ref{model:additive}) is the model with multiple random effects. Consider the vector form analogous to (\ref{model:vector}), 
\begin{equation*}
    X \sim N\left(\theta, (1-\gamma)I_p + \gamma\sum_{k=1}^{R} \mathbf{1}_{B_k}\mathbf{1}_{B_k}^\intercal\right),
\end{equation*} 
where \(B_1,...,B_R\) are known sets of cardinality \(\frac{p}{R}\) (assumed to be an integer) which form a partition of \(\{1,...,p\}\). Here, \(\mathbf{1}_{B_k} \in \R^p\) is the vector with coordinate \(i\) equal to \(\mathbbm{1}_{\{i \in B_k\}}\). The minimax testing rate for the sparse signal detection problem in this model was obtained in \cite{kotekal_minimax_2023} for all configurations of \(p, s, \gamma\), and \(R\). As evidenced by the results of \cite{kotekal_minimax_2023}, new phase transitions emerge for various regimes of \(R\); the sharp minimax testing rate thus exhibits a fairly intricate dependence on \(R\). 

For the problem of estimating a sparse \(\theta\), the methodology developed in this article can be extended to obtain the sharp rate when \(R \lesssim 1\), namely 
\begin{equation*}
    (1-\gamma) s \log\left(\frac{ep}{s \wedge \left(\frac{p}{2R} - s\right)_{+}^2}\right) \wedge p. 
\end{equation*}
In Appendix \ref{appendix:other_correlation}, we discuss the intuition for how this can be established. The case of growing \(R\) is also discussed there. 

\bibliographystyle{skotekal.bst}
\bibliography{sparse_estimation.bib} 

\begin{appendix}

    % \begin{center}
	% 	{\LARGE Supplement to ``Sparsity meets correlation in Gaussian sequence model"		
	% 	}		
	% 	\medskip
		
	% 	{\large Subhodh Kotekal \quad and\quad Chao Gao}
	% 	\medskip
	% \end{center}
	
	% %\begin{abstract}
	% In this supplement, we provide a table of contents and the proofs.
	% %\end{abstract}
	% \tableofcontents

    \section{Continued discussion}\label{appendix:discussion}

    \subsection{Large-scale inference}\label{appendix:LSI}
    We continue the discussion from Section \ref{section:LSI}, namely the estimation of \(\theta\) from the model (\ref{model:additive_LSI}). Consider the transformation \(\widetilde{X} = X - \bar{X}\mathbf{1}_p + \sigma p^{-1/2} \xi \mathbf{1}_p\) where \(\xi \sim N(0, 1)\) is drawn independently of the data. Note \(\widetilde{X} \sim N(\theta - \bar{\theta}\mathbf{1}_p, \sigma^2I_p)\). With the choice \(\lambda = 2(4+\sqrt{2})\sigma \sqrt{\log\left(\frac{2ep}{s}\right)}\), let 
    \begin{equation}\label{eqn:LSI_betahat}
        \hat{\beta} = \argmin_{\beta \in \R^p} \left\{\frac{1}{p} \left|\left|\sqrt{p}\widetilde{X} - \sqrt{p}\left(I_p - \frac{1}{p}\mathbf{1}_p\mathbf{1}_p^\intercal\right)\beta\right|\right|^2 + 2\lambda ||\beta||_1\right\}. 
    \end{equation}
    For estimating \(\theta - \bar{\theta}\mathbf{1}_p\), define the estimator
    \begin{equation}\label{eqn:LSI_vhat}
        \hat{v} :=
        \begin{cases}
            \left(I_p - \frac{1}{p}\mathbf{1}_p\mathbf{1}_p^\intercal\right) \hat{\beta} &\textit{if } 1 \leq s \leq \frac{p}{784}, \\
            \widetilde{X} &\textit{if } \frac{p}{784} < s < \frac{p}{2}.
        \end{cases}
    \end{equation}
    For estimation of \(\bar{\theta}\), if \(1 \leq s \leq \frac{p}{784}\), then \(\frac{1}{p}\sum_{i=1}^{p}\hat{\beta}_i\) will be used. Otherwise, a kernel mode estimator will be used. Define \(h = C_1\sqrt{\sigma^2 \left(1 \vee \log\left(\frac{L_\delta p}{(p-2s)^2}\right)\right)}\) for a confidence level \(\delta\) with \(C_1\) and \(L_\delta\) selected as in Theorem \ref{thm:kme_known_var}. Consider the kernel mode estimator 
    \begin{equation}\label{eqn:LSI_muhat}
        \hat{m} = \argmax_{t \in \R} \frac{1}{2ph} \sum_{i=1}^{p} \mathbbm{1}_{\left\{\left|t - \left(-\widetilde{X}_i\right)\right| \leq h \right\}}.
    \end{equation}
    Note \(\bar{X} \sim N(\bar{\theta} + \mu, \sigma^2p^{-1})\) is useless for estimating \(\bar{\theta}\) since \(\mu \in \R\) is completely unknown and is thus an unbounded nuisance. In contrast, there is information about \(\mu\) (e.g. \(|\mu| \lesssim \gamma\) with high probability) in model (\ref{model:additive}), and so \(\bar{X}\) is utilized in Section \ref{section:linear_functional}. As in Section \ref{section:upper_bound}, the two pieces can be combined to obtain the following result which we state without proof.
    \begin{theorem}\label{thm:LSI}
        Suppose \(1 \leq s < \frac{p}{2}\) and \(\sigma > 0\). Let \(\hat{v}\) be given by (\ref{eqn:LSI_vhat}) and set 
        \begin{equation*}
            \hat{T} = 
            \begin{cases}
                \frac{1}{p}\sum_{i=1}^{p} \hat{\beta}_i &\textit{if } 1 \leq s \leq \frac{p}{784},\\
                \hat{m} &\textit{if } \frac{p}{784} < s < \frac{p}{2}, 
            \end{cases}
        \end{equation*}
        where \(\hat{\beta}\) is given by (\ref{eqn:LSI_betahat}) and \(\hat{m}\) is given by (\ref{eqn:LSI_muhat}). If \(\delta \in (0, 1)\) and \(p\) is sufficiently large depending only on \(\delta\), then 
        \begin{equation*}
            \sup_{\substack{||\theta||_0 \leq s \\ \mu \in \R}} P_{\theta, \mu, \sigma}\left\{||\hat{v} + \hat{T}\mathbf{1}_p - \theta||^2 > C_\delta \varepsilon^*(p, s, \sigma)^2\right\} \leq \delta
        \end{equation*}
        where \(C_\delta > 0\) depends only on \(\delta\) and \(\varepsilon^*(p, s, \sigma)^2\) is given by (\ref{rate:LSI}).
    \end{theorem}
    
    The minimax lower bound can established leveraging the results of Section \ref{section:lower_bounds}. In the model (\ref{model:additive_LSI}), consider placing the prior \(N(0, \tau^2)\) on \(\mu\), so we can write \(\mu = \tau W\) for \(W \sim N(0, 1)\). After normalization, we can write 
    \begin{equation*}
        \frac{X_i}{\sqrt{\tau^2 + \sigma^2}} = \frac{\theta_i}{\sqrt{\tau^2 + \sigma^2}} + \frac{\tau}{\sqrt{\tau^2 + \sigma^2}} W + \frac{\sigma}{\sqrt{\tau^2 + \sigma^2}} Z_i.
    \end{equation*}
    This is precisely model (\ref{model:additive}) with \(\gamma = \frac{\tau}{\sqrt{\tau^2 + \sigma^2}}\). All of the results of Section \ref{section:lower_bounds} can now be applied to obtain the following result, which we state without proof.
    \begin{theorem}\label{thm:LSI_lower_bound}
        Suppose \(1 \leq s \leq p\) and \(\sigma > 0\). For any \(\tau > 0\), if \(\delta \in (0, 1)\) and \(p\) is sufficiently large depending only on \(\delta\), then there exists \(c_\delta > 0\) depending only on \(\delta\) such that
        \begin{equation*}
            \inf_{\hat{\theta}} \sup_{\substack{||\theta||_0 \leq s \\ \mu \in \R}} P_{\theta, \mu, \sigma} \left\{ ||\hat{\theta} - \theta||^2 \geq c_\delta (\tau^2 + \sigma^2)\left( \frac{\sigma^2}{\tau^2 + \sigma^2} s \log\left(\frac{ep}{s}\right)\right)\right\} \geq 1-\delta.
        \end{equation*}
        Furthermore, if \(\frac{p}{2} - \sqrt{p} \leq s < \frac{p}{2}\) then 
        \begin{equation*}
            \inf_{\hat{\theta}} \sup_{\substack{||\theta||_0 \leq s \\ \mu \in \R}} P_{\theta, \mu, \sigma} \left\{ ||\hat{\theta} - \theta||^2 \geq c_\delta (\tau^2 + \sigma^2)\left( \frac{\sigma^2}{\tau^2 + \sigma^2} p \log\left(1 + \frac{p}{(p-2s)^2}\right) \wedge p\right)\right\} \geq 1-\delta
        \end{equation*}
        and if \(s \geq \frac{p}{2}\) then 
        \begin{equation*}
            \inf_{\hat{\theta}} \sup_{\substack{||\theta||_0 \leq s \\ \mu \in \R}} P_{\theta, \mu, \sigma} \left\{ ||\hat{\theta} - \theta||^2 \geq c_\delta (\tau^2 + \sigma^2)p\right\} \geq 1-\delta.
        \end{equation*}
    \end{theorem}
    \noindent Theorem \ref{thm:LSI_lower_bound} establishes a family of lower bounds indexed by \(\tau\), namely
    \begin{equation*}
        \varepsilon^*(p, s, \sigma)^2 \gtrsim 
        \begin{cases}
            \sigma^2 s \log\left(\frac{ep}{s}\right) &\textit{if } 1 \leq s \leq \frac{p}{2} - \sqrt{p}, \\
            \sigma^2 p \log\left(1 + \frac{p}{(p-2s)^2}\right) \wedge (\tau^2 + \sigma^2)p &\textit{if } \frac{p}{2} - \sqrt{p} < s < \frac{p}{2}, \\
            (\tau^2 + \sigma^2) p &\textit{if } \frac{p}{2} \leq s \leq p.
        \end{cases}
    \end{equation*}
    Taking \(\tau \to \infty\) establishes the lower bound (\ref{rate:LSI}). Hence, the minimax rate for estimating \(\theta\) in (\ref{model:additive_LSI}) has been derived from a straightforward application of our results for (\ref{model:additive}).
    
    \subsection{Other correlation structures}\label{appendix:other_correlation}
    Here, we continue the discussion from Section \ref{section:other_correlation}. The rate for estimating the signal \(\theta\) in the \(R\) random effects setting with \(R \lesssim 1\) can be intuited as follows. The data on each block \(B_k\) can be viewed as an independent draw from the equicorrelation model (\ref{model:vector}) with ambient dimension \(\frac{p}{R}\). Thus, a simple idea to estimate \(\theta\) is to employ our methodology on each block separately, and it can be checked this gives the minimax upper bound. The minimax lower bound can be obtained by just focusing on one block and repeating the construction for the equicorrelated case but now in ambient dimension \(\frac{p}{R}\). Focusing on just one block is not suboptimal since \(R \lesssim 1\) implies \(\frac{p}{R} \asymp p\).
    
    It is more interesting to consider potentially growing \(R\), but this simple strategy of treating each block separately turns out to be suboptimal. To see why, consider the special case \(\gamma = 1\). If \(s < \frac{p}{2R}\), then the signal on each block \(B_k\) can be estimated perfectly as described in Section \ref{section:preview}. On the other hand, if \(s \geq \frac{p}{2R}\), then the squared error \(\frac{p}{2R}\) is incurred on each block. In other words, the estimator achieves \(||\hat{\theta} - \theta||^2 = 0\) if \(s < \frac{p}{2R}\) and \(||\hat{\theta} - \theta||^2 \asymp p\) if \(s \geq \frac{p}{2R}\). But it is clearly possible to do much better if \(s \geq \frac{p}{2R}\). For example, if \(R = p\), then the model actually reduces to the usual sparse Gaussian sequence model \(X \sim N(\theta, I_p)\). The minimax rate is well-known to be \(s\log\left(\frac{ep}{s}\right)\), yet our naive estimator only achieves \(p\). The key issue of treating each block separately is that it does not exploit the fact that sparsity \(||\theta||_0 \leq s\) is a global constraint across \emph{all} coordinates. A more clever estimation strategy is needed. We imagine our methodology might be used to construct pilot estimators for each block, which are then fed into a higher level aggregation scheme to exploit the global sparsity condition to achieve the sharp rate. 
    
    It is also interesting to consider spike correlation models in general. For example, the ``rank-one'' correlation model \(X \sim N(\theta, (1-\gamma)I_p + \gamma vv^\intercal)\) was considered in \cite{kotekal_minimax_2023}, where partial results for the testing problem were obtained. The methodology developed in this article is quite specific to the case \(v = \mathbf{1}_p\). The idea to cast estimation of \(\bar{\theta}\) as a robust statistics problem from the data \(X - \bar{X}\mathbf{1}_p\) really relies on the nice structure of \(v = \mathbf{1}_p\), and it is not clear to us that our methodology could be broadly generalized to handle generic \(v\). Furthermore, the minimax rate appears to depend delicately on \(v\) and it may be the case that estimators should be tailored to the correlation direction. In view of the challenges the spiked model presents, developing rate optimal estimators in even more general correlation structures appears to be a very complicated problem. 
    
    \section{Proofs}
    \subsection{Sparse regression}\label{section:sparse_regression_proofs}
    In this section, we prove Theorem \ref{thm:orthog_ubound} and Proposition \ref{prop:small_s_linear}. The proofs are a direct application of the sparse regression results of \cite{bellec_slope_2018}. Before presenting the argument, some preliminary definitions are needed. 
    
    \begin{definition}\label{def:sre}
        The design matrix \(\mathbb{X} \in \R^{n \times p}\) is said to satisfy the \(\SRE(s, c_0)\) condition if \(\frac{1}{n}||\mathbb{X}e_j||^2 \leq 1\) for all \(1 \leq j \leq p\) and 
        \begin{equation*}
            \kappa(s, c_0)^2 := \inf_{\substack{\Delta \in \mathcal{C}_{\SRE}(s, c_0) \\ \Delta \neq 0}} \frac{\frac{1}{n} ||\mathbb{X} \Delta||^2}{||\Delta||^2} > 0,
        \end{equation*}
        where \(\mathcal{C}_{\SRE}(s, c_0) := \left\{ \Delta \in \R^p : ||\Delta||_1 \leq (1+c_0)\sqrt{s}||\Delta||_2 \right\}\).
    \end{definition}
    
    \begin{definition}\label{def:sparse_eigenvalue}
        The design matrix \(\mathbb{X} \in \R^{n \times p}\) is said to satisfy the \(s\)-sparse eigenvalue condition if \(\frac{1}{n}||\mathbb{X}e_j||^2 \leq 1\) for all \(1 \leq j \leq p\) and 
        \begin{equation*}
            \psi_{\min}(\mathbb{X}, s)^2 := \min_{\substack{\Delta \in \R^p \setminus \{0\} \\ ||\Delta||_0 \leq s}} \frac{\frac{1}{n}||\mathbb{X}\Delta||^2}{||\Delta||^2} > 0.
        \end{equation*}
    \end{definition}
    
    \begin{proposition}[Proposition 8.1 Part (iii) in \cite{bellec_slope_2018}]\label{prop:sparse_eigenvalue}
        Let \(\psi_1 > 0, c_0 > 0,\) and \(1 \leq s \leq p\). If the \(s\)-sparse eigenvalue condition holds with \(\psi_{\min}(\mathbb{X}, s) \geq \psi_1\), then the \(\SRE(s_1, c_0)\) condition holds and \(\kappa(s_1, c_0) \geq \psi_1/\sqrt{2}\) for \(s_1 \leq (s-1)\psi_1^2/(2c_0^2)\). 
    \end{proposition}

    With Definitions \ref{def:sre} and \ref{def:sparse_eigenvalue} in hand, the following sparse regression results were obtained by \cite{bellec_slope_2018}.
    
    \begin{theorem}[Corollary 4.3 of \cite{bellec_slope_2018}]\label{thm:bellec_prob}
        Let \(1 \leq s \leq p\). Suppose \(y \sim N(f, \sigma^2 I_n)\) and 
        \begin{equation*}
            \hat{\beta} \in \argmin_{\beta \in \R^p} \left\{ \frac{1}{n} ||\mathbb{X}\beta - y||^2 + 2\lambda ||\beta||_1\right\}. 
        \end{equation*}
        Assume the \(\SRE(s, 7)\) condition holds. Let \(\lambda \geq 2(4+\sqrt{2})\sigma \sqrt{\frac{\log(2ep/s)}{n}}\). Then, with probability at least \(1 - \frac{1}{2}\left(\frac{s}{2ep}\right)^{\frac{s}{\kappa(s, 7)^2}}\), we have 
        \begin{equation*}
            \frac{1}{n}||\mathbb{X}\hat{\beta} - f||^2 \leq \min_{||\beta||_0 \leq s} \frac{1}{n}||\mathbb{X}\beta - f||^2 + \frac{49\lambda^2 s}{16\kappa(s, 7)^2}. 
        \end{equation*}
        Moreover, if \(f = \mathbb{X}\beta^*\) for some \(\beta^* \in \R^p\) with \(||\beta^*||_0 \leq s\), then 
        \begin{equation*}
            P\left\{ ||\hat{\beta} - \beta^*||^2 \leq \frac{(49)^2 \lambda^2 s}{8^2 \kappa(s, 7)^{4}}\right\} \geq 1 - \frac{1}{2}\left(\frac{s}{2ep}\right)^{\frac{s}{\kappa(s, 7)^2}}. 
        \end{equation*}
    \end{theorem}
    \begin{theorem}[Corollary 4.4 of \cite{bellec_slope_2018}]\label{thm:bellec}
        Let \(1 \leq s \leq p\). Suppose \(y \sim N(f, \sigma^2 I_n)\) and 
        \begin{equation*}
            \hat{\beta} \in \argmin_{\beta \in \R^p} \left\{ \frac{1}{n}||\mathbb{X}\beta - y||^2 + 2\lambda ||\beta||_1 \right\}. 
        \end{equation*}
        Assume the \(\SRE(s, 7)\) condition holds. Let \(\lambda \geq 2(4+\sqrt{2})\sigma \sqrt{\frac{\log(2ep/s)}{n}}\). Then 
        \begin{equation*}
            \frac{1}{n} E\left(||\mathbb{X}\hat{\beta} - f||^2\right) \leq \min_{||\beta||_0 \leq s} \frac{1}{n}||\mathbb{X} \beta - f||^2 + \frac{49\lambda^2 s}{16} \left( \frac{1}{\kappa(s, 7)^2} + \frac{1}{2\log(2ep)} \right).
        \end{equation*}
        Moreover, if \(f = \mathbb{X}\beta^*\) for some \(\beta^* \in \R^p\) with \(||\beta^*||_0 \leq s\), then 
        \begin{equation*}
            E\left(||\hat{\beta} - \beta^*||^2\right) \leq \frac{(49)^2\lambda^2 s}{8^2} \left(\frac{1}{\kappa(s, 7)^{4}} + \frac{1}{(\log(2ep))^2}\right).
        \end{equation*}
    \end{theorem}

    Theorem \ref{thm:bellec} will be applied to obtain the error bound in Proposition \ref{prop:lasso_error}. Recall that in Section \ref{section:sparse_regression}, we have defined
    \begin{align*}
        M &= \sqrt{p}\left(I_p - \frac{1}{p}\mathbf{1}_p\mathbf{1}_p^\intercal\right), \\
        \sigma^2 &= (1-\gamma)p.
    \end{align*}
    Further recall in Section \ref{section:sparse_regression} we have defined \(Y = \sqrt{p}\widetilde{X}\) and so \(Y \sim N(M\theta, \sigma^2I_p)\). In order to apply Theorem \ref{thm:bellec}, it must be verified \(M\) satisfies a sparse eigenvalue condition.
    \begin{lemma}\label{lemma:M_SRE}
        If \(s \leq \frac{p}{784}\), then \(M\) satisfies the \(\SRE(s, 7)\) condition with \(\kappa(s, 7) \geq \frac{1}{2}\). 
    \end{lemma}
    \begin{proof}
        First, consider \(\frac{1}{p} ||Me_j||^2 = ||e_j - \bar{e}_j \mathbf{1}_p||^2 \leq ||e_j||^2 \leq 1\) for all \(1 \leq j \leq p\). Next, let \(s^* = \frac{p}{2}\). Consider
        \begin{align*}
            \psi_{\min}(M, s^*)^2 = \min_{\substack{\Delta \in \R^p \setminus \{0\} \\ ||\Delta||_0 \leq s^*}} \frac{\frac{1}{p}||M\Delta||^2}{||\Delta||^2} = \min_{\substack{\Delta \in \R^p \setminus \{0\} \\ ||\Delta||_0 \leq s^*}} \frac{||\Delta - \bar{\Delta}\mathbf{1}_p||^2}{||\Delta||^2} = \frac{p-s^*}{p} \geq \frac{1}{2},
        \end{align*}
        where we have used Corollary 1 of \cite{kotekal_minimax_2023}. Taking \(\psi_1 = \frac{1}{\sqrt{2}}\) and \(c_0 = 7\), observe by Proposition \ref{prop:sparse_eigenvalue} the \(\SRE(s, 7)\) condition holds for all \(s \leq (s^* - 1)\psi_1^2/(2c_0^2) = (s^* - 1) \frac{1}{196}\) with \(\kappa(s, 7) \geq \frac{\psi_1}{\sqrt{2}} = \frac{1}{2}\). Since \(s \leq \frac{p}{784}\) and \(p\) is larger than a sufficiently large universal constant, we have \(s \leq \frac{p}{4 \cdot 196} \leq \left(\frac{p}{2}-1\right)\frac{1}{196} = (s^* - 1)\frac{1}{196}\). Therefore, \(\kappa(s, 7) \geq \frac{1}{2}\) for all \(s \leq \frac{p}{784}\). 
    \end{proof}
    
    \begin{proof}[Proof of Proposition \ref{prop:lasso_error}]
        Consider \(\lambda \geq 2(4+\sqrt{2})\sqrt{\frac{\sigma^2 \log(2ep/s)}{p}}\). By Lemma \ref{lemma:M_SRE}, \(M\) satisfies the \(\SRE(s, 7)\) condition with \(\kappa(s, 7)^2 \geq \frac{1}{4}\), we have by Theorem \ref{thm:bellec} 
        \begin{equation*}
            \sup_{||\theta||_0 \leq s} \frac{1}{p} E\left(||M(\hat{\beta} - \theta)||^2\right) \leq \frac{49\lambda^2 s}{16}\left(4 + \frac{1}{2 \log(2ep)}\right) \leq 16\lambda^2 s. 
        \end{equation*}
        Since $\|\frac{1}{\sqrt{p}}M\hat{\beta}-(\theta - \bar{\theta}\mathbf{1}_p)\|^2=||M(\hat{\beta} - \theta)||^2$, we obtain
        \begin{equation*}
            \sup_{||\theta||_0 \leq s} E\left(\left|\left|\frac{1}{\sqrt{p}}M\hat{\beta} - (\theta - \bar{\theta}\mathbf{1}_p)\right|\right|^2\right) \leq 16\lambda^2 s,
        \end{equation*}
        and so the desired result follows by Markov inequality. 
    \end{proof}
    
    \begin{proof}[Proof of Theorem \ref{thm:orthog_ubound}]
        The result follows directly from Proposition \ref{prop:lasso_error} and the fact that \(E_{\theta, \gamma}(||\widetilde{X} - (\theta - \bar{\theta}\mathbf{1}_p)||^2) = (1-\gamma)p\). 
    \end{proof}
    
    \begin{proof}[Proof of Proposition \ref{prop:small_s_linear}]
        For ease, set \(\hat{T} = \frac{1}{p} \sum_{i=1}^{p} \hat{\beta}_i\). Since the Pythagorean identity implies \(||\hat{T} \mathbf{1}_p - \bar{\theta}\mathbf{1}_p||^2 \leq ||\hat{\beta} - \theta||^2\), the desired result then follows form Markov's inequality and applying the argument of Proposition \ref{prop:lasso_error} using the second conclusion of Theorem \ref{thm:bellec}.
    \end{proof}
    
    \subsection{Kernel mode estimator: known correlation}\label{appendix:kme_known_var}
        Our goal in this section is to prove Theorem \ref{thm:kme_known_var}. Recall the data \(Y\) coming from the model (\ref{model:Gaussian_contamination}). In particular, recall the notation \(\mu = \bar{\theta}\) and \(\eta_i = \bar{\theta} - \theta_i\) for \(i \in \mathcal{O}\), where \(\mathcal{I} = \supp(\theta)^c\) and \(\mathcal{O} = \supp(\theta)\). Recall the kernel mode estimator (\ref{def:mhat}), that is, \(\hat{\mu} = \argmax_{t \in \R} \hat{G}_h(t)\) where 
        \begin{equation*}
            \hat{G}_h(t) = \frac{1}{2ph} \sum_{i=1}^{p} \mathbbm{1}_{\{|t - Y_i| \leq h\}}.
        \end{equation*}
        Let \(G_h(t) = E_{\theta, \gamma}(\hat{G}_h(t))\) denote the expectation. Without loss of generality, we can take \(\gamma = 0\) in (\ref{model:Gaussian_contamination}), as otherwise we could simply work with \(\left\{Y_i/\sqrt{1-\gamma}\right\}_{i=1}^{p}\) and consider estimation of \(\mu/\sqrt{1-\gamma}\). Consequently, we will use \(P_{\theta}\) and \(E_{\theta}\) instead of \(P_{\theta, \gamma}\) and \(E_{\theta, \gamma}\) to reduce notational clutter. Throughout this section, we will assume the bandwidth \(h\) is larger than some sufficiently large universal constant.
        
        Before launching into the proof, let us recall the intuition presented in Section \ref{section:linear_functional}. At a high level, the strategy is to show that with high probability, for every \(t \in \R\) with \(|t - \mu| \geq C_a h\), we can find \(x\) close to \(\mu\), say \(|x - \mu| \leq \frac{h}{4}\), such that \(\hat{G}_h(x) > \hat{G}_h(t)\). Then it would follow that on this event we have \(|\hat{\mu} - \mu| \leq C_a h\). 
        
        Finally, throughout this section we can take without loss of generality \(s \geq \frac{p}{4}\). To see this, consider for all \(s < \frac{p}{4}\), we have \((1-\gamma) \left(1 \vee \log\left(\frac{ep}{(p-2s)^2}\right)\right) \asymp 1-\gamma\), and so it suffices to prove the result of Theorem \ref{thm:kme_known_var} for \(s \geq \frac{p}{4}\). Let us note we can take without loss of generality \(|\mathcal{O}| \geq \frac{p}{4}\). This is because we can simply adjust the sets \(|\mathcal{I}|\) and \(|\mathcal{O}|\) by taking points from \(\mathcal{I}\) and putting them in \(\mathcal{O}\) until \(|\mathcal{O}| = s\). So throughout this section, it will be assumed \(|\mathcal{O}| \geq \frac{p}{4}\). 
    
        \subsubsection{\texorpdfstring{Regime \(\frac{p}{2} - p^{1/4} \leq s < \frac{p}{2}\)}{Regime p/2 - p\textasciicircum (1/4) <= s < p/2}}
        In the regime \(\frac{p}{2} - p^{1/4} \leq s < \frac{p}{2}\), the analysis is standard. The exponent \(\frac{1}{4}\) is not special and could be replaced by any constant in \(\left(\frac{1}{2}, 1\right)\). As remarked after the statement of Theorem \ref{thm:kme_known_var}, we have \(\log\left(\frac{ep}{(p-2s)^2}\right) \asymp \log(ep)\) for \(s \geq \frac{p}{2} - p^{\frac{1}{2}-\delta}\) with any constant \(\delta \in (0, 1/2]\). Hence, it is not so surprising that a delicate argument is not needed.
    
        \begin{proposition}\label{prop:kme_easy}
            Suppose \(\frac{p}{2} - p^{1/4} \leq s < \frac{p}{2}\), and \(C_a\) is a sufficiently large universal constant. There exist universal constants \(C_1\) and \(C_2\) such that the following holds. For any \(\delta \in (0, 1)\), if \(p\) is sufficiently large depending only on \(\delta\) and
            \begin{equation*}
                h = C_1 \sqrt{\log\left(\frac{ep}{(p-2s)^2}\right)},
            \end{equation*}
            then 
            \begin{equation*}
                \sup_{||\theta||_0 \leq s} P_{\theta}\left\{ |\hat{\mu} - \mu| > C_2 h \right\} \leq \delta
            \end{equation*}
            where \(\hat{\mu}\) is given by (\ref{def:mhat}).
        \end{proposition}
        \begin{proof}
            First, consider that 
            \begin{equation*}
                \hat{G}_h(\mu) \geq \frac{1}{2ph} \sum_{i \in \mathcal{I}} \mathbbm{1}_{\{|\mu - Y_i| \leq h\}} = \frac{p-|\mathcal{O}|}{2ph} - \frac{1}{2ph} \sum_{i \in \mathcal{I}} \mathbbm{1}_{\{|\mu - Y_i| > h\}}. 
            \end{equation*}
            For \(t \in \R\) with \(|t - \mu| \geq C_a h\) with \(C_a > 2\) sufficiently large universal constant, consider that 
            \begin{align*}
                \hat{G}_h(t) &\leq \frac{|\mathcal{O}|}{2ph} + \frac{1}{2ph} \sum_{i \in \mathcal{I}} \mathbbm{1}_{\{|t - Y_i| \leq h\}} \\
                &= \frac{|\mathcal{O}|}{2ph} + \frac{1}{2ph} \sum_{i \in \mathcal{I}} \mathbbm{1}_{\{|t - \mu + \mu - Y_i| \leq h\}} \\
                &\leq \frac{|\mathcal{O}|}{2ph} + \frac{1}{2ph} \sum_{i \in \mathcal{I}} \mathbbm{1}_{\{|\mu - Y_i| \geq |t - \mu| - h\}} \\
                &\leq \frac{|\mathcal{O}|}{2ph} + \frac{1}{2ph} \sum_{i \in \mathcal{I}} \mathbbm{1}_{\{|\mu - Y_i| > h\}}.
            \end{align*}
            Note this holds for all \(t\) such that \(|t-\mu| \geq C_a h\), and so we have 
            \begin{equation*}
                \hat{G}_h(\mu) - \sup_{|t-\mu| \geq C_a h} \hat{G}_h(t) \geq \frac{p-2s}{2ph} - \frac{1}{ph} \sum_{i \in \mathcal{I}} \mathbbm{1}_{\{|\mu - Y_i| > h\}}. 
            \end{equation*}
            Note \(Y_i - \mu \sim N(0, 1)\) for \(i \in \mathcal{I}\). Further note \(\Var(\sum_{i \in \mathcal{I}} \mathbbm{1}_{\{|\mu - Y_i| > h\}}) \leq |\mathcal{I}| e^{-Ch^2}\) for some universal constant \(C > 0\) whose value may change from instance to instance. Hence, by Markov's inequality we have for \(u \geq 0\), 
            \begin{equation*}
                \hat{G}_h(\mu) - \sup_{|t - \mu| \geq C_a h} \hat{G}_h(t) \geq \frac{p-2s}{2ph} - \frac{|\mathcal{I}|}{ph} e^{-Ch^2} - u
            \end{equation*}
            with \(P_{\theta}\)-probability at least \(1 - \frac{e^{-Ch^2}}{u\sqrt{p}h}\). Select \(u = \frac{1}{h}e^{-Ch^2}\). Then 
            \begin{equation*}
                \hat{G}_h(\mu) - \sup_{|t-\mu| \geq C_a h} \hat{G}_h(t) \geq \frac{p-2s}{2ph} - \frac{2}{h} e^{-Ch^2}
            \end{equation*}
            with probability at least \(1 - \frac{2}{\sqrt{p}}\). Since \(p\) is sufficiently large depending only on \(\delta\), this probability is greater than or equal to \(1 - \delta\). Select \(C_1\) sufficiently large such that 
            \begin{equation*}
                h^2 \geq\frac{3}{2C}\log\left(\frac{8p}{\left(p-2s\right)^2}\right). 
            \end{equation*}
            Observe that with this choice, since \(p-2s \leq 2p^{1/4}\) we have  
            \begin{align*}
                h^2 &\geq \frac{3}{2C} \left(\log\left(\frac{8p}{p-2s}\right) - \log(2p^{1/4})\right) \\
                &\geq \frac{3}{2C} \left(\log\left(\frac{8p}{p-2s}\right) - \frac{1}{3}\log\left(\frac{8p}{p-2s}\right)\right) \\
                &\geq \frac{1}{C}\log\left(\frac{8p}{p-2s}\right). 
            \end{align*}
            Thus, with \(P_{\theta}\)-probability at least \(1-\delta\), we have 
            \begin{align*}
                \hat{G}_h(\mu) - \sup_{|t - \mu| \geq C_a h} \hat{G}_h(t) &\geq \frac{p-2s}{2ph} - \frac{p-2s}{4ph} = \frac{p-2s}{4ph} > 0.
            \end{align*}
            Hence, \(|\hat{\mu} - \mu| \leq C_a h\) on this event and so the proof is complete.
        \end{proof}
    
        \subsubsection{\texorpdfstring{Regime \(s \leq \frac{p}{2} - p^{1/4}\)}{Regime s <= p/2 - p\textasciicircum (1/4)}} \label{section:peel}
        
        In the regime \(s \leq \frac{p}{2} - p^{1/4}\), the argument is much more involved. Instead of directly showing \(\hat{G}_h(x) > \hat{G}_h(t)\) as in the proof of Proposition \ref{prop:kme_easy}, we will compare the empirical quantities with their population counterparts \(G_h(x)\) and \(G_h(t)\). In particular, consider 
        \begin{equation}\label{eqn:kme_master}
            \hat{G}_h(x) - \hat{G}_h(t) \geq G_h(x) - G_h(t) - \left|\hat{G}_h(x) - G_h(x)\right| - \left|\hat{G}_h(t) - G_h(t)\right|. 
        \end{equation}
        In order to show the right hand side is positive, a population level analysis giving a lower bound on \(G_h(x) - G_h(t)\) is needed. Later, control of the stochastic deviations \(\left|\hat{G}_h(x) - G_h(x)\right|\) and \(\left|\hat{G}_h(t) - G_h(t)\right|\) uniformly over \(|t - \mu| \geq C_a h\) is needed. Recall Proposition \ref{prop:mean_diff_fix} gives a result about the population level quantities.
        \begin{proof}[Proof of Proposition \ref{prop:mean_diff_fix}]
            Recall we have taken \(\gamma = 0\) as stated at the beginning of Section \ref{appendix:kme_known_var}. For ease of notation, denote the functions \(f_{\eta_i}(y) := \Phi(y-\eta_i+h) - \Phi(y-\eta_i-h)\) and \(f_\mu(y) := \Phi(y-\mu+h) - \Phi(y-\mu-h)\). Define the set \(\mathcal{O}' := \left\{i \in \mathcal{O} : |\eta_i - \mu| \geq \frac{C_a}{4}h\right\}\). 
            \begin{equation*}
                x = \argmax_{y \in \R} \left\{ \frac{|\mathcal{O}'|}{2ph} f_\mu(y) + \frac{1}{2ph} \sum_{i \in \mathcal{O}'} f_{\eta_i}(y)\right\}.
            \end{equation*}
            There may be potentially two global maxima, but by Theorem \ref{thm:mixture} we can take \(x\) to be the one near \(\mu\). Since \(C_a\) is sufficiently large, we have \(|x - \mu| \leq \frac{h}{4}\). Fix \(t \in \R\) such that \(|t - \mu| \geq C_a h\). Now, consider 
            \begin{align*}
                G_h(x) - G_h(t) &= \frac{p-|\mathcal{O}|}{2ph} (f_\mu(x) - f_\mu(t)) + \frac{1}{2ph} \sum_{i \in \mathcal{O}} (f_{\eta_i}(x) - f_{\eta_i}(t)) \\
                &= \frac{p-|\mathcal{O}|-|\mathcal{O}'|}{2ph} (f_\mu(x) - f_\mu(t)) + \left(\frac{|\mathcal{O}'|}{2ph} (f_\mu(x) - f_\mu(t)) + \frac{1}{2ph} \sum_{i \in \mathcal{O}'} (f_{\eta_i}(x) - f_{\eta_i}(t)) \right) \\
                &\;\; + \frac{1}{2ph} \sum_{i \in \mathcal{O} \setminus \mathcal{O}'} \left(f_{\eta_i}(x) - f_{\eta_i}(t)\right) \\
                &\geq \frac{p-2|\mathcal{O}|}{2ph} (f_\mu(x) - f_\mu(t)) + \frac{1}{2ph} \sum_{i \in \mathcal{O} \setminus \mathcal{O}'} (f_{\eta_i}(x) - f_{\eta_i}(t))
            \end{align*}
            by the definition of \(x\). We have also used \(f_{\mu}(x) \geq f_{\mu}(t)\) since \(|x - \mu| \leq \frac{h}{4}\) and \(|t - \mu| \geq C_a h\) with \(C_a\) and \(h\) sufficiently large. For \(i \in \mathcal{O}\setminus \mathcal{O}'\) we have \(|\eta_i - \mu| \leq \frac{C_a}{4}h\), and so \(|x - \eta_i| \leq \frac{1+C_a}{4}h\). On the other hand, we have \(|t - \eta_i| \geq \frac{3C_a}{4}h\). Since \(C_a\) and \(h\) are sufficiently large, we have \(f_{\eta_i}(x) \geq f_{\eta_i}(t)\). Therefore, we have the bound 
            \begin{align*}
                G_h(x) - G_h(t) \geq \frac{p-2|\mathcal{O}|}{2ph} (f_\mu(x) - f_\mu(t)) \geq \frac{p-2|\mathcal{O}|}{2ph}(1 -2e^{-Ch^2}) \geq \frac{p-2|\mathcal{O}|}{2ph} \cdot \frac{1}{2}
            \end{align*}
            where \(C > 0\) is some universal constant. Here, we have used \(|t - \mu| \geq C_a h\) and \(|x - \mu| \leq \frac{h}{4}\) to obtain \(f_\mu(t) \leq e^{-Ch^2}\) and \(1-f_\mu(x) \leq e^{-Ch^2}\). We have also used that \(h\) is sufficiently large to obtain \(1 - 2e^{-Ch^2} \geq \frac{1}{2}\). To show the other bound, consider that 
            \begin{align*}
                G_h(x) - G_h(t) &= \frac{p-|\mathcal{O}|}{2ph} (f_\mu(x) - f_\mu(t)) + \frac{1}{2ph} \sum_{i \in \mathcal{O}} (f_{\eta_i}(x) - f_{\eta_i}(t)) \\
                &\geq \frac{p-2|\mathcal{O}|}{2ph} f_{\mu}(x) - \frac{p-|\mathcal{O}|}{2ph} f_\mu(t) + \frac{1}{2ph} \sum_{i \in \mathcal{O}} (f_{\mu}(x) - f_{\eta_i}(x)) \\
                &= \frac{p-2|\mathcal{O}|}{2ph} f_\mu(x) - \frac{p-|\mathcal{O}|}{2ph}f_\mu(t) + \frac{1}{2ph} \sum_{i \in \mathcal{O}} \left(\left(1 - f_{\eta_i}(t)\right) - (1-f_\mu(x))\right) \\
                &= \frac{p-2|\mathcal{O}|}{2ph} f_\mu(x) - \frac{p-|\mathcal{O}|}{2ph}f_\mu(t) - \frac{|\mathcal{O}|}{2ph} (1 - f_\mu(x)) + \frac{1}{2ph} \sum_{i \in \mathcal{O}} \left(1 - f_{\eta_i}(t)\right) \\
                &\geq \frac{p-2|\mathcal{O}|}{2ph} f_\mu(x) - \frac{e^{-Ch^2}}{2h} + \frac{1}{2ph} \sum_{i \in \mathcal{O}} \left(1 - f_{\eta_i}(t)\right) \\
                &\geq \frac{p-2|\mathcal{O}|}{2ph} - \left(1 + \frac{p-2|\mathcal{O}|}{p}\right)\frac{e^{-Ch^2}}{2h} + \frac{1}{2ph} \sum_{i \in \mathcal{O}} \left(1 - f_{\eta_i}(t)\right) \\
                &\geq \frac{p-2|\mathcal{O}|}{2ph} - \frac{e^{-Ch^2}}{h} + \frac{1}{2ph} \sum_{i \in \mathcal{O}} \left(1 - f_{\eta_i}(t)\right)
            \end{align*}
            where \(C > 0\) is some universal constant. Here, we have used \(|t - \mu| \geq C_a h\) and \(|x - \mu| \leq \frac{h}{4}\) to obtain \(f_\mu(t) \leq e^{-Ch^2}\) and \(1-f_\mu(x) \leq e^{-Ch^2}\). The proof is complete.
        \end{proof}
    
        \noindent \textbf{Uniform control of the stochastic error} \newline 
        Proposition \ref{prop:mean_diff_fix} gives us a lower bound for the population quantity (which can be interpreted as the signal) in (\ref{eqn:kme_master}). It remains to uniformly control the stochastic error. In view of Proposition \ref{prop:mean_diff_fix}, define the sets 
        \begin{align*}
            \mathcal{U} &:= \left\{ t \in \R : |t - \mu| \geq C_a h \text{ and } \frac{1}{p} \sum_{i \in \mathcal{O}} P_{\theta}\left\{|t - Y_i| > h\right\} > 4 e^{-Ch^2} \right\}, \\
            \mathcal{V} &:= \left\{t \in \R : |t - \mu| \geq C_a h \text{ and } t \in \mathcal{U}^c\right\}.  
        \end{align*}
        Here, \(C > 0\) is the universal constant from Proposition \ref{prop:mean_diff_fix}, namely it is a universal constant such that for any \(x, t \in \R\) with \(|t - \mu| \geq C_a h\) and \(|x - \mu| \leq \frac{h}{4}\), we have \(P_{\theta}\left\{|t - Y_i| \leq h \right\} \vee P_{\theta}\left\{|x - Y_i| > h\right\} \leq e^{-Ch^2}\) for \(i \in \mathcal{I}\). Note such a \(C\) exists since \(h\) is taken to be larger than a sufficiently large universal constant. Throughout the following sections, \(C\) will refer to this universal constant. 
        
        There are two parts to the lower bound in Proposition \ref{prop:mean_diff_fix}. The first part will be used when dealing with \(\mathcal{V}\) and the second part is used when dealing with \(\mathcal{U}\). \\
    
        \noindent \textbf{Uniform stochastic error control over \(\mathcal{U}\)}
        \begin{proposition}\label{prop:U_control}
            Suppose \(s \leq \frac{p}{2} - p^{1/4}\). Suppose \(C_a\) is a sufficiently large universal constant. Further suppose \(h\) is larger than a sufficiently large universal constant. Let \(x\) such that \(|x - \mu| \leq \frac{h}{4}\) be the point from Proposition \ref{prop:mean_diff_fix}. If \(\delta \in (0, 1)\) and \(p\) is sufficiently large depending only on \(\delta\), then with probability at least \(1-\delta\) we have 
            \begin{equation*}
                \left|\hat{G}_h(t) - G_h(t)\right| < \frac{1}{2}(G_h(x) - G_h(t))
            \end{equation*}
            uniformly over \(t \in \mathcal{U}\). 
        \end{proposition}
        \begin{proof}
            By Proposition \ref{prop:mean_diff_fix}, for any \(t \in \mathcal{U}\)
            \begin{align}
                G_h(x) - G_h(t) &\geq \frac{p-2|\mathcal{O}|}{2ph} - \frac{e^{-Ch^2}}{h} + \frac{1}{2ph} \sum_{i \in \mathcal{O}} P_{\theta}\left\{|t - Y_i| > h\right\} \nonumber\\
                &\geq \frac{p-2|\mathcal{O}|}{2ph} + \frac{1/2}{2ph} \sum_{i \in \mathcal{O}} P_{\theta}\left\{|t - Y_i| > h\right\}. \label{eqn:signalU}
            \end{align}
            Now let us examine the stochastic deviation. Consider that 
            \begin{equation}
                \left|\hat{G}_h(t) - G_h(t)\right| \leq \frac{1}{2ph} \left|\sum_{i \in \mathcal{I}} \mathbbm{1}_{\{|t - Y_i| \leq h\}} - P_{\theta}\{|t - Y_i| \leq h\}\right| + \frac{1}{2ph} \left|\sum_{i \in \mathcal{O}} \mathbbm{1}_{\{|t - Y_i| > h\}} - P_{\theta}\{|t - Y_i| > h\}\right|. \label{eqn:stochU_parts}
            \end{equation}
            Let us bound the second term in (\ref{eqn:stochU_parts}). Taking \(\lambda = \frac{1}{16}\) in Theorem \ref{thm:peel} and noting \(|\mathcal{O}| \asymp p\), we have with probability \(1-\delta/2\) 
            \begin{equation*}
                \frac{1}{2ph} \left|\sum_{i \in \mathcal{O}} \mathbbm{1}_{\{|t - Y_i| > h\}} - P_{\theta}\{|t - Y_i| > h\}\right| \leq \left(\frac{C'\log(ep)}{2ph} + \frac{1/16}{2ph} \sum_{i \in \mathcal{O}} P_{\theta}\left\{|t - Y_i| > h\right\}\right) \left(1 + \frac{\kappa_\delta}{\sqrt{\log(ep)}}\right)
            \end{equation*}
            uniformly over \(t \in \mathcal{U}\). Here \(C' > 0\) is a universal constant. Therefore, for \(p\) sufficiently large depending only on \(\delta\), we have from (\ref{eqn:signalU})
            \begin{align}
                &\left(\frac{C'\log(ep)}{2ph} + \frac{1/16}{2ph} \sum_{i \in \mathcal{O}} P_{\theta}\left\{|t - Y_i| > h\right\}\right) \left(1 + \frac{\kappa_\delta}{\sqrt{\log(ep)}}\right) \nonumber \\
                &\leq 2\left(\frac{C'\log(ep)}{2ph} + \frac{1/16}{2ph} \sum_{i \in \mathcal{O}} P_{\theta}\left\{|t - Y_i| > h\right\}\right) \nonumber \\
                &< \frac{1}{4} \cdot \frac{p-2|\mathcal{O}|}{2ph} + \frac{1/8}{2ph} \sum_{i \in \mathcal{O}} P_{\theta}\left\{|t - Y_i| > h\right\} \nonumber \\
                &\leq \frac{1}{4}\left(G_h(x) - G_h(t)\right) \label{eqn:stochU_bound_2}
            \end{align}
            uniformly over \(t \in \mathcal{U}\) with probability at least \(1-\delta/2\). Here, we have used \(p - 2|\mathcal{O}| \geq p-2s \geq p^{1/4} > 8C'\log(ep)\) since \(p\) is sufficiently large.
    
            It remains to bound the first term in (\ref{eqn:stochU_parts}). Taking \(\lambda = \frac{1}{16}\) in Theorem \ref{thm:peel} and noting \(|\mathcal{I}| \asymp p\), we have with probability \(1 - \delta\) 
            \begin{equation*}
                \frac{1}{2ph} \left| \sum_{i \in \mathcal{I}} \mathbbm{1}_{\{|t - Y_i| \leq h\}} - P_{\theta}\left\{|t - Y_i| \leq h\right\}\right| \leq \left(\frac{C'\log(ep)}{2ph} + \frac{1/16}{2ph} \sum_{i \in \mathcal{I}} P_{\theta}\left\{|t - Y_i| \leq h\right\}\right) \left(1 + \frac{\kappa_\delta}{\sqrt{\log(ep)}}\right)
            \end{equation*}
            uniformly over \(t \in \mathcal{U}\). Since \(|t - \mu| \geq C_a h\), it follows that for \(i \in \mathcal{I}\) we have 
            \begin{align*}
                \frac{1}{p} \sum_{i \in \mathcal{I}} P_{\theta} \left\{|t - Y_i| \leq h\right\} \leq e^{-Ch^2} \leq \frac{1}{p} \sum_{i \in \mathcal{O}} P_{\theta}\left\{|t - Y_i| > h\right\}. 
            \end{align*}
            The final inequality follows from the definition of \(\mathcal{U}\). Consequently, by repeating the argument for (\ref{eqn:stochU_bound_2}), we have with probability \(1-\delta/2\) 
            \begin{equation}
                \frac{1}{2ph} \left| \sum_{i \in \mathcal{I}} \mathbbm{1}_{\{|t - Y_i| \leq h\}} - P_{\theta}\left\{|t - Y_i| \leq h\right\}\right| \leq \frac{1}{4}(G_h(x) - G_h(t)) \label{eqn:stochU_bound_1}
            \end{equation}
            uniformly over \(t \in \mathcal{U}\). Putting together the bounds (\ref{eqn:stochU_bound_2}) and (\ref{eqn:stochU_bound_1}) and applying union bound, it follows from (\ref{eqn:stochU_parts}) that with probability \(1-\delta\)
            \begin{equation*}
                \left|\hat{G}_h(t) - G_h(t)\right| \leq \frac{1}{2} (G_h(x) - G_h(t))
            \end{equation*}
            uniformly over \(t \in \mathcal{U}\). The proof is complete. 
        \end{proof}
    
        \noindent \textbf{Uniform control over \(\mathcal{V}\)}
        \begin{lemma}\label{lemma:stoch_V_exp}
            Suppose \(C_a\) is a sufficiently large universal constant and \(h\) is larger than a sufficiently large universal constant. There exist universal constants \(C', C'' > 0\) such that 
            \begin{equation*}
                E_{\theta}\left(\sup_{t \in \mathcal{V}} \left|\hat{G}_h(t) - G_h(t)\right| \right) \leq C' \left(\frac{h}{p} + \frac{e^{-C'' h^2}}{h\sqrt{p}}\right).
            \end{equation*}
        \end{lemma}
        \begin{proof}
            First, let us split
            \begin{equation}
                \left|\hat{G}_h(t) - G_h(t)\right| \leq \frac{1}{2ph} \left|\sum_{i \in \mathcal{I}} \mathbbm{1}_{\{|t - Y_i| \leq h\}} - P_{\theta}\{|t - Y_i| \leq h\}\right| + \frac{1}{2ph} \left|\sum_{i \in \mathcal{O}} \mathbbm{1}_{\{|t - Y_i| > h\}} - P_{\theta}\{|t - Y_i| > h\}\right|. \label{eqn:stochV_parts}
            \end{equation}
            To bound the second term in (\ref{eqn:stochV_parts}), consider by definition of \(\mathcal{V}\) we have for \(t \in \mathcal{V}\) 
            \begin{equation*}
                \frac{1}{|\mathcal{O}|} \sum_{i\in \mathcal{O}} P_{\theta} \left\{|t - Y_i| > h\right\} \leq \frac{p}{|\mathcal{O}|} \cdot 4e^{-Ch^2} \leq 16 e^{-Ch^2}.
            \end{equation*}
            Here, we have used \(|\mathcal{O}| \geq \frac{p}{4}\) which we can assume without loss of generality as noted at the beginning of Section \ref{section:peel}. Let \(\zeta_t := \frac{1}{2h\sqrt{p}} \sum_{i\in \mathcal{O}} \left(\mathbbm{1}_{\{|t - Y_i| > h\}} - P_{\theta}\left\{|t - Y_i| > h\right\}\right)\). Since \(|\mathcal{O}| \asymp p\) and by Corollaries \ref{corollary:zeta_small_sigma} and \ref{corollary:zeta_big_sigma}, we have 
            \begin{align*}
                E_{\theta}\left(\sup_{t \in \mathcal{V}} \frac{1}{\sqrt{p}} |\zeta_t|\right) \leq \frac{C'}{h\sqrt{p}}\left( \frac{\log\left(\frac{e}{16e^{-Ch^2}}\right)}{\sqrt{p}} \vee 16e^{-Ch^2} \sqrt{\log\left(\frac{e}{16e^{-Ch^2}}\right)}\right) \leq \frac{C' h}{p} \vee \frac{C'e^{-C''h^2}}{h\sqrt{p}}
            \end{align*}
            where \(C', C'' > 0\) are universal constants whose value may change from instance to instance. Note we have used \(e^{-Ch^2} \leq \frac{e^{-C''h^2}}{h}\) for some universal \(C'' > 0\) universal since \(h \gtrsim 1\). Likewise, to bound the first term in (\ref{eqn:stochV_parts}), consider that since \(|t - \mu | \geq C_a h\), 
            \begin{equation*}
                \frac{1}{|\mathcal{I}|} \sum_{i \in \mathcal{I}} P_{\theta}\{|t - Y_i| \leq h\} \leq e^{-Ch^2}. 
            \end{equation*}
            Consequently, the same argument as before yields 
            \begin{equation*}
                E_{\theta}\left(\sup_{t \in \mathcal{V}} \frac{1}{2ph} \left|\sum_{i\in \mathcal{I}} \left(\mathbbm{1}_{\{|t - Y_i| \leq h\}} - P_{\theta}\left\{|t - Y_i| \leq h\right\}\right) \right|\right) \leq \frac{C'h}{p} \vee \frac{C'e^{-C''h^2}}{h\sqrt{p}}. 
            \end{equation*}
            Putting together the bounds yields the desired result. The proof is complete. 
        \end{proof}
    
        \begin{proposition}\label{prop:stoch_V}
            Suppose \(C_a\) is a sufficiently large universal constant and \(h\) is larger than a sufficiently large universal constant. If \(u \geq 0\), then 
            \begin{equation*}
                P_{\theta}\left\{ \sup_{t \in \mathcal{V}} \left|\hat{G}_h(t) - G_h(t)\right| > C' \left(\frac{1}{h\sqrt{p}} e^{-C''h^2} + \frac{h}{p}\right) + u \right\} \leq 2 \exp\left(-c\min\left(\frac{u^2}{\frac{1}{p^2} + \frac{e^{-C'''h^2}}{ph^2}}, uph \right)\right)
            \end{equation*}
            where \(C', C'', C''', c > 0\) are universal constants. 
        \end{proposition}
        \begin{proof}
            Let \(\zeta_t := \frac{1}{2h\sqrt{p}} \sum_{i=1}^{p} \left(\mathbbm{1}_{\{|t - Y_i| \leq h\}} - P_{\theta}\left\{|t - Y_i| \leq h\right\}\right)\). Let \(\rho^2 := \sup_{t \in \mathcal{V}} \Var_{\theta}(2h\sqrt{p} \zeta_t)\). Consider that for \(t \in \mathcal{V}\), arguing as in the proof of Lemma \ref{lemma:stoch_V_exp}
            \begin{align*}
                \Var(2h\sqrt{p} \zeta_t) &= \sum_{i=1}^{p} \Var(\mathbbm{1}_{\{|t - Y_i| \leq h\}}) \\
                &= \sum_{i=1}^{p} P_{\theta}\left\{|t - Y_i| \leq h\right\} \cdot P_{\theta} \left\{|t - Y_i| > h\right\} \\
                &\leq \sum_{i=1}^{p} P_{\theta}\left\{|t - Y_i| \leq h\right\} \wedge P_{\theta} \left\{|t - Y_i| > h\right\}\\
                &\leq \sum_{i \in \mathcal{I}} P_{\theta}\left\{|t - Y_i| \leq h\right\} + \sum_{i \in \mathcal{O}} P_{\theta} \left\{|t - Y_i| > h\right\} \\
                &\leq p e^{-Ch^2} + 4p e^{-Ch^2} \\
                &\leq 5pe^{-Ch^2}. 
            \end{align*}
            The final inequality follows from the definition of \(\mathcal{V}\). We have used \(|t - \mu| \geq C_a h\) to obtain \(P_{\theta}\{|t - Y_i| \leq h\} \leq e^{-Ch^2}\) for \(i \in \mathcal{I}\), and we have used the definition of \(\mathcal{V}\) to bound the second sum. Letting \(c > 0\) change from instance to instance but remain universal, consider that combining Theorems \ref{thm:bernstein_type} and \ref{thm:sup_var}, we have for any \(u \geq 0\), 
            \begin{align*}
                P_{\theta}\left\{ \sup_{t \in \mathcal{V}} \frac{1}{\sqrt{p}} \zeta_t \geq E_{\theta}\left(\sup_{t \in \mathcal{V}} \frac{1}{\sqrt{p}} \zeta_t\right) + u \right\} &\leq \exp\left(-c\min\left(\frac{u^2p^2h^2}{h\sqrt{p}E_{\theta}(\sup_{t \in \mathcal{V}} \zeta_t) + \rho^2}, uph \right)\right) \\
                &\leq \exp\left( - c \min\left( \frac{u^2p^2h^2}{ph\left(\frac{h}{p} + \frac{e^{-C''h^2}}{h\sqrt{p}}\right) + pe^{-Ch^2}}, uph \right)\right) \\
                &\leq \exp\left(- c\min\left( \frac{u^2}{\frac{1}{p^2} + \frac{e^{-C''h^2}}{h^2 p^{3/2}} + \frac{e^{-Ch^2}}{ph^2}}, uph \right) \right) \\
                &\leq \exp\left(-c\min\left(\frac{u^2}{\frac{1}{p^2} + \frac{e^{-C'''h^2}}{ph^2}}, uph \right)\right).
            \end{align*}
            We have used Lemma \ref{lemma:stoch_V_exp} in the course of these calculations. The same calculation can be repeated for \(-\zeta_t\) from which we can obtain the claimed result. The proof is complete. 
        \end{proof}
        \begin{proposition}\label{prop:V_control}
            Suppose \(s \leq \frac{p}{2} - p^{1/4}\). Suppose \(C_a\) is a sufficiently large universal constant. There exists a universal constant \(C_1 > 0\) such that the following holds. If \(\delta \in (0, 1)\), \(p\) is sufficiently large depending only on \(\delta\), and 
            \begin{equation*}
                h = C_1\sqrt{1 \vee \log\left(\frac{L_\delta p}{(p-2s)^2}\right)}
            \end{equation*}
            where \(L_\delta\) depends only on \(\delta\), then, letting \(x\) such that \(|x - \mu| \leq \frac{h}{4}\) denote the point from Proposition \ref{prop:mean_diff_fix}, we have with probability at least \(1-\delta\),
            \begin{equation*}
                \left|\hat{G}_h(t) - G_h(t)\right| \leq \frac{1}{2}\left(G_h(x) - G_h(t)\right) 
            \end{equation*}
            uniformly over \(t \in \mathcal{V}\). 
        \end{proposition}
        \begin{proof}
            By Proposition \ref{prop:mean_diff_fix}, for any \(t \in \mathcal{V}\) we have 
            \begin{align*}
                G_h(x) - G_h(t) &\geq \frac{1}{2} \cdot \frac{p-2|\mathcal{O}|}{2ph}. 
            \end{align*}
            By Proposition \ref{prop:stoch_V}, there exists a constant \(\kappa_\delta\) depending only on \(\delta\) and universal constants \(C', C'' > 0\) such that with probability at least \(1-\delta\), we have uniformly over \(t \in \mathcal{V}\), 
            \begin{align*}
                \left|\hat{G}_h(t) - G_h(t)\right| &\leq C'\left(\frac{1}{h\sqrt{p}} e^{-C''h^2} + \frac{h}{p}\right) + \kappa_\delta \left(\frac{1}{ph} + \frac{1}{p} + \frac{e^{-C''h^2}}{h\sqrt{p}} \right) \\
                &\leq \left(C' + \kappa_\delta\right) \left(\frac{1}{h\sqrt{p}}e^{-C''h^2} + \frac{h}{p}\right) \\
                &\leq \frac{\left(C' + \kappa_\delta\right)}{L_\delta} \frac{p-2s}{ph} + (C' + \kappa_\delta)\cdot \frac{C_1 \sqrt{1 \vee\log(L_\delta p)}}{p} \\
                &\leq \frac{\left(C' + \kappa_\delta\right)}{L_\delta} \frac{p-2s}{ph}.
            \end{align*}
            Note the final inequality follows since \(p-2s \geq p^{1/4}\), and so grows faster than \(\sqrt{\log p}\). Here, the values of \(C', C'',\) and \(L_\delta\) can change from instance to instance. We have also used that \(p\) is sufficiently large depending only on \(\delta\). Clearly taking \(L_\delta\) sufficiently large ensures the desired result. The proof is complete. 
        \end{proof}
    
        \noindent \textbf{Control at \(x\)}\newline 
        Thus far, we have controlled the stochastic deviations uniformly over \(t \in \mathcal{U}\) and \(t \in \mathcal{V}\). It remains to control the deviation at \(x\), i.e. a bound for \(\left|\hat{G}_h(x) - G_h(x)\right|\) is needed. Note that we only need to consider the single point \(x\) given by Proposition \ref{prop:mean_diff_fix}. Importantly, \(x\) does not depend on \(t\) in Proposition \ref{prop:mean_diff_fix}. Note that Proposition \ref{prop:mean_diff_fix} asserts the existence of a single choice of \(x\) which gives the stated bound for all \(t\) with \(|t - \mu| \geq C_a h\). Consequently, empirical process tools are not necessary for bounding \(\left|\hat{G}_h(x) - G_h(x)\right|\). 
    
        \begin{proposition}\label{prop:x_control}
            Suppose \(s \leq \frac{p}{2} - p^{1/4}\). Suppose \(C_a\) is a sufficiently large universal constant. There exists a universal constant \(C_1 > 0\) such that the following holds. If \(\delta \in (0, 1)\), \(p\) is sufficiently large depending only on \(\delta\), and 
            \begin{equation*}
                h = C_1 \sqrt{1 \vee \log\left(\frac{L_\delta p}{(p-2s)^2}\right)}
            \end{equation*}
            where \(L_\delta\) depends only on \(\delta\), then, letting \(x\) such that \(|x - \mu| \leq \frac{h}{4}\) denote the point from Proposition \ref{prop:mean_diff_fix}, we have with probability at least \(1 - \delta\), 
            \begin{equation*}
                \left|\hat{G}_h(x) - G_h(x)\right| \leq \frac{1}{2} (G_h(x) - G_h(t))
            \end{equation*}
            uniformly over \(t \in \mathcal{U} \cup \mathcal{V}\). 
        \end{proposition}
        \begin{proof}
            Define the events 
            \begin{align*}
                \mathcal{E}_{\mathcal{U}} &:= \left\{ \left|\hat{G}_h(x) - G_h(x)\right| < \frac{1}{2}(G_h(x) - G_h(t)) \text{ for all } t\in \mathcal{U}\right\}, \\
                \mathcal{E}_{\mathcal{V}} &:= \left\{ \left|\hat{G}_h(x) - G_h(x)\right| < \frac{1}{2}(G_h(x) - G_h(t)) \text{ for all } t\in \mathcal{V}\right\}.
            \end{align*}
            It suffices by union bound to show each event holds with probability at least \(1-\delta/2\). Let us examine \(\mathcal{E}_{\mathcal{U}}\) first. Fix \(t \in \mathcal{U}\). It follows from Proposition \ref{prop:mean_diff_fix} and \(t \in \mathcal{U}\) that
            \begin{equation}\label{eqn:ux_signal}
                G_h(x) - G_h(t) \geq \frac{p-2|\mathcal{O}|}{2ph} + \frac{1/2}{2ph} \sum_{i \in \mathcal{O}} P_{\theta}\left\{|t - Y_i| > h\right\}. 
            \end{equation}
            Consider that \(\Var\left(2ph \left(\hat{G}_h(x) - G_h(x)\right)\right) = \sum_{i = 1}^{p} \Var\left(\mathbbm{1}_{\{|x - Y_i| \leq h\}}\right) \leq \sum_{i=1}^{p} P_{\theta}\{|x - Y_i| \leq h\}\wedge P_{\theta}\left\{|x - Y_i| > h\right\}\). Therefore, by Theorem \ref{thm:bernstein_bounded} 
            \begin{align*}
                &P_{\theta}\left\{ \frac{1}{2ph} \left| \sum_{i=1}^{p} \mathbbm{1}_{\{|x - Y_i| \leq h\}} - P_{\theta}\left\{|x - Y_i| \leq h \right\}\right| > u \right\} \\
                &\leq 2\exp\left(-c\min\left(\frac{u^2p^2h^2}{\sum_{i=1}^{p} P_{\theta}\{|x - Y_i| \leq h\} \wedge P_{\theta}\left\{|x - Y_i| > h\right\}}, uph\right)\right)
            \end{align*}
            where \(c > 0\) is a universal constant. Let us define the sets 
            \begin{align*}
                E := \left\{ i \in \mathcal{O} : |\mu - \eta_i| > \frac{C_ah}{2}\right\}, \\
                F := \left\{i \in \mathcal{O} : |\mu - \eta_i| \leq \frac{C_a h}{2}\right\}. 
            \end{align*}
            Therefore, with probability at least \(1-\delta/2\), 
            \begin{align*}
                &\left|\hat{G}_h(x) - G_h(x)\right| \\
                &\leq \kappa_\delta \left(\frac{1}{ph} + \sqrt{\frac{1}{p^2h^2} \sum_{i=1}^{p} P_{\theta}\{|x - Y_i| \leq h\}\wedge P_{\theta}\left\{|x - Y_i| > h\right\}}\right) \\
                &\leq \kappa_\delta\left(\frac{1}{ph} + \sqrt{\frac{1}{p^2h^2} \sum_{i \in \mathcal{I}} P_{\theta}\left\{|x - Y_i| > h\right\} + \frac{1}{p^2h^2}\sum_{i \in E} P_{\theta}\left\{|x - Y_i| \leq h\right\}} + \sqrt{\frac{1}{p^2h^2}\sum_{i \in F} P_{\theta}\left\{|x - Y_i| \leq h\right\}}\right) \\
                &\leq \kappa_\delta\left(\frac{1}{ph} + \frac{e^{-Ch^2/2}}{h\sqrt{p}}\right)  + \kappa_\delta \sqrt{\frac{1}{p^2h^2}\sum_{i \in F} P_{\theta}\left\{|x - Y_i| \leq h\right\}} \\
                &\leq \kappa_\delta\left(\frac{1}{ph} + L_{\delta}^{-1} \frac{p-2s}{ph}\right) + \kappa_\delta \sqrt{\frac{1}{p^2h^2}\sum_{i \in F} P_{\theta}\left\{|x - Y_i| \leq h\right\}} \\
                &\leq \kappa_\delta\left(\frac{1}{ph} + L_{\delta}^{-1} \frac{p-2s}{ph}\right) + \frac{8\kappa_\delta^2}{ph} + \frac{1/8}{ph}\sum_{i \in F} P_{\theta}\left\{|x - Y_i| \leq h\right\}.
            \end{align*}
            where \(\kappa_\delta\) is a constant depending only on \(\delta\) which can change from instance to instance. Here, we have taken \(C_1\) sufficiently large. The final inequality follows from the inequality \(ab \leq a^2 + b^2\) for \(a,b \in \R\). Now consider that since \(|t - \mu| \geq C_a h\)  and \(|\mu - \eta_i| \leq \frac{C_a h}{2}\) for \(i \in F\), we have \(|t - \eta_i| \geq \frac{C_ah}{2}\). Since \(C_a\) and \(C_1\) are sufficiently large, we thus have \(P_{\theta}\left\{|t - Y_i| > h\right\} \geq \frac{1}{2}\) for \(i \in F\). Therefore, for any \(t \in \mathcal{U}\) it follows
            \begin{align*}
                \frac{1/8}{ph}\sum_{i \in F} P_{\theta}\left\{|x - Y_i| \leq h\right\} \leq \frac{1/8}{ph}\sum_{i \in F} 2P_{\theta}\left\{|t - Y_i| > h\right\} \leq \frac{1/4}{ph}\sum_{i \in \mathcal{O}} P_{\theta}\left\{|t - Y_i| > h\right\}. 
            \end{align*}
            Thus, we have 
            \begin{equation}\label{eqn:ux}
                \left|\hat{G}_h(x) - G_h(x)\right| \leq \kappa_\delta\left(\frac{1}{ph} + L_{\delta}^{-1} \frac{p-2s}{ph}\right) + \frac{4\kappa_\delta^2}{ph} + \frac{1/4}{2ph} \sum_{i \in \mathcal{O}} P_{\theta}\left\{|t - Y_i| > h\right\} 
            \end{equation}
            with probability at least \(1 - \delta/2\). Since \(L_\delta\) and \(p\) are sufficiently large depending only on \(\delta\), it immediately follows from (\ref{eqn:ux_signal}) and (\ref{eqn:ux}) that the event \(\mathcal{E}_{\mathcal{U}}\) has probability at least \(1 - \delta/2\).
    
            Let us now examine \(\mathcal{E}_{\mathcal{V}}\). Fix \(t \in \mathcal{V}\). It follows from Proposition \ref{prop:mean_diff_fix} and \(t \in \mathcal{V}\) that
            \begin{equation}\label{eqn:vx_signal}
                G_h(x) - G_h(t) \geq \frac{1}{2} \cdot \frac{p-2|\mathcal{O}|}{2ph}.  
            \end{equation}
            Arguing similarly as in the analysis of \(\mathcal{E}_{\mathcal{U}}\), we have with probability at least \(1 - \delta/2\), 
            \begin{align*}
                \left|\hat{G}_h(x) - G_h(x)\right| &\leq \kappa_\delta\left(\frac{1}{ph} + L_{\delta}^{-1} \frac{p-2s}{ph}\right) + \kappa_\delta \sqrt{\frac{1}{p^2h^2}\sum_{i \in F} P_{\theta}\left\{|x - Y_i| \leq h\right\}} \\
                &\leq \kappa_\delta\left(\frac{1}{ph} + L_{\delta}^{-1} \frac{p-2s}{ph}\right) + \kappa_\delta \sqrt{\frac{1}{p^2h^2}\sum_{i \in \mathcal{O}} P_{\theta}\left\{|t - Y_i| > h\right\}} \\
                &\leq \kappa_\delta\left(\frac{1}{ph} + L_{\delta}^{-1} \frac{p-2s}{ph}\right) + \kappa_\delta \sqrt{\frac{4}{ph^2}e^{-Ch^2}}  \\
                &\leq \kappa_\delta\left(\frac{1}{ph} + L_{\delta}^{-1} \frac{p-2s}{ph} + \frac{2e^{-Ch^2/2}}{h\sqrt{p}}\right)
            \end{align*}
            where we have used \(t \in \mathcal{V}\) and the definition of \(\mathcal{V}\) to obtain the penultimate line. As \(L_\delta\) and \(p\) are sufficiently large depending only on \(\delta\), and \(C_1\) is sufficiently large it follows that 
            \begin{equation}\label{eqn:vx}
                \left|\hat{G}_h(x) - G_h(x)\right| < \frac{1}{4} \cdot \frac{p-2s}{ph}
            \end{equation}
            with probability at least \(1-\delta/2\). Therefore, from (\ref{eqn:vx_signal}) and (\ref{eqn:vx}) we have that \(\mathcal{E}_{\mathcal{V}}\) holds with probability at least \(1-\delta/2\). The proof is complete. 
        \end{proof}
    
        \subsubsection{Synthesis}
        With the stochastic errors handled, we are now in position to prove the main result about the kernel mode estimator. 
    
        \begin{proposition}\label{prop:kme_hard}
            Suppose \(s \leq \frac{p}{2} - p^{1/4}\) and \(C_a\) is a sufficiently large universal constant. Fix \(\delta \in (0, 1)\) and suppose 
            \begin{equation*}
                h = C_1 \sqrt{1 \vee \log\left(\frac{L_\delta p}{(p-2s)^2}\right)}
            \end{equation*}
            where \(C_1 > 0\) is a universal constant and \(L_\delta\) is sufficiently large depending only \(\delta\). If \(p\) is sufficiently large depending only on \(\delta\), then there exists a point \(x\) with \(|x - \mu| \leq \frac{h}{4}\) such that 
            \begin{equation*}
                \inf_{||\theta||_0 \leq s} P_{\theta}\left\{ \hat{G}_h(x) > \hat{G}_h(t) \text{ for all } t \in \R \text{ with } |t - \mu| \geq C_a h\right\} \geq 1-\delta. 
            \end{equation*}
        \end{proposition}
        \begin{proof}
            Let \(x\) with \(|x - \mu| \leq \frac{h}{4}\) denote the point from Proposition \ref{prop:mean_diff_fix}. Since \(L_\delta\) and \(p\) are sufficiently large depending only on \(\delta\), it follows from the fact \(\mathcal{U} \cup \mathcal{V} = \left\{ t \in \R : |t - \mu| \geq C_a h\right\}\) and Proposition \ref{prop:x_control} that
            \begin{equation*} 
                \mathcal{E}_{x} = \left\{ \left|\hat{G}_h(x) - G_h(x)\right| < \frac{1}{2} \left(G_h(x) - G_h(t)\right) \text{ for all } |t - \mu| \geq C_a h\right\}
            \end{equation*}
            has probability at least \(1-\delta/3\). Taking \(L_\delta\) sufficiently large and invoking Propositions \ref{prop:U_control} and \ref{prop:V_control}, the events 
            \begin{align*}
                \mathcal{E}_{\mathcal{U}} &= \left\{ \left|\hat{G}_h(t) - G_h(t)\right| < \frac{1}{2}(G_h(x) - G_h(t)) \text{ for all } t \in \mathcal{U} \right\}, \\
                \mathcal{E}_{\mathcal{V}} &= \left\{ \left|\hat{G}_h(t) - G_h(t) \right| < \frac{1}{2}(G_h(x) - G_h(t)) \text{ for all } t \in \mathcal{V} \right\}
            \end{align*}
            each have probability at least \(1 - \delta/3\). Therefore, by union bound it follows that \(\mathcal{E} = \mathcal{E}_{x} \cap \mathcal{E}_{\mathcal{U}} \cap \mathcal{E}_{\mathcal{V}}\) has probability at least \(1-\delta\). Furthermore, on \(\mathcal{E}\) we have for any \(|t - \mu| \geq C_a h\), 
            \begin{align*}
                \hat{G}_h(x) > G_h(x) - \frac{1}{2} \left(G_h(x) - G_h(t)\right) \geq G_h(x) - G_h(t) - \left|\hat{G}_h(t) - G_h(t)\right| - \frac{1}{2} \left(G_h(x) - G_h(t)\right) + \hat{G}_h(t).
            \end{align*}
            Note that \(|t - \mu| \geq C_a h\) implies \(t \in \mathcal{U}\) or \(t \in \mathcal{V}\). Since we are on the event \(\mathcal{E}\), in either case we have
            \begin{equation*}
                \hat{G}_h(x) > G_h(x) - G_h(t) - \frac{1}{2}\left(G_h(x) - G_h(t)\right) -\frac{1}{2} \left(G_h(x) - G_h(t)\right) + \hat{G}_h(t) = \hat{G}_h(t). 
            \end{equation*}
            Since this holds for all \(|t - \mu| \geq C_a h\) and \(\mathcal{E}\) has probability at least \(1-\delta\), the proof is complete. 
        \end{proof}
    
        \subsubsection{Proof of Theorem \ref{thm:kme_known_var}}
    
        \begin{proof}[Proof of Theorem \ref{thm:kme_known_var}]
            Theorem \ref{thm:kme_known_var} follows directly from combining Propositions \ref{prop:kme_easy} and \ref{prop:kme_hard}.
        \end{proof}
    
        \subsection{Lower bound}\label{section:lower_bound_proof}
        The proofs for the results stated in Section \ref{section:lower_bounds} are presented in this section. 
        
        \subsubsection{\texorpdfstring{Regime \(1 \leq s < \frac{p}{2} - \sqrt{p}\)}{Regime 1 <= s < p/2 - sqrt(p)}}
        \begin{proof}[Proof of Proposition \ref{prop:sparse_lbound}]
            Fix \(\delta \in (0, 1)\). The case \(\gamma = 1\) is trivial so it suffices to consider \(\gamma \in [0, 1)\). We break up the analysis into two cases. \newline 
    
            \noindent \textbf{Case 1:} Suppose \(1 \leq s < \frac{p}{2}\). For ease of notation, set \(\varepsilon^2 = (1-\gamma) s \log\left(\frac{ep}{s}\right)\). Let \(c_1\) and \(c_2\) denote the universal positive constants in Corollary \ref{corollary:sparse_gilbert_varshamov}. Let \(c_\delta  = \sqrt{\frac{c_2\delta}{2}}\). Define the finite set \(\mathcal{K} := \left\{\frac{c_\delta \varepsilon}{\sqrt{s}} m : m \in \mathcal{M}\right\}\) where \(\mathcal{M}\) is the set asserted to exist by Corollary \ref{corollary:sparse_gilbert_varshamov} at sparsity level \(s\). Observe that \(||\theta - \theta'||^2 \geq c_1 c_\delta^2\varepsilon^2\) for \(\theta, \theta' \in \mathcal{K}\) with \(\theta \neq \theta'\). Hence, by Proposition \ref{prop:reduce_to_test}, we have 
            \begin{equation*}
                \inf_{\hat{\theta}} \sup_{||\theta||_0 \leq s} P_{\theta, \gamma}\left\{ ||\hat{\theta} - \theta||^2 \geq \frac{c_1c_\delta^2}{4} \varepsilon^2 \right\} \geq \inf_{\varphi} \max_{\theta \in \mathcal{K}} P_{\theta, \gamma}\left\{\varphi(X) \neq \theta\right\}
            \end{equation*}
            where the infimum runs over all measurable functions \(\varphi : \R^p \to \mathcal{K}\). By Proposition \ref{prop:fano} we have 
            \begin{equation*}
                \inf_{\varphi} \max_{\theta \in \mathcal{K}} P_{\theta, \gamma}\left\{\varphi(X) \neq \theta \right\} \geq 1 - \frac{\frac{1}{|\mathcal{K}|^2} \sum_{\theta, \theta' \in \mathcal{K}} \dKL(P_{\theta, \gamma} || P_{\theta', \gamma}) + \log 2}{\log |\mathcal{K}|}.
            \end{equation*}
            Letting \(\Omega^{-1} = (1-\gamma)I_p + \gamma \mathbf{1}_p\mathbf{1}_p^\intercal\), consider for \(\theta, \theta' \in \mathcal{K}\) we have 
            \begin{equation*}
                \dKL(P_{\theta, \gamma}||P_{\theta', \gamma}) = \frac{\langle \theta - \theta', \Omega (\theta - \theta') \rangle}{2} = \frac{1}{2}\left( \langle \theta, \Omega \theta\rangle + \langle \theta', \Omega \theta'\rangle - 2\langle \theta, \Omega \theta'\rangle \right).
            \end{equation*}
            By Lemma \ref{lemma:precision} it follows \(\Omega = \frac{1}{1-\gamma}\left(I_p - \frac{1}{p}\mathbf{1}_p\mathbf{1}_p^\intercal \right) + \frac{1}{1-\gamma+\gamma p} \cdot \frac{1}{p}\mathbf{1}_p\mathbf{1}_p^\intercal\). Since \(||\theta||_0 = ||\theta'||_0 = s\), it follows 
            \begin{equation*}
                \langle \theta, \Omega \theta\rangle = \langle \theta', \Omega \theta' \rangle = \frac{c_\delta^2\varepsilon^2}{s} \left(\frac{s - \frac{s^2}{p}}{1-\gamma} + \frac{s^2}{p(1-\gamma+\gamma p)}\right)
            \end{equation*}
            and by the definition of \(\mathcal{K}\) we have 
            \begin{align*}
                \langle \theta, \Omega \theta' \rangle &= \frac{c_\delta^2\varepsilon^2}{s} \left(\frac{\sum_{i=1}^{p} \mathbbm{1}_{\{m_i = m_i' = 1\}}}{1-\gamma} - \frac{s^2}{p(1-\gamma)} + \frac{s^2}{p(1-\gamma+\gamma p)} \right)\\
                &\geq \frac{c_\delta^2\varepsilon^2}{s} \left( - \frac{s^2}{p(1-\gamma)} + \frac{s^2}{p(1-\gamma+\gamma p)} \right).
            \end{align*}
            Therefore, \(\dKL(P_{\theta, \gamma}||P_{\theta', \gamma}) \leq \frac{c_\delta^2\varepsilon^2}{1-\gamma}\). Since this inequality holds for all \(\theta, \theta' \in \mathcal{K}\), it immediately follows  
            \begin{equation*}
                \inf_{\varphi} \max_{\theta} P_{\theta, \gamma}\left\{\varphi(X) \neq \theta\right\} \geq 1 - \frac{\frac{c_\delta^2\varepsilon^2}{1-\gamma} + \log 2}{c_2 s \log\left(\frac{ep}{s}\right)} \geq 1 - \frac{c_\delta^2}{c_2} - \frac{\log 2}{c_2 \log(ep)}
            \end{equation*}
            where we have used \(\log |\mathcal{K}| \geq c_2 s \log\left(\frac{ep}{s}\right)\) as asserted by Corollary \ref{corollary:sparse_gilbert_varshamov}. Since \(p\) is sufficiently large depending only on \(\delta\), we have \(\frac{\log 2}{c_2 \log(ep)} \leq \frac{\delta}{2}\). We also have \(\frac{c_\delta^2}{c_2} \leq \frac{\delta}{2}\) by our choice of \(c_\delta\). Thus, we have shown 
            \begin{equation*}
                \inf_{\hat{\theta}} \sup_{||\theta||_0 \leq s} P_{\theta, \gamma}\left\{ ||\hat{\theta} - \theta||^2 \geq \frac{c_1 c_\delta^2}{4} \varepsilon^2 \right\} \geq 1 - \delta,
            \end{equation*}
            which is the desired result since \(c_1\) is a universal constant. \newline
    
            \noindent \textbf{Case 2:} Suppose \(s \geq \frac{p}{2}\). Without loss of generality, assume \(\frac{p}{2}\) is an integer. We can repeat exactly the analysis of Case 1 except replacing every instance of \(s\) with \(\frac{p}{2}\), yielding the desired result.
        \end{proof}
    
        \subsubsection{\texorpdfstring{Regime \(\frac{p}{2} - \sqrt{p} < s < \frac{p}{2} \)}{Regime p/2 - sqrt(p) < s < p/2}}
        \begin{lemma}\label{lemma:thetabar_lbound_lemma}
            Suppose \(\frac{p}{2} - \sqrt{p} \leq s < \frac{p}{2}\) and \(\gamma \in [0, 1]\). If \(\delta \in (0, 1)\), then there exists \(c_\delta > 0\) depending only on \(\delta\) such that 
            \begin{equation*}
                \inf_{\hat{\theta}} \sup_{||\theta||_0 \leq s} P_{\theta, \gamma} \left\{||\hat{\theta} - \theta||^2 \geq \frac{c_\delta^2}{16} \left( (1-\gamma)p \log\left(1 + \frac{p}{(p-2s)^2}\right) \wedge (1-\gamma+\gamma p) \right) \right\} \geq 1 - \delta.
            \end{equation*}
        \end{lemma}
        \begin{proof}
            The case \(\gamma = 1\) is trivial so let us focus on the case \(\gamma \in [0, 1)\). Set \(\varepsilon^2 = (1-\gamma)p\log\left(1 + \frac{p}{(p-2s)^2}\right) \wedge \left(1 - \gamma + \gamma p\right)\). For ease of notation, let \(E_0, E_1 \subset [p]\) with \(E_0 = \left\{1,...,\left\lceil \frac{p}{2}\right\rceil\right\}\) and \(E_1 = E_0^c\). Define 
            \begin{equation*}
                c_\delta = \frac{2\delta}{3c_1} \wedge \frac{1}{2} \wedge \sqrt{\frac{\log\left(1 + \frac{4\delta^2}{9}\right)}{16}}
            \end{equation*}
            where \(c_1 > 0\) is a universal constant which will be defined later in the proof. Let \(\pi_0\) be the prior in which a draw \(\theta \sim \pi_0\) is obtained by setting \(\theta = -\frac{c_\delta\varepsilon}{\sqrt{s}} \frac{|E_0|}{p}\mathbf{1}_{S_0}\) where \(S_0 \subset E_0\) is a size \(s\) set drawn uniformly at random. Similarly, let \(\pi_1\) denote the prior in which a draw \(\theta \sim \pi_1\) is obtained by setting \(\theta = \frac{c_\delta\varepsilon}{\sqrt{s}}\frac{|E_1|}{p}\mathbf{1}_{S_1}\) where \(S_1 \subset E_1\) is a size \(s\) set drawn uniformly at random. Note that \(S_0\) and \(S_1\) are disjoint with probability one. Moreover, observe that if \(\theta_0 \sim \pi_0\) and \(\theta_1 \sim \pi_1\), we have 
            \begin{equation*}
                \bar{\theta}_1 - \bar{\theta}_0 = \frac{c_\delta\varepsilon}{\sqrt{s}} \cdot \frac{s}{p} \cdot \frac{|E_0| + |E_1|}{p} \geq \frac{c_\delta\varepsilon}{2\sqrt{p}}
            \end{equation*}
            almost surely where we have used \(s \geq \frac{p}{2} - \sqrt{p} \geq \frac{p}{4}\) as \(p\) is sufficiently large. Consequently, we have 
            \begin{equation}\label{eqn:mean_separation}
                ||\theta_0 - \theta_1||^2 \geq \left|\left| \bar{\theta}_1\mathbf{1}_p - \bar{\theta}_0\mathbf{1}_p\right|\right|^2 \geq \frac{c_\delta^2\varepsilon^2}{4}.
            \end{equation}
            Let \(P_{\pi_0}\) denote \(\int N(\theta, \Sigma) \, \pi_0(d\theta)\) and \(P_{\pi_1}\) denote \(\int N(\theta, \Sigma) \, \pi_1(d\theta)\) where \(\Sigma = (1-\gamma)I_p + \gamma \mathbf{1}_p\mathbf{1}_p^\intercal\). With the separation (\ref{eqn:mean_separation}) in hand, we are able to invoke Proposition \ref{prop:fuzzy_hypotheses} to obtain the bound 
            \begin{equation}\label{eqn:theta2_fuzzy}
                \inf_{\hat{\theta}} \sup_{||\theta||_0 \leq s} P_{\theta, \gamma}\left\{||\hat{\theta} - \theta||^2 \geq \frac{c_\delta^2 \varepsilon^2}{16} \right\} \geq 1 - \dTV\left(P_{\pi_0}, P_{\pi_1}\right).
            \end{equation}
            Consequently, all that remains is to show \(\dTV\left(P_{\pi_0}, P_{\pi_1}\right) \leq \delta\) in order to obtain the desired result. We pursue this objective now. \newline
                
            We will now work with some transformations, so for discussion let us draw \(Y \sim P_{\pi_j}\) for \(j = 0, 1\). Consider we can write \(Y \in \R^p\) as \(Y = (Y_{E_0}, Y_{E_1})\). Writing \(\bar{Y}_{E_0} := \frac{1}{|E_0|} \sum_{i \in E_0} Y_i\) and \(\bar{Y}_{E_1} := \frac{1}{|E_1|} \sum_{i \in E_1} Y_i\), it is straightforward to see from (\ref{model:additive}) that since \(E_0\) and \(E_1\) are disjoint, we have that the random quantities \(Y_{E_0}-\bar{Y}_{E_0}\mathbf{1}_{E_0}, Y_{E_1} - \bar{Y}_{E_1}\mathbf{1}_{E_1}\), and \((\bar{Y}_{E_0}, \bar{Y}_{E_1})\) are mutually independent. Writing \(P_{\pi_j}^{I}, P_{\pi_j}^{II}, P_{\pi_j}^{III}\) to denote the marginal distributions of these three random vectors, we can apply Lemma \ref{lemma:tv_tensorization} to obtain 
            \begin{equation}\label{eqn:tripartite_tv_bound}
                \dTV(P_{\pi_0}, P_{\pi_1}) \leq \dTV(P_{\pi_0}^I, P_{\pi_1}^{I}) + \dTV(P_{\pi_0}^{II}, P_{\pi_1}^{II}) + \dTV(P_{\pi_0}^{III}, P_{\pi_1}^{III}).
            \end{equation}
            We bound each term in (\ref{eqn:tripartite_tv_bound}) separately. We first examine the third term. Consider that 
            \begin{equation*}
                P_{\pi_0}^{III} = N\left( \left(\begin{matrix} -\frac{c_\delta \varepsilon}{\sqrt{s}} \cdot \frac{s}{p} \\ 0  \end{matrix} \right), \left(\begin{matrix} \frac{1-\gamma}{|E_0|} + \gamma & \gamma \\ \gamma & \frac{1-\gamma}{|E_1|} + \gamma \end{matrix}\right)\right)
            \end{equation*}
            and 
            \begin{equation*}
                P_{\pi_1}^{III} = N\left( \left(\begin{matrix} 0 \\ \frac{c_\delta \varepsilon}{\sqrt{s}} \cdot \frac{s}{p} \end{matrix} \right), \left(\begin{matrix} \frac{1-\gamma}{|E_0|} + \gamma & \gamma \\ \gamma & \frac{1-\gamma}{|E_1|} + \gamma \end{matrix}\right)\right).
            \end{equation*}
            Observe that 
            \begin{equation*}
                \left(\begin{matrix} \frac{1-\gamma}{|E_0|} + \gamma & \gamma \\ \gamma & \frac{1-\gamma}{|E_1|} + \gamma \end{matrix}\right)^{-1} = \frac{1}{\left(\frac{1-\gamma}{|E_0|} + \gamma\right)\left(\frac{1-\gamma}{|E_1|} +\gamma\right) - \gamma^2} \left(\begin{matrix} \frac{1-\gamma}{|E_1|} + \gamma & -\gamma \\ -\gamma & \frac{1-\gamma}{|E_0|} + \gamma \end{matrix}\right) 
            \end{equation*}
            and so
            \begin{align*}
                &\left\langle \left(\begin{matrix}\frac{c_\delta \varepsilon}{\sqrt{s}}\cdot \frac{s}{p} \\ \frac{c_\delta \varepsilon}{\sqrt{s}} \cdot \frac{s}{p} \end{matrix}\right), \left(\begin{matrix} \frac{1-\gamma}{|E_0|} + \gamma & \gamma \\ \gamma & \frac{1-\gamma}{|E_1|} + \gamma \end{matrix}\right)^{-1} \left(\begin{matrix}\frac{c_\delta \varepsilon}{\sqrt{s}}\cdot \frac{s}{p} \\ \frac{c_\delta \varepsilon}{\sqrt{s}} \cdot \frac{s}{p} \end{matrix}\right) \right\rangle \\
                &= \frac{c_\delta^2 \varepsilon^2 s}{p^2} \cdot \frac{1}{\left(\frac{1-\gamma}{|E_0|} + \gamma\right)\left(\frac{1-\gamma}{|E_1|} +\gamma\right) - \gamma^2} \left((1-\gamma) \frac{|E_0|}{|E_1|} + (1-\gamma) \frac{|E_1|}{|E_0|} \right) \\
                &\leq \frac{c_\delta^2\varepsilon^2}{\left(\frac{1-\gamma}{|E_0|} + \gamma\right)\left(\frac{1-\gamma}{|E_1|} +\gamma\right) - \gamma^2} \left(\frac{1-\gamma}{|E_1|} + \frac{1-\gamma}{|E_0|} \right)
            \end{align*}
            where we have used \(s \leq p\) as well as \(|E_0| \vee |E_1| \leq p\) to obtain the final line. Consider that \(|E_1| \asymp |E_2| \asymp p\). Hence, there exists some universal constant \(c_1 > 0\) such that 
            \begin{equation*}
                \frac{c_\delta^2\varepsilon^2}{\left(\frac{1-\gamma}{|E_0|} + \gamma\right)\left(\frac{1-\gamma}{|E_1|} +\gamma\right) - \gamma^2} \left(\frac{1-\gamma}{|E_1|} + \frac{1-\gamma}{|E_0|} \right) = \frac{c_\delta^2 \varepsilon^2}{\frac{1-\gamma}{|E_0||E_1|} + \frac{\gamma}{|E_0|} + \frac{\gamma}{|E_1|}} \cdot \frac{|E_0| + |E_1|}{|E_0||E_1|} \leq \frac{c_1^2c_\delta^2 \varepsilon^2}{1-\gamma + \gamma p} \leq c_1^2c_\delta^2
            \end{equation*}
            where we have used \(1-\gamma \leq 1\) and \(\varepsilon^2 \leq 1-\gamma+\gamma p\) in the last step. Therefore, by Pinsker's inequality we have 
            \begin{equation}\label{eqn:tv_bound_III}
                \dTV(P_{\pi_0}^{III}, P_{\pi_1}^{III}) \leq \sqrt{\frac{\dKL(P_{\pi_0}^{III}, P_{\pi_1}^{III})}{2}} \leq \frac{\sqrt{c_1^2c_\delta^2}}{2} = \frac{c_1c_\delta}{2} \leq \frac{\delta}{3}
            \end{equation}
            where we have used the definition of \(c_\delta\). \newline
            
            We now examine the first term in (\ref{eqn:tripartite_tv_bound}). By the definition of \(\pi_0\) and \(\pi_1\), we can write
            \begin{equation*}
                P_{\pi_0}^{I} = \int N\left(\theta_{E_0} - \bar{\theta}_{E_0}\mathbf{1}_{E_0}, (1-\gamma)\left(I_{E_0} - \frac{1}{|E_0|} \mathbf{1}_{E_0}\mathbf{1}_{E_0}^\intercal\right) \right) \, \pi_0(d\theta)
            \end{equation*}
            and 
            \begin{equation*}
                P_{\pi_1}^{I} = N\left( 0, (1-\gamma)\left(I_{E_0} - \frac{1}{|E_0|} \mathbf{1}_{E_0}\mathbf{1}_{E_0}^\intercal\right)\right).
            \end{equation*}
            To bound the total variation distance between these two distributions, we will furnish an upper bound on the \(\chi^2\) divergence. By the Ingster-Suslina method (see Lemma \ref{lemma:ingster_suslina}) we have
            \begin{align}\label{eqn:ingster_suslina_I}
                \chi^2\left(\left.\left.P_{\pi_0}^I \right|\right| P_{\pi_1}^I\right) + 1 &= E\left(\exp\left(\frac{1}{1-\gamma} \left\langle \left(\theta_{E_0} - \bar{\theta}_{E_0} \mathbf{1}_{E_0}\right), \left(I_{E_0} - \frac{1}{|E_0|}\mathbf{1}_{E_0}\mathbf{1}_{E_0}^\intercal\right) \left(\theta'_{E_0} - \bar{\theta'}_{E_0} \mathbf{1}_{E_0}\right)\right\rangle \right)\right)
            \end{align}
            where \(\theta, \theta' \overset{iid}{\sim} \pi_0\). Writing \(\theta = -\frac{c_\delta \varepsilon}{\sqrt{s}}\frac{|E_0|}{p} \mathbf{1}_{S_0}\) and \(\theta' = -\frac{c_\delta \varepsilon}{\sqrt{s}}\frac{|E_0|}{p} \mathbf{1}_{S_0'}\), it follows that 
            \begin{equation}\label{eqn:chisquare_bound_I}
                \chi^2\left(\left.\left. P_{\pi_0}^I \right|\right| P_{\pi_1}^I\right) + 1 = E\left(\exp\left(\frac{c_\delta^2 \varepsilon^2}{s(1-\gamma)}\frac{|E_0|}{p} \left(|S_0 \cap S_0'| - \frac{s^2}{|E_0|}\right) \right)\right)
            \end{equation}
            where \(S_0, S_0'\) are independently and uniformly drawn from the collection of all size \(s\) subsets of \(E_0\). Repeating the same argument yields the analogous expressions
            \begin{align}\label{eqn:ingster_suslina_II}
                \chi^2\left(\left.\left.P_{\pi_1}^{II} \right|\right| P_{\pi_0}^{II}\right) + 1 &= E\left(\exp\left(\frac{1}{1-\gamma} \left\langle \left(\theta_{E_1} - \bar{\theta}_{E_1} \mathbf{1}_{E_1}\right), \left(I_{E_1} - \frac{1}{|E_1|}\mathbf{1}_{E_1}\mathbf{1}_{E_1}^\intercal\right) \left(\theta'_{E_1} - \bar{\theta'}_{E_1} \mathbf{1}_{E_1}\right)\right\rangle \right)\right)
            \end{align}
            and
            \begin{equation}\label{eqn:chisquare_bound_II}
                \chi^2\left(\left.\left.P_{\pi_1}^{II} \right|\right| P_{\pi_0}^{II}\right) + 1 = E\left(\exp\left(\frac{c_\delta^2 \varepsilon^2}{s(1-\gamma)}\frac{|E_1|}{p} \left(|S_1 \cap S_1'| - \frac{s^2}{|E_1|}\right) \right)\right)
            \end{equation}
            where \(S_1, S_1'\) are independently and uniformly drawn from the collection of all size \(s\) subsets of \(E_1\). \newline
            
            We first examine (\ref{eqn:chisquare_bound_I}). For ease of notation, let \(\tau^2 = \frac{c_\delta^2 \varepsilon^2}{s}\left(\frac{|E_0|}{p}\right)^2\). Consider that since \(I_{E_0} - \frac{1}{|E_0|}\mathbf{1}_{E_0}\mathbf{1}_{E_0}^\intercal\) is a projection matrix, we have the equalities 
            \begin{align*}
                \chi^2\left(\left.\left.P_{\pi_0}^I \right|\right| P_{\pi_1}^I\right) + 1 &= E\left(\exp\left(\frac{1}{1-\gamma} \left\langle \left(\theta_{E_0} - \bar{\theta}_{E_0} \mathbf{1}_{E_0}\right), \left(I_{E_0} - \frac{1}{|E_0|}\mathbf{1}_{E_0}\mathbf{1}_{E_0}^\intercal\right) \left(\theta'_{E_0} - \bar{\theta'}_{E_0} \mathbf{1}_{E_0}\right)\right\rangle \right)\right) \\
                &= E\left(\exp\left(\frac{1}{1-\gamma} \left\langle \theta_{E_0}, \left(I_{E_0} - \frac{1}{|E_0|}\mathbf{1}_{E_0}\mathbf{1}_{E_0}^\intercal\right) \theta'_{E_0}\right\rangle \right)\right) \\
                &= E\left(\exp\left(\frac{1}{1-\gamma}\left\langle \theta_{E_0} + \tau\mathbf{1}_{E_0}, \left(I_{E_0} - \frac{1}{|E_0|}\mathbf{1}_{E_0}\mathbf{1}_{E_0}^\intercal\right) \left(\theta'_{E_0} + \tau \mathbf{1}_{E_0}\right)\right\rangle \right)\right). 
            \end{align*}
            Since \(\theta_{E_0} = -\tau \mathbf{1}_{S_0}\) and \(\theta_{E_0}' = -\tau\mathbf{1}_{S_0'}\), we can write \(\theta_{E_0} + \tau \mathbf{1}_{E_0} = \tau \mathbf{1}_{T_0}\) and \(\theta_{E_0}' + \tau \mathbf{1}_{E_0} = \tau \mathbf{1}_{T_0'}\) where \(T_0 = S_0^c\) and \(T_0' = (S_0')^c\). Note that \(T_0\) and \(T_0'\) are independent and uniformly drawn size \(|E_0| - s\) subsets of \(E_0\). Consequently, we have 
            \begin{align*}
                \chi^2\left(\left.\left. P_{\pi_0}^I \right|\right| P_{\pi_1}^I\right) + 1 &= E\left(\exp\left(\frac{\tau^2}{1-\gamma}\left\langle \mathbf{1}_{T_0}, \left(I_{E_0} - \frac{1}{|E_0|}\mathbf{1}_{E_0}\mathbf{1}_{E_0}^\intercal\right) \mathbf{1}_{T_0'} \right\rangle \right)\right) \\
                &= E\left[\exp\left(\frac{\tau^2}{1-\gamma} \left(|T_0 \cap T_0'| - \frac{(|E_0|-s)^2}{|E_0|}\right)\right)\right] \\
                &\leq E\left(\exp\left(\frac{\tau^2}{1-\gamma}|T_0 \cap T_0'|\right)\right) \\
                &\leq \left(1 - \frac{|E_0|-s}{|E_0|} + \frac{|E_0|-s}{|E_0|} e^{\frac{\tau^2}{1-\gamma}}\right)^{|E_0|-s} \\
                &\leq \left(1 - \frac{|E_0|-s}{|E_0|} + \frac{|E_0|-s}{|E_0|} e^{4c_\delta^2\log\left(1 + \frac{p}{(p-2s)^2}\right)}\right)^{|E_0|-s} \\
                &\leq \left(1 - \frac{|E_0|-s}{|E_0|} + \frac{|E_0|-s}{|E_0|} \left(1 + \frac{4c_\delta^2p}{(p-2s)^2}\right)\right)^{|E_0|-s} \\
                &= \left(1 + \frac{1}{|E_0|-s} \left(\frac{4c_\delta^2p(|E_0|-s)^2}{(p-2s)^2|E_0|}\right) \right)^{|E_0|-s}
            \end{align*}
            where we have used \(s \geq \frac{p}{4}\), Lemma \ref{lemma:hypergeometric}, and the inequality \((1+x)^y \leq 1 + yx\) for all \(x > 0\) and \(y \in [0, 1]\). Continuing the calculation yields 
            \begin{align*}
                \chi^2\left(\left.\left. P_{\pi_0}^I \right|\right| P_{\pi_1}^I\right) + 1 &\leq \exp\left(\frac{4c_\delta^2p(|E_0|-s)^2}{(p-2s)^2|E_0|}\right) \\
                &\leq \exp\left(\frac{4c_\delta^2(|E_0|-s)^2}{\left(\frac{p}{2}-s\right)^2}\right) \\
                &\leq e^{16c_\delta^2}
            \end{align*}
            where we have used \(|E_0|/p \leq 2\) and \((|E_0| - s)^2 \leq \left(\frac{p}{2} - s + 1\right)^2 \leq 2\left(\frac{p}{2}-s\right)^2 + 2\). The above arguments can be repeated to also show that \(\chi^2\left(\left.\left. P_{\pi_1}^{II} \right|\right| P_{\pi_0}^{II}\right) + 1 \leq e^{16c_\delta^2}\) as well. Consequently, from these bounds on the \(\chi^2\) divergence we immediately obtain the bounds  
            \begin{align}
                \dTV(P_{\pi_0}^I, P_{\pi_1}^{I}) &\leq \frac{1}{2}\sqrt{\chi^2\left(\left.\left. P_{\pi_0}^I \right|\right| P_{\pi_1}^I\right)} \leq \frac{1}{2}\sqrt{e^{16c_\delta^2} - 1} \leq \frac{\delta}{3} \label{eqn:tv_bound_I_case2}\\
                \dTV(P_{\pi_0}^{II}, P_{\pi_1}^{II}) &\leq \frac{1}{2}\sqrt{\chi^2\left(\left.\left. P_{\pi_1}^{II} \right|\right| P_{\pi_0}^{II}\right)} \leq \frac{1}{2}\sqrt{e^{16c_\delta^2} - 1} \leq \frac{\delta}{3}\label{eqn:tv_bound_II_case2}.
            \end{align}
            where we have used \(16c_\delta^2 \leq \log\left(1 + \frac{4\delta^2}{9}\right)\). Plugging in (\ref{eqn:tv_bound_I_case2}), (\ref{eqn:tv_bound_II_case2}), and (\ref{eqn:tv_bound_III}) into (\ref{eqn:tripartite_tv_bound}) shows that \(1-\dTV(P_{\pi_0}, P_{\pi_1}) \geq 1-\delta\). Combining this with (\ref{eqn:theta2_fuzzy}) yields the desired result.
        \end{proof}
    
        \begin{proof}[Proof of Proposition \ref{prop:thetabar_lbound}]
            Let \(c_\delta = \frac{1}{2}\left(c_{\delta, 1} \vee c_{\delta, 2}^2/16\right)\) where \(c_{\delta, 1}\) and \(c_{\delta, 2}\) are the constants in Proposition \ref{prop:sparse_lbound} and Lemma \ref{lemma:thetabar_lbound_lemma} respectively. Since \(2c_\delta \geq c_{\delta, 1}\) and \(2c_\delta \geq c_{\delta, 2}^2/16\), it follows by Proposition \ref{prop:sparse_lbound} and Lemma \ref{lemma:thetabar_lbound_lemma} that 
            \begin{align*}
                &\inf_{\hat{\theta}} \sup_{||\theta||_0 \leq s} P_{\theta, \gamma}\left\{ ||\hat{\theta} - \theta||^2 \geq c_\delta \left( (1-\gamma)p \log\left(1 + \frac{p}{(p-2s)^2}\right) \wedge p\right) \right\} \\
                &\geq \inf_{\hat{\theta}} \sup_{||\theta||_0 \leq s} P_{\theta, \gamma}\left\{ ||\hat{\theta} - \theta||^2 \geq c_\delta(1-\gamma)p + c_\delta\left( (1-\gamma)p \log\left(1 + \frac{p}{(p-2s)^2}\right) \wedge (1-\gamma+\gamma p) \right)\right\} \\
                &\geq \left( \inf_{\hat{\theta}} \sup_{||\theta||_0 \leq s} P_{\theta, \gamma}\left\{ ||\hat{\theta} - \theta||^2 \geq 2c_\delta(1-\gamma)p \right\} \right) \wedge \\
                &\;\;\; \left( \inf_{\hat{\theta}} \sup_{||\theta||_0 \leq s} P_{\theta, \gamma}\left\{ ||\hat{\theta} - \theta||^2 \geq 2c_\delta\left( (1-\gamma)p \log\left(1 + \frac{p}{(p-2s)^2}\right) \wedge (1-\gamma+\gamma p) \right)\right\} \right) \\
                &\geq 1-\delta
            \end{align*}
            as desired. 
        \end{proof}
    
        \subsubsection{\texorpdfstring{Regime \(\frac{p}{2} \leq s \leq p\)}{Regime p/2 <= s <= p}}
        \begin{lemma}\label{lemma:dense_theta_lbound_lemma}
            Suppose \(\frac{p}{2} \leq s \leq p\) and \(\gamma \in [0, 1]\). If \(\delta \in (0, 1)\), then there exists \(c_\delta > 0\) depending only on \(\delta\) such that 
            \begin{equation*}
                \inf_{\hat{\theta}} \sup_{||\theta||_0 \leq s} P_{\theta, \gamma}\left\{||\hat{\theta} - \theta||^2 \geq \frac{c_\delta^2}{4}(1-\gamma+\gamma p)\right\} \geq 1-\delta.
            \end{equation*}
        \end{lemma}
        \begin{proof}
            Set \(\varepsilon^2 = 1-\gamma+\gamma p\) and \(c_\delta = \frac{\delta}{2}\). Le Cam's two point method will be used to prove the desired lower bound. Let \(S := \{1,...,s\} \subset [p]\) and \(T = S^c\). Note that \(|T| = p - s \leq s\) since \(s \geq \frac{p}{2}\). Define \(\theta_0 = \frac{c_\delta \varepsilon}{\sqrt{s}}\mathbf{1}_S\) and \(\theta_1 = -\frac{c_\delta \varepsilon}{\sqrt{s}}\mathbf{1}_T\). Note \(\theta_0, \theta_1 \in ||\theta||_0 \leq s\) and \(\theta_0 - \theta_1 = \frac{c_\delta\varepsilon}{\sqrt{s}}\mathbf{1}_p\). Specifically, we have \(||\theta_0 - \theta_1||^2 = \frac{c_\delta^2\varepsilon^2p}{s} \geq c_\delta^2\varepsilon^2\). By Le Cam's two point method (e.g. Proposition \ref{prop:reduce_to_test}) and the Neyman-Pearson lemma, it follows 
            \begin{equation*}
                \inf_{\hat{\theta}} \sup_{||\theta||_0 \leq s} P_{\theta, \gamma}\left\{||\hat{\theta} - \theta||^2 \geq \frac{c_\delta^2}{4}\varepsilon^2 \right\} \geq 1 - \dTV(P_{\theta_0, \gamma}, P_{\theta_1,\gamma}).
            \end{equation*}
            Suppose \(\gamma < 1\). Letting \(\Omega^{-1} = (1-\gamma)I_p + \gamma \mathbf{1}_p\mathbf{1}_p^\intercal\), we have by Pinsker's inequality and Lemma \ref{lemma:precision} 
            \begin{align*}
                \dTV(P_{\theta_0, \gamma}, P_{\theta_1, \gamma}) &\leq \sqrt{\frac{\dKL(P_{\theta_1, \gamma} || P_{\theta_0, \gamma})}{2}} \\
                &\leq \frac{\sqrt{\langle \theta_1 - \theta_0, \Omega (\theta_1 - \theta_0)\rangle}}{2} \\
                &= \frac{c_\delta \varepsilon}{2\sqrt{s}} \cdot \sqrt{\frac{p}{1-\gamma+\gamma p}} \\
                &\leq \delta
            \end{align*}
            which yields the desired result. It remains to handle the case \(\gamma = 1\). Consider for any \(v \in \R^p\) we have \(X - \bar{X}\mathbf{1}_p \indep \bar{X}\mathbf{1}_p\) for \(X \sim P_{v, 1}\). Therefore, for any \(\mu, \mu' \in \R^p\) it follows \(\dTV(P_{\mu, 1}, P_{\mu', 1}) \leq \mathbbm{1}_{\left\{\mu - \bar{\mu}\mathbf{1}_p = \mu' - \bar{\mu}'\mathbf{1}_p\right\}} + \dTV(Q_\mu, Q_{\mu'})\) where \(Q_v = N\left(\bar{v}, 1\right)\) for \(v \in \R^p\). Consider \(\theta_0 - \bar{\theta}_0\mathbf{1}_p = \theta_1 - \bar{\theta}_1\mathbf{1}_p\) as well as \(\bar{\theta}_0 = \frac{c_\delta \varepsilon}{\sqrt{s}} \cdot \frac{s}{p}\) and \(\bar{\theta}_1 = -\frac{c_\delta\varepsilon}{\sqrt{s}} \cdot \frac{p-s}{p}\). Therefore by Pinsker's inequality and \(s \geq \frac{p}{2}\) we have
            \begin{align*}
                \dTV(P_{\theta_0, 1}, P_{\theta_1,1}) &\leq \dTV(Q_{\theta_0}, Q_{\theta_1}) \\
                &= \dTV\left(N\left(\frac{c_\delta \varepsilon}{\sqrt{s}} \cdot \frac{s}{p}, 1\right), N\left( -\frac{c_\delta\varepsilon}{\sqrt{s}} \cdot \frac{p-s}{p}, 1\right)\right) \\
                &\leq \sqrt{\frac{c_\delta^2\varepsilon^2}{2s}} \\
                &\leq \sqrt{\frac{c_\delta^2 \varepsilon^2}{p}} \\
                &\leq \delta
            \end{align*}
            as desired. 
        \end{proof}
    
        \begin{proof}[Proof of Proposition \ref{prop:dense_thetabar_lbound}]
            Let \(c_\delta = \frac{1}{2}\left(c_{\delta, 1} \vee c_{\delta, 2}^2/4\right)\) where \(c_{\delta, 1}\) and \(c_{\delta, 2}\) are the constants in Proposition \ref{prop:sparse_lbound} and Lemma \ref{lemma:dense_theta_lbound_lemma} respectively. Since \(2c_\delta \geq c_{\delta, 1}\) and \(2c_\delta \geq c_{\delta,2}^2/4\), it follows by Proposition \ref{prop:sparse_lbound} and Lemma \ref{lemma:dense_theta_lbound_lemma} that 
            \begin{align*}
                &\inf_{\hat{\theta}} \sup_{||\theta||_0 \leq s} P_{\theta, \gamma}\left\{ ||\hat{\theta} - \theta||^2 \geq c_\delta p\right\}\\
                &\geq \inf_{\hat{\theta}} \sup_{||\theta||_0 \leq s} P_{\theta, \gamma}\left\{ ||\hat{\theta} - \theta||^2 \geq c_{\delta} (1-\gamma)p + c_{\delta}(1-\gamma+\gamma p) \right\} \\
                &\geq \left(\inf_{\hat{\theta}} \sup_{||\theta||_0 \leq s} P_{\theta, \gamma}\left\{ ||\hat{\theta} - \theta||^2 \geq 2c_{\delta} (1-\gamma)p \right\}  \right) \wedge \left( \inf_{\hat{\theta}} \sup_{||\theta||_0 \leq s} P_{\theta, \gamma}\left\{ ||\hat{\theta} - \theta||^2 \geq 2c_{\delta} (1-\gamma+\gamma p) \right\} \right) \\
                &\geq 1 - \delta
            \end{align*}
            as desired. 
        \end{proof}

    \subsection{Adaptation to sparsity and correlation}
    This section contains the proofs for the results stated in Section \ref{section:adaptation}.
    \subsubsection{Correlation estimation}    
    \begin{proof}[Proof of Proposition \ref{prop:correlation_estimation}]
        For ease, define \(\mathcal{I} = \supp(\theta)^c\). Denote the event \(\mathcal{E} := \bigcup_{r=1}^{m} \left\{E_r \subset \mathcal{I}\right\}\). By independence, 
        \begin{equation*}
            P_{\theta, \gamma}\left(\mathcal{E}^c\right) = \left(P_{\theta, \gamma}\left\{E_1 \not\subset \mathcal{I}\right\}\right)^m = \left(1 - \frac{\binom{|\mathcal{I}|}{\ell}}{\binom{p}{\ell}}\right)^m.
        \end{equation*}
        Direct calculation yields 
        \begin{align*}
            \frac{\binom{|\mathcal{I}|}{\ell}}{\binom{p}{\ell}} = \frac{(p-\ell)!}{(|\mathcal{I}| - \ell)!} \cdot \frac{|\mathcal{I}|!}{p!} = \frac{(p-\ell)(p-\ell-1) \cdot \cdot \cdot (|\mathcal{I}| - \ell+1)!}{p(p-1)\cdot\cdot\cdot(|\mathcal{I}| + 1)} = \frac{|\mathcal{I}|(|\mathcal{I}| - 1) \cdot\cdot\cdot (|\mathcal{I}| - \ell + 1)}{p(p-1)\cdot\cdot\cdot(p-\ell+1)}.
        \end{align*}
        Since \(2 < \ell \lesssim \log p\) and \(|\mathcal{I}| \geq \frac{p}{2}\), it follows from the above display \(\frac{\binom{|\mathcal{I}|}{\ell}}{\binom{p}{\ell}} \geq \frac{(p/3)^\ell}{p^\ell} = 3^{-\ell}\). To summarize, we have shown \(P_{\theta, \gamma}(\mathcal{E}^c) \leq (1-3^{-\ell})^m \leq \exp\left(-m3^{-\ell}\right)\). Now let us examine the behavior of \(1-\hat{\gamma}\) on the event \(\mathcal{E}\). On \(\mathcal{E}\), there exists \(1 \leq r^* \leq m\) such that \(E_{r^*} \subset \mathcal{I}\). Therefore, 
        \begin{equation*}
            1 - \hat{\gamma} \leq \frac{1}{\ell-1} \sum_{j \in E_{r^*}} \left(X_j - \bar{X}_{E_{r^*}}\right)^2 = \frac{1-\gamma}{\ell-1}\sum_{j\in E_{r^*}} \left(Z_j - \bar{Z}_{E_{r^*}}\right)^2. 
        \end{equation*}
        Since the sequence of subsets \(\left\{E_r\right\}_{r=1}^{m}\) are drawn by the statistician independently of the data, we have  \(\sum_{j\in E_{r^*}} \left(Z_j - \bar{Z}_{E_{r^*}}\right)^2 \sim \chi^2_{\ell-1}\). By Lemma \ref{lemma:chisquare_tail}, we have for any \(L > 0\),
        \begin{align*}
            P_{\theta, \gamma}\left(\mathcal{E} \cap \left\{1-\hat{\gamma} > L (1-\gamma) \right\} \right) \leq P\left\{\frac{1}{\ell-1}\chi^2_{\ell-1} > L \right\} \leq e^{-c_L\ell}
        \end{align*}
        for some \(c_L > 0\) depending only on \(L\). 
        
        On the other hand, let us now examine the lower bound. From the independence of \(\left\{E_r\right\}_{r=1}^{m}\) and the data \(\left\{X_j\right\}_{j=1}^{p}\), it is clear \(\sum_{j\in E_r}(X_j - \bar{X}_{E_r})^2\) stochastically dominates \((1-\gamma) \sum_{j \in E_r} (Z_j - \bar{Z}_{E_r})^2\) for all \(1 \leq r \leq m\). From an application of Lemma \ref{lemma:chisquare_lower_tail}, we have for any \(L > 0\), 
        \begin{align*}
            P_{\theta, \gamma}\left\{1-\hat{\gamma} \leq (1-\gamma) L^{-1} \right\} &\leq \sum_{r=1}^{m} P_{\theta, \gamma}\left\{\frac{1}{\ell-1} \sum_{j\in E_r} (X_j - \bar{X}_{E_r})^2 \leq (1-\gamma) L^{-1} \right\} \\
            &\leq \sum_{r=1}^{m} P_{\theta, \gamma}\left\{\frac{1}{\ell-1} \sum_{j\in E_r} (Z_j - \bar{Z}_{E_r})^2 \leq L^{-1}\right\} \\
            &\leq m(eL^{-1})^{\frac{\ell-1}{2}} \\
            &\leq m e^{-c_L' \ell}
        \end{align*}
        where \(c_L' > 0\) depends only on \(L\). Note \(c_L'\) can be made to be as large as needed by taking \(L\) sufficiently large. Collecting our bounds, we have shown
        \begin{equation*}
            P_{\theta, \gamma}\left\{ L^{-1} \leq \frac{1-\hat{\gamma}}{1-\gamma} \leq L\right\} \geq 1 - e^{-m3^{-\ell}} - e^{-c_L \ell} - me^{-c_L' \ell}. 
        \end{equation*}
         Since \(\ell = \left\lceil C_2 \log p\right\rceil\) and \(m = \left\lceil p^{C_1}\right\rceil\), we choose \(C_1, C_2\) depending only on \(\delta\) and we choose \(L\) sufficiently large (to make \(c_L'\) large enough) depending only on \(C_1,C_2\) to give
        \begin{equation*}
            P_{\theta, \gamma}\left\{ L^{-1} \leq \frac{1-\hat{\gamma}}{1-\gamma} \leq L\right\} \geq 1 - \delta.
        \end{equation*}
        Taking infimum over \(||\theta||_0 \leq s\) and \(\gamma \in [0,1)\) yields the desired result.  
    \end{proof}
    
    \subsubsection{Sparse regression}
    In this section, we show Theorem 4.2 of \cite{bellec_slope_2018} holds with the choice of design matrix \(M = \sqrt{p}\left(I_p - \frac{1}{p}\mathbf{1}_p\mathbf{1}_p^\intercal\right)\) even when the noise is correlated with covariance matrix proportional to \(I_p - \frac{1}{p}\mathbf{1}_p\mathbf{1}_p^\intercal\). The proof is almost exactly the same. As the authors of \cite{bellec_slope_2018} note, the conclusions of Theorem 4.2 of \cite{bellec_slope_2018} hold deterministically on the event (4.1) in \cite{bellec_slope_2018}. In particular, no matter what the problem parameters are, the conclusions hold provided (4.1) is in force. Hence, it suffices to show the event (4.1) holds with high probability under the correlated noise.
    
    \begin{proposition}\label{prop:correlated_event}
        Let \(\delta_0 \in (0, 1)\) and let \(H(\cdot), G(\cdot)\) be defined in (2.8) of \cite{bellec_slope_2018}. If \(\xi \sim N\left(0, \sigma^2\left(I_p - \frac{1}{p}\mathbf{1}_p\mathbf{1}_p^\intercal\right)\right)\), then the event
        \begin{equation*}
            \left\{ \frac{1}{p} \xi^\intercal M u \leq H(u) \vee G(u) \text{ for all } u \in \R^p \right\}
        \end{equation*}
        is of probability at least \(1 - \delta_0/2\) where \(M = \sqrt{p}\left(I_p - \frac{1}{p}\mathbf{1}_p\mathbf{1}_p^\intercal\right)\). 
    \end{proposition}
    \begin{proof}
        First, consider we can write \(\xi = \left(I_p - \frac{1}{p}\mathbf{1}_p\mathbf{1}_p^\intercal\right)g\) where \(g \sim N(0, \sigma^2 I_p)\). Therefore, 
        \begin{align*}
            \xi^\intercal M u &= \sqrt{p} \xi^\intercal \left(I_p - \frac{1}{p}\mathbf{1}_p\mathbf{1}_p^\intercal\right)u \\
            &= \sqrt{p} g^\intercal \left(I_p - \frac{1}{p}\mathbf{1}_p\mathbf{1}_p^\intercal\right)^2u \\
            &= \sqrt{p} g^\intercal \left(I_p - \frac{1}{p}\mathbf{1}_p\mathbf{1}_p^\intercal\right)u \\
            &= g^\intercal M u
        \end{align*}
        where we have used \(I_p - \frac{1}{p}\mathbf{1}_p\mathbf{1}_p^\intercal\) is symmetric and idempotent. By Theorem 4.1 of \cite{bellec_slope_2018}, it follows the event 
        \begin{equation*}
            \left\{ \frac{1}{p} g^\intercal M u \leq H(u) \vee G(u) \text{ for all } u \in \R^p\right\}
        \end{equation*}
        has probability at least \(1 - \delta_0/2\), as desired. 
    \end{proof}
    
    The following statement is essentially Corollary 4.3 of \cite{bellec_slope_2018} applied to \(M = \sqrt{p}\left(I_p - \frac{1}{p}\mathbf{1}_p\mathbf{1}_p^\intercal\right)\), with a small modification of the statement to emphasize a random penalty can be accommodated.
    
    \begin{corollary}\label{corollary:bellec_prob_correlation}
        Let \(1 \leq s \leq p\). Suppose \(y \sim N\left(f, \sigma^2 \left(I_p - \frac{1}{p}\mathbf{1}_p\mathbf{1}_p^\intercal\right)\right)\) and
        \begin{equation*}
            \hat{\beta} \in \argmin_{\beta \in \R^p} \left\{ \frac{1}{p} ||M\beta - y||^2 + 2\lambda ||\beta||_1\right\}
        \end{equation*}
        where \(M = \sqrt{p}\left(I_p - \frac{1}{p}\mathbf{1}_p\mathbf{1}_p^\intercal\right)\) and \(\lambda > 0\) is potentially random. If the \(\SRE(s, 7)\) condition holds, then 
        \begin{align*}
            &P\left(\left\{\frac{1}{p}||M\hat{\beta} - f||^2 > \min_{||\beta||_0 \leq s} \frac{1}{p}||M\beta - f||^2 + \frac{49\lambda^2 s}{16\kappa(s, 7)^2} \right\} \cap \left\{\lambda \geq  2(4+\sqrt{2})\sigma \sqrt{\frac{\log(2ep/s)}{p}} \right\} \right) \\
            &\leq \frac{1}{2}\left(\frac{s}{2ep}\right)^{\frac{s}{\kappa(s, 7)^2}}.
        \end{align*}
        Moreover, if \(f = M\beta^*\) for some \(\beta^* \in \R^p\) with \(||\beta^*||_0 \leq s\), then 
        \begin{equation*}
            P\left( \left\{  ||\hat{\beta} - \beta^*|| > \frac{49\lambda s^{1/2}}{8\kappa(s, 7)^{2}}\right\} \cap \left\{\lambda \geq  2(4+\sqrt{2})\sigma \sqrt{\frac{\log(2ep/s)}{p}} \right\} \right) \leq \frac{1}{2}\left(\frac{s}{2ep}\right)^{\frac{s}{\kappa(s, 7)^2}}
        \end{equation*}
    \end{corollary}
    \begin{proof}
        With the choice of \(\gamma = \frac{1}{2}\) in Theorem 4.2 of \cite{bellec_slope_2018}, note the conclusion of Theorem 4.2 is a deterministic result which holds whenever the event (4.1) holds and for any value of penalty of at least \(2(4+\sqrt{2})\sigma\sqrt{\log(2ep/s)/p}\). In the setting of correlated noise, it follows from Proposition \ref{prop:correlated_event} that the same conclusion continues to hold on the intersection of the event (4.1) and the event \(\{\lambda > 2(4+\sqrt{2})\sigma \sqrt{\log(2ep/s)/p}\}\). 
    \end{proof}
    
    \begin{proof}[Proof of Proposition \ref{prop:lasso_lepski}]
        Let \(\tilde{s}\) denote the smallest element in \(\mathcal{S}\) which is greater than or equal to \(s^*\). For ease of notation, set \(r(s) = \sqrt{13s\hat{\lambda}^2(s)}\). Define the event 
        \begin{align*}
            \mathcal{G} := &\left\{ \left|\left| \hat{v}(s) - (\theta - \bar{\theta}\mathbf{1}_p) \right|\right| \leq r(s)  \text{ for all } s \geq \tilde{s} \text{ such that } s \in \mathcal{S} \setminus \{p\} \right\} \\
            &\cap \left\{ ||\hat{v}(p) - (\theta - \bar{\theta}\mathbf{1}_p)|| \leq 2\delta^{-1/2}\sqrt{(1-\gamma)p}\right\}.
        \end{align*}
        Let us break the analysis into two cases. \newline 
    
        \noindent \textbf{Case 1:} Suppose \(s^* < \frac{p}{2}\). By Proposition \ref{prop:correlation_estimation}, the event \(\mathcal{E}_{\text{var}} = \left\{ L_\eta^{-1} \leq \frac{1-\hat{\gamma}}{1-\gamma} \leq L_\eta\right\}\) has \(P_{\theta,\gamma}\)-probability of at least \(1-\eta\) uniformly over \(||\theta||_0 < \frac{p}{2}\) and \(\gamma \in [0, 1)\). 
        Consider 
        \begin{align*}
            &P_{\theta, \gamma}\left\{ ||\hat{v} - (\theta - \bar{\theta}\mathbf{1}_p)||^2 > C_{\delta, \eta}(1-\gamma) s^*\log\left(\frac{ep}{s^*}\right) \right\} \\
            &\leq P_{\theta, \gamma}\left(\left\{ ||\hat{v} - (\theta - \bar{\theta}\mathbf{1}_p)||^2 > C_{\delta, \eta}(1-\gamma) s^*\log\left(\frac{ep}{s^*}\right) \right\} \cap \mathcal{G} \cap \mathcal{E}_{\text{var}}\right) + P_{\theta, \gamma}(\mathcal{G}^c \cap \mathcal{E}_{\text{var}}) + \eta.
        \end{align*}
        Let us examine the first term. If \(s^* < \frac{p}{2 \cdot 784}\), it follows by definition of \(\tilde{s}\) and \(\mathcal{S}\) that we have \(s^* \leq \tilde{s} \leq 2s^*\) and so on the event \(\mathcal{E}_{\text{var}}\) it follows 
        \begin{equation*}
            r(\tilde{s}) \leq r(2s^*) \leq \tilde{C}_{\delta, \eta} \sqrt{(1-\gamma) s^*\log\left(\frac{ep}{s^*}\right)}
        \end{equation*}
        where \(\tilde{C}_{\delta, \eta} > 0\) is a constant depending only on \(\delta\) and \(\eta\) whose value may change from instance to instance. On the other hand, if \(\frac{p}{2} > s^* \geq \frac{p}{2 \cdot 784}\), then \(r\left(2^{\lfloor \log_2 (p/784) \rfloor}\right) \vee r(\tilde{s}) \leq r(p) \leq 8(4+\sqrt{2})L_\eta\sqrt{(1-\gamma)p \log(2e)} \leq \tilde{C}_{\delta, \eta} \sqrt{(1-\gamma) s^* \log\left(\frac{ep}{s^*}\right)}\) on the event \(\mathcal{E}_{\text{var}}\). Note, on the event \(\mathcal{G}\), if \(\tilde{s} \in \mathcal{S} \setminus\{p\}\), then we have \(s' \leq \tilde{s}\) and so \(\hat{v} \in B\left(\hat{v}(\tilde{s}), r(\tilde{s})\right)\). On the other hand, if \(\tilde{s} = p\) then \(s^* \geq 2^{\lfloor \log_2 (p/784) \rfloor}\) by definition of \(\tilde{s}\). Consider then that either \(s'\) does not exist or \(s' \leq 2^{\lfloor \log_2 (p/784) \rfloor}\). If \(s'\) does not exist, then \(\hat{v} = \hat{v}(p)\) and so \(\hat{v} \in B(\hat{v}(\tilde{s}), r(\tilde{s}))\). If \(s'\) does exist, then \(\hat{v} \in B\left(\hat{v}\left(2^{\lfloor \log_2 (p/784) \rfloor}\right), r\left(2^{\lfloor \log_2 (p/784) \rfloor}\right)\right)\). Therefore, on \(\mathcal{G} \cap \mathcal{E}_{\text{var}}\) it is the case that
        \begin{align*}
            ||\hat{v} - (\theta - \bar{\theta}\mathbf{1}_p)|| \leq \tilde{C}_{\delta, \eta} \sqrt{(1-\gamma) s^*\log\left(\frac{2ep}{s^*}\right)}.
        \end{align*}
        Therefore, we have 
        \begin{equation*}
            P_{\theta, \gamma}\left(\left\{ ||\hat{v} - (\theta - \bar{\theta}\mathbf{1}_p)||^2 > C_{\delta, \eta}(1-\gamma) s^*\log\left(\frac{ep}{s^*}\right) \right\} \cap \mathcal{G} \cap \mathcal{E}_{\text{var}}\right) = 0
        \end{equation*}
        by taking \(C_{\delta, \eta} > 0\) sufficiently large. 
    
        Let us now turn to bounding \(P_{\theta, \gamma}(\mathcal{G}^c \cap \mathcal{E}_{\text{var}})\). Union bound yields 
        \begin{align*}
            P_{\theta, \gamma}(\mathcal{G}^c \cap \mathcal{E}_{\text{var}}) &\leq P_{\theta, \gamma} \left\{||\hat{v}(p) - (\theta - \bar{\theta}\mathbf{1}_p)||^2 > 4\delta^{-1}(1-\gamma)p\right\} \\
            &+ \sum_{\substack{s \in \mathcal{S} \setminus \{p\} \\ s \geq \tilde{s}}} P_{\theta, \gamma}\left(\left\{ ||\hat{v}(s) - (\theta - \bar{\theta}\mathbf{1}_p)||^2 > r(s)^2\right\} \cap \mathcal{E}_{\text{var}}\right). 
        \end{align*}
        To bound each term in the sum, we will apply Corollary \ref{corollary:bellec_prob_correlation} since \(||\theta||_0 \leq s^*\) implies \(||\theta||_0 \leq s\) for \(s \geq \tilde{s}\). Note on \(\mathcal{E}_{\text{var}}\) we have 
        \begin{align*}
            \hat{\lambda}(s) \geq 2(4+\sqrt{2}) \sqrt{(1-\gamma) \log\left(\frac{2ep}{s}\right)} \geq 2(4+\sqrt{2})\sqrt{\frac{\sigma^2 \log\left(2ep/s\right)}{p}}.
        \end{align*}
        Since \(s \in \mathcal{S} \setminus \{p\}\) implies \(s \leq \frac{p}{784}\), it is clear \(M\) satisfies the \(\SRE(s, 7)\) condition by Lemma \ref{lemma:M_SRE} with \(\kappa(s, 7) \geq \frac{1}{2}\). It is also clear from the definition of \(\kappa(s, 7)\) and \(M\) that \(\kappa(s, 7)^2 \leq 1\). Hence, by Corollary \ref{corollary:bellec_prob_correlation} we have 
        \begin{align*}
            P_{\theta, \gamma}\left(\left\{||\hat{v}(s) - (\theta - \bar{\theta}\mathbf{1}_p)||^2 > r(s)^2\right\} \cap \mathcal{E}_{\text{var}}\right) &\leq P_{\theta, \gamma}\left(\left\{ ||\hat{v}(s) - (\theta - \bar{\theta}\mathbf{1}_p)||^2 > \frac{49\hat{\lambda}(s)^2 s}{16\kappa(s, 7)^2} \right\} \cap \mathcal{E}_{\text{var}}\right) \\
            &\leq \frac{1}{2}\left(\frac{s}{2ep}\right)^{\frac{s}{\kappa(s, 7)^2}} \\
            &\leq \left(\frac{s}{p}\right)^s. 
        \end{align*}
        Therefore, 
        \begin{align*}
            \sum_{\substack{s \in \mathcal{S} \setminus \{p\} \\ s \geq \tilde{s}}} P_{\theta, \gamma}\left(\left\{ ||\hat{v}(s) - (\theta - \bar{\theta}\mathbf{1}_p)||^2 > r(s)^2\right\} \cap \mathcal{E}_{\text{var}}\right) &\leq \sum_{\substack{s \in \mathcal{S} \setminus \{p\} \\ s \geq \tilde{s}}} \left(\frac{s}{p}\right)^s \\
            &= \sum_{\substack{s \in \mathcal{S} \setminus \{p\} \\ s \geq \tilde{s}}} \exp\left(-s \log\left(\frac{p}{s}\right)\right) \\
            &\leq C \exp\left(-s^*\log\left(\frac{p}{s^*}\right)\right) \\
            &\leq C \exp\left(-\log p\right) \\
            &\leq \frac{\delta}{2}.
        \end{align*}
        where \(C > 0\) is a universal constant. Here, we have used \(p\) is sufficiently large. Likewise, consider by Markov's inequality 
        \begin{align*}
            P_{\theta, \gamma} \left\{||\hat{v}(p) - (\theta - \bar{\theta}\mathbf{1}_p)||^2 > 4\delta^{-1}(1-\gamma)p\right\} \leq \frac{2(1-\gamma)(p-1)}{4\delta^{-1}(1-\gamma)p} \leq \frac{\delta}{2}.
        \end{align*}
        To summarize, we have shown \(P_{\theta, \gamma}(\mathcal{G}^c \cap \mathcal{E}_{\text{var}}) \leq \delta\). Hence, we have shown
        \begin{equation*}
            \sup_{\substack{||\theta||_0 \leq s^* \\ \gamma \in [0, 1)}} P_{\theta, \gamma}\left\{ ||\hat{v} - (\theta - \bar{\theta}\mathbf{1}_p)||^2 > C_{\delta, \eta} (1-\gamma) s^*\log\left(\frac{ep}{s^*}\right) \right\} \leq \delta + \eta. 
        \end{equation*}
        The analysis for this case is complete. \newline 
    
        \noindent \textbf{Case 2:} Suppose \(s^* \geq \frac{p}{2}\). Consider 
        \begin{align*}
            &P_{\theta, \gamma}\left\{||\hat{v} - (\theta - \bar{\theta}\mathbf{1}_p)||^2 > C_{\delta, \eta}(1-\gamma) s^*\log\left(\frac{ep}{s^*}\right)\right\}  \\
            &\leq P_{\theta, \gamma}\left(\left\{||\hat{v} - (\theta - \bar{\theta}\mathbf{1}_p)||^2 > C_{\delta, \eta}(1-\gamma) s^*\log \left(\frac{ep}{s^*}\right)\right\} \cap \mathcal{G}\right) + P_{\theta, \gamma}\left(\mathcal{G}^c\right).
        \end{align*}
        Since \(\tilde{s} \geq s^* \geq \frac{p}{2}\), it follows by definition of \(\mathcal{S}\) that \(\tilde{s} = p\). Hence, from our calculation in Case 1 we have \(P_{\theta, \gamma}(\mathcal{G}^c) = P_{\theta, \gamma} \left\{||\hat{v}(p) - (\theta - \bar{\theta}\mathbf{1}_p)||^2 > 4\delta^{-1}(1-\gamma)p\right\} \leq \delta\). Further, note since \(s^* \geq \frac{p}{2}\) we have \(  4\delta^{-1}(1-\gamma)p \leq \tilde{C}_{\delta, \eta}\sqrt{(1-\gamma) s^* \log\left(\frac{ep}{s^*}\right)}\) where \(\tilde{C}_{\delta, \eta}\) is a constant depending only on \(\delta\) and \(\eta\). Taking \(C_{\delta, \eta}\) sufficiently large concludes the proof following the same reasoning employed in Case 1.
    \end{proof}
    \subsubsection{Linear functional: kernel mode estimator}
    
    \begin{proof}[Proof of Proposition \ref{prop:linear_lepski}]
        Recall from Section \ref{section:kme_adaptation} we can write \(Y = (\bar{X} - X_i) + \frac{\sqrt{1-\gamma}}{\sqrt{p}} \xi\) where \(\xi \sim N(0, 1)\). Let \(\tilde{s}\) denote the smallest element in \(\mathcal{S}\) which is greater than or equal to \(s^*\). Define the event
        \begin{equation*}
            \mathcal{G} = \left\{|\hat{T}(s) - \bar{\theta}| \leq r(s) \text{ for all } s \geq \tilde{s} \text{ such that } s \in \mathcal{S}\right\}
        \end{equation*}
        We break up the analysis into two cases. \newline
    
        \noindent \textbf{Case 1:} Suppose \(s^* < \left\lfloor\frac{p}{2}\right\rfloor\). By Proposition \ref{prop:correlation_estimation}, there exists \(L_\eta \geq 1\) such that the event \(\mathcal{E}_{\text{var}} = \left\{L_\eta^{-1} \leq \frac{1-\hat{\gamma}}{1-\gamma} \leq L_\eta\right\} \cap \left\{|\xi| \leq L_\eta\right\}\) has \(P_{\theta, \gamma}\)-probability of at least \(1-\eta\) uniformly over \(||\theta||_0 < \frac{p}{2}\) and \(\gamma \in [0, 1)\). Let us now turn to bounding \(P_{\theta, \gamma}(\mathcal{G}^c \cap \mathcal{E}_{\text{var}})\), which will be useful later. Define the events 
        \begin{align*}
            \mathcal{B}_0 &:= \left\{|\hat{T}(s) - \bar{\theta}| \leq r(s) \text{ for all } \frac{p}{784} \geq s \geq \tilde{s} \text{ such that } s \in \mathcal{S} \right\}, \\
            \mathcal{B}_1 &:= \left\{|\hat{T}(s) - \bar{\theta}| \leq r(s) \text{ for all } s \geq \tilde{s} \vee \frac{p}{784} \text{ such that } s \in \mathcal{S} \text{ and } 1-\hat{\gamma} > C_\eta \log^{-1}\left(\frac{ep}{(p-2s)^2}\right)\right\}, \\
            \mathcal{B}_2 &:= \left\{|\hat{T}(s) - \bar{\theta}| \leq r(s) \text{ for all } s \geq \tilde{s} \vee \frac{p}{784} \text{ such that } s \in \mathcal{S} \text{ and } 1-\hat{\gamma} \leq C_\eta \log^{-1}\left(\frac{ep}{(p-2s)^2}\right)\right\}.
        \end{align*}
        Observe by definition of \(\hat{T}(s)\) and \(r(s)\) we have
        \begin{align*}
            P_{\theta, \gamma}(\mathcal{G}^c \cap \mathcal{E}_{\text{var}}) &\leq P_{\theta, \gamma}\left(\mathcal{B}_0^c \cap \mathcal{E}_{\text{var}}\right) + P_{\theta, \gamma}\left( \mathcal{B}_1^c\right) + P_{\theta, \gamma}\left(\mathcal{B}_2^c \cap \mathcal{E}_{\text{var}}\right) \\
            &\leq \sum_{\substack{s \in \mathcal{S} \\ \tilde{s} \leq s \leq \frac{p}{784}}}P_{\theta, \gamma} \left( \left\{ \left|\frac{1}{p}\sum_{i=1}^{p}\hat{\beta}_i(s) - \bar{\theta}\right| > r(s)\right\} \cap \mathcal{E}_{\text{var}}\right)\\
            &+ P_{\theta, \gamma}\left\{ |\hat{T}(p) - \bar{\theta}| > r(p) \right\} + \sum_{\substack{s \in \mathcal{S} \setminus \{p\} \\ s \geq \tilde{s}\vee \frac{p}{784}}} P_{\theta, \gamma}\left(\left\{ \left|\left(\argmax_{t \in \R} \hat{G}_{\hat{h}(s)}(t)\right) - \bar{\theta}\right| > r(s)\right\} \cap \mathcal{E}_{\text{var}}\right). 
        \end{align*}
        Since \(\tilde{s} \geq s^*\) and \(||\theta||_0 \leq s^*\), it follows \(||\theta||_0 \leq s\) for \(s \geq \tilde{s}\). The argument in the proof of Proposition \ref{prop:lasso_lepski} (but now using the second conclusion of Corollary \ref{corollary:bellec_prob_correlation}) along with \(||(p^{-1}\sum_{i=1}^{p} \hat{\beta}_i(s)) \mathbf{1}_p - \bar{\theta}\mathbf{1}_p||^2 \leq ||\hat{\beta}(s) - \theta||^2\) from the Pythagorean identity can be used to show 
        \begin{equation*}
            \sum_{\substack{s \in \mathcal{S} \\ \tilde{s} \leq s \leq \frac{p}{784}}}P_{\theta, \gamma} \left( \left\{ \left|\frac{1}{p}\sum_{i=1}^{p}\hat{\beta}_i(s) - \bar{\theta}\right| > r(s)\right\} \cap \mathcal{E}_{\text{var}}\right) \leq p^{-C}. 
        \end{equation*}
        for some universal constant \(C > 0\). Note since \(p\) can be taken sufficiently large depending only on \(K\), we have \(p^{-C} \leq \frac{1}{e^{K/4}-1}\). 
        
        Turning to the third term, consider by (\ref{def:kme_bypass_decorrelation}), the definition of \(\delta_s\), and the proof of Theorem \ref{thm:kme_unknown_var}, we have for \(s \in \mathcal{S} \setminus \{p\}\), 
        \begin{align*}
            \sup_{\substack{||\theta||_0 \leq s \\ \gamma \in [0, 1)}} P_{\theta, \gamma}\left(\left\{ \left|\left(\argmax_{t \in \R} \hat{G}_{\hat{h}(s)}(t)\right) - \bar{\theta}\right| > \tilde{C}_{K, \eta}\sqrt{(1-\gamma)\left(1 \vee \log\left(\frac{ep}{(p-2s)^2}\right)\right)}\right\} \cap \mathcal{E}_{\text{var}}\right) \leq \delta_s
        \end{align*}
        where \(\tilde{C}_{K, \eta} > 0\) is sufficiently large depending only on \(K\) and \(\eta\). Here, we have used \(p\) is sufficiently large depending only on \(K\). On the event \(\mathcal{E}_{\text{var}}\), by taking \(R_{K, \eta}\) sufficiently large depending only on \(K\) and \(\eta\) we have 
        \begin{equation*}
            r(s) = R_{K, \eta}\hat{h}(s) \geq\tilde{C}_{K, \eta} \sqrt{(1-\gamma)\left(1 \vee \log\left(\frac{ep}{(p-2s)^2}\right)\right)}.
        \end{equation*}
        Therefore, by definition of \(\mathcal{S}\) it follows 
        \begin{align*}
            &\sum_{\substack{s \in \mathcal{S} \setminus \{p\} \\ s \geq \tilde{s}\vee \frac{p}{784}}} P_{\theta, \gamma}\left(\left\{ \left|\left(\argmax_{t \in \R} \hat{G}_{\hat{h}(s)}(t)\right) - \bar{\theta}\right| > r(s)\right\} \cap \mathcal{E}_{\text{var}}\right) \\
            &\leq e^{-K} + \sum_{\substack{s \in \mathcal{S} \setminus \{\lfloor p/2\rfloor - 1, p\} \\ s \geq \tilde{s}\vee \frac{p}{784}}} \exp\left(-K\left(\frac{p}{(p-2s)^2} \wedge p^{1/32}\right)\right) \\
            &\leq e^{-K} + p\exp\left(-Kp^{1/32}\right) + \sum_{k=0}^{\lfloor \log_2 \sqrt{p} \rfloor} \exp\left(-\frac{Kp}{2^2(k+1)}\right) \\
            &\leq e^{-K} + p\exp\left(-Kp^{1/32}\right) + \sum_{k=0}^{\lfloor\log_2\sqrt{p}\rfloor} \exp\left(-\frac{K}{4} 2^{2\log_2\sqrt{p} - 2k}\right) \\
            &\leq e^{-K} + p\exp\left(-Kp^{1/32}\right) + \sum_{k=0}^{\infty} \exp\left(-\frac{K}{4} 2^k\right) \\
            &\leq \frac{3}{e^{K/4} - 1}. 
        \end{align*}
        Likewise, consider \(P_{\theta, \gamma}\left\{|\hat{T}(p) - \bar{\theta}| > r(p)\right\} \leq \frac{1}{e^{K/4} - 1}\) by taking \(R_{K, \eta}\) large enough since \(\bar{X} \sim N(\bar{\theta}, \gamma + \frac{1-\gamma}{p})\). Hence, we have shown 
        \begin{equation}\label{eqn:lepski_bad_event_bound}
            P_{\theta, \gamma}\left(\mathcal{G}^c \cap \mathcal{E}_{\text{var}}\right) \leq \frac{5}{e^{K/4}-1}. 
        \end{equation}
        With this probability bound in hand, let us now split the analysis into two subcases. \newline
    
        \noindent \textbf{Case 1.1:} Suppose \(s^* \leq \frac{p}{2 \cdot 784}\). By union bound, consider 
        \begin{align*}
            &P_{\theta, \gamma}\left\{|\hat{T} - \bar{\theta}| > C_{K, \eta}\sqrt{(1-\gamma) \frac{s^*}{p}\log\left(\frac{ep}{s^*}\right)}\right\} \\
            &\leq P_{\theta, \gamma}\left(\left\{|\hat{T} - \bar{\theta}| > C_{K, \eta}\sqrt{(1-\gamma) \frac{s^*}{p}\log\left(\frac{ep}{s^*}\right)}\right\} \cap \mathcal{E}_{\text{var}} \cap \mathcal{G}\right) + P_{\theta, \gamma}(\mathcal{G}^c \cap \mathcal{E}_{\text{var}}) + \eta. 
        \end{align*}
        Since \(\hat{T}(s) = \frac{1}{p} \sum_{i=1}^{p} \hat{\beta}_i(s)\) for \(s \leq \frac{p}{2 \cdot 784}\), the analysis of Proposition \ref{prop:lasso_lepski} can be repeated to obtain 
        \begin{equation*}
            P_{\theta, \gamma}\left(\left\{|\hat{T} - \bar{\theta}| > C_{K, \eta}\sqrt{(1-\gamma) \frac{s^*}{p}\log\left(\frac{ep}{s^*}\right)}\right\} \cap \mathcal{E}_{\text{var}} \cap \mathcal{G}\right) = 0
        \end{equation*}
        for \(C_{K, \eta}\) sufficiently large. Then plugging in the bound (\ref{eqn:lepski_bad_event_bound}) concludes the analysis of this case. \newline
    
        \noindent \textbf{Case 1.2:} Suppose \(\frac{p}{2 \cdot 784} < s^* < \frac{p}{2}\). Without loss of generality, we can assume \(\lfloor \frac{p}{2} \rfloor - 2^{\lfloor \log_2 \sqrt{p}\rfloor} \leq s^*\). To see this, consider if \(s^* < \lfloor\frac{p}{2} \rfloor - 2^{\lfloor \log_2 \sqrt{p}\rfloor}\), then certainly \(||\theta||_0 \leq s^*\) certainly implies \(||\theta||_0 \leq \lfloor\frac{p}{2} \rfloor - 2^{\lfloor \log_2 \sqrt{p}\rfloor}\). Since \(s^* \geq \frac{p}{2 \cdot 784}\), we also have 
        \begin{equation*}
            (1-\gamma)\frac{s^*}{p} \log\left(\frac{es^*}{p}\right) \asymp (1-\gamma).   
        \end{equation*}
        Hence, the desired rate for \(s^* < \lfloor\frac{p}{2} \rfloor - 2^{\lfloor \log_2 \sqrt{p}\rfloor}\) is the same as if \(s^* = \lfloor\frac{p}{2} \rfloor - 2^{\lfloor \log_2 \sqrt{p}\rfloor}\). Therefore, we can take \(s^* \geq \lfloor \frac{p}{2} \rfloor - 2^{\lfloor \log_2 \sqrt{p}\rfloor}\) without loss of generality.
        
        By union bound we have 
        \begin{align*}
            &P_{\theta, \gamma}\left\{ |\hat{T} -  \bar{\theta}| > C_{K, \eta} \sqrt{1 \wedge (1-\gamma)\left(1 \vee \log\left(\frac{ep}{(p-2s^*)^2}\right)\right)} \right\} \\
            &\leq P_{\theta, \gamma}\left( \left\{ |\hat{T} -  \bar{\theta}| > C_{K, \eta} \sqrt{1 \wedge (1-\gamma)\left(1 \vee \log\left(\frac{ep}{(p-2s^*)^2}\right)\right)} \right\} \cap \mathcal{E}_{\text{var}} \cap \mathcal{G}\right) + P_{\theta, \gamma}(\mathcal{G}^c \cap \mathcal{E}_{\text{var}}) + \eta. 
        \end{align*}
        Let us examine the first term. By definition of \(\mathcal{S}\), if \(s^* \leq \lfloor \frac{p}{2} \rfloor - 2^{\lceil \log_2 p^{1/4}\rceil}\), we have \(p-2\tilde{s} \leq p-2s^* \leq2(p-2\tilde{s})\). On the other hand, if \(s^* > \lfloor \frac{p}{2} \rfloor - 2^{\lceil \log_2 p^{1/4}\rceil}\), then \(\tilde{s} = \lfloor \frac{p}{2} \rfloor - 1\) since we know \(s^* < \lfloor \frac{p}{2} \rfloor\) in this case. Thus in both cases, on the event \(\mathcal{E}_{\text{var}}\), it can be shown taking \(C_\eta\) sufficiently large that 
        \begin{align*}
            r(\tilde{s}) \leq \tilde{C}_{K, \eta} \sqrt{1 \wedge (1-\gamma)\left(1 \vee \log\left(\frac{ep}{(p-2s^*)^2} \right)\right)}
        \end{align*}
        where \(\tilde{C}_{K, \eta} > 0\) is sufficiently large depending only on \(K\) and \(\eta\), and whose value may change from instance to instance. On the event \(\mathcal{G}\), we have \(s' \leq \tilde{s}\) and so \(\hat{T} \in B\left(\hat{T}(\tilde{s}), r(\tilde{s})\right)\). Therefore, on the event \(\mathcal{G} \cap \mathcal{E}_{\text{var}}\) we have 
        \begin{align*}
            |\hat{T} - \bar{\theta}| \leq |\hat{T} - \hat{T}(\tilde{s})| + |\hat{T}(\tilde{s}) - \bar{\theta}|
            \leq 2r(\tilde{s})
            \leq \tilde{C}_{K, \eta} \sqrt{1 \wedge (1-\gamma)\left(1 \vee \log\left(\frac{ep}{(p-2s^*)^2} \right)\right)}.
        \end{align*}
        Therefore, 
        \begin{equation*}
            P_{\theta, \gamma}\left( \left\{ |\hat{T} -  \bar{\theta}| > C_{K, \eta} \sqrt{1 \wedge (1-\gamma)\left(1 \vee \log\left(\frac{ep}{(p-2s^*)^2}\right)\right)} \right\} \cap \mathcal{E}_{\text{var}} \cap \mathcal{G}\right) = 0
        \end{equation*}
        by taking \(C_{K, \eta}\) sufficiently large. Plugging in the bound (\ref{eqn:lepski_bad_event_bound}), the analysis for this case is complete. \newline 
    
        \noindent \textbf{Case 2:} Suppose \(s^* \geq \lfloor \frac{p}{2}\rfloor\). By definition of \(\mathcal{S}\), we have \(\tilde{s} = p\). An argument modeled on the Case 2 analysis of the proof of Proposition \ref{prop:lasso_lepski} can be employed to obtain the desired result. We omit the details.
    \end{proof}
    
    \section{Kernel mode estimator: unknown correlation}\label{appendix:kme_unknown_var}
    The arguments to show Theorem \ref{thm:kme_unknown_var} are a straightforward extension of what was done in Section \ref{appendix:kme_known_var}. The main difference in Theorem \ref{thm:kme_unknown_var} is that the bandwidth \(\hat{h}\) is now random as it depends on the data through \(\hat{\gamma}\). 
    
    Recall the kernel mode estimator \(\tilde{\mu}\) is an ``oracle" estimator applied to the data from (\ref{model:Gaussian_contamination}) and is given by (\ref{def:mhat}). Specifically, \(\tilde{\mu} = \argmax_{t \in \R} \hat{G}(t, h)\) where 
    \begin{equation*}
        \hat{G}(t, h) = \frac{1}{2ph} \sum_{i=1}^{p} \mathbbm{1}_{\{|t - Y_i| \leq h\}}.
    \end{equation*}
    Here, we use the notation \(\hat{G}(t, h)\) rather than \(\hat{G}_h(t)\) as in Section \ref{appendix:kme_known_var} to further emphasize the dependence on \(h\). For a deterministic \(h\), let \(G(t, h) = E_{\theta, \gamma}(\hat{G}(t ,h))\) denote the expectation. Note we will, without loss of generality, take \(s \geq \frac{p}{4}\) and \(|\mathcal{O}| \geq \frac{p}{4}\) as in Section \ref{appendix:kme_known_var}.
    
    Let \(\underline{h} \leq \bar{h}\) denote two deterministic points. Eventually, they will be chosen so that the plug-in bandwidth \(\hat{h}\) satisfies \(\hat{h} \in [\underline{h}, \bar{h}]\) with high probability. The strategy for proving Theorem \ref{thm:kme_unknown_var} is the same as that in Section \ref{appendix:kme_known_var}. To obtain the desired result, it suffices to show that with probability at least \(1-\delta\), there exists \(x = x(\hat{h}) \in \R\) with \(|x - \mu| \leq \hat{h}/4\) such that for all \(t \in \R\) with \(|t - \mu| \geq C_a \hat{h}\), we have \(\hat{G}(x, \hat{h}) > \hat{G}(t, \hat{h})\). Mimicking Section \ref{appendix:kme_known_var}, it suffices to show 
    \begin{align*}
        G(x, h) - G(t, h) - \left|\hat{G}(x, h) - G(x, h)\right| - \left|\hat{G}(t, h) - G(t, h)\right| > 0
    \end{align*}
    uniformly over \(\left\{(x, t, h) \in \R^3 : \underline{h} \leq h \leq \bar{h}, x = x(h), |t - \mu| \geq C_a h\right\}\) with probability at least \(1-\delta\). Notice we aim for uniform control over \(\underline{h} \leq h \leq \bar{h}\) in order to deal with the random bandwidth. 
    
    The result about the signal (Proposition \ref{prop:mean_diff_fix}) is deterministic and so requires no modification. However, the empirical process terms now require control over an extra parameter, namely the bandwidth. This is the only salient difference between the situation here and the situation in Section \ref{appendix:kme_known_var}. This difference does not pose any substantial conceptual difficulty. As it is a straightforward extension, in the interest of brevity we will only provide proof sketches for most of the following assertions. 
    
    \subsection{\texorpdfstring{Regime \(\frac{p}{2} - p^{1/4} \leq s < \frac{p}{2}\)}{Regime p/2 - p\textasciicircum (1/4) <= s < p/2}}
    In the regime \(\frac{p}{2} - p^{1/4} \leq s < \frac{p}{2}\), the analysis is standard as in Section \ref{appendix:kme_known_var}.
    \begin{proposition}\label{prop:kme_easy_unknown_var}
        Suppose \(\frac{p}{2} - p^{1/4} \leq s < \frac{p}{2}\) and \(C_a\) is a sufficiently large universal constant. Fix \(\delta, \eta \in (0, 1)\), and let \(\hat{\gamma}\) be the estimator from Proposition \ref{prop:correlation_estimation} at confidence level \(\eta\). Define the (random) bandwidth 
        \begin{equation*}
            \hat{h} := C_\eta\sqrt{(1-\hat{\gamma})\left(\log\left(\frac{ep}{(p-2s)^2}\right) + \log\log\left(\frac{1}{\delta}\right)\right)}    
        \end{equation*}
        where \(C_\eta\) is sufficiently large depending only on \(\eta\). If \(p \geq (2/\delta)^2\), then the kernel mode estimator with (random) bandwidth \(\hat{h}\), 
        \begin{equation*}
            \tilde{\mu} := \argmax_{x \in \R} \frac{1}{2p\hat{h}} \sum_{i=1}^{n} \mathbbm{1}_{\{|x - Y_i| \leq \hat{h}\}}
        \end{equation*}
        satisfies 
        \begin{equation*}
            \sup_{\substack{||\theta||_0 \leq s, \\ \gamma \in [0, 1)}} P_{\theta, \gamma}\left\{ \frac{|\tilde{\mu} - \mu|}{\sqrt{1-\gamma}} > C_\eta'\sqrt{\log\left(\frac{ep}{(p-2s)^2}\right) + \log\log\left(\frac{1}{\delta}\right)} \right\} \leq \delta + \eta
        \end{equation*}
        where \(C_\eta' > 0\) is a constant depending only on \(\eta\). 
    \end{proposition}
    \begin{proof}
        First, consider by Proposition \ref{prop:correlation_estimation} there exists \(R \geq 1\) depending only on \(\eta\) such that the event \(\mathcal{E}_\gamma := \left\{R^{-1} (1-\gamma) \leq 1-\hat{\gamma} \leq R(1-\gamma)\right\}\) has probability at least \(1-\eta\). Define 
        \begin{align*}
            \underline{h} &:= C_\eta R^{-1}\sqrt{(1-\gamma)\left(\log\left(\frac{ep}{(p-2s)^2}\right) + \log\log\left(\frac{1}{\delta}\right)\right)}, \\
            \bar{h} &:= C_\eta R \sqrt{(1-\gamma)\left(\log\left(\frac{ep}{(p-2s)^2}\right) + \log\log\left(\frac{1}{\delta}\right)\right)}.
        \end{align*}
        Next, consider that for any \(\underline{h} \leq h \leq \bar{h}\), we have 
        \begin{equation*}
            \hat{G}(\mu, h) \geq \frac{1}{2ph} \sum_{i \in \mathcal{I}} \mathbbm{1}_{\{|\mu - Y_i| \leq h\}} = \frac{p-|\mathcal{O}|}{2ph} - \frac{1}{2ph} \sum_{i \in \mathcal{I}} \mathbbm{1}_{\{|\mu - Y_i| > h\}}. 
        \end{equation*}
        For \(t \in \R\) with \(|t - \mu| \geq C_a h\) with \(C_a > 2\) sufficiently large universal constant, consider that 
        \begin{align*}
            \hat{G}(t, h) &\leq \frac{|\mathcal{O}|}{2ph} + \frac{1}{2ph} \sum_{i \in \mathcal{I}} \mathbbm{1}_{\{|t - Y_i| \leq h\}} \\
            &\leq \frac{|\mathcal{O}|}{2ph} + \frac{1}{2ph} \sum_{i \in \mathcal{I}} \mathbbm{1}_{\{|t - \mu + \mu - Y_i| \leq h\}} \\
            &\leq \frac{|\mathcal{O}|}{2ph} + \frac{1}{2ph} \sum_{i \in \mathcal{I}} \mathbbm{1}_{\{|\mu - Y_i| \geq |t - \mu| - h\}} \\
            &\leq \frac{|\mathcal{O}|}{2ph} + \frac{1}{2ph} \sum_{i \in \mathcal{I}} \mathbbm{1}_{\{|\mu - Y_i| > h\}}.
        \end{align*}
        Note this holds for all \(t\) such that \(|t-\mu| \geq C_a h\), and so we have for all \(\underline{h} \leq h \leq \bar{h}\), 
        \begin{equation*}
            \hat{G}(\mu, h) - \sup_{|t - \mu| \geq C_a h} \hat{G}(t, h) \geq \frac{p-2s}{2ph} - \frac{1}{ph} \sum_{i \in \mathcal{I}} \mathbbm{1}_{\{|\mu - Y_i| > h\}} \geq \frac{p-2s}{2ph} - \frac{1}{ph} \sum_{i \in \mathcal{I}} \mathbbm{1}_{\{|\mu - Y_i| > \underline{h}\}}. 
        \end{equation*}
        Note \(Y_i - \mu \sim N(0, 1-\gamma)\) for \(i \in \mathcal{I}\). Further note \(\Var(\sum_{i \in \mathcal{I}} \mathbbm{1}_{\{|\mu - Y_i| > \underline{h}\}}) \leq |\mathcal{I}| e^{-C\underline{h}^2/(1-\gamma)}\) for some universal constant \(C > 0\) whose value may change from instance to instance. Hence, by Markov's inequality we have for \(u \geq 0\), 
        \begin{equation*}
            \hat{G}(\mu, h) - \sup_{|t - \mu| \geq C_a h} \hat{G}(t, h) \geq \frac{p-2s}{2ph} - \frac{|\mathcal{I}|}{ph} e^{-C\underline{h}^2/(1-\gamma)} - \frac{u}{h}
        \end{equation*}
        uniformly over \(\underline{h} \leq h \leq \bar{h}\) with probability at least \(1 - \frac{e^{-C\underline{h}^2/(1-\gamma)}}{u\sqrt{p}}\). Therefore, 
        \begin{equation*}
            \hat{G}(\mu, h) - \sup_{|t-\mu| \geq C_a h} \hat{G}(t, h) \geq \frac{p-2s}{2ph} - \frac{2}{h} e^{-C\underline{h}^2/(1-\gamma)}
        \end{equation*}
        uniformly over \(\underline{h} \leq h \leq \bar{h}\) with probability at least \(1 - \frac{2}{\sqrt{p}}\). Call this event \(\mathcal{E}_{\mu}\). Since \(p \geq (2/\delta)^2\), it follows \(\mathcal{E}_{\mu}\) has probability at least \(1-\delta\). Select \(C_\eta\) sufficiently large depending only on \(R\). Following an argument similar to that of Proposition \ref{prop:kme_easy}, we have on the event \(\mathcal{E}_{\gamma} \cap \mathcal{E}_{\mu}\),  
        \begin{equation*}
            \hat{G}(\mu, h) - \sup_{|t - \mu| \geq C_a h} \hat{G}(t, h) \geq \frac{p-2s}{4ph} > 0
        \end{equation*}
        uniformly over \(\underline{h} \leq h \leq \bar{h}\). By union bound, the event \(\mathcal{E}_\mu \cap \mathcal{E}_{\gamma}\) has probability at least \(1-\delta-\eta\). Since \(\hat{h} \in [\underline{h}, \bar{h}]\) on \(\mathcal{E}_\gamma\), it thus follows that on \(\mathcal{E}\) we have \(\hat{G}(\mu, \hat{h}) > \hat{G}(t, \hat{h})\) for all \(|t - \mu| \geq C_a \hat{h}\). Therefore, it follows that \(|\tilde{\mu} - \mu| \leq C_a \hat{h}\) with probability at least \(1-\delta-\eta\). The proof is complete. 
    \end{proof}
    
    \subsection{\texorpdfstring{Regime \(s \leq \frac{p}{2} - p^{1/4}\)}{Regime s <= p/2-p\textasciicircum (1/4)}}
    As noted in Section \ref{appendix:kme_unknown_var}, Proposition \ref{prop:mean_diff_fix} is a deterministic result and thus gives a lower bound for the signal for all bandwidths. It remains to bound the stochastic error. Mimicking Section \ref{appendix:kme_known_var}, define the sets 
    \begin{align*}
        \mathcal{U}(h) &:= \left\{ t \in \R : |t - \mu| \geq C_a h \text{ and } \frac{1}{p} \sum_{i \in \mathcal{O}} P_{\theta, \gamma}\left\{|t - Y_i| > h\right\} > 4 e^{-Ch^2/(1-\gamma)}\right\}, \\
        \mathcal{V}(h) &:= \left\{t \in \R : |t - \mu| \geq C_a h \text{ and } t \in \mathcal{U}(h)^c\right\}. 
    \end{align*}
    Here, \(C > 0\) is a universal constant such that for all \(\underline{h} \leq h \leq \bar{h}\) and \(x, t \in \R\) with \(|t - \mu| \geq C_a h\) and \(|x - \mu| \leq \frac{h}{4}\), we have \(P_{\theta, \gamma}\left\{|t - Y_i| \leq h\right\} \vee P_{\theta, \gamma}\left\{|x - Y_i| > h\right\} \leq e^{-Ch^2/(1-\gamma)}\) for \(i \in \mathcal{I}\). Such a \(C\) exists since \(\underline{h}/\sqrt{1-\gamma}\) will be taken to be larger than a sufficiently large universal constant. In the following sections, \(C\) will refer to this constant. \newline
    
    \noindent \textbf{Uniform stochastic error control over \(\mathcal{U}\)}
    \begin{proposition}\label{prop:U_control_unknown_var}
        Suppose \(s \leq \frac{p}{2} - p^{1/4}\). Suppose \(C_a\) and \(\tilde{C}\) are sufficiently large universal constants. For \(\underline{h} \leq h \leq \bar{h}\), let \(x = x(h)\) with \(|x - \mu| \leq \frac{h}{4}\) denote the point from Proposition \ref{prop:mean_diff_fix}. Suppose \(\underline{h}/\sqrt{1-\gamma}\) is larger than a sufficiently large universal constant. If \(\delta \in (0, 1)\) and \(p \geq \tilde{C}\log^{16}\left(1/\delta\right)\), then with probability at least \(1-\delta\) we have 
        \begin{equation*}
            \left|\hat{G}(t, h) - G(t, h)\right| \leq \frac{1}{2}\left(G(x, h) - G(t, h)\right)
        \end{equation*}
        uniformly over \(\left\{(t, h) \in \R^2 : \underline{h} \leq h \leq \bar{h} \text{ and } t \in \mathcal{U}(h)\right\}\). 
    \end{proposition}
    \begin{proof}
        The proof is exactly the same as the proof of Proposition \ref{prop:U_control}, except invoking Corollary \ref{corollary:peel_tail} instead of Theorem \ref{thm:peel}. Note the corollary can be applied since the VC dimensions of the classes of sets \(\left\{[t - h, t+h] : \underline{h} \leq h \leq \bar{h}, t \in \R\right\}\) and \(\left\{[t - h, t+h]^c : \underline{h} \leq h \leq \bar{h}, t \in \R\right\}\) are bounded by a universal constant. 
    \end{proof}
    
    \noindent \textbf{Uniform stochastic error control over \(\mathcal{V}\)}
    
    \begin{lemma}\label{lemma:stoch_V_exp_unknown_var}
        Suppose \(C_a\) is a sufficiently large universal constant and \(\underline{h}/\sqrt{1-\gamma}\) is larger than a sufficiently large universal constant. There exist universal constants \(C', C'' > 0\) such that 
        \begin{equation*}
            E_{\theta, \gamma}\left( \sup_{\substack{\underline{h} \leq h \leq \bar{h}, \\ t \in \mathcal{V}(h)}} \left|\hat{G}(t, h) - G(t, h)\right| \right) \leq C' \left( \frac{\bar{h}^2/(1-\gamma)}{p\underline{h}} + \frac{e^{-C''\underline{h}^2/(1-\gamma)}}{\underline{h}\sqrt{p}}\right).
        \end{equation*}
    \end{lemma}
    \begin{proof}[Proof sketch]
        The proof is very similar to that of Lemma \ref{lemma:stoch_V_exp}. Consequently, we omit details and only point out some differences. As done in the proof of Lemma \ref{lemma:stoch_V_exp}, we first split 
        \begin{align}
            &\sup_{\substack{\underline{h} \leq h \leq \bar{h}, \\ t \in \mathcal{V}(h)}}\left|\hat{G}(t, h) - G(t, h)\right| \leq \nonumber \\
            &\frac{1}{2p\underline{h}} \sup_{\substack{\underline{h} \leq h \leq \bar{h}, \\ t \in \mathcal{V}(h)}} \left| \sum_{i \in \mathcal{I}} \mathbbm{1}_{\{|t - Y_i| \leq h\}} - P_{\theta, \gamma}\left\{|t - Y_i| \leq h\right\}\right| + \frac{1}{2p\underline{h}} \sup_{\substack{\underline{h} \leq h \leq \bar{h}, \\ t \in \mathcal{V}(h)}} \left| \sum_{i \in \mathcal{O}} \mathbbm{1}_{\{|t - Y_i| > h\}} - P_{\theta, \gamma}\left\{ |t - Y_i| > h\right\}\right|. \label{eqn:stochV_exp_parts_unknown_var}  
        \end{align}
        To bound the second term, we would like to apply Corollaries \ref{corollary:zeta_small_sigma} and \ref{corollary:zeta_big_sigma} but now taking supremum over \(\{(t, h) \in \R^2 : \underline{h} \leq h \leq \bar{h} \text{ and } t \in \mathcal{V}(h)\}\). To do so, consider that arguing as in the proof of Lemma \ref{lemma:stoch_V_exp}
        \begin{equation*}
            \sup_{\substack{\underline{h} \leq h \leq \bar{h}, \\ t \in \mathcal{V}(h)}} \frac{1}{|\mathcal{O}|} \sum_{i \in \mathcal{O}} P_{\theta, \gamma} \left\{ |t - Y_i| > h\right\} \leq \sup_{\underline{h} \leq h \leq \bar{h}} 16e^{-C\frac{h^2}{1-\gamma}} \leq 16e^{-C\frac{\underline{h}^2}{1-\gamma}}.
        \end{equation*}
        The rest of the argument, as well as the argument for bounding the first term of (\ref{eqn:stochV_exp_parts_unknown_var}), follows in exactly the same manner as in the proof of Lemma \ref{lemma:stoch_V_exp}.
    \end{proof}
    
    \begin{proposition}\label{prop:stochV_unknown_var}
        Suppose \(C_a\) is a sufficiently large universal constant and \(\underline{h}/\sqrt{1-\gamma}\) is larger than a sufficiently large universal constant. If \(u \geq 0\), then 
        \begin{align*}
            &P_{\theta, \gamma}\left\{ \sup_{\substack{\underline{h} \leq h \leq \bar{h}, \\ t \in \mathcal{V}(h)}} \left|\hat{G}(t, h) - G(t, h)\right| > C'\left(\frac{1}{\underline{h}\sqrt{p}}e^{-C''\underline{h}^2/(1-\gamma)} + \frac{\bar{h}^2/(1-\gamma)}{n\underline{h}}\right)\right\} \\
            &\leq 2\exp\left(-c\min\left( \frac{u^2}{\frac{\bar{h}^2/(1-\gamma)}{p^2\underline{h}^2} + \frac{e^{-C'''\underline{h}^2/(1-\gamma)}}{\underline{h}^2p^{3/2}}}, up\underline{h}\right)\right)
        \end{align*}
        where \(C', C'', C''', c > 0\) are universal constants. 
    \end{proposition}
    \begin{proof}[Proof of sketch]
        The proof is essentially the same as the proof of Proposition \ref{prop:stoch_V}, except Lemma \ref{lemma:stoch_V_exp_unknown_var} is invoked instead of Lemma \ref{lemma:stoch_V_exp}. 
    \end{proof}
    
    \begin{proposition}\label{prop:V_control_unknown_var}
        Suppose \(s \leq \frac{p}{2} - p^{1/4}\). Suppose \(C_a\) and \(\tilde{C}\) are sufficiently large universal constants. For \(\underline{h} \leq h \leq \bar{h}\), let \(x = x(h)\) with \(|x - \mu| \leq \frac{h}{4}\) denote the point from Proposition \ref{prop:mean_diff_fix}. If \(\delta \in (0, 1)\), \(R \geq 1\), and \(p \geq \tilde{C}\log^2\left(1/\delta\right)\), then there exists \(C_1 > 0\) depending only on \(R\) such that for
        \begin{align*}
            \underline{h} &:= C_1R^{-1}\sqrt{(1-\gamma)\left(1 \vee \left(\log\left(\frac{ep}{(p-2s)^2}\right)+\log\log\left(\frac{1}{\delta}\right)\right)\right)}, \\
            \bar{h} &:= C_1 R\sqrt{(1-\gamma)\left(1 \vee \left(\log\left(\frac{ep}{(p-2s)^2}\right) + \log\log\left(\frac{1}{\delta}\right)\right)\right)},
        \end{align*}
        we have with probability at least \(1-\delta\), 
        \begin{equation*}
            \left|\hat{G}(t, h) - G(t, h)\right| < \frac{1}{2} (G(x, h) - G(t, h))
        \end{equation*}
        uniformly over \(\{(t, h) \in \R^2 : \underline{h} \leq h \leq \bar{h} \text{ and } t \in \mathcal{V}(h)\}\). 
    \end{proposition}
    \begin{proof}[Proof sketch]
        The proof is the same as the proof of Proposition \ref{prop:V_control} except in invoking Proposition \ref{prop:stochV_unknown_var} and noting that \(\kappa_\delta \asymp \log\left(1/\delta\right)\) can be taken in the calculations of the proof of Proposition \ref{prop:V_control}.
    \end{proof}
    
    \noindent \textbf{Uniform stochastic error control over \(x(h)\)}
    
    One substantial difference between the current setting of unknown correlation and the setting of Section \ref{appendix:kme_known_var} lies in controlling the stochastic error at \(x = x(h)\) provided by Proposition \ref{prop:mean_diff_fix}. In Section \ref{appendix:kme_known_var}, the bandwidth is not random and so the stochastic deviation \(\left|\hat{G}(x, h) - G(x, h)\right|\) needed only to be controlled at a single point; empirical process tools were not necessary. However, in the current situation, we need to control the deviation uniformly over \(\underline{h} \leq h \leq \bar{h}\). Though it is an important conceptual difference from Section \ref{appendix:kme_known_var}, it does not add substantial technical difficulty since our empirical process tools can be easily employed. 
    
    \begin{proposition}\label{prop:x_control_unknown_var}
        Suppose \(s \leq \frac{p}{2} - p^{1/4}\). Suppose \(C_a\) and \(\tilde{C}\) are sufficiently large universal constants. For \(\underline{h} \leq h \leq \bar{h}\), let \(x(h)\) such that \(|x(h) - \mu| \leq \frac{h}{4}\) denote the point from Proposition \ref{prop:mean_diff_fix}. Suppose \(\delta \in (0, 1)\), \(R \geq 1\), and \(p \geq \tilde{C}\log^{16}(1/\delta)\) is sufficiently large depending only on \(\delta\). Further suppose 
        \begin{align*}
            \underline{h} &:= C_1R^{-1}\sqrt{(1-\gamma)\left(1 \vee \left(\log\left(\frac{ep}{(p-2s)^2}\right) + \log\log\left(\frac{1}{\delta}\right)\right)\right)}, \\
            \bar{h} &:= C_1 R \sqrt{(1-\gamma)\left(1 \vee \left(\log\left(\frac{ep}{(p-2s)^2}\right) + \log\log\left(\frac{1}{\delta}\right)\right)\right)}.
        \end{align*}
        Then there exists \(C_1 > 0\) sufficiently large depending only on \(R\) such that with probability at least \(1-\delta\), we have 
        \begin{equation*}
            \left|\hat{G}(x(h), h) - G(x(h), h)\right| < \frac{1}{2}(G(x(h), h) - G(t, h))
        \end{equation*}
        uniformly over \(\left\{(t, h) \in \R^2 : \underline{h} \leq h \leq \bar{h} \text{ and } |t - \mu| \geq C_a h\right\}\). 
    \end{proposition}
    \begin{proof}
        Since \(\left\{t \in \R : |t - \mu| \geq C_ah \right\} = \mathcal{U}(h) \cup \mathcal{V}(h)\), consider the two events 
        \begin{align*}
            \mathcal{E}_{\mathcal{U}} &:= \left\{ \left|\hat{G}(x(h), h) - G(x(h), h)\right| < \frac{1}{2}(G(x(h), h) - G(t, h)) \text{ for all } (t, h) \in \R^2 \text{ with }\underline{h} \leq h \leq \bar{h} \text{ and } t \in \mathcal{U}(h) \right\},\\
            \mathcal{E}_{\mathcal{V}} &:= \left\{ \left|\hat{G}(x(h), h) - G(x(h), h)\right| < \frac{1}{2}(G(x(h), h) - G(t, h)) \text{ for all } (t, h) \in \R^2 \text{ with }\underline{h} \leq h \leq \bar{h} \text{ and } t \in \mathcal{V}(h) \right\}.
        \end{align*}
        To prove the desired result, it suffices via union bound to show each event has probability at least \(1-\delta/2\). Define the sets 
        \begin{align*}
            E(h) &:= \left\{i \in \mathcal{O} : |\mu - \eta_i| \geq \frac{C_a h}{2}\right\}, \\
            F(h) &:= \left\{i \in \mathcal{O} : |\mu - \eta_i| < \frac{C_ah}{2}\right\}. 
        \end{align*}
        Further define 
        \begin{equation*}
            \mathcal{H} := \left\{ \underline{h} \leq h \leq \bar{h} : \sum_{i \in E(h)} P_{\theta,\gamma}\left\{|x(h) - Y_i| \leq h\right\} \leq \sum_{i \in F(h)} P_{\theta,\gamma}\left\{|x(h) - Y_i| \leq h\right\} \right\}.
        \end{equation*}
        Let us first examine \(\mathcal{E}_{\mathcal{U}}\). For any \(\underline{h} \leq h \leq \bar{h}\) and \(t \in \mathcal{U}(h)\), we have by Proposition \ref{prop:mean_diff_fix} 
        \begin{equation}\label{eqn:uxh_signal}
            G(x(h), h) - G(t, h) \geq \frac{p-2|\mathcal{O}|}{2ph} + \frac{1/2}{2ph} \sum_{i \in \mathcal{O}} P_{\theta, \gamma}\left\{|t - Y_i| > h\right\}. 
        \end{equation}
        Let us examine the stochastic deviation. Consider that
        \begin{align}\label{eqn:uxh_parts}
            &\left|\hat{G}(x(h), h) - G(x(h), h)\right| \leq \nonumber \\
            &\frac{1}{2ph} \left|\sum_{i \in \mathcal{I}} \mathbbm{1}_{\{|x(h) - Y_i| > h\}} - P_{\theta, \gamma}\left\{|x(h) - Y_i| > h\right\}\right| + \frac{1}{2ph} \left|\sum_{i \in \mathcal{O}} \mathbbm{1}_{\{|x(h) - Y_i| \leq h\}} - P_{\theta, \gamma}\left\{ |x(h) - Y_i| \leq h \right\}\right|. 
        \end{align}
        Let us now bound the first term in (\ref{eqn:uxh_parts}). Note that since \(|x(h) - \mu| \leq \frac{h}{4}\), we have 
        \begin{align*}
            \sup_{\substack{\underline{h} \leq h \leq \bar{h}}} \frac{1}{|\mathcal{I}|} \sum_{i \in \mathcal{I}} P_{\theta, \gamma}\left\{ |x(h) - Y_i| > h\right\} \leq \sup_{\underline{h} \leq h \leq \bar{h}} e^{-C\frac{h^2}{1-\gamma}} \leq e^{-C\frac{\underline{h}^2}{(1-\gamma)}}.
        \end{align*}
        Since this bound is analogous to the one in the proof of Lemma \ref{lemma:stoch_V_exp_unknown_var}, it follows that an analogue of Proposition \ref{prop:stochV_unknown_var} holds for \(\sup_{\underline{h} \leq h \leq \overline{h}} \frac{1}{2ph} \left|\sum_{i=1}^{p} \mathbbm{1}_{\{|x(h) - Y_i| > h\}} - P_{\theta, \gamma}\{|x(h)-Y_i| > h\} \right|\). It is then straightforward (e.g. mimicking the argument of Proposition \ref{prop:V_control_unknown_var}) to show that with probability at least \(1-\delta/4\), 
        \begin{equation}\label{eqn:ix_bound}
            \frac{1}{2ph} \left|\sum_{i \in \mathcal{I}} \mathbbm{1}_{\{|x(h) - Y_i| > h\}} - P_{\theta, \gamma}\left\{|x(h) - Y_i| > h\right\}\right| < \frac{1}{4} \cdot \frac{p-2s}{2ph}
        \end{equation}
        uniformly over \(\underline{h} \leq h \leq \overline{h}\). From (\ref{eqn:uxh_signal}), it thus follows the first term of (\ref{eqn:uxh_parts}) is indeed bounded by \(\frac{1}{4}(G(x(h), h) - G(t, h))\). 
        
        Let us now bound the second term in (\ref{eqn:uxh_parts}). Taking \(\lambda = \frac{1}{64}\) in Theorem \ref{thm:peel} (applied as per the argument in the proof of Proposition \ref{prop:U_control_unknown_var}) and noting \(|\mathcal{O}| \asymp p\), we have with probability at least \(1-\delta/8\)
        \begin{align*}
            \frac{1}{2ph} \left|\sum_{i \in \mathcal{O}} \mathbbm{1}_{\{|x(h) - Y_i| \leq h\}} - P_{\theta, \gamma}\left\{ |x(h) - Y_i| \leq h \right\}\right| &\leq 2\left(\frac{C'\log(ep)}{2ph} + \frac{1/64}{2ph} \sum_{i \in \mathcal{O}} P_{\theta, \gamma}\left\{|x(h) - Y_i| \leq h \right\}\right) 
        \end{align*}
        uniformly over \((t, h)\) such that \(\underline{h} \leq h \leq \bar{h}\) and \(t \in \mathcal{U}(h)\). The inequality follows from mimicking the argument of Proposition \ref{prop:U_control}. Consider that for \(t \in \mathcal{U}(h)\) and \(i \in F(h)\), we have \(P_{\theta, \gamma}\left\{|t - Y_i| > h\right\} \geq \frac{1}{2}\) since \(|t - \mu| \geq C_a h\) and \(C_a, h\) are sufficiently large. Therefore, for \(h \in \mathcal{H}\) we have 
        \begin{align*}
            \frac{1/64}{2ph} \sum_{i \in \mathcal{O}} P_{\theta, \gamma}\left\{|x(h) - Y_i| \leq h \right\} \leq \frac{1/32}{2ph} \sum_{i \in F(h)} P_{\theta, \gamma}\left\{|x(h) - Y_i| \leq h \right\} \leq \frac{1/16}{2ph} \sum_{i\in \mathcal{O}} P_{\theta, \gamma}\left\{|t - Y_i| > h\right\}. 
        \end{align*}
        Hence, with probability at least \(1-\delta/8\) we have 
        \begin{align}
            \frac{1}{2ph} \left|\sum_{i \in \mathcal{O}} \mathbbm{1}_{\{|x(h) - Y_i| \leq h\}} - P_{\theta, \gamma}\left\{ |x(h) - Y_i| \leq h \right\}\right| &\leq 2\left(\frac{C'\log(ep)}{2ph} + \frac{1/16}{2ph} \sum_{i\in \mathcal{O}} P_{\theta, \gamma}\left\{|t - Y_i| > h\right\}\right) \nonumber \\
            &< \frac{1}{4}(G(x(h), h) - G(t, h)) \label{eqn:uxh_parts_2a}
        \end{align}
        uniformly over \((t, h)\) with \( h \in \mathcal{H}\) and \(t \in \mathcal{U}(h)\). The final inequality follows from (\ref{eqn:uxh_signal}). It remains to obtain the same bound over \((t, h)\) with \(h \in \mathcal{H}^c\) and \(t \in \mathcal{U}(h)\). Consider that since \(|x(h) - \mu| \leq h/4\), we have
        \begin{align*}
            \sup_{h \in \mathcal{H}^c} \frac{1}{|\mathcal{O}|}\sum_{i \in \mathcal{O}} P_{\theta, \gamma}\left\{|x(h) - Y_i| \leq h\right\} \leq \sup_{h \in \mathcal{H}^c} \frac{2}{|\mathcal{O}|}\sum_{i \in E(h)} P_{\theta, \gamma}\left\{|x(h) - Y_i| \leq h\right\} \leq 2e^{-C\underline{h}^2/(1-\gamma)}.
        \end{align*}
        Therefore, the argument yielding (\ref{eqn:ix_bound}) yields 
        \begin{equation}\label{eqn:uxh_parts_2b}
            \frac{1}{2ph} \left|\sum_{i \in \mathcal{O}} \mathbbm{1}_{\{|x(h) - Y_i| \leq h\}} - P_{\theta, \gamma}\left\{ |x(h) - Y_i| \leq h \right\}\right| < \frac{1}{4} \cdot \frac{p-2s}{ph}
        \end{equation}
        uniformly over \(h \in \mathcal{H}^c\) with probability at least \(1-\delta/8\). Therefore, taking union bound over (\ref{eqn:ix_bound}), (\ref{eqn:uxh_parts_2a}), and (\ref{eqn:uxh_parts_2b}) gives us that \(\mathcal{E}_{\mathcal{U}}\) holds with probability at least \(1-\delta/2\). \newline
    
        Handling \(\mathcal{E}_{\mathcal{V}}\) is similar to handling \(\mathcal{E}_{\mathcal{U}}\). First, consider that for any \(\underline{h} \leq h \leq \bar{h}\) and \(t \in \mathcal{V}(h)\), we have by Proposition \ref{prop:mean_diff_fix} 
        \begin{equation}\label{eqn:vxh_signal}
            G(x(h), h) - G(t, h) \geq \frac{1}{2} \cdot \frac{p-2|\mathcal{O}|}{2ph}. 
        \end{equation}
        As in (\ref{eqn:uxh_parts}), we split into two parts 
        \begin{align}\label{eqn:vxh_parts}
            &\left|\hat{G}(x(h), h) - G(x(h), h)\right| \leq \nonumber \\
            &\frac{1}{2nh}\left|\sum_{i \in \mathcal{I}} \mathbbm{1}_{\{|x(h) - Y_i| > h\}} - P_{\theta, \gamma} \left\{|x(h) - Y_i| > h\right\}\right| + \frac{1}{2ph}\left|\sum_{i \in \mathcal{O}} \mathbbm{1}_{\{|x(h) - Y_i| \leq h\}} - P_{\theta, \gamma} \left\{|x(h) - Y_i| \leq h\right\}\right|. 
        \end{align}
        The first term can be bounded by mimicking the proof of (\ref{eqn:ix_bound}) to obtain 
        \begin{equation*}
            \frac{1}{2ph}\left|\sum_{i \in \mathcal{I}} \mathbbm{1}_{\{|x(h) - Y_i| > h\}} - P_{\theta, \gamma}\left\{|x(h) - Y_i| > h\right\}\right| < \frac{1}{8} \cdot \frac{p-2s}{2ph}    
        \end{equation*}
        uniformly over \(\underline{h} \leq h \leq \bar{h}\) with probability at least \(1-\delta/4\). To bound the second term, consider that
        \begin{align*}
            &\sup_{\underline{h} \leq h \leq \overline{h}} \frac{1}{|\mathcal{O}|} \sum_{i \in \mathcal{O}} P_{\theta,\gamma}\left\{|x(h) - Y_i| \leq h\right\} \\
            &\leq \sup_{h \in \mathcal{H}} \frac{1}{|\mathcal{O}|} \sum_{i \in \mathcal{O}} P_{\theta,\gamma}\left\{|x(h) - Y_i| \leq h\right\} + \sup_{h \in \mathcal{H}^c} \frac{1}{|\mathcal{O}|} \sum_{i \in \mathcal{O}} P_{\theta,\gamma}\left\{|x(h) - Y_i| \leq h\right\} \\
            &\leq \sup_{h \in \mathcal{H}} \frac{2}{|\mathcal{O}|} \sum_{i \in F(h)} P_{\theta,\gamma}\left\{|x(h) - Y_i| \leq h\right\} + \sup_{h \in \mathcal{H}^c} \frac{2}{|\mathcal{O}|} \sum_{i \in E(h)} P_{\theta,\gamma}\left\{|x(h) - Y_i| \leq h\right\}.
        \end{align*}
        Clearly we have \(\sup_{h \in \mathcal{H}^c} \frac{2}{|\mathcal{O}|} \sum_{i \in E(h)} P_{\theta,\gamma}\left\{|x(h) - Y_i| \leq h\right\} \leq 2e^{-C\underline{h}^2/(1-\gamma)}\). On the other hand, consider that for \(h \in \mathcal{H}\) and for any \(t \in \mathcal{V}(h)\) we have the following as argued in the analysis of \(\mathcal{E}_{\mathcal{U}}\). Namely, for \(i \in F(h)\), we have \(P_{\theta, \gamma}\{|t - Y_i| > h\} \geq \frac{1}{2}\). Therefore, 
        \begin{align*}
            \frac{2}{|\mathcal{O}|} \sum_{i \in F(h)} P_{\theta,\gamma}\left\{|x(h) - Y_i| \leq h\right\} &\leq \frac{2}{|\mathcal{O}|} \sum_{i \in F(h)} 2P_{\theta,\gamma}\left\{|t - Y_i| > h\right\} \\
            &\leq \frac{4}{|\mathcal{O}|} \sum_{i \in \mathcal{O}} P_{\theta, \gamma}\left\{|t - Y_i| > h\right\} \\
            &\leq \frac{16p}{|\mathcal{O}|} e^{-C\underline{h}^2/(1-\gamma)} \\
            &\leq 64 e^{-C\underline{h}^2/(1-\gamma)}
        \end{align*}
        where we have used the definition of \(\mathcal{V}(h)\) and \(|\mathcal{O}| \geq \frac{p}{4}\). Therefore, we have shown 
        \begin{equation*}
            \sup_{\underline{h} \leq h \leq \overline{h}} \frac{1}{|\mathcal{O}|} \sum_{i \in \mathcal{O}} P_{\theta,\gamma}\left\{|x(h) - Y_i| \leq h\right\} \leq 66e^{-C\underline{h}^2/(1-\gamma)}
        \end{equation*}
        and so the argument yielding (\ref{eqn:ix_bound}) can be mimicked to yield 
        \begin{equation*}
            \frac{1}{2ph} \left|\sum_{i \in \mathcal{O}} \mathbbm{1}_{\{|x(h) - Y_i| \leq h\}} - P_{\theta, \gamma}\left\{ |x(h) - Y_i| \leq h \right\}\right| < \frac{1}{8} \cdot \frac{p-2s}{ph}
        \end{equation*}
        uniformly over \(\underline{h} \leq h \leq \bar{h}\) with probability at least \(1-\delta/8\). Taking union over the two events yielding bounds for the pieces in (\ref{eqn:vxh_parts}) gives us that \(\mathcal{E}_{\mathcal{V}}\) has probability at least \(1-\delta/2\). The proof is complete. 
    \end{proof} 
    
    \subsection{Synthesis}
    All of the pieces are now combined to establish the accuracy of the kernel mode estimator in the regime \(s \leq \frac{p}{2} - p^{1/4}\). 
    
    \begin{proposition}\label{prop:kme_hard_unknown_var}
        Suppose \(s \leq \frac{p}{2}-p^{1/4}\). Further suppose \(C_a\) and \(\tilde{C}\) are sufficiently large universal constants. Fix \(\delta, \eta \in (0, 1)\) and let \(\hat{\gamma}\) be the estimator from Proposition \ref{prop:correlation_estimation} defined at confidence level \(\eta\). Define the (random) bandwidth 
        \begin{equation*}
            \hat{h} := C_1 \sqrt{(1-\hat{\gamma})\left(1 \vee \left(\log\left(\frac{ep}{(p-2s)^2}\right) + \log\log\left(\frac{1}{\delta}\right)\right)\right)} 
        \end{equation*}
        where \(C_1 > 0\) is sufficiently large depending only on \(\eta\). If \(p \geq \tilde{C}\log^{16}(1/\delta)\), then the kernel mode estimator with (random) bandwidth \(\hat{h}\), 
        \begin{equation*}
            \tilde{\mu} := \argmax_{x \in \R} \frac{1}{2p\hat{h}}\sum_{i=1}^{p} \mathbbm{1}_{\{|x - Y_i| \leq \hat{h}\}},
        \end{equation*}
        satisfies 
        \begin{equation*}
            \sup_{\substack{||\theta||_0 \leq s \\ \gamma \in [0, 1)}} P_{\theta, \gamma}\left\{ \frac{|\tilde{\mu} - \mu|}{\sqrt{1-\gamma}} > C' \sqrt{1 \vee \log\left(\frac{ep}{(p-2s)^2}\right)}\right\} \leq \delta + \eta
        \end{equation*}
        where \(C' > 0\) is a constant depending only on \(\eta\). 
    \end{proposition}
    \begin{proof}
        First, consider that by Proposition \ref{prop:correlation_estimation}, there exists \(R \geq 1\) depending only on \(\eta\) such that the event \(\mathcal{E}_{\gamma} := \left\{R^{-1}(1-\gamma) \leq 1-\hat{\gamma} \leq R(1-\gamma)\right\}\) has probability at least \(1-\eta\). Define 
        \begin{align*}
            \underline{h} &:= C_1R^{-1}\sqrt{(1-\gamma)\left(1\vee \left(\log\left(\frac{ep}{(p-2s)^2}\right) + \log\log\left(\frac{1}{\delta}\right)\right)\right)}, \\
            \bar{h} &:= C_1 R\sqrt{(1-\gamma)\left(1 \vee \left(\log\left(\frac{ep}{(p-2s)^2}\right) + \log\log\left(\frac{1}{\delta}\right)\right)\right)}
        \end{align*}
        for sufficiently large \(C_1 > 0\) depending only on \(\delta\). From Propositions \ref{prop:U_control_unknown_var}, \ref{prop:V_control_unknown_var}, and \ref{prop:x_control_unknown_var}, it follows that the events 
        \begin{align*}
            \mathcal{E}_{\mathcal{U}} &:=  \left\{ \left|\hat{G}(x(h),h) - G(x(h), h)\right| < \frac{1}{2}(G(x(h), h) - G(t, h)) \text{ for all } (t, h) \text{ with } t \in \mathcal{U}(h), \underline{h} \leq h \leq \bar{h} \right\},\\
            \mathcal{E}_{\mathcal{V}} &:=  \left\{ \left|\hat{G}(x(h),h) - G(x(h), h)\right| < \frac{1}{2}(G(x(h), h) - G(t, h)) \text{ for all } (t, h) \text{ with } t \in \mathcal{V}(h), \underline{h} \leq h \leq \bar{h} \right\},\\
            \mathcal{E}_{x} &:= \left\{ \left|\hat{G}(x(h),h) - G(x(h), h)\right| < \frac{1}{2}(G(x(h), h) - G(t, h)) \text{ for all } (t, h) \text{ with } |t - \mu| \geq C_a h, \underline{h} \leq h \leq \bar{h} \right\}
        \end{align*}
        each have probability at least \(1-\delta/6\). Thus by union bound it follows that the event \(\mathcal{E} = \mathcal{E}_{\gamma} \cap \mathcal{E}_{\mathcal{U}} \cap \mathcal{E}_{\mathcal{V}} \cap \mathcal{E}_x\) has probability at least \(1-\delta-\eta\). Furthermore, on \(\mathcal{E}\) we have uniformly over \((t, h)\) such that \(|t - \mu| \geq C_a h\) and \(\underline{h} \leq h \leq \bar{h}\), 
        \begin{align*}
            \hat{G}(x(h), h) &> G(x(h), h) - \frac{1}{2}(G(x(h), h) - G(t, h)) \\
            &\geq G(x(h), h) - G(t, h) - \left|\hat{G}(t, h) - G(t, h)\right| - \frac{1}{2}\left(G(x(h), h) - G(t, h)\right) + \hat{G}(t, h).
        \end{align*}
        Note \(|t - \mu| \geq C_a h\) implies \(t \in \mathcal{U}(h)\) or \(t \in \mathcal{V}(h)\). In either case, since we are on \(\mathcal{E}\) we have 
        \begin{equation*}
            \hat{G}(x(h), h) > G(x(h), h) - G(t, h) - \frac{1}{2}\left(G(x(h), h) - G(t, h)\right) - \frac{1}{2}\left(G(x(h), h) - G(t, h)\right) + \hat{G}(t, h) = \hat{G}(t, h).
        \end{equation*}
        Since this holds uniformly over \((t, h)\) and since \(\hat{h} \in [\underline{h}, \bar{h}]\) because \(\mathcal{E} \subset \mathcal{E}_{\gamma}\), it follows that 
        \begin{equation*}
            \hat{G}(x(\hat{h}), \hat{h}) > \hat{G}(t, \hat{h})
        \end{equation*}
        for all \(|t - \mu| \geq C_a \hat{h}\). It thus follows that on the event \(\mathcal{E}\), we have \(|\tilde{\mu} - \mu| \leq C_a \hat{h} \leq C_a \bar{h}\). As \(\mathcal{E}\) has probability at least \(1-\delta-\eta\), the proof is complete. 
    \end{proof}

    \subsection{Proof of Theorem \ref{thm:kme_unknown_var}}
    
    \begin{proof}[Proof of Theorem \ref{thm:kme_unknown_var}]
        Theorem \ref{thm:kme_unknown_var} follows directly from Propositions \ref{prop:kme_easy_unknown_var} and \ref{prop:kme_hard_unknown_var}. 
    \end{proof}

    \section{Miscellaneous results}\label{section:misc}
    In this section, proofs of Proposition \ref{prop:sample_median} and Theorem \ref{thm:adapt_in-exp} are established. 
    
    \begin{proof}[Proof of Proposition \ref{prop:sample_median}]
        Recall \(Y\) is given by (\ref{model:Gaussian_contamination}). Without loss of generalization, assume \(\gamma = 0\) otherwise we can simply work with the normalized data \(\{Y_j/(1-\gamma)\}_{j=1}^{n}\) since \(\median(Y_1/(1-\gamma),\ldots,Y_p/(1-\gamma)) = \median(Y_1,\ldots,Y_p)/(1-\gamma) = \hat{T}/(1-\gamma)\). The analysis is split into three cases. \newline

        \noindent \textbf{Case 1:} Suppose \(1 \leq s \leq \frac{p}{4}\). Let \(\Phi\) denote the cumulative distribution function of the standard Gaussian distribution and let \(\Phi^{-1}\) denote its quantile function. Let \(\pi = \frac{1}{2} - \left(\frac{C_\delta}{\sqrt{p}} + \frac{s}{2(p-s)}\right)\) where \(C_\delta > 0\) is sufficiently large depending only on \(\delta\). Note we have \(\pi \gtrsim 1\) since \(p\) is sufficiently large and \(\frac{s}{2(p-s)} \leq \frac{1}{4}\). Denote the event \(E_j = \left\{Y_j > \bar{\theta} + \Phi^{-1}(1-\pi)\right\}\) and note \(\pi = P_{\theta_0}(E_j)\) for \(j \in \mathcal{I}\). Consider by Hoeffding's inequality, 
        \begin{align*}
            P_{\theta, 0}\left\{ \hat{T} > \bar{\theta} + \Phi^{-1}(\pi)\right\} &\leq P_{\theta, 0}\left\{ \sum_{j=1}^{p} \mathbbm{1}_{E_j} > \frac{p}{2} \right\} \\
            &\leq P_{\theta, 0}\left\{ \sum_{j \in \mathcal{I}} \mathbbm{1}_{E_j} - \pi > \frac{p}{2} - |\mathcal{O}| - |\mathcal{I}|\pi \right\} \\
            &=  P_{\theta, 0}\left\{ \frac{1}{|\mathcal{I}|}\sum_{j \in \mathcal{I}} \mathbbm{1}_{E_j} - \pi > \frac{1}{2} - \frac{|\mathcal{O}|}{2|\mathcal{I}|} - \pi \right\} \\
            &\leq \exp\left(-c'' |\mathcal{I}|\left(\frac{1}{2} - \frac{s}{2(p-s)} - \pi\right)^2\right) \\
            &\leq \exp\left(-\frac{c''C_\delta^2}{2}\right) \\
            &\leq \frac{\delta}{2} 
        \end{align*}
        where \(c'' > 0\) is a universal constant. Here, we have used \(C_\delta > 0\) is sufficiently large.  Applying a similar argument to bound \(P_{\theta, 0}\left\{\hat{T} < \bar{\theta} - \Phi^{-1}(\pi)\right\}\), we can conclude by union bound 
        \begin{equation*}
            P_{\theta, 0}\left\{|\hat{T} - \bar{\theta}| > \Phi^{-1}(1-\pi)\right\} \leq \delta.     
        \end{equation*}
        By Taylor expansion, we have \(\Phi^{-1}(1-\pi) = \Phi^{-1}(\frac{1}{2}) + \frac{\left(1 - \pi - \frac{1}{2}\right)}{\Phi'(\Phi^{-1}(\xi))} = \left(\frac{C_\delta}{\sqrt{p}} + \frac{s}{2(p-s)}\right)\frac{1}{\Phi'(\Phi^{-1}(\xi))}\) for some \(\xi\) between \(\frac{1}{2}\) and \(1-\pi\). Since \(\pi \gtrsim 1\), it follows \(\frac{1}{\Phi'(\Phi^{-1}(\xi))} \asymp 1\), and so \(\Phi^{-1}(1-\pi) \asymp \frac{C_\delta}{\sqrt{p}} + \frac{s}{p}\). Since \(||\hat{T}\mathbf{1}_p - \bar{\theta}\mathbf{1}_p||^2 = p |\hat{T} - \bar{\theta}|^2\), the desired result follows. \newline
    
        \noindent \textbf{Case 2:} Suppose \(\frac{p}{4} < s < \frac{p}{2} - \sqrt{p}\). Define the interval 
        \begin{equation*}
            E := \left[\bar{\theta} - \sqrt{2 \log\left(\frac{4|\mathcal{I}|}{\frac{p}{2} - |\mathcal{O}|}\right)}, \bar{\theta} + \sqrt{2\log\left(\frac{4|\mathcal{I}|}{\frac{p}{2} - |\mathcal{O}}\right)} \right]. 
        \end{equation*}
        Consider 
        \begin{align*}
            P_{\theta, 0}\left\{\hat{T} \not \in E\right\} &\leq P_{\theta, 0} \left\{ \sum_{j=1}^{p} \mathbbm{1}_{\{Y_j \not \in E\}} > \frac{p}{2} \right\} \\
            &\leq P_{\theta, 0} \left\{ \sum_{j \in \mathcal{I}} \mathbbm{1}_{\{Y_j \not \in E\}} > \frac{p}{2} - |\mathcal{O}| \right\} \\
            &\leq P_{\theta, 0} \left\{ \sum_{j \in \mathcal{I}} \mathbbm{1}_{\{Y_j \not \in E\}} - \pi > \frac{p}{2} - |\mathcal{O}| - |\mathcal{I}|\pi\right\}
        \end{align*}
        where \(\pi = P\left\{|Z| > \sqrt{2 \log\left(\frac{4|\mathcal{I}|}{\frac{p}{2} - |\mathcal{O}|}\right)}\right\}\) with \(Z \sim N(0, 1)\). By Bernstein's inequality (Theorem \ref{thm:bernstein_bounded}), we have for a universal constant \(c > 0\), 
        \begin{equation*}
            P_{\theta, 0}\left\{ \sum_{j \in \mathcal{I}} \mathbbm{1}_{\{Y_j \not \in E\}} - \pi > \frac{p}{2} - |\mathcal{O}| - |\mathcal{I}|\pi \right\} \leq \exp\left(-c \min\left( \frac{\left(\frac{p}{2} - |\mathcal{O}| - |\mathcal{I}|\pi\right)^2}{|\mathcal{I}|\pi(1-\pi)} \right)\right). 
        \end{equation*} 
        Since \(\pi \leq 2\exp\left(-\frac{1}{2} \cdot 2 \log\left(\frac{4|\mathcal{I}|}{\frac{p}{2} - |\mathcal{O}|}\right) = \frac{\frac{p}{2} - |\mathcal{O}|}{2|\mathcal{I}|}\right)\). Hence, 
        \begin{equation*}
            P_{\theta, 0} \left\{\hat{T} \not \in E\right\} \leq \exp\left(-\frac{c}{2} \left(\frac{p}{2} - |\mathcal{O}|\right)\right) \leq \exp\left(-\frac{c(p-2s)}{4}\right) \leq \exp\left(-\frac{c\sqrt{p}}{4}\right) \leq \delta
        \end{equation*}
        as \(p\) is sufficiently large depending on \(\delta\). Consider \(\log\left(\frac{4|\mathcal{I}|}{\frac{p}{2} - |\mathcal{O}|}\right) \lesssim \log\left(\frac{ep}{p-2s}\right) \asymp \frac{1}{p} + \frac{s^2}{p^2} \log\left(\frac{ep}{p-2s}\right)\) because \(s \asymp p\). Thus, the desired result follows. \newline 
    
        \noindent \textbf{Case 3:} Suppose \(s \geq \frac{p}{2} - \sqrt{p}\). Since \(|\mathcal{I}| > \frac{p}{2}\), it follows 
        \begin{equation*}
            P_{\theta,0}\left\{ |\hat{T} - \bar{\theta}| > \sqrt{4\log(ep)} \right\} \leq P_{\theta, 0}\left\{ \max_{j \in \mathcal{I}} |X_j - \bar{\theta}| > \sqrt{4\log(ep)} \right\} \leq 2p \exp\left(-2\log(ep)\right) \leq \frac{1}{p} < \delta
        \end{equation*} 
        for \(p \geq \frac{1}{\delta}\). Since \(\log\left(\frac{ep}{p-2s}\right) \asymp \log(ep)\) because \(p-2s \leq \sqrt{p}\), the claimed result follows.
    \end{proof}
    
    \begin{proof}[Proof of Theorem \ref{thm:adapt_in-exp}]
        To prove the desired result, we first prove an intermediate statement. Let \(\theta_1 := \frac{a}{\sqrt{p}}\mathbf{1}_p\) with \(a = \sqrt{(1-\gamma+\gamma p)r}\). For any rates \(f(0)\) and \(f(p)\) and for any estimator \(\hat{\theta}\) we have
        \begin{align*}
            &\sup_{||\theta||_0 \leq s} \frac{E_{\theta, \gamma}\left(||\hat{\theta} - \theta||^2\right)}{f(0)} + \sup_{||\theta||_0 \leq p} \frac{E_{\theta, \gamma}\left( ||\hat{\theta} - \theta||^2 \right)}{f(p)} \\
            &\geq \frac{E_{0, \gamma}(||\hat{\theta}||^2)}{f(0)} + \frac{E_{\theta_1, \gamma}\left(||\hat{\theta} - \theta_1||^2\right)}{f(p)} \\
            &\geq \frac{E_{0, \gamma}\left(||\hat{\theta}||^2 \mathbbm{1}_{\{||\hat{\theta}||^2 \geq a^2/4\}}\right)}{f(0)} + \frac{E_{\theta_1, \gamma}\left(||\hat{\theta} - \theta_1||^2 \mathbbm{1}_{\{||\hat{\theta}||^2 < a^2/4\}}\right)}{f(p)} \\
            &\geq \frac{a^2}{4} \left( \frac{P_{0, \gamma}\left\{||\hat{\theta}||^2 \geq \frac{a^2}{4}\right\}}{f(0)} + \frac{P_{\theta_1, \gamma}\left\{ ||\hat{\theta}||^2 < \frac{a^2}{4} \right\}}{f(p)}\right) \\
            &\geq  \frac{a^2}{4f(p)} \inf_{\mathcal{A}} \left\{ \frac{f(p)}{f(0)} P_{0, \gamma}\left(\mathcal{A}\right) + P_{\theta_1, \gamma}(\mathcal{A}^c) \right\}.
        \end{align*}
        Here, the infimum runs over all events \(\mathcal{A}\). Also, we have used the Pythagorean identity to argue that on the event \(\{||\hat{\theta}||^2 < a^2/4\}\) we have \(\overline{\hat{\theta}}^2 \leq a^2/(4p)\), and so \(||\hat{\theta} - \theta_1||^2 \geq \left(\overline{\hat{\theta}} - a/\sqrt{p}\right)^2 ||\mathbf{1}_p||^2 \geq \frac{a^2}{4}\). Let \(q = \frac{f(p)}{f(0)}\). By Lemma \ref{lemma:adapt_exp_lbound}, we have
        \begin{equation*}
            \inf_{\mathcal{A}}\left\{ q P_{0, \gamma}(\mathcal{A}) + P_{\theta_1, \gamma}(\mathcal{A}^c) \right\} \geq \sup_{0 < \tau < 1} \left\{\frac{q\tau}{1 + q\tau} \left(1 - \tau\left(\chi^2(P_{\theta_1, \gamma}\,||\, P_{0, \gamma}) + 1\right)\right) \right\}.
        \end{equation*}
        Noting \(\Sigma^{-1} = \frac{1}{1-\gamma}\left(I_p - \frac{1}{p}\mathbf{1}_p\mathbf{1}_p^\intercal\right) + \frac{1}{1-\gamma+\gamma p} \cdot \frac{1}{p}\mathbf{1}_p\mathbf{1}_p^\intercal\), consider by the Ingster-Suslina method (see Lemma \ref{lemma:ingster_suslina})
        \begin{equation*}
            \chi^2(P_{\theta_1, \gamma}\,||\, P_{0, \gamma}) + 1 = \exp\left(\langle \theta_1, \Sigma^{-1} \theta_1\rangle\right) = \exp\left(\frac{a^2}{1-\gamma+\gamma p}\right) = e^r. 
        \end{equation*}
        Therefore,
        \begin{equation*}
            \sup_{0 < \tau < 1} \left\{ \frac{q\tau}{1 + q\tau}\left(1 - \tau(\chi^2(P_{\theta_1, \gamma}\,||\, P_{0, \gamma})+ 1) \right)\right\} \geq \sup_{0 < \tau < 1} \left\{\frac{q\tau}{1+q\tau} (1-\tau e^r)\right\} \geq \frac{qe^{-r}/2}{2(1+qe^{-r}/2)}
        \end{equation*}
        where we have taken \(\tau = e^{-r}/2\) to obtain the final inequality. Therefore, we have
        \begin{equation*}
            \sup_{||\theta||_0 \leq s} \frac{E_{\theta, \gamma}\left(||\hat{\theta} - \theta||^2\right)}{f(0)} + \sup_{||\theta||_0 \leq p} \frac{E_{\theta, \gamma}\left( ||\hat{\theta} - \theta||^2\right)}{f(p)} \geq \frac{a^2}{4f(p)} \cdot \frac{qe^{-r}/2}{2(1+qe^{-r}/2)} \gtrsim \frac{a^2}{f(p)} \cdot \frac{qe^{-r}}{1 + qe^{-r}}.
        \end{equation*}
        Note \(\frac{qe^{-r}}{1 + qe^{-r}} \gtrsim 1\) if and only if \(qe^{-r} \gtrsim 1\), which in turn holds if and only if \(f(p) \gtrsim f(0)e^r\). Furthermore, \(\frac{a^2}{f(p)} \gtrsim 1\) if and only if \(a^2 \gtrsim f(p)\). Consequently, with the choice \(f(p) \asymp a^2\) and \(f(0) \asymp a^2 e^{-r}\), we have
        \begin{equation*}
            \sup_{||\theta||_0 \leq s} \frac{E_{\theta, \gamma}(||\hat{\theta} - \theta||^2)}{f(0)} + \sup_{||\theta||_0 \leq p} \frac{E_{\theta, \gamma}(||\hat{\theta} - \theta||^2)}{f(p)} \gtrsim 1. 
        \end{equation*}
        In other words, there exists a universal constant \(C_0 > 0\) such that
        \begin{equation*}
            \sup_{||\theta||_0 \leq s} \frac{E_{\theta, \gamma}(||\hat{\theta} - \theta||^2)}{a^2 e^{-r}} + \sup_{||\theta||_0 \leq p} \frac{E_{\theta, \gamma}(||\hat{\theta} - \theta||^2)}{a^2} \geq 2C_0.
        \end{equation*}
        Since \(a^2 = (1-\gamma + \gamma p)r\) and \(r \geq 1\), it is clear \(a^2 e^{-r} \geq (1-\gamma+\gamma p)e^{-C_1 r}\) for a universal constant \(C_1 > 0\). Consequently, if \(\hat{\theta}\) is an estimator such that 
        \begin{equation*}
            \sup_{||\theta||_0 \leq s} E_{\theta, \gamma}\left(||\hat{\theta} - \theta||^2\right) \leq C_0 (1-\gamma+\gamma p)e^{-C_1 r},
        \end{equation*}
        then 
        \begin{equation*}
            \sup_{||\theta||_0 \leq p} E_{\theta, \gamma}\left(||\hat{\theta} - \theta||^2\right) \geq C_0(1-\gamma+\gamma p)r. 
        \end{equation*}
        The proof is complete.
    \end{proof}

        \section{Auxiliary results}\label{section:auxiliary}
    
        \begin{lemma}\label{lemma:tv_tensorization}
            Suppose \(\mu = \mu_1 \otimes \mu_2\) and \(\nu = \nu_1 \otimes \nu_2\) are two product measures on \(\mathcal{X} \times \mathcal{X}\). Suppose \(\mu_1, \mu_2, \nu_1,\) and \(\nu_2\) are absolutely continuous with respect to a common measure \(\lambda\) on \(\mathcal{X}\). Then \(\dTV(\mu, \nu) \leq \dTV(\mu_1, \nu_1) + \dTV(\mu_2, \nu_2)\). 
        \end{lemma}
    
        \begin{proposition}[Lemma 4.10 \cite{massart_concentration_2007} - Sparse Gilbert-Varshamov]
            Given \(1 \leq s < p\), define \(\{0, 1\}^p_s := \left\{x \in \{0, 1\}^p : ||x||_0 = s\right\}\). For every \(\alpha \in (0, 1)\) and \(\beta \in (0, 1)\) such that \(s \leq \alpha \beta p\), there exists some subset \(\mathcal{M}\) of \(\{0, 1\}^p_s\) such that 
            \begin{enumerate}[label=(\roman*)]
                \item \(\sum_{i=1}^{p} \mathbbm{1}_{\{m_i \neq m_i'\}} > 2(1-\alpha)s\) for all \(m, m' \in \mathcal{M}\) with \(m \neq m'\),
                \item \(\log |\mathcal{M}| \geq \rho s \log\left(\frac{p}{s}\right)\)
            \end{enumerate}
            where \(\rho = \frac{\alpha}{-\log(\alpha \beta)} (-\log(\beta) + \beta - 1)\).
        \end{proposition}
    
        \begin{corollary}\label{corollary:sparse_gilbert_varshamov}
            If \(1 \leq s \leq \frac{p}{2}\), then there exist universal positive constants \(c_1, c_2\) and a subset \(\mathcal{M} \subset \{x \in \{0, 1\}^p : ||x||_0 = s\}\) such that 
            \begin{enumerate}[label=(\roman*)]
                \item \(\sum_{i=1}^{p} \mathbbm{1}_{\{m_i \neq m_i'\}} > c_1 s\) for all \(m, m' \in \mathcal{M}\) with \(m \neq m'\), 
                \item \(\log |\mathcal{M}| \geq c_2 s \log\left(\frac{ep}{s}\right)\).
            \end{enumerate}
        \end{corollary}
    
        \begin{lemma}\label{lemma:precision}
            If \(\gamma \in [0, 1)\), then 
            \begin{equation*}
                \left((1-\gamma)I_p + \gamma \mathbf{1}_p\mathbf{1}_p^\intercal \right)^{-1} = \frac{1}{1-\gamma}\left(I_p - \frac{1}{p}\mathbf{1}_p\mathbf{1}_p^\intercal \right) + \frac{1}{1-\gamma + \gamma p} \cdot \frac{1}{p}\mathbf{1}_p\mathbf{1}_p^\intercal.
            \end{equation*}
        \end{lemma}
    
        \begin{proposition}\label{prop:reduce_to_test}
            Suppose \(\mathcal{P}\) is a collection of distributions on a sample space \(\mathcal{X}\) and suppose \((\Upsilon, \rho)\) is a metric space. Let \(\tau : \mathcal{P} \to \Upsilon\) be a function and let \(\phi : \R_{+} \to \R_{+}\) be a non-decreasing function with \(\phi(0) = 0\). If \(\delta > 0\) and \(\mathcal{V}\) is a finite index set with \(\{P_v\}_{v \in \mathcal{V}} \subset \mathcal{P}\) such that \(\rho(\tau(P_v), \tau(P_{v'})) \geq 2\delta\) for all \(v \neq v'\), then 
            \begin{equation*}
                \inf_{\hat{\tau}} \sup_{P_v \in \mathcal{P}} P_v\left\{\phi(\rho(\hat{\tau}(X), \tau(P_v))) \geq \phi(\delta) \right\} \geq \inf_{\varphi} \max_{v \in \mathcal{V}} P_v\left\{\varphi(X) \neq v\right\}
            \end{equation*}
            where the infimums run over all estimators \(\hat{\tau} : \mathcal{X} \to \Upsilon\) and measurable functions \(\varphi : \mathcal{X} \to \mathcal{V}\) respectively. 
        \end{proposition}
        \begin{proof}
            For any \(\hat{\tau}\), define the function \(\varphi : \mathcal{X} \to \mathcal{V}\) with \(\varphi(x) = \argmin_{v \in \mathcal{V}} \rho(\hat{\tau}(x), \tau(P_v))\). Since \(\rho(\tau(P_v), \tau(P_{v'}) \geq 2\delta\) for all \(v \neq v'\), it follows by triangle inequality the inclusion of events
            \begin{equation*}
                \left\{\varphi(X) \neq v\right\} \subset \left\{ \rho(\hat{\tau}(X), \tau(P_v)) \geq \delta \right\}.
            \end{equation*}
            Since \(\mathcal{V} \subset \mathcal{P}\) and \(\phi\) is a non-decreasing function, we thus have 
            \begin{align*}
                \inf_{\hat{\tau}} \sup_{v \in \mathcal{P}}  P_v\left\{\phi(\rho(\hat{\tau}(X), \tau(P_v))) \geq \phi(\delta) \right\} &\geq \inf_{\hat{\tau}} \max_{v \in \mathcal{V}}  P_v\left\{\phi(\rho(\hat{\tau}(X), \tau(P_v))) \geq \phi(\delta) \right\} \\
                &\geq \inf_{\hat{\tau}} \max_{v \in \mathcal{V}}  P_v\left\{\rho(\hat{\tau}(X), \tau(P_v)) \geq \delta \right\} \\
                &\geq \inf_{\varphi} \max_{v \in \mathcal{V}} P_v\left\{\varphi(X) \neq v\right\}
            \end{align*}
            as desired. 
        \end{proof}
    
        \begin{proposition}[Fano's lemma] \label{prop:fano}
            If \(P_1,...,P_M\) for \(M \geq 2\) are probability distributions on a measurable space \((\mathcal{X}, \mathcal{A})\) such that \(P_j \ll P_k\) for all \(1 \leq j, k \leq M\), then 
            \begin{equation*}
                \inf_{\varphi} \max_{1 \leq j \leq M} P_j\left\{\varphi(X) \neq j \right\} \geq 1 - \frac{\frac{1}{M^2} \sum_{1 \leq j,k \leq M} \dKL(P_j || P_k) + \log 2}{\log M}
            \end{equation*}
            where the infimum runs over all measurable functions \(\varphi : \mathcal{X} \to \{1,...,M\}\). 
        \end{proposition}
    
        \begin{proposition}[Method of two fuzzy hypotheses]\label{prop:fuzzy_hypotheses}
            Suppose \(\mathcal{P}\) is a collection of distributions on a sample space \(\mathcal{X}\) and \((\Upsilon, \rho)\) is a metric space. Let \(\tau : \Theta \to \Upsilon\) be a function, \(\phi : \R_+ \to \R_+\) be a non-decreasing function with \(\phi(0) = 0\), and let \(\mathcal{P}_0, \mathcal{P}_1 \subset \mathcal{P}\). If \(\pi_0\) and \(\pi_1\) are two priors supported on \(\Theta_0\) and \(\Theta_1\) respectively, then 
            \begin{equation*}
                \inf_{\hat{\tau}}\sup_{P \in \mathcal{P}} P\left\{ \phi\left(\rho(\hat{\tau}(X), \tau(P))\right) \geq \phi(\delta) \right\} \geq 1 - \dTV(P_{\pi_0}, P_{\pi_1})
            \end{equation*}
            where \(2\delta := \inf_{\substack{P \in \mathcal{P}_0 \\ P' \in \mathcal{P}_1}} \rho(\tau(P), \tau(P'))\) and \(P_{\pi_j} = \int P \, \pi_j(dP)\) for \(j = 0, 1\). 
        \end{proposition}
        \begin{proof}
            If \(\delta = 0\) then the conclusion trivially holds. Suppose \(\delta > 0\). Consider 
            \begin{align*}
                &\inf_{\hat{\tau}}\sup_{P \in \mathcal{P}} P\left\{ \phi(\rho(\hat{\tau}(X), \tau(P))) \geq \phi(\delta) \right\} \\
                &\geq \inf_{\hat{\tau}}\sup_{P \in \mathcal{P}} P\left\{ \rho(\hat{\tau}(X), \tau(P)) \geq \delta \right\} \\
                &\geq \inf_{\hat{\tau}} \max\left\{ \int_{P \in \mathcal{P}_0} P\left\{ \rho(\hat{\tau}(X), \tau(P)) \geq \delta \right\} \, \pi_0(dP) , \int_{P \in \mathcal{P}_1} P\left\{\rho(\hat{\tau}(X), \tau(P)) \geq \delta\right\} \, \pi_1(dP) \right\} \\
                &\geq \inf_{\hat{\tau}} \max\left\{ \int_{P \in \mathcal{P}_0} P\left\{ \rho(\hat{\tau}(X), \tau(P)) \geq \delta \right\} \, \pi_0(dP) , \int_{P \in \mathcal{P}_1} P\left\{\rho(\hat{\tau}(X), \tau(P)) \geq \delta\right\} \, \pi_1(dP) \right\} 
            \end{align*}
            where we have used that \(\phi\) is a nondecreasing function to obtain the first inequality. Consider the following test 
            \begin{equation*}
                \psi^*(X) := \argmin_{j \in \{0, 1\}} \left\{\inf_{P \in \mathcal{P}_j} \rho(\hat{\tau}(X), \tau(P))\right\}. 
            \end{equation*}
            Then we have by triangle inequality
            \begin{align*}
                &\inf_{\hat{\tau}} \max\left\{ \int_{P \in \mathcal{P}_0} P\left\{ \rho(\hat{\tau}(X), \tau(P)) \geq \delta \right\} \, \pi_0(dP) , \int_{P \in \mathcal{P}_1} P\left\{\rho(\hat{\tau}(X), \tau(P)) \geq \delta\right\} \, \pi_1(dP) \right\} \\
                &\geq \inf_{\hat{\tau}} \max\left\{ \int_{P \in \mathcal{P}_0} P\left\{\psi^* = 1 \right\} \pi_0(dP), \int_{P \in \mathcal{P}_1} P\left\{\psi^* = 0\right\} \pi_1(dP)\right\} \\
                &\geq \inf_{\psi} \max\left\{ \int_{P \in \mathcal{P}_0} P\left\{\psi = 1 \right\} \pi_0(dP), \int_{P \in \mathcal{P}_1} P\left\{\psi = 0\right\} \pi_1(dP)\right\} \\
                &= 1 - \dTV(P_{\pi_0}, P_{\pi_1})
            \end{align*}
            where in the penultimate line the infimum runs over all tests \(\psi\) and the final line follows from Neyman-Pearson lemma. The proof is complete.
        \end{proof}
    
        \begin{lemma}[Ingster-Suslina method \cite{ingster_nonparametric_2003}]\label{lemma:ingster_suslina}
            Suppose \(\Sigma \in \R^{p \times p}\) is a positive definite matrix and \(\Theta \subset \R^p\) is a parameter space. Let \(P_\theta\) denote the distribution \(N(\theta, \Sigma)\). If \(\pi\) is a probability distribution supported on \(\Theta\), then 
            \begin{equation*}
                \chi^2\left(P_{\pi}||P_0\right) = E\left(\exp\left(\left\langle \theta, \Sigma^{-1}\tilde{\theta}\right\rangle\right)\right) - 1
            \end{equation*}
            where \(\theta, \tilde{\theta} \overset{iid}{\sim} \pi\). Here, \(P_\pi = \int P_\theta \, \pi(d\theta)\) denotes the mixture induced by \(\pi\). 
        \end{lemma}

        \begin{lemma}[\cite{collier_minimax_2017}]\label{lemma:hypergeometric}
            Suppose \(s \leq p\). If \(Y\) is distributed according to the hypergeometric distribution with probability mass function \(P\{Y = k\} = \frac{\binom{s}{k} \binom{p-s}{s-k}}{\binom{p}{s}}\) for \(0 \leq k \leq s\), then \(E(Y) = \frac{s^2}{p}\) and \(E(\exp(\lambda^2 Y)) \leq \left(1 - \frac{s}{p} + \frac{s}{p}e^{\lambda^2}\right)^s\). 
        \end{lemma}
    
        \begin{lemma}[Lemma 1 \cite{laurent_adaptive_2000}]\label{lemma:chisquare_tail}
            For any positive integer \(d\) and \(t > 0\), we have 
            \begin{equation*}
                P\left\{\chi^2_d \geq d + 2\sqrt{dt} + 2t\right\} \leq e^{-t}. 
            \end{equation*}
        \end{lemma}
        
        \begin{lemma}[Lemma 11 \cite{verzelen_minimax_2012}]\label{lemma:chisquare_lower_tail}
            For any positive integer \(d\) and \(t > 0\), we have 
            \begin{equation*}
                P\left\{\chi^2_d \leq e^{-1} dt^{2/d}\right\} \leq t. 
            \end{equation*}
        \end{lemma}
    
        \begin{lemma}[Lemma 8 \cite{collier_optimal_2018}]\label{lemma:adapt_exp_lbound}
            Let \(P\) and \(Q\) be two probability measures on a measurable space \((X, \mathcal{U})\). Then, for any \(q > 0\), 
            \begin{equation*}
                \inf_{\mathcal{A} \in \mathcal{U}} \left\{P(\mathcal{A})q + Q(\mathcal{A}^c) \right\} \geq \sup_{0 < \tau < 1} \left\{ \frac{q\tau}{1 + q\tau} \left(1 - \tau(\chi^2(Q\,||\,P) + 1)\right)\right\}.
            \end{equation*}
        \end{lemma}

        \section{Global maxima of a mixture associated to kernel mode estimation}\label{appendix:mixture}
        In this section, we will establish a result about the population-level mixture arising from the kernel model estimator proposed in Section \ref{section:linear_functional}. Throughout this section, \(\Phi\) and \(\varphi\) denote the cumulative distribution function and probability density function respectively of the standard normal distribution. Let \(\mu, \eta_1,...,\eta_k \in \R\) and \(h > 0\). Define the function \(f : \R \to [0, \infty)\) with 
        \begin{equation}\label{def:f}
            f(x) = \sum_{i=1}^{k} f_i(x)
        \end{equation}
        where 
        \begin{equation}\label{def:fi}
            f_i(x) = (\Phi(x-\mu+h) - \Phi(x-\mu-h)) + (\Phi(x-\eta_i+h) - \Phi(x-\eta_i-h)). 
        \end{equation}
        We will prove the following result. 
        \begin{theorem}\label{thm:mixture}
            There exists a sufficiently large universal constant \(C_b > 0\) such that if \(\min_i |\mu - \eta_i| > C_b h\), then 
            \begin{equation*}
                \max_{|x-\mu| \leq \frac{h}{4}} f(x) = \max_{x \in \R} f(x). 
            \end{equation*}
        \end{theorem}
    
        \subsection{A special case: separation and ordering}\label{section:mixture_special_case}
        We will first prove some intermediate results en route to proving Theorem \ref{thm:mixture}. Define 
        \begin{equation}\label{def:Delta_h}
            \Delta(h) = \argmax_{\delta \in \R} -(\varphi(\delta+h) - \varphi(\delta-h)). 
        \end{equation}
    
        Throughout this section, we will make the following assumption. 
    
        \begin{assumption}\label{assumption:ordering}
            Assume \(|\mu - \eta_1| > 2\Delta(h)\) and \(\mu < \eta_1 \leq \eta_2 \leq \ldots \leq \eta_k\). 
        \end{assumption}
    
        \begin{proposition}\label{prop:fmax_order}
            Suppose Assumption \ref{assumption:ordering} holds. Then 
            \begin{equation*}
                \max_{x \in [\mu, \mu+\Delta(h)]} f(x) = \max_{x \in \R} f(x)
            \end{equation*}
            where \(f\) is given by (\ref{def:f}). 
        \end{proposition}
        \begin{proof}
            Define the functions \(u_i : \left[\frac{\mu+\eta_i}{2}, \infty\right) \to \R\) with \(u_i(v) = \mu+\eta_i -v\). Note this is just the reflection about the point \(\frac{\mu+\eta_i}{2}\). Note in particular that \(f_i(v) = f_i(u_i(v))\) for all \(v \in \left[\frac{\mu+\eta_i}{2}, \infty\right)\) where \(f_i\) is given by (\ref{def:fi}). Notably, we have \(x_i^* = u_i(y_i^*)\) where \(x_i^*\) and \(y_i^*\) are the global maximizers of \(f_i\) (see Lemma \ref{lemma:fi_maximizers}). Further notice that \(u_i(v) \leq u_j(v)\) for all \(i \leq j \) since \(\eta_i \leq \eta_j\). 
    
            We first sketch a broad overview of our proof strategy. First it is clear that \(\max_{v \in \R} f(v) = \max_{v \in [\mu, \eta_k]}f(v)\) so it suffices to focus attention on \(v \in [\mu, \eta_k]\). For every point \(v > \mu + \Delta(h)\), we will find a point \(w(v) \in [\mu, \mu+\Delta(h)]\) such that \(f(v) \leq f(w(v))\). This will prove that \(\max_{x \in [\mu, \mu+\Delta(h)]} f(x) \geq \max_{x > \mu + \Delta(h)}f(x)\). 
    
            Fix \(v \in (\mu+\Delta(h), \eta_k]\). To find a suitable point \(w(v) \in [\mu, \mu+\Delta(h)]\), we break up the analysis into various exhaustive cases. Throughout, let \(x_i^*\) and \(y_i^*\) denote the global maximizers of \(f_i\) given by Lemma \ref{lemma:fi_maximizers}.
            
            Figure \ref{fig:gaussian_mixture} presents a cartoon schematic in the special case \(k = 2\) which illustrates how the cases of our analysis are organized; the dashed lines demarcate the regions considered in the cases. For the \(v\) in hand, we try to find a suitable \(w(v)\) contained in the shaded region (representing the region \([\mu, \mu+\Delta(h)]\)) such that \(f(w(v)) \geq f(v)\). Note \(x_1^*\) and \(x_2^*\) are contained in the shaded region in Figure \ref{fig:gaussian_mixture}. Before going into the rigorous proof, we roughly describe how the cases are organized and analyzed in the special case presented in Figure \ref{fig:gaussian_mixture}.
            
            The analysis of which \(w(v)\) to pick depends on the location of \(v\). As seen in Figure \ref{fig:gaussian_mixture}, Case 1 deals with the case where \(v\) is to the left of the midpoint between \(\mu\) and \(\eta_1\). Intuitively, you can find a suitable \(w(v)\) just by moving left from \(v\); so, we pick \(w(v)= x_1^*\). Case 2 addresses the situation where \(v\) may potentially close to \(y_1^*\), but on the left hand side. Intuitively, to pick \(w(v)\), intuitively one can just reflect \(v\) across the midpoint of \(\mu\) and \(\eta_1\), then move to the left; so, we pick \(w(v) = x_1^*\) again. 
            
            In the case of general \(k\), Case 3 successively deals with the case where \(v\) is contained in \([y_i^*, y_{i+1}^*]\). The analysis proceeds by considering each \(\eta_j\) in turn, and splits into two subcases depending on the location of \(v\). In the special case of \(k = 2\) illustrated in Figure \ref{fig:gaussian_mixture}, Case 3 addresses the case where \(v\) is contained in \([y_1^*, y_2^*]\). Here, we select \(w(v)\) to be \(x_2^* \vee u_1(v)\) (recall \(u_1(v)\) denotes the reflection of \(v\) across the midpoint \(\frac{\mu+\eta_1}{2}\) that is marked by the dashed line between Case 1 and Case 2 in Figure \ref{fig:gaussian_mixture}). To show this choice of \(w(v)\) satisfies \(f(w(v)) \geq f(v)\), the analysis splits into subcases demarcated by the dotted line in Figure \ref{fig:gaussian_mixture}, which is located at the midpoint \(\frac{\mu+\eta_2}{2}\). Suppose \(v\) is contained in the region under Case 3.1. Noting \(f = f_1 + f_2\), let us look at \(f_1\) first. Observe \(f_1\) is increasing on \((-\infty, x_1^*]\). Since \(v \geq y_1^*\), we have \(u_1(v) \leq x_1^*\). Furthermore, since \(x_2^* \leq x_1^*\), we have \(u_1(v) \leq w(v) \leq x_1^*\), and so it follows \(f_1(v) = f_1(u_1(v)) \leq f_1(w(v))\). Looking at \(f_2\), note that \(v \in \left[\frac{\eta+\mu_2}{2}, y_2^*\right]\) in Case 3.1, and so \(u_2(v)\) (the reflection of \(v\) across the dotted line), satisfies \(x_2^* \leq u_2(v)\). Note we also have \(u_1(v) \leq u_2(v)\), and since \(f_2\) is monotone decreasing on \([x_2^*, \frac{\mu+\eta_2}{2}]\), it follows from \(w(v) = x_2^* \vee u_1(v) \leq u_2(v)\) that \(f_2(v) = f_2(u_2(v)) \leq f_2(w(v))\). Hence, \(f(v) \leq f(w(v))\). Case 3.2 is analyzed similarly, but now there is no need to consider the reflection \(u_2(v)\) since \(v\) is already to the left of the dotted line in Case 3.2 as seen in Figure \ref{fig:gaussian_mixture}. 
            
            Finally, Case 4 addresses the case where \(v\) is to the right of the final \(y^*_k\) (which is \(y^*_2\) in Figure \ref{fig:gaussian_mixture}); \(w(v)\) can be selected by moving to the left from \(v\) until \(y^*_k\), which only increases the value of \(f\), then applying the argument in the previous case for \(y_k^*\) to obtain \(w(y_k^*)\) contained in the shaded region; hence, our choice is \(w(v) = w(y_k^*)\). 
    
            \begin{figure}
                \centering
                \includegraphics{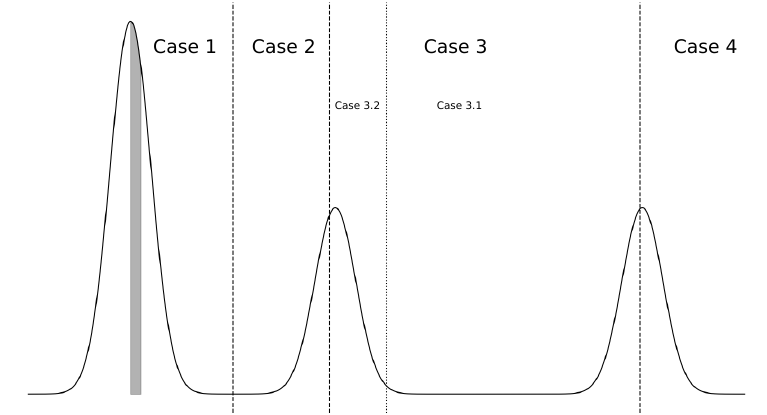}
                \caption{Cartoon schematic of the case organization of the proof of Proposition \ref{prop:fmax_order} in the special case \(k = 2\). The solid line represents the function \(f\). For a given point \(v\), the strategy in the proof is to find a \(w(v)\) in the shaded region such that \(f(v) \leq f(w(v))\). The shaded region represents \([\mu, \mu+\Delta(h)]\), the dashed line between Case 1 and Case 2 is located at \(\frac{\mu+\eta_1}{2}\), the other two dashed lines are located at \(y_1^*\) and \(y_2^*\) respectively, and the dotted line is located at \(\frac{\mu+\eta_2}{2}\). }\label{fig:gaussian_mixture}
            \end{figure}
    
            Having discussed the organization of the cases, we now proceed with the proof. \newline 
    
            \textbf{Case 1:} Suppose \(v < \frac{\mu+\eta_1}{2}\). Let \(w(v) = x_1^*\) and note \(\mu \leq w(v) \leq \mu+\Delta(h)\). Since \(v > \mu + \Delta(h)\), we have \(v > x_1^* \geq x_2^* \geq ... \geq x_k^*\). Since, by Lemma \ref{lemma:fi_monotone} we have \(f_i\) is decreasing on the interval \([x_i^*, \frac{\mu+\eta_i}{2}]\) (which contains \([x_1^*, \frac{\mu+\eta_1}{2}]\) by Lemma \ref{lemma:xi_ordering}) for all \(i\), it immediately follows that \(f_i(v) \leq f_i(x_1^*) \leq f_i(w(v))\) for all \(i\). Hence, we have \(f(v) \leq f(w(v))\) and so the analysis for this subcase is complete. 
    
            \textbf{Case 2:} Suppose \(\frac{\mu+\eta_1}{2} \leq v \leq y_1^*\). Let \(w(v) = x_1^*\) and again note \(\mu \leq w(v) \leq \mu + \Delta(h)\). Clearly we have \(f_1(v) = f_1(u_1(v))\). We now claim \(f_i(v) \leq f_i(w(v))\) for \(i \geq 2\). Recall that \(f_i(v) = f_i(u_i(v))\). Now consider \(u_i(v) \geq u_1(v) \geq x_1^* \geq x_i^*\) where the second inequality follows from \(v \leq y_1^*\) and the final inequality follows from Lemma \ref{lemma:xi_ordering}. By Lemma \ref{lemma:fi_monotone}, we have \(f_i(u_i(v)) \leq f_i(x_1^*)\). Since \(f_i(u_i(v)) = f_i(v)\) and \(f_i(x_1^*) = f_i(w(v))\), we have \(f_i(v) \leq f_i(w(v))\). As this holds for all \(i\), we have \(f(v) \leq f(w(v))\) as desired. The analysis for this subcase is complete. 
            
            \textbf{Case 3:} Suppose \(y_i^* \leq v \leq y_{i+1}^*\) for some \(1 \leq i \leq k-1\). Let \(w(v) = x_{i+1}^* \vee u_i(v)\). Since \(u_i(v) \leq x_i^* \leq \mu+\Delta(h)\) (because \(v \geq y_i^*\)) and since \(\mu \leq x_{i+1}^* \leq \mu+\Delta(h)\) it follows that \(\mu \leq w(v) \leq \mu + \Delta(h)\). With this established, we now work to show \(f(v) \leq f(w(v))\). We first claim that for all \(j \leq i\), we have \(f_j(v) \leq f_j(w(v))\). Fix \(j \leq i\). Since \(v \geq y_i^*\), we have \(v \geq y_j^*\) by Lemma \ref{lemma:xi_ordering}. Since \(v \geq y_j^*\), we have \(v > \frac{\mu+\eta_j}{2}\). Now consider \(u_j(v) \leq u_i(v) \leq x_i^* \leq x_j^*\). The first inequality is since \(j \leq i\). The second inequality is because \(v \geq y_i^*\). The last inequality is due to Lemma \ref{lemma:xi_ordering}. Furthermore, we have \(x_{i+1}^* \leq x_j^*\). Hence, we have \(w(v) = x_{i+1}^* \vee u_i(v) \leq x_j^*\). Since \(u_j(v) \leq w(v) \leq x_j^*\), it follows by Lemma \ref{lemma:fi_monotone} that \(f_j(v) = f_j(u_j(v)) \leq f_j(w(v))\). Thus our first claim is proved. 
    
            We now claim that for any \(j \geq i+1\), we have \(f_j(v) \leq f_j(w(v))\). There are two subcases to consider. 
    
            \textbf{Case 3.1:} Suppose \(v > \frac{\mu+\eta_j}{2}\). Since \(j \geq i+1\), we have \(x_j^* \leq x_{i+1}^*\) by Lemma \ref{lemma:xi_ordering}. Hence we have \(x_j^* \leq w(v)\). Now we claim \(w(v) \leq u_j(v)\). First, observe that \(u_i(v) \leq u_j(v)\) since \(j \geq i+1\). To show that we also have \(x_{i+1}^* \leq u_j(v)\), consider that \(u_j(v) = \mu+\eta_j-v\) and \(x_{i+1}^* = \mu+\eta_{i+1}-y_{i+1}^*\). Hence \(u_j(v) \geq x_{i+1}^*\) if and only if \(\eta_j - v \geq \eta_{i+1} - y_{i+1}^*\). Since \(v \leq y_{i+1}^* \leq \eta_{i+1} \leq \eta_j\), we clearly have \(\eta_j - v \geq \eta_{i+1} - y_{i+1}^*\). Hence, \(x_{i+1}^* \leq u_j(v)\) and so \(w(v) \leq u_j(v)\). To summarize, we have shown \(x_j^* \leq w(v) \leq u_j(v)\). Thus, by Lemma \ref{lemma:xi_ordering} we have \(f_j(v) = f_j(u_j(v)) \leq f_j(w(v))\). Therefore, for this subcase we have shown \(f_j(v) \leq f_j(w(v))\) as desired. 
    
            \textbf{Case 3.2:} Suppose \(v \leq \frac{\mu+\eta_j}{2}\). Since \(j \geq i+1\), we have \(x_j^* \leq x_{i+1}^* \leq x_i^*\). Moreover, since \(y_i^* \leq v\), we have \(u_i(v) \leq x_i^*\). Hence, we have \(x_j^* \leq x_{i+1}^* \vee u_i(v) \leq x_i^*\). Now, since \(v \geq \frac{\mu+\eta_i}{2}\) we have \(v \geq x_i^*\). Thus it follows \(x_j^* \leq x_{i+1}^* \vee u_i(v) \leq v\). Since \(f_j\) is decreasing to the right of \(x_j^*\) by Lemma \ref{lemma:fi_monotone}, it follows that \(f_j(w(v)) = f_j(x_{i+1}^* \vee u_i(v)) \geq f_j(v)\) as desired. Therefore, the analysis for this subcase is complete. 
    
            To summarize our work for Case 3, we have shown that for any \(j \leq i\) we have \(f_j(v) \leq f_j(w(v))\). We have also shown that for any \(j \geq i+1\) we have \(f_j(v) \leq f_j(w(v))\). Therefore, it immediately follows that \(f(v) \leq f(w(v))\), as desired. 
            
            \textbf{Case 4:} Suppose \(y_k^* \leq v\). Then \(f_j(y_k^*) \geq f_j(v)\) for all \(j\). Therefore \(f(y_k^*) \geq f(v)\). By Case 3 applied to the point \(y_k^*\), there exists \(w(y_k^*) \in [\mu, \mu+\Delta(h)]\) such that \(f(w(y_k^*)) \geq f(y_k^*)\). So if we take \(w(v) = w(y_k^*)\), then we have \(f(w(v)) \geq f(y_k^*) \geq f(v)\), as desired. 
    
            With all of the cases handled, the proof is complete.
        \end{proof}
    
        \begin{lemma}\label{lemma:fi_maximizers}
            Suppose Assumption \ref{assumption:ordering} holds. The function \(f_i\) exhibits two global maximizers \(x_i^*\) and \(y_i^* = \eta_i + \mu - x_i^*\) with \(\mu \leq x_i^* \leq \mu + \Delta(h)\) and \(\eta_i - \Delta(h) \leq y_i^* \leq \eta_i\). 
        \end{lemma}
        \begin{proof}
            First, consider that 
            \begin{equation*}
                f_i'(t) = (\varphi(t-\mu+h) - \varphi(t-\mu-h)) + (\varphi(t-\eta_i+h) - \varphi(t-\eta_i-h)). 
            \end{equation*}
            Consequently, \(f_i'(\mu) = \varphi(\mu-\eta_i+h) - \varphi(\mu-\eta_i-h) \geq 0\) since \(\mu < \eta_i\). On the other hand, consider that 
            \begin{align*}
                f_i'(\mu+\Delta(h)) &= (\varphi(\Delta(h) + h) - \varphi(\Delta(h) - h)) + (\varphi(\mu+\Delta(h)-\eta_i+h) - \varphi(\mu+\Delta(h)-\eta_i-h)) \\
                &\leq (\varphi(\Delta(h) + h) - \varphi(\Delta(h) - h)) + \sup_{\delta \in \R} (\varphi(\delta+h) - \varphi(\delta-h)) \\
                &\leq 0
            \end{align*}
            by definition of \(\Delta(h)\). So \(f_i'(\mu) \geq 0\) and \(f_i'(\mu+\Delta(h)) \leq 0\), and thus it follows by intermediate value theorem that there exists \(x_i^* \in [\mu, \mu+\Delta(h)]\) such that \(f_i'(x_i^*) = 0\). Moreover, it can be readily seen \(x_i^*\) is a local maximizer. Furthermore, by Lemma \ref{lemma:Gamma_h}, the map \( t \mapsto \varphi(t-\mu+h) - \varphi(t-\mu-h)\) is strictly decreasing for \(t \in [\mu, \mu+\Delta(h)]\) and the map \(t \mapsto \varphi(t-\eta_i+h) - \varphi(t-\eta_i-h)\) is strictly increasing (since \(|\mu-\eta_i| > 2\Delta(h)\)). Therefore, \(x_i^*\) is the unique local maximizer in the interval \([\mu, \mu+\Delta(h)]\). Now consider that for any \(t \in \R\), we have that \(f_i(t) = f_i(\mu+\eta_i-t)\). Consequently, we also have \(f_i'(t) = -f_i(\mu+\eta_i-t)\) and \(f_i''(t) = f_i''(\mu+\eta_i-t)\). Therefore, letting \(y_i^* = \mu+\eta_i-x_i^*\), we see that \(y_i^*\) is also a local maximizer and \(y_i^* \in [\eta_i-\Delta(h), \eta_i]\). It remains to show there are no local maximizers in the interval \((\mu+\Delta(h), \eta_i-\Delta(h))\) (since it is straightforward to see there are no maximizers to the left of \(\mu\) nor to the right of \(\eta_i\)). By symmetry about the midpoint \(\frac{\mu+\eta_i}{2}\), it suffices to only consider \((\mu+\Delta(h), \frac{\mu+\eta_i}{2}]\). Note that the point \(\frac{\mu+\eta_i}{2}\) is indeed a critical point (i.e. \(f_i'\) is equal to zero at this point), but \(f_i''\) vanishes at this point so it is not a local maximizer. Hence it remains to consider the open interval \((\mu+\Delta(h), \frac{\mu+\eta_i}{2})\). Define the functions \(g_1(u) = \varphi(\Delta(h) + u + h) - \varphi(\Delta(h)+u-h)\) and \(g_2 = \varphi(\mu+\Delta(h)+u-\eta_i+h) - \varphi(\mu+\Delta(h)+u-\eta_i-h)\). Note \(f_i'(\mu+\Delta(h)+u) = g_1(u) + g_2(u)\). By Lemma \ref{lemma:Gamma_h} observe that for \(u > 0\), the function \(g_1\) is strictly increasing in \(u\). Likewise, since \(|\mu-\eta_i| > 2\Delta(h)\), we also have that \(g_2\) is strictly increasing in \(u\). Hence \(f_i'\) is strictly increasing on the interval \((\mu+\Delta(h), \frac{\mu+\eta_i}{2})\). Since \(f_i'(\frac{\mu+\eta_i}{2}) = 0\), it follows \(f_i' < 0\) on our interval of interest, and so there are no local maximizers on this interval. Hence, the only local maximizers are located in \([\mu, \mu+\Delta(h)]\) and \([\eta_i - \Delta(h), \eta_i]\). Since \(f_i(x_i^*) = f_i(y_i^*)\), it follows \(x_i^*\) and \(y_i^*\) are global maximizers, not just local ones. 
        \end{proof}
    
        \begin{lemma}\label{lemma:xi_ordering}
            Suppose Assumption \ref{assumption:ordering} holds. If \(\mu \leq x_i^* \leq \mu+\Delta(h)\) and \(\eta_i - \Delta(h) \leq y_i^* \leq \eta_i\) denote the global maximizers of \(f_i\), then \(x_1^* \geq x_2^* \geq ... \geq x_k^*\) and \(y_1^* \leq y_2^* \leq ... \leq y_k^*\). 
        \end{lemma}
        \begin{proof}
            Note by the proof of Lemma \ref{lemma:fi_maximizers} that \(y_i^* = \mu + \eta_i - x_i^*\). Hence we have \(y_i^* \leq y_{i+1}^*\) if and only if \(\eta_i - \eta_{i+1} \leq x_i^* - x_{i+1}^*\). Since \(\eta_i - \eta_{i+1} \leq 0\), it suffices to show \(x_i^* - x_{i+1}^* \geq 0\) in order to establish the desired ordering for \(\{y_i^*\}_{i=1}^{k}\). In other words, it suffices to show \(x_1^* \geq ... \geq x_k^*\) in order to prove \(y_1^* \leq ... \leq y_k^*\). 
    
            Recall from the proof of Lemma \ref{lemma:fi_maximizers} that \(f_i'(\mu) \geq 0\). Now fix \(\delta \in [0, \Delta(h)]\) and note that we have \(f_i'(\mu+\delta) = (\varphi(\delta+h) - \varphi(\delta-h))+(\varphi(\mu+\delta-\eta+h) - \varphi(\mu+\delta-\eta-h))\). For \(\eta\) such that \(|\eta - \mu| > 2\Delta(h)\), let us define the mapping \(\eta \mapsto \delta^*(\eta)\) to be such that \(\delta^*(\eta) \in [0, \Delta(h)]\) solves 
            \begin{equation*}
                0 = (\varphi(\delta+h)-\varphi(\delta-h))+(\varphi(\mu+\delta-\eta+h)-\varphi(\mu+\delta-\eta-h)). 
            \end{equation*}
            Since \(|\eta_i - \mu| > 2\Delta(h)\) and \(\delta \leq \Delta(h)\), it follows from Lemma \ref{lemma:Gamma_h} that the map \(\eta_i \mapsto \varphi(\mu+\delta-\eta_i+h) - \varphi(\mu+\delta - \eta_i-h)\) is monotone decreasing in \(\eta_i\) since \(\eta_i-\mu-\delta > \Delta(h)\). Consequently, \(\delta^*(\eta_i)\) is monotone decreasing in \(\eta_i\). Since \(x_i^* = \mu + \delta^*(\eta_i)\), it follows from \(\eta_1 \leq ... \leq \eta_k\) that \(x_1^* \geq ... \geq x_k^*\). The proof is complete. 
        \end{proof}
    
        \begin{lemma}\label{lemma:fi_monotone}
            Suppose Assumption \ref{assumption:ordering} holds. The function \(f_i\) is decreasing on the interval \(\left[x_i^*, \frac{\mu+\eta_i}{2}\right]\) and increasing on the interval \((-\infty, x_i^*]\) where \(\mu \leq x_i^* \leq \mu+\Delta(h)\) is the global maximizer of \(f_i\). 
        \end{lemma}
        \begin{proof}
            For sake of contradiction, suppose there exists some \(t^* \in [x_i^*, \frac{\mu+\eta_i}{2}]\) such that \(f_i'(t^*) > 0\). Recall from the proof of Lemma \ref{lemma:Gamma_h} that \(f_i'(t) = \Gamma_h(t-\mu) + \Gamma_h(t-\eta_i)\) where \(\Gamma_h\) is the function defined in the statement of Lemma \ref{lemma:Gamma_h}. Consider that \(f_i'(x_i^*) = 0\) and \(f_i'\left(\frac{\mu+\eta_i}{2}\right) = 0\). Since \(f_i'\) is continuous and differentiable, it follows by intermediate value theorem that there exists \(z^* \in \left[t^*, \frac{\mu+\eta_i}{2}\right]\) such that \(f_i'(z^*) = 0\) and \(f_i''(z^*) < 0\). Hence, \(z^*\) is a local maximizer of \(f_i\) in addition to \(x_i^*\) and \(y_i^* = \eta_i + \mu - x_i^*\), which is a contradiction. Hence, \(f_i \leq 0\) on the interval \(\left[x_i^*, \frac{\mu+\eta_i}{2}\right]\), yielding the first claim. The second claim can be shown by a similar argument. 
        \end{proof}
        
        \begin{lemma}\label{lemma:Gamma_h}
            Define \(\Gamma_h : \R \to \R\) to be the function \(\Gamma_h(x) = \varphi(x+h) - \varphi(x-h)\). Then \(\Gamma_h\) attains its global minimum and maximum at \(\Delta(h)\) and \(-\Delta(h)\) respectively, where \(\Delta(h)\) is given by (\ref{def:Delta_h}). Respectively, these are the positive and negative solutions of the equation \(x = \frac{h}{\tanh(xh)}\). Additionally, \(\Delta(h) \geq h\) and \(\lim_{h \to \infty} \frac{\Delta(h)}{h} = 1\). Furthermore, \(\Gamma_h\) is strictly increasing on the intervals \((-\infty, -\Delta(h))\) and \((\Delta(h), \infty)\), and \(\Gamma_h\) is strictly decreasing on the interval \((-\Delta(h), \Delta(h))\). 
        \end{lemma}
        \begin{proof}
            By definition of \(\Delta(h)\) and the fact \(\Gamma_h\) is an odd function, it is clear that \(\Delta(h)\) and \(-\Delta(h)\) are the global minimum and maximum respectively of \(\Gamma_h\). Moreover, consider that \(\Gamma_h(x) < 0\) for positive \(x\) and \(\Gamma_h(x) > 0\) for negative \(x\). Hence, it follows that \(\Delta(h) > 0\). It remains to show \(\Delta(h)\) and \(-\Delta(h)\) are the solutions of the claimed equation. Consider that 
            \begin{equation*}
                \Gamma_h'(x) = -x\Gamma_h(x) - h(\varphi(x+h) + \varphi(x-h)).
            \end{equation*}
            Therefore, \(\Gamma_h'(x) = 0\) if and only if \(-x\Gamma_h(x) = h(\varphi(x+h) + \varphi(x-h))\). In particular, this holds if and only if 
            \begin{equation*}
                -x = h \cdot \frac{\varphi(x+h) + \varphi(x-h)}{\varphi(x+h) - \varphi(x-h)} = h \cdot \frac{e^{-xh} + e^{xh}}{e^{-xh}-e^{xh}} = -h\cdot \frac{1}{\tanh(xh)}. 
            \end{equation*}
            Hence, \(\Gamma_h'(x) = 0\) if and only if \(x = \frac{h}{\tanh(xh)}\). Hence, \(\Delta(h)\) and \(-\Delta(h)\) are the positive and negative solutions of \(x = \frac{h}{\tanh(xh)}\). Further, it follows \(\Gamma_h' > 0\) on \((-\infty, \Delta(h))\) and \((\Delta(h), \infty)\). Likewise, \(\Gamma_h' < 0\) on \((-\Delta(h), \Delta(h))\). 
    
            To show \(\Delta(h) \geq h\), first consider that \(\Delta(h)\tanh(h\Delta(h)) = h\). Since \(|\tanh(y)| \leq 1\) for all \(y \in \R\), we have immediately \(\Delta(h) \geq h\). To show that \(\lim_{h \to \infty} \frac{\Delta(h)}{h} = 1\), observe that \(\lim_{h \to \infty} \Delta(h)\tanh(h\Delta(h)) = \lim_{h \to \infty} h = \infty\). Therefore, we also have \(\lim_{h \to \infty} \Delta(h) = \infty\) since \(\Delta(h) \geq h\). Hence, \(\lim_{h \to \infty} \tanh(h\Delta(h)) = 1\) and so \(\lim_{h \to \infty} \frac{\Delta(h)}{h} = \lim_{h \to \infty} \frac{1}{\tanh(h\Delta(h))} = 1\). The proof is complete. 
        \end{proof}
        
        \subsubsection{Relaxing the ordering assumption}
    
        We now relax Assumption \ref{assumption:ordering}.
    
        \begin{proposition}\label{prop:fmax_sep}
            If \(\min_i |\mu - \eta_i| \geq 2\Delta(h)\), then 
            \begin{equation*}
                \max_{|x-\mu| \leq \Delta(h)} f(x) = \max_{x \in \R} f(x)
            \end{equation*}
            where \(f\) is given by (\ref{def:f}).
        \end{proposition}
        \begin{proof}
            The proof simply involves invoking Proposition \ref{prop:fmax_order} twice. For every point \(v\) with \(|v-\mu| > \Delta(h)\), we will find a point \(w\) with \(|w - \mu| \leq \Delta(h)\) such that \(f(v) \leq f(w)\). Fix such a point \(v\). Write \(f = f_L + f_R\) where 
            \begin{align*}
                f_L &:= \sum_{i : \eta_i < \mu} f_i, \\
                f_R &:= \sum_{i : \eta_i > \mu} f_i. 
            \end{align*}
            We consider two cases. Suppose \(v > \mu+\Delta(h)\). By Proposition \ref{prop:fmax_order}, there exists \(w \in [\mu, \mu+\Delta(h)]\) such that \(f_R(v) \leq f_R(w)\). Now, consider that \(f_L\) is monotonically decreasing on \([\mu, \infty)\). Consequently, we also have \(f_L(v) \leq f_L(w)\). Therefore, \(f(v) \leq f(w)\). The case \(v < \mu - \Delta(h)\) is handled by a very similar argument. Hence, we have proved the desired result. 
        \end{proof}
    
        \subsection{Proof of Theorem \ref{thm:mixture}}
        We are now in position to prove Theorem \ref{thm:mixture}. In order to do so, analogues of lemmas stated in Section \ref{section:mixture_special_case} are needed.
    
        \begin{lemma}\label{lemma:fi_maximizers_Cb}
            Assume \(\mu < \eta_1 \leq \eta_2 \leq ... \leq \eta_k\). There exists a sufficiently large universal constant \(C_b > 0\) such that if \(\min_i |\mu - \eta_i| \geq C_b h\), then \(f_i\) exhibits global maximizers \(x_i^* \in \left[\mu, \mu+\frac{h}{4}\right]\) and \(y_i^* \in \left[\eta_i - \frac{h}{4}, \eta_i\right]\). 
        \end{lemma}
        \begin{proof}
            We will select \(C_b > 2 \cdot \sup_{r > 0} \frac{\Delta(r)}{r}\) suitably large. Note \(\sup_{r > 0} \frac{\Delta(r)}{r} \lesssim 1\) by Lemma \ref{lemma:Gamma_h} and so \(C_b\) can be indeed taken to be a universal positive constant. The proof is essentially the same as the proof of Lemma \ref{lemma:fi_maximizers}. 
    
            Recall that \(f_i'(t) = (\varphi(t - \mu + h) - \varphi(t - \mu - h)) + (\varphi(t-\eta_i+h) - \varphi(t-\eta_i-h))\). As argued in the proof of Lemma \ref{lemma:fi_maximizers} we have \(f_i'(\mu) \geq 0\). On the other hand, consider 
            \begin{align*}
                f_i'\left(\mu + \frac{h}{4}\right) &= \left(\varphi\left(\frac{h}{4} + h\right) - \varphi\left(\frac{h}{4} - h\right)\right) + \left( \varphi\left(\mu - \eta_i + \frac{h}{4} + h \right) - \varphi\left(\mu - \eta_i + \frac{h}{4} - h\right)\right) \\
                &\leq \left(\varphi\left(\frac{h}{4} + h\right) - \varphi\left(\frac{h}{4} - h\right)\right) + \left( \varphi\left( \left(\frac{1}{4} - C_b\right)h + h \right) - \varphi\left(\left(\frac{1}{4} - C_b\right)h- h\right)\right).
            \end{align*} 
            By continuity of \(\varphi\), we can clearly take \(C_b\) to be a sufficiently large universal constant such that \(f_i'\left(\mu + \frac{h}{4}\right) \leq 0\). Thus, by intermediate value theorem, there exists \(x_i^* \in \left[\mu, \mu+\frac{h}{4}\right]\) such that \(f_i'(x_i^*) = 0\). The argument in the proof of Lemma \ref{lemma:fi_maximizers} can be essentially repeated to establish \(y_i^* = \mu+\eta_i-x_i^*\) and \(x_i^*\) are the global maximizers of \(f_i\). We omit details for brevity. 
        \end{proof}
    
        \begin{lemma}\label{lemma:xi_ordering_Cb}
            Assume \(\mu < \eta_1 \leq \eta_2 \leq ... \leq \eta_k\). There exists a large universal constant \(C_b > 0\) such that if \(\min_i |\mu - \eta_i| \geq C_b h\), then the global maximizers \(x_i^* \in \left[\mu, \mu + \frac{h}{4}\right]\) and \(y_i^* \in \left[\eta_i - \frac{h}{4}, \eta_i\right]\) of \(f_i\) satisfy \(x_1^* \geq ... \geq x_k^*\) and \(y_1^* \leq ... \leq y_k^*\). 
        \end{lemma}
        \begin{proof}
            The argument is exactly the same as the proof of Lemma \ref{lemma:xi_ordering} when we pick \(C_b\) as in Lemma \ref{lemma:fi_maximizers_Cb}. 
        \end{proof}
    
        \begin{proof}[Proof of Theorem \ref{thm:mixture}]
            Picking \(C_b\) as in Lemma \ref{lemma:fi_maximizers_Cb}, the proof of Proposition \ref{prop:fmax_sep} (using Lemmas \ref{lemma:fi_maximizers_Cb} and \ref{lemma:xi_ordering_Cb} instead of the analogues in Section \ref{section:mixture_special_case}) can be repeated to conclude the desired result.
        \end{proof}

    \section{Empirical process theory}
    This section contains the empirical process results used in our analysis. Specifically, Theorem \ref{thm:peel} is the main tool we use when investigating the kernel mode estimator. We primarily draw from the book of Boucheron, Lugosi, and Massart \cite{boucheron_concentration_2013}, and so we adopt their setting in what follows. Namely, \(\mathcal{T}\) will denote a countable index set for an associated empirical process. As noted by the authors of \cite{boucheron_concentration_2013}, they state results for countable \(\mathcal{T}\) to avoid fine technical points regarding measurability. However, well-known standard arguments can extend to our setting where \(\mathcal{T}\) is a subset of \(\R\). Consequently, though the following empirical process theory results are stated for countable \(\mathcal{T}\), we will invoke them for our purposes without comment.  
    
    \subsection{Standard tools}
    This section collects standard results, mainly from \cite{boucheron_concentration_2013}. 
    
    \begin{definition}
        For a pseudometric space \((\mathcal{T}, d)\) and for \(\varepsilon > 0\), let \(\mathcal{D}(\mathcal{T}, d, \varepsilon)\) denote the packing number of \(\mathcal{T}\) at level \(\varepsilon\) with respect to \(d\). 
    \end{definition}
    
    \begin{theorem}[Dudley's entropy integral - Corollary 13.2 \cite{boucheron_concentration_2013}]\label{thm:dudley}
        Let \((\mathcal{T}, d)\) be a separable pseudometric space and let \(\left\{X_t\right\}_{t \in \mathcal{T}}\) be a collection of random variables such that 
        \begin{equation*}
            \log E\left(e^{\lambda(X_t - X_{t'})}\right) \leq \frac{\lambda^2 d^2(t, t')}{2}
        \end{equation*}
        for all \(t, t' \in \mathcal{T}\) and all \(\lambda > 0\). Then for any \(t_0 \in \mathcal{T}\), 
        \begin{equation*}
            E\left(\sup_{t \in \mathcal{T}} X_t - X_{t_0}\right) \leq 12 \int_{0}^{\frac{\delta}{2}} \sqrt{\log \mathcal{D}(\mathcal{T}, d, \varepsilon)} d\varepsilon
        \end{equation*}
        where \(\delta = \sup_{t \in \mathcal{T}} d(t, t_0)\). 
    \end{theorem}
    
    \begin{definition}[Uniform entropy]\label{def:uniform_entropy}
        Let \(\mathcal{A} = \{A_t\}_{t \in \mathcal{T}}\) denote a collection of measurable subsets of \(\R\). For \(\delta > 0\) and a probability measure \(Q\) on \(\R\), let \(\mathcal{D}(\delta, \mathcal{A}, Q)\) denote the maximum cardinality of \(N\) of a subset \(\{t_1,...,t_N\}\) such that \(Q(A_{t_i} \Delta A_{t_j}) > \delta^2 \) for every \(i \neq j\). The uniform \(\delta\)-metric entropy of \(\mathcal{A}\) is defined as 
        \begin{equation*}
            H(\delta, \mathcal{A}) := \sup_{Q} \log \mathcal{D}(\delta, \mathcal{A}, Q)
        \end{equation*}
        where the supremum is taken over all probability measures \(Q\) supported on some finite subset of \(\R\). 
    \end{definition}
    
    \begin{definition}[VC dimension \cite{boucheron_concentration_2013}]
        Let \(\mathcal{A}\) denote a collection of subsets of \(\mathcal{X}\). For \(x = (x_1,...,x_n) \in \mathcal{X}^n\), denote the trace of \(\mathcal{A}\) on \(x\) as 
        \begin{equation*}
            \Tr(x) := \{A \cap \{x_1,...,x_n\} : A \in \mathcal{A}\}. 
        \end{equation*}
        Let \(D(x)\) denote the cardinality of \(k\) of the largest subset \(\{x_{i_1},...,x_{i_k}\}\) for which \(2^k = |\Tr((x_{i_1},...,x_{i_k}))|\). The VC dimension of \(\mathcal{A}\) is given by \(\sup_{n \geq 1} \sup_{x \in \mathcal{X}^n} D(x)\). Furthermore, \(\mathcal{A}\) is said to be a VC class if it has a finite VC dimension. 
    \end{definition}
    
    \noindent The following proposition is useful. Note it is essentially Lemma 13.5 in \cite{boucheron_concentration_2013}, but with slight modification to suit our needs. 
    \begin{proposition}\label{prop:empirical_process}
        Let \(\mathcal{A} = \{A_t\}_{t \in \mathcal{T}}\) denote a collection of measurable subsets of \(\R\). Let \(X_1,...,X_n\) be independent random variables. Define 
        \begin{equation*}
            \rho^2 := \sup_{t \in \mathcal{T}} \frac{1}{n} \sum_{i=1}^{n} P\left\{X_i \in A_t\right\} 
        \end{equation*}
        and assume \(\rho > 0\). Let 
        \begin{equation*}
            Z := \sup_{t \in \mathcal{T}} \frac{1}{\sqrt{n}} \sum_{i=1}^{n}\left(\mathbbm{1}_{\{X_i \in A_t\}} - P\{X_i \in A_t\}\right) 
        \end{equation*}
        and denote 
        \begin{equation*}
            D_\rho = \int_{0}^{1/2} \sqrt{H(\rho \varepsilon, \mathcal{A})} \, d\varepsilon.
        \end{equation*}
        Then 
        \begin{equation*}
            E(Z) \leq \frac{576D_\rho^2}{\sqrt{n}} \left(1 + \sqrt{\frac{\rho^2 n}{192 D_\rho^2}} \right).
        \end{equation*}
        Moreover, the same upper bound holds for \(Z^{-} := \sup_{t \in \mathcal{T}} \frac{1}{\sqrt{n}} \sum_{i=1}^{n} \left(P\{X_i \in A_t\} - \mathbbm{1}_{\{X_i \in A_t\}}\right)\). 
    \end{proposition}
    \begin{proof}
        The proof is very similar to that of Lemma 13.5 in \cite{boucheron_concentration_2013} with some slight modifications. Drawing \(r_1,...,r_n \overset{iid}{\sim} \Rademacher\left(\frac{1}{2}\right)\) independently of \(\{X_i\}_{i=1}^{n}\), consider the process \(\{\Xi_t\}_{t \in \mathcal{T}}\) given by 
        \begin{equation*}
            \Xi_t := \frac{1}{\sqrt{n}} \sum_{i=1}^{n} r_i \mathbbm{1}_{\{X_i \in A_t\}}.
        \end{equation*}
        Note that \(E(Z) \leq 2E(\sup_{t \in \mathcal{T}} \Xi_t)\) by Lemma \ref{lemma:symmetrization}. Lemma \ref{lemma:symmetrization} also gives 
        \begin{equation}\label{eqn:subgaussian_increment}
            E\left(\exp\left(\lambda(\Xi_t - \Xi_{t'})\right)\,|\, \{X_i\}_{i=1}^{n}\right) \leq \exp\left(\frac{\lambda^2 d^2(t, t')}{2}\right)
        \end{equation}
        for any \(t, t' \in \mathcal{T}\) and \(\lambda > 0\), where \(d(t, t') = \sqrt{\frac{1}{n} \sum_{i=1}^{n} \left(\mathbbm{1}_{\{X_i \in A_t\}} - \mathbbm{1}_{\{X_i \in A_{t'}\}}\right)^2}\). Further consider 
        \begin{equation*}
            \sup_{t,t' \in \mathcal{T}} d^2(t, t') = \sup_{t, t' \in \mathcal{T}} \frac{1}{n} \sum_{i=1}^{n} \mathbbm{1}_{\{X_i \in A_t \Delta A_{t'}\}} \leq \sup_{t \in \mathcal{T}} \frac{2}{n} \sum_{i=1}^{n} \mathbbm{1}_{\{X_i \in A_t\}}.
        \end{equation*}
        Now define
        \begin{equation*}
            \delta_n^2 := \rho^2 \vee \left(\sup_{t \in \mathcal{T}} \frac{2}{n} \sum_{i=1}^{n} \mathbbm{1}_{\{X_i \in A_t\}}\right).
        \end{equation*}
        Note \(\delta_n\) is a random variable. Conditional on \(\{X_i\}_{i=1}^{n}\), the process \(\{\Xi_t\}_{t\in \mathcal{T}}\) is a sub-Gaussian process (i.e. satisfies (\ref{eqn:subgaussian_increment})), \(\{X_1,...,X_n\}\) is a finite point set, and \(\Xi_{t_0}\) is mean zero for any \(t_0 \in \mathcal{T}\). Letting \(H(\varepsilon, \mathcal{A})\) denote the uniform \(\varepsilon\)-entropy of \(\mathcal{A}\) (see Definition \ref{def:uniform_entropy}), we have by Theorem \ref{thm:dudley} and the fact that \(\delta_n \geq \sup_{t, t' \in \mathcal{T}} d(t, t')\), 
        \begin{align*}
            E\left(\sup_{t \in \mathcal{T}} \Xi_t \,|\, \{X_i\}_{i=1}^{n}\right) &\leq 12 \int_{0}^{\delta_n/2} \sqrt{H(\varepsilon, \mathcal{A})} \, d\varepsilon, \\
            &\leq 12 \delta_n \int_{0}^{1/2} \sqrt{H(\delta_n u, \mathcal{A})} \, du, \\
            &\leq 12 \delta_n \int_{0}^{1/2} \sqrt{H(\rho u, \mathcal{A})} \, du, \\
            &= 12 \delta_n D_\rho
        \end{align*}
        where we have used that \(H(\delta, \mathcal{A})\) is a nonincreasing function of \(\delta\). Consider that \(\delta_n^2 \leq 3\rho^2 + 2Z/\sqrt{n}\) due to the definition of \(\rho^2\). To see this, consider that 
        \begin{align*}
            3\rho^2 + \frac{2}{\sqrt{n}} Z &= 3\rho^2 + \sup_{t \in \mathcal{T}} \left\{ \left(\frac{2}{n} \sum_{i=1}^{n} \mathbbm{1}_{\{X_i \in A_t\}} \right) - \left(\frac{2}{n} \sum_{i=1}^{n} P\{X_i \in A_t\} \right)\right\} \\
            &\geq 3\rho^2 + \sup_{t \in \mathcal{T}} \left\{ \left(\frac{2}{n}\sum_{i=1}^{n} \mathbbm{1}_{\{X_i \in A_t\}}\right) - 2\rho^2 \right\} \\
            &= 3\rho^2 + \left( \sup_{t \in \mathcal{T}} \frac{2}{n} \sum_{i=1}^{n} \mathbbm{1}_{\{X_i \in A_t\}} \right) - 2\rho^2 \\
            &= \rho^2 + \left(\sup_{t \in \mathcal{T}} \frac{2}{n} \sum_{i=1}^{n} \mathbbm{1}_{\{X_i \in A_t\}}\right) \\
            &\geq \delta_n^2.  
        \end{align*}
        Therefore, 
        \begin{align*}
            E(Z) &\leq 2E\left(\sup_{t \in \mathcal{T}} \Xi_t\right) \\
            &= 2E\left(E\left(\sup_{t \in \mathcal{T}} \Xi_t\,|\, \{X_i\}_{i=1}^{n}\right)\right) \\
            &\leq 24D_\rho E(\delta_n) \\
            &\leq 24D_\rho \sqrt{E(\delta_n^2)} \\
            &\leq 24D_\rho \sqrt{3\rho^2 + 2E(Z)/\sqrt{n}}. 
        \end{align*}
        Solving for \(E(Z)\) gives 
        \begin{equation*}
            E(Z) \leq \frac{576D_\rho^2}{\sqrt{n}} \left(1 + \sqrt{1 + \frac{\rho^2 n}{192 D_\rho^2}}\right). 
        \end{equation*}
        An upper bound on \(E(Z^{-})\) can be deduced by arguing just as in the proof of Lemma 13.5 in \cite{boucheron_concentration_2013}.
    \end{proof}
    
    \begin{corollary}\label{corollary:zeta_small_sigma}
        Consider the setup of Proposition \ref{prop:empirical_process}. Assume that \(\mathcal{A}\) is a VC class with VC dimension \(V\). If \(L > 0\) is a universal constant and \(\rho^2 \leq L \frac{D_\rho^2}{n}\), then 
        \begin{equation*}
            \max(E(Z), E(Z^-)) \lesssim \frac{D_\rho^2}{\sqrt{n}} \lesssim \frac{V \log\left(\frac{e}{\rho}\right)}{\sqrt{n}}.
        \end{equation*}
    \end{corollary}
    \begin{proof}
        The first inequality follows from Proposition \ref{prop:empirical_process}. The second inequality follows from the bound on \(D_\rho\) implied by Corollary \ref{corollary:J_sigma_bound}.
    \end{proof}
    
    \begin{corollary}\label{corollary:zeta_big_sigma}
        Consider the setup of Proposition \ref{prop:empirical_process}. Assume that \(\mathcal{A}\) is a VC class with VC dimension \(V\). Then 
        \begin{equation*}
            \max(E(Z), E(Z^-)) \lesssim \rho \sqrt{V\log\left(\frac{e}{\rho}\right)}
        \end{equation*}
        provided that \(\rho^2 \geq L \frac{D_\rho^2}{n}\) for a sufficiently large universal constant \(L > 0\). 
    \end{corollary}
    \begin{proof}
        By Proposition \ref{prop:empirical_process}, we have the bound 
        \begin{equation*}
            E(Z) \leq (576 \cdot \sqrt{192}) \cdot D_\rho \tau \left(1 + \sqrt{1 + \frac{\rho^2}{\tau^2}}\right)
        \end{equation*}
        where \(\tau = \frac{D_\rho}{\sqrt{192 n}}\). Note \(\tau \leq \rho\) when \(L\) is a suitably large universal constant. Since the map \(t \mapsto t\left(1 + \sqrt{1 + \frac{\rho^2}{t^2}}\right)\) is increasing for \(0 < t \leq \rho\), it immediately follows that \(E(Z) \lesssim D_\rho \rho\). Applying Corollary \ref{corollary:J_sigma_bound} yields 
        \begin{equation*}
            E(Z) \lesssim \rho \sqrt{V \log\left(\frac{e}{\rho}\right)}.
        \end{equation*}
        The same argument applies to \(E(Z^-)\) and so the proof is complete. 
    \end{proof}

    \begin{lemma}[Haussler's VC Bound - Lemma 13.6 \cite{boucheron_concentration_2013}]\label{lemma:Haussler}
        Let \(\mathcal{A}\) denote a VC class of subsets of \(\R\) with VC dimension \(V\). For every positive \(\delta > 0\), 
        \begin{equation*}
            H(\delta, \mathcal{A}) \leq 2V \log\left(\frac{e}{\delta}\right) + \log(e(V+1)) \leq 2V \log\left(\frac{e^2}{\delta}\right). 
        \end{equation*}
    \end{lemma}
    
    \begin{corollary}\label{corollary:J_sigma_bound}
        Suppose \(\mathcal{A}\) is a VC class of subsets of \(\R\) with VC dimension \(V\). Suppose \(0 \leq \rho \leq 1\). Let 
        \begin{equation*}
            J_\rho = \int_{0}^{1} \sqrt{H(\rho \varepsilon, \mathcal{A})} \, d\varepsilon. 
        \end{equation*}
        Then 
        \begin{equation*}
            J_\rho \leq \sqrt{6V\log\left(\frac{e}{\rho}\right)}. 
        \end{equation*}
        \begin{proof}
            By Lemma \ref{lemma:Haussler} and Jensen's inequality, 
            \begin{equation*}
                J_\rho \leq \int_{0}^{1} \sqrt{2V \log\left(\frac{e^2}{\rho \varepsilon}\right)} \, d\varepsilon \leq \sqrt{2V + 2V \log\left(\frac{e^2}{\rho}\right)} = \sqrt{4V + 2V\log\left(\frac{e}{\rho}\right)} \leq \sqrt{6V\log\left(\frac{e}{\rho}\right)}
            \end{equation*}
            as desired. 
        \end{proof}
    \end{corollary}
    
    \subsection{Peeling}
    This section contains the main empirical process results (based on the technique commonly called \textit{peeling} in the literature) used in the analysis of the kernel mode estimator. The following definition of a sublinear function and the subsequent lemma are taken from Section 13.7 of \cite{boucheron_concentration_2013}. 
    \begin{definition}[Sublinear function]
        A function \(\psi : [0, \infty) \to [0, \infty)\) is said to be \textbf{sublinear} if it is nondecreasing, continuous, \(\psi(x)/x\) is nonincreasing, and \(\psi(1) \geq 1\).
    \end{definition}
    \begin{lemma}\label{lemma:sublinear_subadditive}
        Suppose \(\psi\) is a sublinear function. Then \(\psi\) is subadditive, that is, \(\psi(x + y) \leq \psi(x) + \psi(y)\). 
    \end{lemma}
    
    \noindent The following lemma is essentially Lemma 13.18 from \cite{boucheron_concentration_2013}, with some slight modification and specialization for our purposes. 
    \begin{lemma}\label{lemma:peeling}
        Let \(\mathcal{A} = \left\{A_t\right\}_{t \in \mathcal{T}}\) denote a collection of measurable subsets of \(\R\). Let \(X_1,...,X_n\) be independent random variables. Define the function \(L : \mathcal{T} \to [0, \infty)\) with  
        \begin{equation*}
            L(t) = \frac{1}{n} \sum_{i=1}^{n} P\left\{X_i \in A_t\right\}. 
        \end{equation*}
        Let \(Z_t = \frac{1}{\sqrt{n}} \sum_{i = 1}^{n} \left(\mathbbm{1}_{\{X_i \in A_t\}} - P\{X_i \in A_t\}\right)\) for \(t \in \mathcal{T}\). Assume there exists a sublinear function \(\psi\) and \(r_* > 0\) such that for all \(r \geq r_*\), 
        \begin{equation*}
            E\left( \sup_{\substack{t \in \mathcal{T}, \\ L(t) \leq r^2}} |Z_t| \right) \leq \psi(r). 
        \end{equation*}
        Then, for all \(r \geq r_*\), 
        \begin{equation*}
            E\left(\sup_{t \in \mathcal{T}} \frac{r^2}{r^2 + L(t)} |Z_t|\right) \leq 4\psi(r). 
        \end{equation*}
    \end{lemma}
    \begin{proof}
        Let \(r \geq r_*\), define \(T_0 := \left\{ t \in \mathcal{T} : L(t) \leq r^2\right\}\) and \(T_k = \left\{t \in \mathcal{T} : r^2 2^{2(k-1)} < L(t) \leq r^2 2^{2k}\right\}\) for \(k \geq 1\). Then,
        \begin{align*}
            E\left( \sup_{t \in \mathcal{T}} \frac{r^2}{r^2 + L(t)} |Z_t| \right) &= E\left(\max_{k \geq 0} \sup_{t \in T_k} \frac{r^2}{r^2+L(t)}|Z_t|\right) \\
            &\leq \sum_{k=0}^{\infty} E\left(\sup_{\substack{t \in T_k}}  \frac{r^2}{r^2 + L(t)} |Z_t|\right) \\
            &\leq \psi(r) + \sum_{k=1}^{\infty} \frac{r^2}{r^2 + r^22^{2(k-1)}} E\left(\sup_{t \in T_k} |Z_t|\right) \\
            &\leq \psi(r) + \sum_{k=1}^{\infty} \frac{1}{1 + 2^{2(k-1)}} \psi(2^k r) \\
            &\leq \psi(r) + \sum_{k=1}^{\infty} \frac{2^k}{1 + 2^{2(k-1)}} \psi(r) \\
            &\leq 2\left(1 + \sum_{k=0}^{\infty} 2^{-k}\right) \psi(r) \\
            &= 4\psi(r). 
        \end{align*}
        where we have used that \(\psi\) is sublinear and so \(\psi(2^{k} r) \leq 2^k \psi(r)\) by Lemma \ref{lemma:sublinear_subadditive}.
    \end{proof}
    
    \noindent The following proposition is essentially Theorem 13.19 from \cite{boucheron_concentration_2013}, with some slight modification to suit our purposes. 
    \begin{proposition}\label{prop:uniform_peel}
        Let \(\mathcal{A} = \left\{A_t\right\}_{t \in \mathcal{T}}\) denote a collection of measurable subsets of \(\R\). Let \(X_1,...,X_n\) be independent random variables. Define the function \(L : \mathcal{T} \to [0, \infty)\) with  
        \begin{equation*}
            L(t) = \frac{1}{n} \sum_{i=1}^{n} P\left\{X_i \in A_t\right\}. 
        \end{equation*}
        Let \(Z_t = \frac{1}{\sqrt{n}} \sum_{i = 1}^{n} \left(\mathbbm{1}_{\{X_i \in A_t\}} - P\{X_i \in A_t\}\right)\) for \(t \in \mathcal{T}\). Assume there exists a sublinear function \(\psi\) and \(r_* > 0\) such that for all \(r \geq r_*\), 
        \begin{equation*}
            E\left( \sup_{\substack{t \in \mathcal{T}, \\ L(t) \leq r^2}} |Z_t| \right) \leq \psi(r). 
        \end{equation*}
        Then for \(r \geq r_*\) and \(u \geq 0\), we have with probability \(1 - 2\exp\left( - c \min\left( \frac{n^2u^2}{\sqrt{n}\psi(r) + nr^2}, n u \right)\right)\) 
        \begin{equation*}
            \frac{|Z_t|}{\sqrt{n}} \leq \frac{r^2 + L(t)}{r^2} \left(\frac{4\psi(r)}{\sqrt{n}} + u\right)
        \end{equation*}
        uniformly over all \(t \in \mathcal{T}\). 
    \end{proposition}
    \begin{proof}
        Let \(r \geq r_*\) and consider by Lemma \ref{lemma:peeling}
        \begin{equation}\label{eqn:peel_expectation}
            E\left(\sup_{t \in \mathcal{T}} \frac{r^2}{r^2 + L(t)} \frac{1}{\sqrt{n}}Z_t\right) \leq \frac{4\psi(r)}{\sqrt{n}}. 
        \end{equation}
        Further consider since \(\frac{r^2}{r^2 + L(t)} \leq 1\), we have 
        \begin{align}
            \sup_{t \in \mathcal{T}} \Var\left( \frac{r^2}{r^2 + L(t)} \sqrt{n} Z_t\right) &= \sup_{t \in \mathcal{T}}\left(\frac{r^2}{r^2+L(t)}\right)^2\Var\left(\sqrt{n}Z_t\right) \nonumber \\
            &\leq \sup_{t \in \mathcal{T}} \frac{r^2}{r^2+L(t)}\Var\left(\sqrt{n}Z_t\right) \nonumber \\
            &= \sup_{t \in \mathcal{T}} \frac{r^2}{r^2 + L(t)} \sum_{i=1}^{n} \Var\left(\mathbbm{1}_{\{X_i \in A_t\}}\right) \nonumber \\ 
            &\leq \sup_{t \in \mathcal{T}} \frac{r^2}{r^2 + L(t)} \sum_{i=1}^{n} P\{X_i \in A_t\} \nonumber \\
            &= \sup_{t \in \mathcal{T}} \frac{r^2}{r^2 + L(t)} \cdot n L(t) \nonumber \\
            &\leq nr^2. \label{peel_var}
        \end{align}
        Since \(\frac{r^2}{r^2 + L(t)} \leq 1\), it follows that \(\left|\frac{r^2}{r^2 + L(t)} \left(\mathbbm{1}_{\{X_i \in A_t\}} - P\{X_i \in A_t\}\right)\right| \leq 1\). Furthermore, note 
        \begin{equation*}
            E\left(\frac{r^2}{r^2 + L(t)} \left(\mathbbm{1}_{\{X_i \in A_t\}} - P\{X_i \in A_t\}\right)\right) = 0.
        \end{equation*}
        Therefore, we can apply Theorem \ref{thm:bernstein_type} and Theorem \ref{thm:sup_var} to obtain 
        \begin{align*}
            &P\left\{ \sup_{t \in \mathcal{T}} \frac{r^2}{r^2 + L(t)} \frac{Z_t}{\sqrt{n}} \geq E\left(\sup_{t \in \mathcal{T}} \frac{r^2}{r^2 + L(t)}\frac{Z_t}{\sqrt{n}}\right) + u \right\} \\
            &\leq \exp\left(-c \min\left( \frac{n^2u^2}{\sqrt{n}E\left(\sup_{t \in \mathcal{T}} \frac{r^2}{r^2 + L(t)} Z_t\right) + \sup_{t \in \mathcal{T}} \Var\left( \frac{r^2}{r^2 + L(t)} \sqrt{n} Z_t\right) }, nu \right)\right)
        \end{align*}
        for \(u \geq 0\). The desired result then follows from (\ref{eqn:peel_expectation}) and (\ref{peel_var}), and repeating the argument for \(-Z_t\).
    \end{proof}
    
    \begin{theorem}\label{thm:peel}
        Let \(\mathcal{A} = \left\{A_t\right\}_{t \in \mathcal{T}}\) denote a VC class with VC dimension bounded by a universal constant. Let \(X_1,...,X_n\) be independent random variables. For any \(\delta, \lambda \in (0, 1)\), there exists a universal constant \(C' > 0\) and a constant \(\kappa_\delta > 0\) depending only on \(\delta\)  such that with probability at least \(1 - \delta\), we have 
        \begin{equation*}
            \frac{1}{n} \left|\sum_{i=1}^{n} (\mathbbm{1}_{\{X_i \in A_t\}} - P\{X_i \in A_t\})\right| \leq \left(\frac{C' \log(en)}{n\lambda} + \frac{\lambda}{n} \sum_{i=1}^{n} P\{X_i \in A_t\} \right)\left(1 + \frac{\kappa_\delta}{\sqrt{\log(en)}}\right)
        \end{equation*}
        uniformly over \(t \in \mathcal{T}\). 
    \end{theorem}
    \begin{proof}
        Let \(Z_t = \frac{1}{\sqrt{n}} \sum_{i = 1}^{n} \left(\mathbbm{1}_{\{X_i \in A_t\}} - P\{X_i \in A_t\}\right)\) for \(t \in \mathcal{T}\). Let \(L : \mathcal{T} \to [0, \infty)\) be the function \(L(t) = \frac{1}{n} \sum_{i=1}^{n} P\{X_i \in A_t\}\). Since \(V \asymp 1\), consider from Corollaries \ref{corollary:zeta_small_sigma} and \ref{corollary:zeta_big_sigma} we have 
        \begin{equation*}
            E\left(\sup_{\substack{t \in \mathcal{T}, \\ L(t) \leq r^2}} |Z_t| \right) \leq C\left(\frac{\log\left(\frac{e}{r}\right)}{\sqrt{n}} \vee r\sqrt{\log\left(\frac{e}{r}\right)}\right).
        \end{equation*}
        where \(C \geq 1\) is some universal constant. Consider that \(r\sqrt{\log\left(\frac{e}{r}\right)} \geq \frac{\log\left(\frac{e}{r}\right)}{\sqrt{n}}\) if and only if \(r^2 n \geq \log\left(\frac{e}{r}\right)\). Let \(r_* := \inf\left\{r \geq 0 : r^2 n \geq \log\left(\frac{e}{r}\right)\right\}\) and set 
        \begin{equation*}
            \psi(r) = 
            \begin{cases}
                C r \sqrt{\log\left(\frac{e}{r}\right)} &\textit{if } r \leq 1, \\
                C &\textit{if } r > 1.
            \end{cases}
        \end{equation*}
        Note \(\psi\) is indeed a sublinear function (note \(\psi(1) \geq 1\) is satisfied since \(C \geq 1\)). Further note that for \(r \geq r_*\), we have 
        \begin{equation*}
            E\left(\sup_{\substack{t \in \mathcal{T}, \\ L(t) \leq r^2}} |Z_t|\right) \leq \psi(r).
        \end{equation*}
        Consequently, Proposition \ref{prop:uniform_peel} gives for \(r \geq r_*\), with probability at least \(1 - \delta\), uniformly over \(t \in \mathcal{T}\) 
        \begin{align*}
            \frac{1}{n} \left|\sum_{i=1}^{n} \left(\mathbbm{1}_{\{X_i \in A_t\}} - P\{X_i \in A_t\}\right)\right| &\leq \left(r^2 + L(t)\right) \left(\frac{4\psi(r)}{r^2\sqrt{n}} + \kappa_\delta\left(\frac{1}{r^2n} + \sqrt{\frac{\psi(r)}{r^4n^{3/2}}} + \frac{1}{r\sqrt{n}}\right)\right) \\
            &\leq \left(r^2 + L(t)\right) \left(\frac{4C\sqrt{\log\left(\frac{e}{r}\right)}}{r\sqrt{n}} + \kappa_\delta\left(\frac{1}{r^2n} + \sqrt{\frac{\psi(r)}{r^4n^{3/2}}} + \frac{1}{r\sqrt{n}}\right)\right)
        \end{align*}
        Consider that \(r_*^2 \leq \frac{\log(en)}{n} \leq \frac{(4C)^2}{\lambda^2}\frac{\log(en)}{n}\). Selecting \(r^2 = \frac{(4C)^2}{\lambda^2}\frac{\log(en)}{n}\), we have
        \begin{align*}
            &\frac{1}{n} \left|\sum_{i=1}^{n} \left(\mathbbm{1}_{\{X_i \in A_t\}} - P\{X_i \in A_t\}\right)\right| \\
            &\leq \left(\frac{(4C)^2 \log(en)}{n\lambda^2} + L(t)\right) \left(\lambda + \kappa_\delta \left(\frac{1}{r^2n} + \sqrt{\frac{\psi(r)}{r^4n^{3/2}}} + \frac{1}{r\sqrt{n}}\right)\right) \\
            &\leq \left(\frac{(4C)^2 \log(en)}{n\lambda} + \lambda L(t)\right)\left(1 + \lambda^{-1}\kappa_\delta\left(\frac{1}{r^2n} + \sqrt{\frac{\psi(r)}{r^2\sqrt{n}} \cdot \frac{1}{r^2n}} + \frac{1}{r\sqrt{n}} \right)\right) \\
            &\leq  \left(\frac{(4C)^2 \log(en)}{n\lambda} + \lambda L(t)\right)\left(1 + \kappa_\delta\left(\frac{1}{\log(en)} + \frac{1}{\sqrt{\log(en)}} + \frac{1}{\sqrt{\log(en)}} \right)\right)
        \end{align*} 
        which holds uniformly over \(t \in \mathcal{T}\) with probability at least \(1-\delta\). Since \(\frac{1}{\sqrt{\log(en)}}\) dominates in the tail, we have the desired result. 
    \end{proof}
    
    \begin{corollary}\label{corollary:peel_tail}
        Consider the setup of Theorem \ref{thm:peel}. For any \(\delta, \lambda \in (0, 1)\), there exist universal constants \(C', C'' > 0\) such that with probability at least \(1-\delta\) we have 
        \begin{equation*}
            \frac{1}{n}\left|\sum_{i=1}^{n} \mathbbm{1}_{\{X_i \in A_t\}} - P\{X_i \in A_t\}\right| \leq \left(\frac{C' n^{1/8}}{n\lambda} + \frac{\lambda}{n}\sum_{i=1}^{n} P\{X_i \in A_t\} \right)\left(1 + \frac{C''\log\left(1/\delta\right)}{n^{1/16}}\right)
        \end{equation*}
        uniformly over \(t \in \mathcal{T}\). 
    \end{corollary}
    \begin{proof}
        The proof is the same as the proof of Theorem \ref{thm:peel}, with the slight modification of taking \(r^2 = \frac{(4C)^2}{\lambda^2} \cdot \frac{n^{1/8}}{n}\) as well as taking \(\kappa_\delta = C''\log\left(1/\delta\right)\). 
    \end{proof}

    \subsection{Tail probabilities}
    For the following two results, consider the following setup. Suppose \(X_i = \left\{X_{i,t}\right\}_{t \in \mathcal{T}}\) is a collection of random variables for \(1 \leq i \leq n\). Assume \(X_1,...,X_n\) are mutually independent. Define 
    \begin{align*}
        \Sigma^2 &:= E\left(\sup_{t \in \mathcal{T}} \sum_{i=1}^{n} X_{i,t}^2\right), \\
        \rho^2 &:= \sup_{t \in \mathcal{T}} \sum_{i=1}^{n} E(X_{i,t}^2).
    \end{align*}
    
    \begin{theorem}[Bernstein-type - Theorem 12.2 \cite{boucheron_concentration_2013}]\label{thm:bernstein_type}
        Assume \(E(X_{i,t}) = 0\) and \(|X_{i,t}| \leq 1\) for all \(t \in \mathcal{T}\) and \(1 \leq i \leq n\). Let \(Z = \sup_{t \in \mathcal{T}} \sum_{i=1}^{n} X_{i, t}\). If \(u \geq 0\), then
        \begin{equation*}
            P\left\{Z \geq E(Z) + u \right\} \leq \exp\left(-\frac{u^2}{2(2(\Sigma^2 + \rho^2) + u)}\right).
        \end{equation*}
    \end{theorem}
    
    \begin{theorem}[Theorem 11.8 \cite{boucheron_concentration_2013}]\label{thm:sup_var}
        Assume \(E(X_{i,t}) = 0\) and \(|X_{i,t}| \leq 1\) for all \(t \in \mathcal{T}\) and \(1 \leq i \leq n\). Let \(Z = \sup_{t \in \mathcal{T}} \sum_{i=1}^{n} X_{i, t}\). Then 
        \begin{equation*}
            \Sigma^2 \leq 8E(Z) + \rho^2. 
        \end{equation*}
    \end{theorem}
    
    \begin{theorem}[Bernstein's inequality - Theorem 2.8.4 \cite{vershynin_high-dimensional_2018}]\label{thm:bernstein_bounded}
        Let \(Y_1,...,Y_n\) be independent mean zero random variables such that \(|Y_i| \leq 1\) for all \(i\). If \(u \geq 0\), then 
        \begin{equation*}
            P\left\{ \left|\sum_{i=1}^{n} Y_i\right| \geq u \right\} \leq 2 \exp\left(-\frac{u^2/2}{\tau^2 + u/3}\right)
        \end{equation*}
        where \(\tau^2 = \sum_{i=1}^{n} E(Y_i^2)\). 
    \end{theorem}
    
    \begin{theorem}[Bounded differences - Theorem 6.2 \cite{boucheron_concentration_2013}]\label{thm:bounded_differences}
        Suppose \(f : \mathcal{X}^n \to \R\) satisfies the bounded differences inequality for some nonnegative \(d_1,...,d_n\), that is 
        \begin{equation*}
            \sup_{\substack{x_1,...,x_n \in \mathcal{X}, \\ x_i' \in \mathcal{X}}} |f(x_1,...,x_{i-1},x_i,x_{i+1}...,x_n) - f(x_1,...,x_{i-1},x_{i}',x_{i+1},...,x_n)| \leq d_i
        \end{equation*}
        for all \(1 \leq i \leq n\). Let \(Z = f(X_1,...,X_n)\) where \(X_1,...,X_n\) are independent \(\mathcal{X}\)-valued random variables. If \(u \geq 0\), then 
        \begin{equation*}
            P\left\{|Z - E(Z)| > u\right\} \leq 2 \exp\left(-\frac{2u^2}{\sum_{i=1}^{n} d_i^2}\right).
        \end{equation*}
    \end{theorem}
    
    \subsection{Symmetrization}
    Suppose we have a stochastic process \(\{X_t\}_{t \in \mathcal{T}}\) given by 
    \begin{equation*}
        X_t = \frac{1}{\sqrt{n}} \sum_{i=1}^{n} (f_t(Y_i) - E(f_t(Y_i)))
    \end{equation*}
    where \(Y_1,...,Y_n\) are independent random variables and \(\{f_t\}_{t\in \mathcal{T}}\) is a collection of real-valued functions. Consider the symmetrized process, that is, draw \(\varepsilon_1,...,\varepsilon_n \overset{iid}{\sim} \Rademacher\left(\frac{1}{2}\right)\) independently of \(\{Y_i\}_{i=1}^{n}\) and consider 
    \begin{equation*}
        \xi_t := \frac{1}{\sqrt{n}} \sum_{i=1}^{n} \varepsilon_i f_t(Y_i).
    \end{equation*}
    The symmetrized process \(\{\xi_t\}_{t\in \mathcal{T}}\) is closely related to the original process \(\{X_t\}_{t \in \mathcal{T}}\). In particular, the following lemma is well-known. 
    
    \begin{lemma}[Symmetrization]\label{lemma:symmetrization}
        We have \(E\left(\sup_{t \in \mathcal{T}} X_t\right) \leq 2 E\left(\sup_{t \in \mathcal{T}} \xi_t\right)\) and \(E\left(\sup_{t \in \mathcal{T}} |X_t|\right) \leq 2 E\left(\sup_{t \in \mathcal{T}} |\xi_t|\right)\). Furthermore, for any \(t, t' \in \mathcal{T}\) and \(\lambda > 0\), we have 
        \begin{equation*}
            E\left(\left. e^{\lambda(\xi_t - \xi_{t'})} \right| \left\{Y_i\right\}_{i=1}^{n}\right) \leq e^{\frac{\lambda^2 d_Y^2(t, t')}{2}}
        \end{equation*}
        where \(d_Y(t, t') = \sqrt{\frac{1}{n} \sum_{i=1}^{n} (f_t(Y_i) - f_{t'}(Y_i))^2}\). 
    \end{lemma}
    
    \end{appendix}
\end{document}